\documentclass[a4paper,11pt]{article}
\usepackage{bbm}
\usepackage{mathrsfs}
\usepackage{amsfonts}
\usepackage{amssymb}
\usepackage{tabularx}
\usepackage{enumerate}
\usepackage{graphicx}
\usepackage{texdraw}
\usepackage{tikz}
\usepackage{cite}
\usepackage{hyperref}
\usepackage{amsthm}
\usepackage[top=1.38in, bottom=1.39in, left=1.18in, right=1.17in]{geometry}
\newtheorem{theo}{Theorem}[section]
\newtheorem{prop}[theo]{Proposition}
\newtheorem{lem}[theo]{Lemma}
\newtheorem{rem}[theo]{Remark}
\newtheorem{col}[theo]{Corollary}
\newtheorem{defi}[theo]{Definition}
\newtheorem{que}[theo]{Question}
\newtheorem{claim}[theo]{Claim}
\newtheorem{exa}[theo]{Example}
\newcommand{\be}{\begin{equation}}
\newcommand{\ee}{\end{equation}}
\newcommand\bes{\begin{eqnarray}}
\newcommand\ees{\end{eqnarray}}
\newcommand{\bess}{\begin{eqnarray*}}
\newcommand{\eess}{\end{eqnarray*}}

\usepackage[toc,page,title,titletoc]{appendix}
\def\theequation{\arabic{section}.\arabic{equation}}

\begin{document}
 \setlength{\baselineskip}{15pt} \pagestyle{myheadings}

 \title{On Tamed Almost Complex Four-Manifolds\thanks{The work
 is supported by PRC grant NSFC 11701226 (Tan), 11371309, 11771377 (Wang), 11426195 (Zhou), 11471145 (Zhu);
  Natural Science Foundation of Jiangsu Province BK20170519 (Tan); University Scinence Research Project of Jiangsu Province 15KJB110024 (Zhou);
 Foundation of Yangzhou University 2015CXJ003 (Zhou). }}

 \author{Qiang Tan, Hongyu Wang\thanks{E-mail:hywang@yzu.edu.cn}, Jiuru Zhou \,and Peng Zhu\\
  {\small Dedicated to professor Kung-Ching Chang on the occasion of his 85th birthday}}

 \date{}
 \maketitle

 \noindent {\bf Abstract:} {
 This paper proves that on any tamed closed almost complex four-manifold $(M,J)$ whose dimension of $J$-anti-invariant cohomology is equal to the self-dual second Betti number minus one, there exists a new symplectic form compatible with the given almost complex structure $J$. In particular,
 if the self-dual second Betti number is one,
 we give an affirmative answer to a question of Donaldson for tamed closed almost complex four-manifolds.
  Our approach is along the lines used by Buchdahl to give a unified proof of the Kodaira conjecture.
 }\\

 \noindent {\bf Keywords:} $\omega$-tame(compatible) almost complex structure; $J$-anti-invariant cohomology; positive $(1,1)$ current; local symplectic property; $J$-holomorphic curve.\\

 \noindent{\bf 2010 Mathematics Subject Classification:} 53D35, 53C56, 53C65

 \section{Introduction}\label{1}
  Suppose that $M$ is a closed, oriented, smooth $4$-manifold and suppose that $\omega$ is a symplectic form on $M$ that is compatible with the orientation.
  An endomorphism, $J$, of $TM$ is said to be an almost complex structure when $J^2=-id_{TM}$.
  Such an almost complex structure is said to be tamed by $\omega$ when the bilinear form $\omega(\cdot,J\cdot)$ is positive definite.
  The almost complex structure $J$ is said to be compatible (or calibrate) with $\omega$ when this same bilinear form is also symmetric,
  that is, $\omega(\cdot,J\cdot)>0$ and $\omega(J\cdot,J\cdot)=\omega(\cdot,\cdot)$.
  M. Gromov \cite{G3} observed that tamed almost complex structures and also compatible almost complex structures always exist.
  Let $\mathcal{J}(M)$ be the space of all almost complex structures on $M$,
   $\mathcal{J}_c(M,\omega)$ the space of all $\omega$-compatible
   almost complex structures on $M$ and $\mathcal{J}_\tau(M,\omega)$ the space of all $\omega$-tame almost complex structures on $M$.
   Note that $\mathcal{J}_c(M,\omega)$ and $\mathcal{J}_\tau(M,\omega)$ are even contractible,
   and $\mathcal{J}_\tau(M,\omega)$ is open in the space $\mathcal{J}(M)$ (This is defined using the $C^\infty$-Fr\'{e}chet space topology (cf. \cite{A2})).
   S. K. Donaldson \cite{D6} posed the following question:
   If an almost complex structure is tamed by a given symplectic form $\omega$, must it be compatible with a new symplectic form?
  That is, which tamed almost complex 4-manifolds can be calibrated?
  This is a natural question to arise in the context of calibrated geometries \cite{HL1,HL4,HL5}.
  Since any almost complex 4-manifold $(M,J)$ has the local symplectic property \cite{L2,Sikorav}, that is, for any $p\in M$,
  there exists a $J$-compatible symplectic $2$-form $\omega_p$ on a neighborhood $U_p$ of $p$ which can be viewed as a calibration on $U_p$ \cite{HL1,HL4,HL5}.

  Note that there are topological obstructions to the existence of almost complex structures on an even dimensional manifold.
  For a closed 4-manifold, a necessary condition is that $1-b^1+b^+$ be even \cite{BHPV},
   where $b^1$ is the first Betti number and $b^+$ is the number of positive eigenvalues of
   the quadratic form on $H^2(M;\mathbb{R})$ defined by the cup product, hence the condition is either $b^1$ be even and $b^+$ odd,
    or $b^1$ be odd and $b^+$ even.
  It is a well-known fact (that is the Kodaira conjecture \cite{KM}) that any closed complex surface with $b^1$ even is K\"{a}hler.
  The direct proofs have been given by N. Buchdahl \cite{B} and A. Lamari \cite{L1}.
  R. Harvey and H. B. Lawson, Jr. (Theorems 26 and 38 in \cite{HL2}) proved that for any closed complex surface $(M,J)$ with $b^1$ even,
   there exists a symplectic form $\omega$ on $M$ by which $J$ is tamed.
  Thus, Donaldson's question for tamed almost complex 4-manifolds (in particular, $b^+=1$) is related to the
   Kodaira conjecture for complex surfaces (cf. \cite{DLZ1}).

  When $M=\mathbb{C}P^2$ for every tamed almost complex structure $J$, there exists a symplectic form $\Omega$ on $\mathbb{C}P^2$ with which $J$ is compatible.
  It follows from M. Gromov's result \cite{G3} on pseudoholomorphic curves and C. H. Taubes' result \cite{T1} on symplectic forms on $\mathbb{C}P^2$.

  Donaldson suggests in \cite{D6} an approach to his question, one along the lines used by S.-T. Yau in \cite{Y} to prove the Calabi conjecture.
  This approach is considered by V. Tosatti, B. Weinkove, and S.-T. Yau in \cite{TWY,W}.

  Taubes considered in \cite{T2} Donaldson's question as follows:
  Fix a closed almost complex 4-manifold $M$ with $b^+=1$ and with a given symplectic form $\omega$.
  He proves in \cite{T2} the following: The Fr\'{e}chet space, $\mathcal{J}_\tau(M,\omega)$,
     of tamed almost complex structures as defined by $\omega$ has an open and dense subset whose almost
       complex structures are compatible with a new symplectic form that is cohomologous to $\omega$.

  Very recently, T.-J. Li and W. Zhang \cite{LZ2} studied Nakai-Moishezon type question and Donaldson's ``tamed to compatible" question
  for almost complex structures on rational $4$-manifolds.
  By extending Taubes' subvarieties-current-form technique to J-nef genus $0$ classes, they gave affirmative answers of these two questions for all
  tamed almost complex structures on $S^2$ bundles over $S^2$ as well as for many geometrically interesting
   tamed almost complex structures on other rational four manifolds.

  \vskip 6pt

  For a closed almost complex 4-manifold $(M,J)$, T.-J. Li and W. Zhang \cite{LZ} introduced subgroups $H^+_J$ and $H^-_J$,
  of the real degree 2 de Rham cohomology group $H^2(M;\mathbb{R})$,
  as the sets of cohomology classes which can be represented by $J$-invariant and $J$-anti-invariant real 2-forms.
  Let us denote by $h^+_J$ and $h^-_J$ the dimensions of $H^+_J$ and $H^-_J$, respectively.
  T. Draghici, T.-J. Li and W. Zhang \cite{DLZ1} proved that for a closed almost complex 4-manifold $(M,J)$, $$H^2(M;\mathbb{R})=H^+_J\oplus H_J^-.$$
  If $J$ is integrable, the induced decomposition is nothing but the classical real Hodge-Dolbeault
  decomposition of $H^2(M;\mathbb{R})$ (cf. \cite{BHPV,DLZ1}), that is,
  $$
  H^+_J=H^{1,1}_{\bar{\partial}}\cap H^2(M;\mathbb{R})\,\,\, and \,\,\,H^-_J=(H^{2,0}_{\bar{\partial}}\oplus H^{0,2}_{\bar{\partial}})\cap H^2(M;\mathbb{R}).
  $$

  In this paper, we give an affirmative answer to Donaldson's question when $h^-_J=b^+-1$ by using very different approach.
  In particular, if the self-dual second Betti number is one, we give an affirmative answer to the conjecture of Tosatti, Weinkove and Yau \cite{TWY}.
   Our approach is along the lines used by Buchdahl in \cite{B} to give a unified proof of the Kodaira conjecture.
  \begin{theo}\label{1t1}
  Let $M$ be a closed symplectic 4-manifold with symplectic form $\omega$.
  Suppose that $J$ is an $\omega$-tame almost complex structure on $M$ and $h^-_J=b^+-1$.
  Then there exists a new symplectic form $\Omega$ that is compatible with $J$.
  \end{theo}

  \begin{rem}\label{1r2}
  If $(M,J)$ is a closed complex surface with $b^1$ even,
   then there exists a symplectic form $\omega$ by which $J$ is tamed (see Theorem {\rm 26} and {\rm 38} in {\rm \cite{HL2})} and $h^-_J=b^+-1$.
   Thus, the above theorem gives an affirmative answer to the Kodaira conjecture in symplectic version.
  \end{rem}

  Note that if $(M,J)$ is a tamed, closed almost complex 4-manifold,
   then it is easy to see that $0\leq h^-_J\leq b^+-1$ (cf. \cite{Wang,TWZZ}),
  thus $h^-_J=b^+-1$ is a technical condition.
  Hence if $b^+=1$, then $h^-_J=b^+-1=0$.
  As a direct consequence of Theorem \ref{1t1}, we have the following corollary which gives an affirmative answer to Conjecture $1.2$
  in \cite{TWY} (see also the description in \cite{W}).

  \begin{col}\label{1c3}
  Let $(M,J)$ be a tamed, closed, almost complex 4-manifold with a taming form $\omega$.
  When $b^+=1$, then exists a new symplectic form $\Omega$ that is compatible with almost complex structure $J$ and cohomologous to $\omega$.
  \end{col}

  We have shown that generically $h^-_J=0$ (cf. \cite{TWZZ,TWZ}).
  So when $b^+>1$ the hypothesis of Theorem \ref{1t1} can at best be satisfied by very special almost complex structures (for example, $J$ is integrable).
  Hence, it is natural to ask the following question,

  \begin{que}
  (1) Which is the sufficient and necessary condition for Donaldson's ``tamed to compatible" question?

  (2) Is it possible to construct a closed symplectic $4$-manifold $(M,\omega)$ with $b^+>1$ such that for any $\omega$-compatible almost complex structure $J$,
     $h^-_J$ is strictly less than $b^+-1$?
   \end{que}

   The remainder of the paper is organized as follows:

   \vskip 6pt

   {\bf Section 2: Preliminaries.} In this section, it is similar to $\partial\bar{\partial}$ operator in classical complex analysis,
     we introduce the operators $\widetilde{\mathcal{D}}^+_J$ and $\mathcal{D}^+_J$ on tamed almost complex $4$-manifolds.

   \vskip 6pt

  {\bf Section 3: The intersection pairing on weakly $\widetilde{\mathcal{D}}^+_J$-closed (1,1)-forms.}
   In this section, as done in complex surfaces, we give the notion of weakly $\widetilde{\mathcal{D}}^+_J$-closed $(1,1)$-form which is similar to the weakly
   $\partial\bar{\partial}$-closed $(1,1)$-form in classical complex analysis.
   We investigate the intersection pairing on weakly $\widetilde{\mathcal{D}}^+_J$-closed (1,1)-forms,
   and obtain a key lemma (Lemma \ref{3l13}) as done in compact complex surfaces.

  \vskip 6pt

   {\bf Section 4: The tamed almost complex $4$-manifolds with $h_J^{-}=b^{+}-1$.}
   In this section, based on the key lemma proved in Section $3$,
   we give a proof of our main theorem which follows mainly Buchdahl's proof of the fact that compact complex surfaces with $b_1$ even is K\"{a}hler.

   \vskip 6pt

   To prove the main result, we extend several notions and important theorems from complex analysis to the almost complex setting
   which are necessary for the proof of the main theorem.
   Many of them are interesting by themselves.
   The rest of this paper contains three appendices:

    \vskip 6pt

  {\bf Appendix A: Elementary pluripotential theory}

    A.1: $J$-plurisubharmonic functions on almost complex manifolds.

    A.2: Kiselman's minimal principle for $J$-plurisubharmonic functions.

    A.3: H\"ormander's $L^2$-estimates on tamed almost complex $4$-manifolds.

    A.4: The singularities of $J$-plurisubharmonic functions on tamed almost complex $4$-manifolds.

     \vskip 6pt

     {\bf Appendix B: Siu's decomposition theorem on tamed almost complex $4$-manifolds }

    B.1: Lelong numbers of closed positive $(1,1)$-currents on tamed complex $4$-manifolds.

    B.2: Siu's decomposition formula of closed positive $(1,1)$-currents on tamed almost complex $4$-manifolds.

    \vskip 6pt

     {\bf Appendix C: Demailly's approximation theorem on tamed almost complex $4$-manifolds}

   C.1: Exponential map associated to the second canonical connection.

   C.2: Regularization of quasi-$J$-plurisubharmonic functions on tamed almost Hermitian $4$-manifolds.

   C.3: Regularization of closed positive $(1,1)$-currents on tamed almost complex $4$-manifolds.

   C.4: Demailly's approximation theorem on tamed almost complex $4$-manifolds.

 \section{Preliminaries}\label{2}
 \setcounter{equation}{0}
   Suppose that $M$ is  an almost complex manifold with almost complex structure $J$,
   then for any $x\in M$, $T_x(M)\otimes_\mathbb{R}\mathbb{C}$ which is the complexification of $T_x(M)$,
   has the following decomposition (cf. \cite{A2,KN,LZ}):
   \begin{equation}\label{2eq1}
  T_x(M)\otimes_\mathbb{R}\mathbb{C}=T^{1,0}_x+T^{0,1}_x,
  \end{equation}
  where $T^{1,0}_x$ and $T^{0,1}_x$ are the eigenspaces of $J$ corresponding to the eigenvalues $\sqrt{-1}$ and $-\sqrt{-1}$, respectively.
   A complex tangent vector is of type $(1,0)$ (resp. $(0,1)$) if it belongs to $T^{1,0}_x$ (resp. $T^{0,1}_x$).
   Let $T(M)\otimes_\mathbb{R}\mathbb{C}$ be the complexification of the tangent bundle.
   Similarly, let $T^*M\otimes_{\mathbb{R}}\mathbb{C}$ denote the complexification of the cotangent bundle $T^*M$.
   $J$ can act on $T^*M\otimes_{\mathbb{R}}\mathbb{C}$ as follows:
   $$\forall \alpha\in T^*M\otimes_{\mathbb{R}}\mathbb{C},\,\,\, J\alpha(\cdot)=-\alpha(J\cdot).$$
   Hence $T^*M\otimes_{\mathbb{R}}\mathbb{C}$ has the following decomposition according to the eigenvalues $\mp\sqrt{-1}$:
   \begin{equation}\label{2eq2}
   T^*M\otimes_{\mathbb{R}}\mathbb{C}=\Lambda^{1,0}_J\oplus\Lambda^{0,1}_J.
   \end{equation}
   We can form exterior bundle $\Lambda^{p,q}_J=\Lambda^p\Lambda^{1,0}_J\otimes\Lambda^q\Lambda^{0,1}_J$.
   Let $\Omega^{p,q}_J(M)$ denote the space of $C^\infty$ sections of the bundle $\Lambda^{p,q}_J$.
   The exterior differential operator acts on $\Omega^{p,q}_J$ as follows:
   \begin{equation}\label{2eq3}
   d\Omega^{p,q}_J\subset\Omega^{p-1,q+2}_J+\Omega^{p+1,q}_J+\Omega^{p,q+1}_J+\Omega^{p+2,q-1}_J.
   \end{equation}
  Hence, $d$ has the following decomposition:
  \begin{equation}\label{2eq4}
  d=A_J\oplus\partial_J\oplus\bar{\partial}_J\oplus\bar{A}_J.
  \end{equation}
  Recall that on an almost complex manifold $(M,J)$, there exists the Nijenhuis tensor $\mathcal{N}_J$
   as follows:
  \begin{equation}\label{2eq5}
  4\mathcal{N}_J=[JX,JY]-[X,Y]-J[X,JY]-J[JX,Y],
  \end{equation}
   where $X, Y\in TM$.
   By the Newlander-Nirenberg Theorem \cite{A2}, $\mathcal{N}_J=0$ if and only if $J$ is integrable, that is, $J$ is a complex structure.
   If $J$ is integrable, then $d=\partial_J\oplus\bar{\partial}_J$ (For details, see \cite{A2,KN,LZ}).
   By a direct calculation, we have: For any $\alpha\in(\Omega^{p,q}_J+\Omega^{q,p}_J)_\mathbb{R}\subset\Omega^{p+q}_\mathbb{R}$,
  \begin{equation}\label{2eq6}
  (A_J+\bar{A}_J)(\alpha)(X_1,...,X_{p+q+1})=\sum_{i<j} (-1)^{i+j+1} \alpha(\mathcal{N}_J(X_i,X_j),X_1,...,\hat{X_i},...,\hat{X_j},...,X_{p+q+1}),
  \end{equation}
   where $X_1,...,X_{p+q+1}\in T(M)$ (cf. \cite{KN,TWY,WZ2}).

   \vskip 6pt

   Let $(M,J)$ be an almost complex 4-manifold. After a simple calculation, we can get the following properties:
  \begin{equation}\label{2eq7}
   d:\Omega^0_\mathbb{R}\longrightarrow\Omega^1_\mathbb{R},\,\,\, d=\partial_J+\bar{\partial}_J.
   \end{equation}
   \begin{equation}\label{2eq8}
     A_J\circ\partial_J+\bar{\partial}^2_J+\bar{A}_J\circ\bar{\partial}_J+\partial^2_J=0
     : \Omega^0_\mathbb{R}\longrightarrow(\Omega^{2,0}_J+\Omega^{0,2}_J)_\mathbb{R}.
   \end{equation}
  \begin{equation}\label{2eq9}
  \partial_J\circ\bar{\partial}_J+\bar{\partial}_J\circ\partial_J=0: \Omega^0_\mathbb{R}\longrightarrow\Omega^{1,1}_\mathbb{R}.
  \end{equation}
   \begin{equation}\label{2eq10}
   d:\Omega^1_\mathbb{R}\longrightarrow\Omega^2_\mathbb{R},\,\,\, d=A_J+\partial_J+\bar{\partial}_J+\bar{A}_J.
   \end{equation}
  \begin{equation}\label{2eq11}
   d:(\Omega^{2,0}+\Omega^{0,2})_\mathbb{R}\longrightarrow(\Omega^{1,2}+\Omega^{2,1})_\mathbb{R}, \,\,\,d=A_J+\partial_J+\bar{\partial}_J+\bar{A}_J.
   \end{equation}
   \begin{equation}\label{2eq12}
   d: \Omega^{1,1}_\mathbb{R}\longrightarrow(\Omega^{1,2}+\Omega^{2,1})_\mathbb{R},\,\,\,d=\partial_J+\bar{\partial}_J.
   \end{equation}
   \begin{equation}\label{2eq13}
   \partial_J\circ\bar{\partial}_J+\bar{\partial}_J\circ\partial_J=0: \Omega^{1,1}_\mathbb{R}\longrightarrow\Omega^4_\mathbb{R}.
   \end{equation}

  Suppose that $(M,J)$ is an almost complex 4-manifold.
  One can construct a $J$-invariant Riemannian metric $g$ on $M$, namely, $g(JX,JY)=g(X,Y)$ for all tangent vector fields $X$ and $Y$ on $M$.
  Such a metric $g$ is called an almost Hermitian metric (real) on $(M,J)$.
   This then in turn gives a $J$-compatible nondegenerate 2-form $F$ on $M$ by $F(X,Y)=g(JX,Y)$, called the fundamental 2-form.
  Such a quadruple $(M,g,J,F)$ is called an almost Hermitian 4-manifold.
  Thus an almost Hermitian structure on $M$ is a triple $(g,J,F)$.
  If $J$ is integrable, the triple $(g,J,F)$ is called an Hermitian structure
  (In complex coordinate system, the almost Hermitian metric is written as $h=g-\sqrt{-1}F$.).
   By using almost Hermitian structure $(g,J,F)$, we can define a volume form $d\mu_g=F^2/2$ with $$\int_Md\mu_g=1$$
   by rescaling in the conformal equivalent class $[g]$.
  If the 2-form $F$ is closed, then the triple $(g,J,F)$ is called an almost K\"{a}hler structure.
  When the two conditions hold simultaneously, the $(g,J,F)$ defines a K\"{a}hler structure on $M$ (cf. \cite{A2,KN}).
  Note that although $M$ need not admit a symplectic condition (i.e. $dF=0$),
    P. Gauduchon \cite{G1} has shown that  for a closed almost Hermitian 4-manifold $(M,g,J,F)$ there is a conformal rescaling of the metric $g$,
     unique up to positive constant, such that the associated form satisfies $\bar{\partial}_J\partial_J F=0$.
     This metric is called Gauduchon metric.

   \vskip 6pt

   Let $\Omega^2_\mathbb{R}(M)$ denote the space of real smooth 2-forms on $M$, that is,
  the real $C^\infty$ sections of the bundle $\Lambda^2_\mathbb{R}(M)$.
   The almost complex structure $J$ acts on $\Omega^2_\mathbb{R}(M)$ as an involution by $\alpha(\cdot,\cdot)\mapsto\alpha(J\cdot,J\cdot)$,
   thus we have the splitting into $J$-invariant and $J$-anti-invariant 2-forms respectively
  \begin{equation}\label{2eq14}
  \Lambda^2_\mathbb{R}=\Lambda^+_J\oplus\Lambda^-_J,
  \end{equation}
   where the bundles $\Lambda^{\pm}_J$ are defined by
  $$
  \Lambda^{\pm}_J=\{\alpha\in\Lambda^2_\mathbb{R}\mid \alpha(J\cdot,J\cdot)=\pm\alpha(\cdot,\cdot)\}.
  $$
  We will denote by $\Omega^+_J$ and $\Omega^-_J$, respectively,
   the $C^\infty$ sections of the bundles $\Lambda^+_J$ and $\Lambda^-_J$.
  For $\alpha\in\Omega^2_\mathbb{R}(M)$, denote by $\alpha^+_J$ and $\alpha^-_J$,
   respectively, the $J$-invariant and $J$-anti-invariant components of $\alpha$ with respect to the decomposition (\ref{2eq14}).
   We will also use the notation $\mathcal{Z}^2_\mathbb{R}$ for the space of real closed 2-forms on $M$
   and $\mathcal{Z}^\pm_J=\mathcal{Z}^2_\mathbb{R}\cap\Omega^\pm_J$ for the corresponding projections.

   Li and Zhang have defined in \cite{LZ} the $J$-invariant and $J$-anti-invariant cohomology subgroups $H^\pm_J$ of $H^2(M;\mathbb{R})$ as follows:
  $$H^\pm_J=\{\mathfrak{a}\in H^2(M;\mathbb{R})\mid \exists \alpha\in\mathcal{Z}^\pm_J \,\,\,{\rm such} \,\,\, {\rm that} \,\,\, [\alpha]=\mathfrak{a}\};$$
  $J$ is said to be $C^\infty$-pure if $H^+_J\cap H^-_J=\{0\}$, $C^\infty$-full if $H^+_J+H^-_J=H^2(M;\mathbb{R})$.
  $J$ is $C^\infty$-pure and full if and only if $H^2(M;\mathbb{R})=H^+_J\oplus H^-_J$.
   \begin{prop}\label{2p1}
  (Theorem {\rm 2.2} in {\rm\cite{DLZ1}}) If $M$ is a  closed almost complex 4-manifold $(M,J)$,
   then the almost complex structure $J$ on $M$ is $C^\infty$-pure and full.
   Thus, there is a direct sum cohomology decomposition
   $$H^2(M;\mathbb{R})=H^+_J\oplus H^-_J.$$
   Let us denote by $h^+_J$ and $h^-_J$ the dimensions of $H^+_J$ and $H^-_J$, respectively.
   Then we have $b^2=h^+_J+h^-_J$, where $b^2$ is the second Betti number.
   \end{prop}

  When $J$ is integrable, there is the Dolbeault decomposition which has long been discovered.

  \begin{rem}\label{2r2}
  (cf. {\rm\cite{BHPV,DLZ1}}) If $J$ is integrable on a closed 4-manifold, then
  $$H^+_J=H^{1,1}_{\bar{\partial}_J}\cap H^2(M;\mathbb{R})\,\,;
   \,\,\, H^-_J=(H^{2,0}_{\bar{\partial}_J}\oplus H^{0,2}_{\bar{\partial}_J})\cap H^2(M;\mathbb{R}).$$
   Let us denote the dimension of $H^{p,q}_{\bar{\partial}_J}$ by $h^{p,q}_{\bar{\partial}_J}$.
   So if $J$ is integrable,
    it follows from the above proposition that $h^+_J=h^{1,1}_{\bar{\partial}_J} ,\,\,\,\,\, h^-_J=2h^{2,0}_{\bar{\partial}_J}$.
   So in this case, using the signature theorem we get
   $$
   h^{+}_{J}=\left\{
    \begin{array}{ll}
   b^{-}+1 & \textrm{if}~ b_1 ~\textrm{even} \\
   b^- & \textrm{if}~ b_1 ~\textrm{odd},
  \end{array}
   \right.
   ~h^{-}_{J}=\left\{
  \begin{array}{ll}
   b^{+}-1 & \textrm{if}~ b_1 ~\textrm{even} \\
   b^+ & \textrm{if} ~b_1~\textrm{odd}.
  \end{array}
  \right.
  $$
  \end{rem}

   Since $(M,g,J,F)$ is a closed almost Hermitian 4-manifold, the Hodge star operator $*_g$ gives the self-dual,
   anti-self-dual decomposition of the bundle of 2-forms (see \cite{D6,DK}):
  \begin{equation}\label{2eq15}
  \Lambda^2_\mathbb{R}=\Lambda^+_g\oplus\Lambda^-_g.
  \end{equation}
   We denote by $\Omega^\pm_g$ the spaces of smooth sections of $\Lambda^\pm_g$,
  and by $\alpha^+_g$ and $\alpha^-_g$ respectively the self-dual and anti-self-dual components of a 2-form $\alpha$.
   Since the Hodge-de Rham Laplacian $\Delta_g=dd^*+d^*d$, where $d^*=-*_gd*_g$ is the codifferential operator with respect to the metric $g$,
    commutes with $*_g$, the decomposition (\ref{2eq15}) holds for the space $\mathcal{H}_g$ of harmonic 2-forms as well.
    By Hodge theory, this induces cohomology decomposition by the metric $g$:
  $$
    \mathcal{H}_g=\mathcal{H}_g^+\oplus\mathcal{H}_g^-.
  $$
  Suppose $\alpha\in \Omega^+_g$ and its Hodge decomposition \cite{D6,DK} is:
  $$
    \alpha=\alpha_h+d\theta+d^*\psi=\alpha_h+d\theta+*_gd\varphi,
  $$
   where $\alpha_h$ is a harmonic 2-form and $\varphi=-*_g\psi$ .
   Then, since $*_g\alpha=\alpha$, the uniqueness of the Hodge decomposition gives that $\theta=\varphi$,
    and $\alpha_h=*_g\alpha_h$, so $\alpha=\alpha_h+d^+_g(2\theta)$, where
  $$
  d^\pm_g:\Omega^1_\mathbb{R}\rightarrow\Omega^\pm_g
  $$
  is the first-order differential operator formed from the composite of the exterior derivative
   $d: \Omega^1_\mathbb{R}\rightarrow\Omega^2_\mathbb{R}$ with the algebraic projections $P^\pm_g=\frac{1}{2}(1\pm*_g)$
   from $\Omega^2_\mathbb{R}$ to $\Omega^\pm_g$, where $d_g^\pm=P^\pm_gd$.
   So we can get the following Hodge decompositions (see \cite{DK}):
   \begin{equation}\label{2eq16}
   \Omega^+_g=\mathcal{H}^+_g\oplus d^+_g(\Omega^1), \,\,\, \Omega^-_g=\mathcal{H}^-_g\oplus d^-_g(\Omega^1).
  \end{equation}
    Note that
   \begin{equation}\label{2eq16'}
  d^\pm_gd^*:\Omega^\pm_g\rightarrow\Omega^\pm_g
   \end{equation}
  are self-adjoint strongly elliptic operators and $\ker d^\pm_gd^*=\mathcal{H}^\pm_g$.
  If $d^+_gu$ is $d$-closed, that is, $dd^+_gu=0$, then
  $$
   0=\int_Mdd^+_gu\wedge u=-\int_Md^+_gu\wedge du=-\int_M|d^+_gu|^2,
  $$
  so $d^+_gu=0$. Similarly, for any $u\in \Omega^1_\mathbb{R}$, if $d^+_gu=0$,
    \begin{equation}\label{2eq17}
   0=\int_Mdu\wedge du=\int_M|d^+_gu|^2-\int_M|d^-_gu|^2=-\int_M|d^-_gu|^2,
  \end{equation}
  so $d^-_gu=0$ too, therefore we can get $du=0$ (cf. \cite{D6,DK}).

   We define,
  $$
    H^\pm_g=\{ \mathfrak{a}\in H^2(M;\mathbb{R})\mid\exists\alpha\in\mathcal{Z}^\pm_g:=\mathcal{Z}_{\mathbb{R}}^2\cap\Omega^\pm_g
     \,\,\, {\rm such} \,\,\, {\rm that} \,\,\, \mathfrak{a}=[\alpha]\}.
   $$
    There are the following relations between the decompositions (\ref{2eq14}) and (\ref{2eq15}) on an almost Hermitian 4-manifold:
   \begin{equation}\label{2eq18}
  \Lambda^+_J=\mathbb{R}\cdot F\oplus\Lambda^-_g,\,\,\,
  \Lambda^+_g=\mathbb{R}\cdot F\oplus\Lambda^-_J,
  \end{equation}
  \begin{equation}\label{2eq19}
  \Lambda^+_J\cap\Lambda^+_g=\mathbb{R}\cdot F, \,\,\, \Lambda^-_J\cap\Lambda^-_g=\{0\}.
  \end{equation}
  It is easy to see that $H^-_J\subset H^+_g$ and $H^-_g\subset H^+_J$ (cf. \cite{DLZ2,TWZZ}).

   \vskip 6pt

   Let $b^+$ the self-dual Betti number, and $b^-$ the anti-self-dual Betti number of $M$, hence $b^2=b^++b^-$.
   Thus, for a closed almost Hermitian 4-manifold $(M,g,J,F)$, we have (cf. \cite{TWZZ}):
   $$
   \mathcal{Z}^-_J\subset\mathcal{Z}^+_g, \,\,\, \mathcal{Z}^-_g\subset\mathcal{Z}^+_J, b^++b^-=h^+_J+h^-_J, \,\,\, h^+_J\geq b^-, \,\,\, 0\leq h^-_J\leq b^+.
   $$
   M. Lejmi \cite{L3} recognizes $\mathcal{Z}^-_J$ as the kernel of an elliptic operator on $\Omega^-_J$.
  \begin{lem}\label{2l3}
  (Lemma {\rm4.1} in {\rm\cite{L3}}) Let $(M,g,J,F)$ be a closed almost Hermitian 4-manifold.
  Let operator $P:\Omega^-_J\rightarrow\Omega^-_J$ be defined by
  \begin{eqnarray*}
  P(\psi)=P^-_J(dd^*\psi),
  \end{eqnarray*}
  where $P^-_J: \Omega^2_\mathbb{R}\rightarrow\Omega^-_J$ is the projection.
   Then $P$ is a self-adjoint strongly elliptic linear operator with kernel the $g$-self-dual-harmonic, $J$-anti-invariant 2-forms.
  Hence,
  \begin{eqnarray*}
  \Omega^-_J=kerP\oplus P^-_J(d\Omega^1_\mathbb{R})=H^-_J\oplus P^-_J(d\Omega^1_\mathbb{R}).
  \end{eqnarray*}
   \end{lem}

   Suppose that $(M,J)$ is a closed complex surface, that is, $J$ is integrable.
   Theorem 2.13 of \cite{BHPV} shows that the cup product form on $H^2(M,\mathbb{R})$,
  restricted to $H^{1,1}_{\mathbb{R}}(M)$, is nondegenerate of type $(1, h^{1,1}-1)$ if $b^1$ is even and of type $(0, h^{1,1})$ if $b^1$ is odd.
   For closed almost complex 4-manifolds, by using Proposition \ref{2p1} and Lemma \ref{2l3}, we have the following analogous theorem:

   \begin{theo}\label{2t4}
 (Signature Theorem) Let $(M,J)$ be a closed almost complex 4-manifold. Then the cup-product form on $H^2(M;\mathbb{R})$
   restricted to $H^+_J$ is nondegenerate of type $(b^+-h^-_J,b^-)$.
   \end{theo}
  \begin{proof}
   We define an almost Hermitian structure $(g,J,F)$ on $M$.
   By Proposition \ref{2p1}, we have
   $$
   H^2(M;\mathbb{R})=H^+_g\oplus H^-_g=H^+_J\oplus H^-_J.
  $$
   So we can get
    $$
  H^+_J=H^-_g\oplus(H^+_J\cap H^+_g), \,\,\, dim(H^+_J\cap H^+_g)=b^+-h^-_J.
  $$
  For any $[\gamma]\in H^+_g$, $\gamma\in \mathcal{H}^+_g$,
  $$\gamma_J^-=\frac{1}{2}(\gamma(\cdot,\cdot)-\gamma(J\cdot,J\cdot))\in \Omega^-_J,$$
  by Lemma \ref{2l3} ,
  $$\gamma_J^-=\gamma_h+d^-_J(v_\gamma+\bar{v}_\gamma),$$
  where $$\gamma_h\in \mathcal{Z}^-_J\subseteq \mathcal{H}^+_g, \,\,\, v_\gamma\in\Omega^{0,1}_J.$$
  $\gamma-\gamma_h$ is still a self-dual harmonic 2-form.
  $$\gamma-\gamma_h-d(v_\gamma+\bar{v}_\gamma)\in H^+_J.$$

   By the discussion above, we can choose $[\omega_1],...,[\omega_{b^+-h^-_J}]$,
   where $(\omega_i,\omega_j)_g=\delta_{ij}$ for a standard orthonormal basis of $H^+_J\cap H^+_g$  with respect to the cup product.
   Let $\widetilde{\omega}_i\in\mathcal{Z}^+_J$ cohomologous to $\omega_i$.
   So
   \begin{equation}\label{2eq20}
   \int_M\widetilde{\omega}_i\wedge\widetilde{\omega}_j=\int_M\omega_i\wedge\omega_j=\int_M\omega_i\wedge *_g\omega_j=(\omega_i,\omega_j)_g=\delta_{ij}.
   \end{equation}

   Let $\beta_1,...,\beta_{b^-}\in \mathcal{H}^-_g$ be a standard orthonormal basis of $\mathcal{H}^-_g$ with respect to the integration by $g$, i.e. ,
   \begin{equation}\label{2eq21}
  (\beta_i,\beta_j)_g=\int_M\beta_i\wedge\ast_g\beta_j =\delta_{ij}.
  \end{equation}
    So $[\beta_1],...,[\beta_{b^-}]$ is standard orthonormal basis of $H^-_g$ with respect to the cup product.

   It is easy to see that $(\widetilde{\omega}_i,\beta_j)_g=0$ pointwise.
   So $\{\widetilde{\omega}_1,...,\widetilde{\omega}_{b^+-h^-_J},\beta_1,...,\beta_{b^-}\}$ is a standard orthonormal
   basis of $\mathcal{Z}^+_J$ with respect to the cup product.
  The matrix of the cup-product form on $H^2(M;\mathbb{R})$ restricted to $H^+_J$ under the above basis is
  \begin{equation}\label{2eq22}
  \left(
  \begin{array}{cc}
    I_{b^+-h^-_J} & 0 \\
    0 & -I_{b^-} \\
  \end{array}
  \right).
  \end{equation}
  This completes the proof of Theorem \ref{2t4}.
   \end{proof}

   We define the following operators:
   \begin{eqnarray}\label{2eq23}
  d^+_J&=&P^+_Jd: \,\,\,\Omega^1_{\mathbb{R}}\longrightarrow\Omega^{1,1}_\mathbb{R}, \nonumber\\
  d^-_J&=&P^-_Jd: \,\,\,\Omega^1_{\mathbb{R}}\longrightarrow(\Omega^{2,0}_J+\Omega^{0,2}_J)_\mathbb{R},
  \end{eqnarray}
  where $P^\pm_J: \Omega^2_\mathbb{R}\longrightarrow\Omega^\pm_J$.

  Suppose that $(M,g,J,F)$ is a closed almost Hermitian 4-manifold, and that the given almost complex structure $J$ is also tamed by a symplectic form $\omega$.
  By Lemma \ref{2l3}, $\omega$ can be decomposed as follows:
   $$\omega=F+d^-_J(v+\bar{v})+\alpha_\omega,$$
  where $\alpha_\omega\in\mathcal{Z}^-_J\subset\mathcal{H}^+_g$, $v\in\Omega^{0,1}_J$, $F^2>0$.
  Set $\omega_1=\omega-\alpha_\omega$.
  It is clear that $J$ is also an $\omega_1$-tame almost complex structure.
  Set
  $$\widetilde{\omega}_1=\omega_1-d(v+\bar{v})=F-d^+_J(v+\bar{v})\in\mathcal{Z}^+_J.$$
  Thus $[\widetilde{\omega}_1]\in H^+_g\cap H^+_J$.
  It is easy to see that $0\leq h^-_J\leq b^+-1$ (cf. \cite{TWZZ}).
  We may assume without loss of generality that
  $$\int_MF^2=2$$ and $$\int_M|d^-_J(v+\bar{v})|^2d\mu_g=2a>0,$$
  for if $a=0$, then $F$ is a symplectic form compatible with $J$.

   \vskip 6pt

   Let $(g,J,F)$ be an almost Hermitian structure on a closed 4-manifold $M$,
   $\omega_1=F+d^-_J(v+\bar{v})$ a symplectic form on $M$ by which $J$ is tamed, where $v\in\Omega^{0,1}_J$.
   Suppose $\psi\in\Lambda^{1,1}_\mathbb{R}\otimes L^2(M)$ is $d$-exact with
   \begin{equation}\label{2eq24}
   \psi=d(u+\bar{u})=d^+_J(u+\bar{u}), \,\,\,i.e., \,\,\,d^-_J(u+\bar{u})=0,
   \end{equation}
   for some $u\in\Lambda^{0,1}_J\otimes L^2_1(M)$.
   Let $$f_\psi=\frac{1}{2}\psi\wedge F/d\mu_g-\frac{1}{2}\int_M\psi\wedge F,$$
   then $$\int_Mf_\psi d\mu_g=0.$$
   Define $$L^2_2(M)_0:=\{f\in L^2_2(M)| \int_Mfd\mu_g=0 \}.$$
   It is easy to see that $f_\psi\in L^2_2(M)_0$.
   Recall that if $J$ is integrable, in classical complex analysis,
    it follows that $dJdf_\psi=2\sqrt{-1}\partial_J\bar{\partial}_Jf_\psi$.
   For general case (i.e., $J$ is not integrable),
   by Lemma \ref{2l3}, there exists $\eta^1_\psi\in\Lambda^{0,2}_J\otimes L^2_2(M)$ such that
    $$d^-_JJdf_\psi+d^-_Jd^*(\eta^1_\psi+\overline{\eta}^1_\psi)=0.$$
     Then, by Lemma \ref{2l3} and the Hodge decomposition
    $\Omega^+_g=\mathcal{H}^+_g\oplus d^+_g(\Omega^1)$ (cf. \cite{D6,DK}),
   since
   $$d_g^+d^*: \Omega^+_g\longrightarrow\Omega^+_g$$
   is a strongly self-adjoint elliptic operator,
  there are
   $\eta^2_\psi\in\Lambda^{0,2}_J\otimes L^2_2(M)4$
  satisfying
  \begin{equation}\label{2eq25}
   d^+_g(u+\bar{u})=d^+_gd^*[f_\psi\omega_1+(\eta^1_\psi+\eta^2_\psi+\overline{\eta}^1_\psi+\overline{\eta}^2_\psi)],
  \end{equation}
   where $$f_\psi\omega_1+(\eta^1_\psi+\eta^2_\psi)+(\overline{\eta}^1_\psi+\overline{\eta}^2_\psi)\in\Omega^+_g.$$
   Note that
   \begin{eqnarray}
     d^*(f_\psi\omega_1) &=& -*_gd(f_\psi\omega_1) \nonumber \\
      &=&  -*_g(d f_\psi\wedge\omega_1) \nonumber \\
      &=& -*_g(d f_\psi\wedge F) \nonumber \\
      &=& Jdf_\psi- *_g(df_\psi\wedge d^-_J(v+\bar{v})).
   \end{eqnarray}
   By (\ref{2eq17}) and (\ref{2eq25}), we have
   \begin{eqnarray*}
  \psi&=&d(u+\bar{u})\\
   &=&dd^*[f_\psi\omega_1+(\eta^1_\psi+\eta^2_\psi+\overline{\eta}^1_\psi+\overline{\eta}^2_\psi)] \\
    &=& dJdf_\psi+dd^*(\eta^1_\psi+\overline{\eta}^1_\psi)-d*_g(df_\psi\wedge d^-_J(v+\bar{v}))+dd^*(\eta^2_\psi+\overline{\eta}^2_\psi),
   \end{eqnarray*}
   where, by Lemma \ref{2l3},
  $$-d^-_J*_gdf_\psi\wedge d^-_J(v+\bar{v})+d^-_Jd^*(\eta^2_\psi+\overline{\eta}^2_\psi)=0.$$
  Thus, by the above discussion,
  we can define two operators $$\mathcal{D}^+_J \,\,{\rm and}\,\, \widetilde{\mathcal{D}}^+_J: L^2_2(M)_0\longrightarrow\Lambda^{1,1}_\mathbb{R}\otimes L^2(M).$$
  \begin{defi}\label{2d5}
  Set $\mathcal{W}: L^2_2(M)_0\longrightarrow\Lambda^1_\mathbb{R}\otimes L^2_1(M)$,
   $$\mathcal{W}(f)=Jdf+d^*(\eta^1_f+\overline{\eta}^1_f),\,\,\,\, \eta^1_f\in\Lambda^{0,2}_J\otimes L^2_2(M),$$ satisfying
   $$d^-_J\mathcal{W}(f)=0.$$
  \,\,\,\,\,\,\, Define $\mathcal{D}^+_J: L^2_2(M)_0\longrightarrow\Lambda^{1,1}_\mathbb{R}\otimes L^2(M)$, \,\,\,$\mathcal{D}^+_J(f)=d\mathcal{W}(f)$.\\

  Set $\mathcal{\widetilde{W}}: L^2_2(M)_0\longrightarrow\Lambda^1_\mathbb{R}\otimes L^2_1(M)$,
  $$\mathcal{\widetilde{W}}(f)=\mathcal{W}(f)-*_g(df\wedge d^-_J(v+\bar{v}))+d^*(\eta^2_f+\overline{\eta}^2_f),
   \,\,\, \eta^2_f\in\Lambda^{0,2}_J\otimes L^2_2(M),$$ satisfying
   $$d^*\mathcal{\widetilde{W}}(f)=0, \,\,\, d^-_J\mathcal{\widetilde{W}}(f)=0.$$
  \,\,\,\,\,\,\, Define $\widetilde{\mathcal{D}}^+_J: L^2_2(M)_0\longrightarrow\Lambda^{1,1}_\mathbb{R}\otimes L^2(M)$,
   \,\,\,\, $\widetilde{\mathcal{D}}^+_J(f)=d\mathcal{\widetilde{W}}(f)$.
  \end{defi}
  \begin{rem}\label{2r6}
  Notice that $d^-_J\mathcal{\widetilde{W}}=0=d^-_J\mathcal{W}$, by the above formula,
  it implies that $$d^-_J(*_g(df\wedge d^-_J(v+\bar{v}))+d^*(\eta^2_f+\overline{\eta}^2_f))=0.$$
  If $dF=0$, then $\mathcal{D}^+_J=\widetilde{\mathcal{D}}^+_J$ since $d^-_J(v+\bar{v})=0$.
  If $J$ is integrable, $\bar{\partial}^2_J=\partial^2_J=0$ and $\partial_J\bar{\partial}_J+\bar{\partial}_J\partial_J=0$,
  then $dJdf=2\sqrt{-1}\partial_J\bar{\partial}_Jf=\mathcal{D}^+_J(f)$, that is, $\eta^1_f=0$.
  (cf. {\rm\cite{TWY,WZ2}}).
  For the higher dimensional closed almost K\"{a}hler manifold $(M,g,J,\omega)$, could one define the similar operator
   $\mathcal{D}^+_J$ with the strongly self-adjoint elliptic operator?
  \end{rem}

   Denote by $\mathbb{G}$ the Green operator associated to $\Delta_g$ (cf. \cite{Kod}).
   The Hodge operator $*_g$ commutes with $\Delta_g$.
  It follows that $*_g$ commutes with $\mathbb{G}$.
  It is clear that $d$ and $d^*$ commute with $\mathbb{G}$. Lejmi \cite{L3} proved a generalized $\partial\bar{\partial}$-lemma for almost K\"ahler 4-manifolds under the condition $h_J^{-}=b^{+}-1$, and in the following, we generalize this result to almost Hermitian manifolds $(M,g,J,F)$ with $J$ tamed by $\omega_1$,
  where $\omega_1$ is the form defined earlier.
  \begin{prop}\label{3p6'}
  (cf. Proposition 2.5 in {\rm \cite{L4}})
  If $h^-_J=b^+-1$, then $\widetilde{\mathcal{D}}^+_J(f)$ can be rewritten as
  $$\widetilde{\mathcal{D}}^+_J(f)=2d\mathbb{G}d^*(f'F)=2\mathbb{G}dd^*(f'F)=2 dd^*\mathbb{G}(f'F),$$ and
  $\widetilde{\mathcal{W}}(f)$ can be rewritten as $$\widetilde{\mathcal{W}}(f)=2\mathbb{G}d^*(f'F)=2d^*\mathbb{G}(f'F),$$ where
  $f'\in L^2(M)_0,~f\in L^2_2(M)_0$.
  \end{prop}
  \begin{proof}
  First of all, we prove that for any $f'\in L^2(M)_0$, $d\mathbb{G}d^*(f'F)$ is $J$-invariant if $h^-_J=b^+-1$.
   Without loss of generality, we choose $f'\in C^\infty(M)_0$.
   \begin{eqnarray*}
   (dGd^*(f'F))^-_J &=& P^+_g(d\mathbb{G}d^*(f'F))-\frac{1}{2}(P^+_g(d\mathbb{G}d^*(f'F)),F)_gF  \\
                   &=&  \frac{1}{2}(1+*_g)(-\mathbb{G}d*_gd*_g(f'F))-\frac{1}{4}(1+*_g)(-\mathbb{G}d*_gd*_g(f'F),F)_gF  \\
                   &=&  \frac{1}{2}\mathbb{G}\Delta_g(f'F)-\frac{1}{4}(\mathbb{G}\Delta_g(f'F),F)_gF \\
                   &=&  \frac{1}{2}(f'F)- \frac{1}{2}(f'F)_H-\frac{1}{4}(f'F-(f'F)_H,F)_gF \\
                   &=&  \frac{1}{2}(f'F)- \frac{1}{2}(f'F)_H-\frac{1}{2}(f'F)+\frac{1}{4}((f'F)_H,F)_gF \\
                   &=& - \frac{1}{2}(f'F)_H+\frac{1}{4}((f'F)_H,F)_gF,
   \end{eqnarray*}
   where $(f'F)_H$ denotes the harmonic part with respect to $\Delta_g$.
  Under the assumption $h^-_J=b^+-1$, it follows that $(f'F)_H=0$ for any smooth function $f'$ with zero integral for the following reason.
  In this case, $$\mathcal{H}^2_g=\mathbb{R}\cdot\omega_1\oplus\mathcal{H}^-_J\oplus\mathcal{H}^-_g.$$
   Since
  $$\int_Mf'F\wedge\omega_1=\int_Mf'F\wedge F=2\int_Mf'd\mu_g=0,$$
  $f'F\wedge\alpha\equiv 0$ for any $\alpha\in\mathcal{H}^-_J$
  and $f'F\wedge\beta\equiv 0$ for any $\beta\in\mathcal{H}^-_g$,
  by Hodge decomposition (cf. \cite{DK}), we can get $(f'F)_H=0$.
  By the above calculation, it is easy to see that
  \begin{equation}\label{3eq11'}
  P^+_g(2d\mathbb{G}d^*(f'F))=P^+_g(2dd^*\mathbb{G}(f'F))=\mathbb{G}\Delta_g(f'F)=f'F.
   \end{equation}

 Second, let $\psi$ be a smooth $J$-invariant $2$-form which is $d$-exact,
 i.e., $\psi=d(u+\bar{u})$ and $d^-_J(u+\bar{u})=0$, where $u\in\Omega^{0,1}_J$.
 Then $P^+_g(\psi)=f'_{\psi}F$, $f'_{\psi}\in C^\infty(M)_0$,
 since $\omega_1=F+d^-_J(v+\bar{v})$, $v\in\Omega^{0,1}_J$ and
 $$2\int_Mf'_{\psi}d\mu_g=\int_M\psi\wedge F=\int_M\psi\wedge\omega_1=\int_Md(u+\bar{u})\wedge\omega_1=0.$$
 Therefore, by (\ref{3eq11'}), $$P^+_g(\psi)=f'_{\psi}F=P^+_g(d2\mathbb{G}d^*(f'_{\psi}F)).$$
 Hence $$\psi=d(u+\bar{u})=d2\mathbb{G}d^*(f'_{\psi}F),$$ since
 $P^+_g(\psi-d2\mathbb{G}d^*(f'_{\psi}F))=0$ and $\psi-d2\mathbb{G}d^*(f'_{\psi}F)$ is $d$-exact (cf. (\ref{2eq17}) or \cite{DK}).
 According to the construction of $\widetilde{\mathcal{D}}^+_J$, there exists a function $f_{\psi}\in L^2_2(M)_0$ such that
  $\psi=\widetilde{\mathcal{D}}^+_J(f_{\psi})=2dd^*\mathbb{G}(f'_{\psi}F)$.
 \end{proof}

 \begin{rem}\label{3r7'}
 (1) If $(M,g,J,\omega)$ is a K\"{a}hler surface, then $h^-_J=b^+-1$ and
  $$\mathcal{D}^+_J(f)=\widetilde{\mathcal{D}}^+_J(f)=2d\mathbb{G}d^*(f'\omega)
  =2d\mathbb{G}J(df')=2d\mathbb{G}d^cf'=2 dd^c\mathbb{G}f'=2\sqrt{-1}\partial_J\bar{\partial}_Jf,$$ where $f=\mathbb{G}f'$.
 Hence, the above proposition can be viewed as a generalized $\partial\bar{\partial}$-lemma and
 $$P^+_g(2d\mathbb{G}d^*(f'_{\psi}F))=P^+_g(2d d^*\mathbb{G}(f'_{\psi}F))=P^+_g(2\mathbb{G}dd^*(f'_{\psi}F))=f'_{\psi}F.$$

 (2) $\mathbb{G}(f'_{\psi}F)\in\Lambda^2_{\mathbb{R}}\otimes L^2_2(M)$, where $f'_{\psi}\in L^2(M)_0$.
 \end{rem}

  Suppose that $(M,g,J,F)$ is tamed by $\omega_1=F+d^-_J(v+\bar{v})$, where $v\in \Omega^{0,1}_J$,
   suppose that $[\alpha_1],\cdot\cdot\cdot,[\alpha_{h_J^-}]$ is a basis of $H^-_J$,
   and $[\omega_1],\cdot\cdot\cdot,[\omega_{b^+-h_J^-}]$ is a basis of $H^+_g\cap H^+_J$, where $0\leq h^-_J\leq b^+-1$.
   Let $\psi\in\Lambda^{1,1}_\mathbb{R}\otimes L^2(M)$ be a real d-exact (1,1)-form, that is,
   there exists $u_\psi\in\Omega^{0,1}_J$ such that $\psi=d(u_\psi+\bar{u}_\psi)$, hence $d^-_J(u_\psi+\bar{u}_\psi)=0$.
    It is clear that
   $$
  \psi\wedge\alpha_j=0,\,\,\,1\leq j\leq h_J^-.
   $$
   Hence,
   \begin{equation}\label{2e26}
  \int_M\psi\wedge\alpha_j=0,\,\,\,1\leq j\leq h_J^-,
  \end{equation}
  \begin{equation}\label{2e27}
  \int_M\psi\wedge\omega_i=0,\,\,\,1\leq i\leq b^+-h_J^-.
  \end{equation}
  Thus $\psi$ is orthogonal to the self-dual harmonic 2-forms, $\mathcal{H}^+_g$, with respect to the cup product.
  By Hodge decomposition (cf. \cite{DK}), there exist $$f_\psi\in L^2_2(M)_0,\,\,\, \eta^1_\psi,\,\,\, \eta^2_\psi\in \Lambda^-_J\otimes L^2_2(M)$$ such that
  $$P^+_g\psi=d^+_g(u_\psi+\bar{u}_\psi)=d^+_gd^*(f_\psi\omega_1+(\eta^1_\psi+\bar{\eta}^1_\psi)+(\eta^2_\psi+\bar{\eta}^2_\psi))$$
  satisfying
  \begin{equation}\label{2e28}
  d^-_Jd^*(f_\psi\omega_1+(\eta^1_\psi+\bar{\eta}^1_\psi)+(\eta^2_\psi+\bar{\eta}^2_\psi))=0,
  \end{equation}
   and it follows that
  \begin{equation}\label{2e29}
  \psi=dd^*(f_\psi\omega_1+(\eta^1_\psi+\bar{\eta}^1_\psi)+(\eta^2_\psi+\bar{\eta}^2_\psi)).
  \end{equation}

    By Definition \ref{2d5} and Proposition \ref{3p6'}, we have the following lemma,

  \begin{lem}\label{2l10}
  Let $(M,J)$ be a tamed closed almost complex $4$-manifold with $h^-_J=b^+-1$.
  Suppose that $\psi\in\Lambda^{1,1}_\mathbb{R}\otimes L^2(M)$ is $d$-exact.
  Then there exists $f_{\psi}\in L^2_2(M)_0$ and $f'_\psi\in L^2(M)_0$
   such that $$\psi=\widetilde{\mathcal{D}}^+_J(f_\psi)=d\widetilde{\mathcal{W}}(f_\psi)=2dd^*\mathbb{G}(f'_\psi F).$$
  \end{lem}

 \section{The intersection pairing on weakly $\widetilde{\mathcal{D}}^+_J$-closed (1,1)-forms }\label{3}
 \setcounter{equation}{0}
  In this section, we shall investigate the intersection paring on weakly $\widetilde{\mathcal{D}}^+_J$-closed $(1,1)$-forms defined below as done in Buchdahl's paper \cite{B}.
  First, we consider the following technical lemma (compare Lemma 1 in \cite{B} or $\S$ 3.2 in \cite{GH}):
  \begin{lem}\label{3l1}
  Suppose that $(M,g,J,F)$ is a closed almost Hermitian 4-manifold.
  Then $$d^+_J: \Lambda^1_\mathbb{R}\otimes L^2_1(M)\longrightarrow\Lambda^{1,1}_\mathbb{R}\otimes L^2(M)$$ has closed range.
  \end{lem}
  \begin{proof}
  Let $\{w_i\}$ be a sequence of real 1-forms on $M$ with coefficients
   in $L^2_1$ such that $\psi_i=d^+_Jw_i$ is converging in $L^2$ to some $\psi\in\Lambda^{1,1}_\mathbb{R}\otimes L^2(M)$.
   Write $w_i=u_i+\bar{u}_i$ for some $(0,1)$-form $u_i$, so $\psi_i=d^+_J(u_i+\bar{u}_i)=\partial_J u_i+\bar{\partial}_J\bar{u}_i$.

 By smoothing and diagonalising, it can be assumed without loss of generality that $u_i$ is smooth for each $i$.
 Note that
 \begin{equation}\label{3eq1}
 F\wedge\psi_i=(\wedge\psi_i)F^2/2,
 \end{equation}
 \begin{equation}\label{3eq2}
 *_g\psi_i=(\wedge\psi_i)F-\psi_i,
 \end{equation}
 \begin{equation}\label{3eq3}
 |\psi_i|^2d\mu_g=(\wedge\psi_i)^2F^2/2-\psi_i^2,
 \end{equation}
 where $\wedge:\Omega_{\mathbb{R}}^{1,1}\longrightarrow\Omega^0_{\mathbb{R}}$ is an algebraic operator in Lefschetz decomposition (cf. \cite{GH}).
 Using Stokes' Theorem,
 $$\int_M|\psi_i|^2d\mu_g=\int_M(\wedge\psi_i)^2d\mu_g+2\int_M(\bar{\partial}_Ju_i+A_J\bar{u}_i)^2,$$
 $$\int_Mdw_i\wedge\ast_gdw_i=\int_M\psi_i\wedge*_g\psi_i+2\int_M(\bar{\partial}_Ju_i+A_J\bar{u}_i)^2.$$
 So it follows that $dw_i=d^+_Jw_i+d^-_Jw_i$ is bounded in $L^2$.
 Let $\widetilde{w}_i$ be the $L^2$-projection of $w_i$ perpendicular to the kernel of $d$, so $d^*\widetilde{w}_i=0$ and $\widetilde{w}_i$ is perpendicular to the harmonic 1-forms.
 Hence $d\widetilde{w}_i=dw_i$ and there exists a constant $C$ such that
 \begin{equation}\label{3eq4}
 \|\widetilde{w}_i\|^2_{L^2_1(M)}\leq C(\|d\widetilde{w}_i\|^2_{L^2(M)}+\|d^*\widetilde{w}_i\|^2_{L^2(M)})=C\|dw_i\|^2_{L^2(M)}<Const.,
 \end{equation}
 so a subsequence of the sequence $\{\widetilde{w}_i\}$ converges weakly in $L^2_1$ to some $\widetilde{w}\in\Lambda^1_\mathbb{R}\otimes L^2_1(M)$. Since $d^+_J\widetilde{w}_i=d^+_Jw_i=\psi_i$, it follows $d^+_J\widetilde{w}=\psi$, proving the claim.
 \end{proof}
 We now consider the closed tamed almost Hermitian 4-manifold $(M,g,J,F)$.
 We may assume without loss of generality that $\omega_1=F+d^-_J(v+\bar{v})$, $v\in\Omega^{0,1}_J$, $F$ is the fundamental form with
 \begin{equation}\label{3eq5}
 \int_MF^2=2, g(\cdot,\cdot)=F(\cdot,J\cdot),\,\,\, \int_M\omega^2_1=2(1+a),\,\,\, 2a=\int_M|d^-_J(v+\bar{v})|^2d\mu_g>0,
 \end{equation}
 where $d\mu_g$ is the volume form defined by $g$;
 if $a=0$, then $F$ is a $J$-compatible symplectic form.
 It is clear that $0\leq h^-_J\leq b^+-1$ (cf. \cite{TWZZ}).
 Denote by
 \begin{equation}\label{3eq6}
 \widetilde{\omega}_1:=\omega_1-d(v+\bar{v})=F-d^+_J(v+\bar{v}),
 \end{equation}
 then $\widetilde{\omega}_1\in\mathcal{Z}^+_J$ being cohomologous to $\omega_1$,
 \begin{equation}\label{3eq7}
 \int_M\widetilde{\omega}_1^2=\int_M\omega^2_1=2(1+a),
 \end{equation}
 $$
 -\int_M(d^+_J(v+\bar{v}))^2=\int_M|d^-_J(v+\bar{v})|^2d\mu_g=2a>0,
 $$
 and
 \begin{equation}\label{3eq8}
 \int_Md^+_J(v+\bar{v})\wedge F=-2a.
 \end{equation}
 Choose $\alpha_j\in \mathcal{Z}_J^-\subset\mathcal{Z}_g^+=\mathcal{H}^+_g$ such that
 $$\int_M\alpha_i\wedge\alpha_j=\delta_{ij},\,\,\, 1\leq j\leq h^-_J.$$
 We can find $\omega_2,\cdot\cdot\cdot,\omega_{b^+-h^-_J}\in \mathcal{Z}^+_g\setminus \mathcal{Z}^-_J$, such that $$\int_M\omega_j\wedge\omega_k=\delta_{jk},\,\,\, 2\leq j, k\leq b^+-h^-_J,$$
 $$\int_M\omega_1\wedge\omega_j=0,\,\,\, 2\leq j\leq b^+-h^-_J.$$
 Hence $\mathcal{H}^+_g=Span\{\omega_1,\cdot\cdot\cdot,\omega_{b^+-h^-_J},\alpha_1,\cdot\cdot\cdot,\alpha_{h^-_J}\}$.
 Let $\widetilde{\omega}_i\in\mathcal{Z}^+_J$ be cohomologous to $\omega_i$, $1\leq i\leq b^+-h^-_J$,
 so
 \begin{equation}\label{3eq9}
 \int_M\widetilde{\omega}_1\wedge F=2(1+a)
 \end{equation}
 and
 \begin{equation}\label{3eq10}
 \int_M\widetilde{\omega}_j\wedge F=0,\,\,\,2\leq j\leq b^+-h^-_J.
 \end{equation}

 In Section \ref{2}, we define $\mathcal{D}^+_J$ and $\widetilde{\mathcal{D}}^+_J: L^2_2(M)_0\longrightarrow\Lambda^{1,1}_\mathbb{R}\otimes L^2(M)$.
 Analogous to Lemma \ref{3l1}, we have:
 \begin{lem}\label{3l2}
 $\widetilde{\mathcal{D}}^+_J: L^2_2(M)_0\longrightarrow\Lambda^{1,1}_\mathbb{R}\otimes L^2(M)$ has closed range.
 If $J$ is integrable, then $$\mathcal{D}^+_J=dJdf=2\sqrt{-1}\partial_J\bar{\partial}_Jf,$$
 hence $\mathcal{D}^+_J$ has closed range too.
 \end{lem}
 \begin{proof}
 Let $\{f_i\}$ be a sequence of real functions on $M$ in $L^2_2(M)_0$.
 By Definition \ref{2d5}, $\{\mathcal{\widetilde{W}}(f_i)\}$ is a sequence of real 1-forms on $M$ with coefficients in $L^2_1$ such that $$\psi_i=d\mathcal{\widetilde{W}}(f_i)=\widetilde{\mathcal{D}}^+_J(f_i)\in\Lambda^{1,1}_\mathbb{R}\otimes L^2(M)$$
 is converging in $L^2$ to some $\psi\in\Lambda^{1,1}_\mathbb{R}\otimes L^2(M)$.
 It is clear that $d^*\mathcal{\widetilde{W}}(f_i)=0$.
 By the proof of Lemma \ref{3l1}, $\{\mathcal{\widetilde{W}}(f_i)\}$ is bounded in $L^2_1$, so a subsequence of $\{\mathcal{\widetilde{W}}(f_i)\}$ converges weakly in $L^2_1$ to some $\mathcal{\widetilde{W}}\in\Lambda^1_\mathbb{R}\otimes L^2_1(M)$.
 Since $d\mathcal{\widetilde{W}}(f_i)\in\Lambda^{1,1}_\mathbb{R} \otimes L^2(M)$, it follows that $$d\mathcal{\widetilde{W}}=\psi\in\Lambda^{1,1}_\mathbb{R}\otimes L^2(M).$$
 To complete the proof of Lemma \ref{3l2}, we need the following claim:

 {\bf Claim} (cf. Lemma \ref{2l10}): {\it Suppose that $\psi\in\Lambda^{1,1}_{\mathbb{R}}\otimes L^2(M)$ is $d$-exact,
 that is, there is $u_{\psi}\in\Lambda^{0,1}_J\otimes L^2_1(M)$ such that $\psi=d(u_{\psi}+\bar{u}_{\psi})$.
 Then $\psi$ is $\widetilde{\mathcal{D}}^+_J$-exact, that is, there exists $f_{\psi}\in L^2_2(M)_0$ such that $\psi=\widetilde{\mathcal{D}}^+_J(f_{\psi})$.}

 Indeed, let $A\in\Omega^1_{\mathbb{R}}(M)$, $dA=d^+_JA+d^-_JA$.
 By (\ref{3eq1})-(\ref{3eq3}), we have
 $$
 \int_M|d^+_JA|^2d\mu_g=\int_M(\wedge d^+_JA)^2d\mu_g+\int_M|d^-_JA|^2d\mu_g,
 $$
 $$
 \int_M|dA|^2d\mu_g=\int_M|d^+_JA|^2d\mu_g+\int_M|d^-_JA|^2d\mu_g.
 $$
 Let $\tilde{A}$ be the $L^2$-projection of $A$ perpendicular to the kernel of $d$,
 by Hodge decomposition, $d^*\tilde{A}=0$ and $\tilde{A}$ are
  perpendicular to the harmonic $1$-forms.
  Hence $d\tilde{A}=dA$ and there exists a constant $C$ such that
  \begin{equation}\label{non equ}
    \|\tilde{A}\|^2_{L^2}\leq \|\tilde{A}\|^2_{L^2_1}\leq C(\|d\tilde{A}\|^2_{L^2}+\|d^*\tilde{A}\|^2_{L^2})\leq {\rm Const.}(dA).
  \end{equation}
  Recall the definition of $\mathcal{\widetilde{W}}$ (cf. Definition \ref{2d5}):
  $f\in L^2_2(M)_0$, $\eta^1_f,\eta^2_f\in\Lambda^{0,2}_J\otimes L^2_2(M)$ such that
  $$\mathcal{\widetilde{W}}(f)=d^*(f\omega_1+(\eta^1_f+\bar{\eta}^1_f)+(\eta^2_f+\bar{\eta}^2_f))$$
  satisfying $d^-_J\mathcal{\widetilde{W}}(f)=0$, $d^*\mathcal{\widetilde{W}}(f)=0$
  and $d\mathcal{\widetilde{W}}(f)=d^+_J\mathcal{\widetilde{W}}(f)\in\Lambda^{1,1}_{\mathbb{R}}\otimes L^2(M)$.
  As done in Appendix \ref{Hormander}, without loss of generality, we may assume that if $A\in \Omega^1_{\mathbb{R}}(M)$, $d^*A=0$ and $d^-_JA=0$,
  then
   \begin{eqnarray*}
    (\mathcal{\widetilde{W}}(f),A) &=& -\int_M A\wedge d[f\omega_1+(\eta^1_f+\bar{\eta}^1_f)+(\eta^2_f+\bar{\eta}^2_f)]\\
     &=& -\int_M d(A)\wedge[f\omega_1+(\eta^1_f+\bar{\eta}^1_f)+(\eta^2_f+\bar{\eta}^2_f)]  \\
     &=& -\int_M d^+_J(A)\wedge fF\\
     &=& (f, \mathcal{\widetilde{W}}^*A).
  \end{eqnarray*}
  Thus, the formal $L^2$-adjoint operator of $\mathcal{\widetilde{W}}$ is
  \begin{equation}\label{formal ope}
 \mathcal{\widetilde{W}}^*A=\frac{-2F\wedge d^+_JA}{F^2}=-(\wedge d^+_JA).
  \end{equation}
  By (\ref{non equ}), (\ref{formal ope}), we have: If $A\in\Lambda^1_{\mathbb{R}}\otimes L^2_1(M)$,
  $d^*A=0$, then
 \begin{equation}\label{non equ1}
    \|A\|^2_{L^2}\leq C(\|\mathcal{\widetilde{W}}^*A\|^2_{L^2}+2\|d^-_JA\|^2_{L^2})\leq {\rm Const.}(\wedge d^+_JA, d^-_JA).
  \end{equation}
    Now suppose that $\psi\in\Lambda^{1,1}_{\mathbb{R}}\otimes L^2(M)$ is $d$-exact,
    then there exists $u_{\psi}\in\Lambda^{0,1}_J\otimes L^2_1(M)$ such that $\psi=d(u_{\psi}+\bar{u}_{\psi})$,
    $d^-_J(u_{\psi}+\bar{u}_{\psi})=0$. By Hodge decomposition, there exists $\tilde{u}_{\psi}\in\Lambda^{0,1}_J\otimes L^2_1(M)$
    satisfying that
    $$
    \psi=d(\tilde{u}_{\psi}+\bar{\tilde{u}}_{\psi}),\,\,\,d^-_J(\tilde{u}_{\psi}+\bar{\tilde{u}}_{\psi})=0,\,\,\,d^*(\tilde{u}_{\psi}+\bar{\tilde{u}}_{\psi})=0.
    $$
    By (\ref{non equ1}),
    $$
    \|\tilde{u}_{\psi}+\bar{\tilde{u}}_{\psi}\|_{L^2}\leq C \|\wedge\psi\|_{L^2}=C\|P^+_g\psi\|_{L^2}.
    $$
    Since $d^+_g\oplus d^*:\Lambda^1_{\mathbb{R}}\rightarrow\Lambda^{1,1}_{\mathbb{R}}\oplus\Lambda^0_{\mathbb{R}}$
    is an elliptic system,
    we can solve $\mathcal{\widetilde{W}}, d^-_J$-problem
    (that is similar to $\bar{\partial}$-problem in classical complex analysis \cite{Hormander})
    for closed almost Hermitian $4$-manifold $(M,g,J,F)$ tamed by the symplectic form $\omega_1$ (more details see Appendix \ref{Hormander}), that is,
    there exists $f_{\psi}\in L^2_2(M)_0$ such that $\mathcal{\widetilde{W}}(f_{\psi})=\tilde{u}_{\psi}+\bar{\tilde{u}}_{\psi}$,
    $P^+_gd\mathcal{\widetilde{W}}(f_{\psi})=P^+_g\psi$.
    Since $\psi\in \Lambda^{1,1}_{\mathbb{R}}\otimes L^2(M)$ is $d$-exact, it follows that
    $d\mathcal{\widetilde{W}}(f_{\psi})=\psi$. This completes the proof of the above Claim.

  We now return to the proof of Lemma \ref{3l2}. By the above claim which is similar to
  Lemma \ref{2l10}, there exists $f\in L^2_2(M)_0$ such that $\widetilde{\mathcal{D}}^+_J(f)=d\mathcal{\widetilde{W}}(f)=\psi$.

 If $J$ is integrable, after a simple calculation, we can get
 $$\mathcal{D}^+_J(f)=dJdf=2\sqrt{-1}\partial_J\bar{\partial}_Jf$$ and $$2\sqrt{-1}\partial_J\bar{\partial}_Jf\wedge F=\Delta_gf\cdot\frac{F^2}{2}.$$
 So by Poincar\'{e}'s Inequality and Interpolation Inequality, we can immediately get that $\mathcal{D}^+_J$ has closed range.
  \end{proof}

 \begin{defi}\label{3d3}
 $\psi\in\Lambda^{1,1}_\mathbb{R}\otimes L^2(M)$ is said to be weakly $\widetilde{\mathcal{D}}^+_J$-closed if and only if for any $f\in L^2_2(M)_0$,
 $$\int_M\psi\wedge\widetilde{\mathcal{D}}^+_J(f)=0.$$
 \end{defi}
 Let $(\Lambda^{1,1}_\mathbb{R}\otimes L^2(M))_{w}$ denote the space of weakly $\widetilde{\mathcal{D}}^+_J$-closed $(1,1)$-forms.
 It is easy to get the following lemma since
 $$\widetilde{\mathcal{D}}^+_J(f)=d\mathcal{\widetilde{W}}(f)\in\Lambda^{1,1}_\mathbb{R}\otimes L^2(M).$$
 \begin{lem}\label{3l4}
 $F$, $d^+_J(u+\bar{u})$ where $u\in\Lambda^{0,1}_J\otimes L^2_1(M)$ are weakly $\widetilde{\mathcal{D}}^+_J$-closed.
 \end{lem}
 \begin{proof}
 Notice that
 $$
 \int_MF\wedge\widetilde{\mathcal{D}}^+_J(f)=\int_M\omega_1\wedge\widetilde{\mathcal{D}}^+_J(f)=0,
 $$
 and
 $$
 \int_Md^+_J(u+\bar{u})\wedge\widetilde{\mathcal{D}}^+_J(f)=\int_Md(u+\bar{u})\wedge\widetilde{\mathcal{D}}^+_J(f)=0.
 $$
 \end{proof}
 \begin{rem}\label{3r5}
 If $J$ is integrable, then $\partial_J^2=0=\bar{\partial}_J^2$, $\partial_J\bar{\partial}_J+\bar{\partial}_J\partial_J=0$.
 Hence $d^+_J(u+\bar{u})$ is also weakly $\partial_J\bar{\partial}_J$-closed.
 Since $\widetilde{\omega}_1=F-d^+_J(v+\bar{v})$ is a smooth d-closed $(1,1)$-form, $\widetilde{\omega}_1$ is also $\partial_J\bar{\partial}_J$-closed, hence, $F$ is weakly $\partial_J\bar{\partial}_J$-closed.
 Thus, the notation of weakly $\widetilde{\mathcal{D}}^+_J$-closed is a generalization of the notation of weakly $\partial_J\bar{\partial}_J$-closed
  defined in {\rm \cite{B}} (also see {\rm\cite{HL2}}).
 \end{rem}

 \begin{defi}
  $(\Lambda^{1,1}_\mathbb{R}\otimes L^2(M))^0_{w}:=\{cF+\psi \mid \,\,\,c\in\mathbb{R},\,\,\,
   \psi\in(\Lambda^{1,1}_\mathbb{R}\otimes L^2(M))_{w}$
  $$
   {\rm satisfies}\,\,\,
    P^+_g(\psi)\perp\mathcal{H}^+_g  \,\,\,{\rm with\,\,\, respect\,\,\, to\,\,\, the\,\,\, integration}\}
  $$
  \end{defi}

  It is clear that $(\Lambda^{1,1}_\mathbb{R}\otimes L^2(M))^0_{w}\subset(\Lambda^{1,1}_\mathbb{R}\otimes L^2(M))_{w}$,
  since $F\in(\Lambda^{1,1}_\mathbb{R}\otimes L^2(M))_{w}$.
  Let $\psi\in(\Lambda^{1,1}_\mathbb{R}\otimes L^2(M))^0_{w}$ and
  set
  $$
  c_{\psi}=\frac{1}{2}\int_M\psi\wedge F.
   $$
  Since $\psi\in(\Lambda^{1,1}_\mathbb{R}\otimes L^2(M))^0_{w}$ and
  $$
  \Lambda^+_g=\mathbb{R}\cdot F\oplus \Lambda^-_J,\,\,\, \Lambda^+_J=\mathbb{R}\cdot F\oplus \Lambda^-_g
  $$
  we can get that $P^+_g(\psi-c_{\psi}F)$ is orthogonal to $\mathcal{H}_g^+(M)$ with respect to the integration.
  By Hodge decomposition,
  there exists $f_{\psi}\in L^2_2(M)_0$
  such that
  \begin{equation}\label{3eq15}
   P^+_g(\psi-c_{\psi}F)=P^+_g(\widetilde{\mathcal{D}}_J^+(f_{\psi}))
  \end{equation}
  holds in $\Lambda^{1,1}_\mathbb{R}\otimes L^2(M)$.
  If $\psi$ is smooth, then $f_{\psi}$ is also smooth.
    By (\ref{3eq15}), we will find that
  $$
   \psi-c_{\psi}F-\widetilde{\mathcal{D}}^+_J(f_{\psi})=P^-_g(\psi-c_{\psi}F-\widetilde{\mathcal{D}}^+_J(f_{\psi}))\in\Lambda^-_g\otimes L^2(M)
   $$
  since $P^+_g(\psi-c_{\psi}F-\widetilde{\mathcal{D}}^+_J(f_{\psi}))=0$.
  By Hodge decomposition again, we have the following decomposition
  $$
   \psi-c_{\psi}F-\widetilde{\mathcal{D}}^+_J(f_{\psi})=\beta_{\psi}+d^-_g(\gamma_\psi)
   $$
  where $\beta_{\psi}\in\mathcal{H}^-_g(M)$, $\gamma_\psi\in\Lambda^1_\mathbb{R}\otimes L^2_1(M)$.
   Hence,
  $$
  \psi=c_{\psi}F+\beta_{\psi}+d^-_g(\gamma_\psi)+\widetilde{\mathcal{D}}^+_J(f_{\psi}).
   $$
  It is easy to see that $d^-_g(\gamma_\psi)\in(\Lambda^{1,1}_\mathbb{R}\otimes L^2(M))^0_{w}$,
        since $\psi$, $F$, $\beta_{\psi}$, $\widetilde{\mathcal{D}}^+_J(f_{\psi})\in(\Lambda^{1,1}_\mathbb{R}\otimes L^2(M))^0_{w}$.
   Let $$\psi'=\psi-d^-_g(\gamma_\psi)=c_{\psi}F+\beta_{\psi}+\widetilde{\mathcal{D}}^+_J(f_{\psi}).$$
  $\psi'$ is also in $(\Lambda^{1,1}_\mathbb{R}\otimes L^2(M))^0_{w}$.
  If $\psi$ is smooth,  both $\psi'$ and $f_{\psi}$ are smooth.
    Then, we have the following equation
  \begin{equation}\label{smooth case}
  F\wedge(\psi'-c_{\psi}F-\widetilde{\mathcal{D}}^+_J(f_{\psi}))=0.
  \end{equation}
   If $\psi$ is not smooth, in $\Lambda^{1,1}_\mathbb{R}\otimes L^2(M)$,
   we still have
   $$
  \psi=c_{\psi}F+\beta_{\psi}+d^-_g(\gamma_\psi)+\widetilde{\mathcal{D}}^+_J(f_{\psi}),
   $$
    where $\beta_{\psi}\in\mathcal{H}^-_g(M)$, $c_\psi$ is a constant, $f_\psi\in L^2_2(M)_0$,
     $\gamma_\psi\in\Lambda^1_\mathbb{R}\otimes L^2_1(M)$, and $d^-_g(\gamma_\psi)\in(\Lambda^{1,1}_\mathbb{R}\otimes L^2(M))^0_{w}$.
     Let  $ \psi'=c_{\psi}F+\beta_{\psi}+\widetilde{\mathcal{D}}^+_J(f_{\psi})$,
     then $\psi=\psi'+d^-_g(\gamma_\psi)$.
     Since $d^-_g(\gamma_\psi)\in(\Lambda^{1,1}_\mathbb{R}\otimes L^2(M))^0_{w}$, it is easy to see that
     $$
     \int_M\psi'\wedge d^-_g(\gamma_\psi)=0
     $$
     and
  \begin{eqnarray*}
     \int_M\psi^2&=& \int_M(\psi'+d^-_g(\gamma_\psi))^2 \\
     &=&\int_M\psi'^2 -\|d^-_g(\gamma_\psi)\|^2_{L^2(M)}.
  \end{eqnarray*}
  Also, we can find a smooth sequence of $\{f_{\psi,j}\}\subset C^\infty(M)_0$ such that
  $$\psi'_j=c_{\psi}F+\beta_{\psi}+\widetilde{\mathcal{D}}^+_J(f_{\psi,j})$$
  is converging to $\psi'$ in $L^2(M)$.
   By the above statement, we get the following lemma,
   \begin{lem}\label{3l7}
  If $\psi\in(\Lambda^{1,1}_\mathbb{R}\otimes L^2(M))^0_{w}$, then $\psi$ could be written as
  $$
  \psi=cF+\beta_\psi+\widetilde{\mathcal{D}}^+_J(f_\psi)+d^-_g(\gamma_\psi),
  $$
  where $f_\psi\in L^2_2(M)_0$, $\beta_{\psi}\in\mathcal{H}^-_g(M)$,
    $d^-_g(\gamma_\psi)\in(\Lambda^{1,1}_\mathbb{R}\otimes L^2(M))^0_{w}$,
    $\gamma_\psi\in\Lambda^1_\mathbb{R}\otimes L^2_1(M)$ and $c$ is a constant.
  Denote $\psi-d^-_g(\gamma_\psi)$ by $\psi'$.
  Then
   $$
     \int_M\psi^2=\int_M\psi'^2 -\|d^-_g(\gamma_\psi)\|^2_{L^2(M)},
  $$
  and there is a smooth sequence of $\{f_{\psi,j}\}\subset C^\infty(M)_0$ such that
  $$
  \psi'_j=cF+\beta_\psi+\widetilde{\mathcal{D}}^+_J(f_{\psi,j})
  $$
  is converging to $\psi'$ in $L^2$.
  \end{lem}

It is similar to the argument of Buchdahl in \cite{B}, we need the following lemmas and propositions,

  \begin{lem}\label{3l10}
  (cf. Lemma \rm{4} in \rm{\cite{B}}) If $\psi\in(\Lambda^{1,1}_\mathbb{R}\otimes L^2(M))^0_{w}$,
   then
  $$
  (\int_MF\wedge\psi)^2\geq(\int_MF^2)(\int_M\psi^2)
  $$
  with equality if and only if $\psi=cF+\widetilde{\mathcal{D}}^+_J(f)$ for some constant $c$ and some $f\in L^2_2(M)_0$.
  \end{lem}
  \begin{proof}
  Let
  $$c=\frac{1}{2}\int_MF\wedge\psi.$$
  By Lemma \ref{3l7}, we can get
  \begin{eqnarray*}
      \psi &=& \psi'+d^-_g(\gamma_\psi)\\
           &=& cF+\beta_\psi+\widetilde{\mathcal{D}}^+_J(f_\psi)+d^-_g(\gamma_\psi),
  \end{eqnarray*}
  where $f_\psi\in L^2_2(M)_0$, $\beta_{\psi}\in\mathcal{H}^-_g(M)$,
    $d^-_g(\gamma_\psi)\in(\Lambda^{1,1}_\mathbb{R}\otimes L^2(M))^0_{w}$
   and $\gamma_\psi\in\Lambda^1_\mathbb{R}\otimes L^2_1(M)$.
   Then $$P^+_g(\psi'-cF-\widetilde{\mathcal{D}}^+_J(f_\psi))=0.$$
  If $\psi'$ is smooth,
   there is a smooth solution $f_\psi$ to the equation
  $$F\wedge(\psi'-cF-\widetilde{\mathcal{D}}^+_J(f_\psi))=0.$$
   Hence,
  \begin{eqnarray*}
  \|\psi'-cF-\widetilde{\mathcal{D}}^+_J(f_\psi)\|^2_{L^2(M)} &=& -\int_M(\psi'-cF-\widetilde{\mathcal{D}}^+_J(f_\psi))^2 \\
  &=& -\int_M(\psi')^2+2c\int_MF\wedge\psi'-2c^2\\
  &=& -\int_M(\psi')^2+2c^2\\
  &=& -\int_M(\psi')^2+(\int_MF\wedge\psi')^2/(\int_MF^2).
  \end{eqnarray*}
  Since
  $$\|\psi'-cF-\widetilde{\mathcal{D}}^+_J(f_\psi)\|^2_{L^2(M)}\geq 0,$$
  we can easily get $$(\int_MF\wedge\psi')^2\geq (\int_MF^2)\int_M(\psi')^2.$$
 If $\psi'$ is not smooth, the inequality follows from smooth case after approximating $\psi'$ by using Lemma \ref{3l7}.
   Hence
   $$(\int_MF\wedge\psi)^2=(\int_MF\wedge\psi')^2\geq(\int_MF^2)\int_M(\psi')^2\geq(\int_MF^2)(\int_M\psi^2).$$

  Suppose $$(\int_MF\wedge\psi)^2=(\int_MF^2)(\int_M\psi^2).$$
  By Lemma \ref{3l7},
    $$\int_M\psi^2=\int_M(\psi')^2-\|d^-_g(\gamma_\psi)\|^2$$
  and
  \begin{eqnarray*}
     (\int_MF\wedge\psi)^2 &=& (\int_MF\wedge\psi')^2 \\
       &\geq& (\int_MF^2)\int_M(\psi')^2 \\
       &\geq& (\int_MF^2)(\int_M\psi^2),
   \end{eqnarray*}
  which implies that
      \begin{equation}\label{equality}
   d^-_g(\gamma_\psi)=0, \,\,\,\,\,(\int_MF\wedge\psi')^2=(\int_MF^2)\int_M(\psi')^2\,\,\,\,\,
   {\rm and}\,\,\, \,\,\psi=\psi'.
      \end{equation}
  By $$(\int_MF\wedge\psi')^2=(\int_MF^2)\int_M(\psi')^2,$$ we have
 $4c^2=4c^2-2\|\beta_\psi\|^2_{L^2(M)}$, which implies that $\beta_\psi=0$.
  Hence, $\psi=cF+\widetilde{\mathcal{D}}^+_J(f_\psi)$.
   \end{proof}

  By Lemma \ref{3l7}, we have the following proposition.
  \begin{prop}\label{3p11}
  Let $\psi_1,\,\,\,\psi_2\in(\Lambda^{1,1}_\mathbb{R}\otimes L^2(M))^0_{w}$
    and satisfy $$\int_M\psi^2_j\geq 0 \,\,\,and\,\,\, \int_MF\wedge\psi_j\geq 0$$ for $j=1,2$.
  Then
  $$
  \int_M\psi_1\wedge\psi_2\geq(\int_M\psi_1^2)^{\frac{1}{2}}(\int_M\psi_2^2)^{\frac{1}{2}},
  $$
  with equality if and only if $\psi_1$ and $\psi_2$ are linearly dependent modulo the image of $\widetilde{\mathcal{D}}^+_J$.
  \end{prop}
  \begin{proof}
   It can be assumed that $$a_j=\frac{1}{2}\int_MF\wedge\psi_j$$ are strictly positive for $j=1,2$ else $\psi_j$ are $\widetilde{\mathcal{D}}^+_J$-exact for $j=1,2$.
  Indeed, if $a_j=0$ for $j=1,2$, then by Lemma \ref{3l7}, we have
  \begin{eqnarray*}
      \psi_j &=& \psi'_j+d^-_g(\gamma_{\psi_j})\\
           &=& \beta_{\psi_j}+\widetilde{\mathcal{D}}^+_J(f_{\psi_j})+d^-_g(\gamma_{\psi_j}),
  \end{eqnarray*}
  where $f_{\psi_j}\in L^2_2(M)_0$, $\beta_{\psi_j}\in\mathcal{H}^-_g(M)$,
    $d^-_g(\gamma_{\psi_j})\in(\Lambda^{1,1}_\mathbb{R}\otimes L^2(M))^0_{w}$
   and $\gamma_{\psi_j}\in\Lambda^1_\mathbb{R}\otimes L^2_1(M)$ for $j=1,2$.
  Hence $\psi'_j-\widetilde{\mathcal{D}}^+_J(f_{\psi_j})=\beta_{\psi_j}$ are anti-self-dual smooth harmonic $2$-forms, $j=1,2$.
  Then, by Lemma \ref{3l7},
  \begin{eqnarray*}
    0 &\geq& -\|\psi'_j-\widetilde{\mathcal{D}}^+_J(f_{\psi_j})\|^2_{L^2(M)} \\
     &=& \int_M(\psi'_j-\widetilde{\mathcal{D}}^+_J(f_{\psi_j}))^2 \\
     &=& \int_M(\psi'_j)^2\\
     &=& \int_M\psi_j^2+\|d^-_g(\gamma_{\psi_j})\|^2_{L^2(M)}\geq 0,
  \end{eqnarray*}
  and it follows that $d^-_g(\gamma_{\psi_j})=0$, $\beta_{\psi_j}=0$ and $\psi_j=\psi'_j=\widetilde{\mathcal{D}}^+_J(f_{\psi_j})$ for $j=1,2$.

  To prove the inequality, after replacing $\psi_j$ by $\psi_j+\varepsilon F$ and taking the limit as $\varepsilon\rightarrow 0$,
  it can be assumed that $$\int_M\psi^2_j>0$$ and $$a_j=\frac{1}{2}\int_MF\wedge\psi_j>0$$ for $j=1,2$.
  By Lemma \ref{3l7}, we have the following decompositions
         \begin{equation}\label{key deco}
              \psi_j= a_jF+\beta_{\psi_j}+\widetilde{\mathcal{D}}^+_J(f_{\psi_j})+d^-_g(\gamma_{\psi_j}),
            \end{equation}
  where $$f_{\psi_j}\in L^2_2(M)_0, \,\,\beta_{\psi_j}\in\mathcal{H}^-_g(M),\,\,
    d^-_g(\gamma_{\psi_j})\in(\Lambda^{1,1}_\mathbb{R}\otimes L^2(M))^0_{w},$$
    and $\gamma_{\psi_j}\in\Lambda^1_\mathbb{R}\otimes L^2_1(M)$ for $j=1,2$.

   By (\ref{key deco}), we have
   \begin{equation}\label{3eq17}
      a_2\psi_1-a_1\psi_2=a_2\beta_{\psi_1}-a_1\beta_{\psi_2}+\widetilde{\mathcal{D}}^+_J(a_2f_{\psi_1}-a_1f_{\psi_2})+d^-_g(a_2\gamma_{\psi_1}-a_1\gamma_{\psi_2}).
  \end{equation}
  It follows that
   $$
    a_2\psi_1-a_1\psi_2-\widetilde{\mathcal{D}}^+_J(a_2f_{\psi_1}-a_1f_{\psi_2})=(a_2\beta_{\psi_1}-a_1\beta_{\psi_2})+d^-_g(a_2\gamma_{\psi_1}-a_1\gamma_{\psi_2})
    $$
   is an anti-self-dual $2$-form.
   So
               \begin{eqnarray*}
                  0 &\geq&  -\|a_2\beta_{\psi_1}-a_1\beta_{\psi_2}\|^2_{L^2(M)}-\|d^-_g(a_2\gamma_{\psi_1}-a_1\gamma_{\psi_2})\|^2_{L^2(M)}\\
                       &=& \int_M(a_2\psi_1-a_1\psi_2-\widetilde{\mathcal{D}}^+_J(a_2f_{\psi_1}-a_1f_{\psi_2}))^2  \\
                   &=& \int_M(a_2\psi_1-a_1\psi_2)^2\\
                   &=& a_2^2\int_M\psi_1^2+a_1^2\int_M\psi_2^2-2a_1a_2\int_M\psi_1\wedge\psi_2 \\
                   &\geq& 2a_1a_2(\int_M\psi_1^2)^{\frac{1}{2}}(\int_M\psi_2^2)^{\frac{1}{2}}-2a_1a_2\int_M\psi_1\wedge\psi_2,
                \end{eqnarray*}
  giving the desired inequality $$\int_M\psi_1\wedge\psi_2\geq(\int_M\psi_1^2)^{\frac{1}{2}}(\int_M\psi_2^2)^{\frac{1}{2}}.$$
  If
  \begin{equation}
  \int_M\psi_1\wedge\psi_2=(\int_M\psi_1^2)^{\frac{1}{2}}(\int_M\psi_2^2)^{\frac{1}{2}},
  \end{equation}
  we obtain that
  $a_2\beta_{\psi_1}-a_1\beta_{\psi_2}=0$ and $d^-_g(a_2\gamma_{\psi_1}-a_1\gamma_{\psi_2})=0$.
  Hence, by (\ref{3eq17}),
   we get $$a_2\psi_1-a_1\psi_2=\widetilde{\mathcal{D}}^+_J(a_2f_{\psi_1}-a_1f_{\psi_2}).$$
  This completes the proof of Proposition \ref{3p11}.
 \end{proof}

  It is easy to see the following corollary,

 \begin{col}\label{3c12}
 If $\psi\in(\Lambda^{1,1}_\mathbb{R}\otimes L^2(M))^0_{w}$ and satisfies
 $$\int_M\psi^2>0 \,\,\,and \,\,\,\int_M\psi\wedge F>0,$$
 then $$\int_M\psi\wedge\varphi>0$$ for any other such form $\varphi\in(\Lambda^{1,1}_\mathbb{R}\otimes L^2(M))^0_{w}$
  satisfying $$\int_M\varphi^2\geq 0\,\,\, and \,\,\,\int_M\varphi\wedge F>0.$$
 \end{col}
 In order to get the desired key lemma (Lemma \ref{3l13}), we need the following technical lemma,
 \begin{lem}\label{particular case}
  If $h^-_J=b^+-1$, then
  $$
  (\Lambda^{1,1}_\mathbb{R}\otimes L^2(M))^0_{w}=(\Lambda^{1,1}_\mathbb{R}\otimes L^2(M))_{w}.
  $$
 \end{lem}
  \begin{proof}
   It is clear that $(\Lambda^{1,1}_\mathbb{R}\otimes L^2(M))^0_{w}\subset(\Lambda^{1,1}_\mathbb{R}\otimes L^2(M))_{w}$.
   For any $\varphi\in(\Lambda^{1,1}_\mathbb{R}\otimes L^2(M))_{w}$, set $$c=\frac{1}{2}\int_MF\wedge\varphi$$
   and let $\widetilde{\varphi}=\varphi-cF$.
   Then we will find that $$\int_M\widetilde{\varphi}\wedge \omega_1=\int_M\widetilde{\varphi}\wedge F=0.$$
   Thus, $P^+_g(\widetilde{\varphi})\bot\mathcal{H}^+_g$ since $h^-_J=b^+-1$,
    that is, $$\mathcal{H}^+_g=Span\{\omega_1,\alpha_1,\cdot\cdot\cdot,\alpha_{h^-_J}\}.$$
  $\varphi=cF+\widetilde{\varphi}\in(\Lambda^{1,1}_\mathbb{R}\otimes L^2(M))^0_{w}$.
   Hence $(\Lambda^{1,1}_\mathbb{R}\otimes L^2(M))^0_{w}=(\Lambda^{1,1}_\mathbb{R}\otimes L^2(M))_{w}$.
  \end{proof}

   With Corollary \ref{3c12} and Lemma \ref{particular case}, as done in the proof of Lemma 7 in \cite{B}, we can get the following key lemma,

  \begin{lem}\label{3l13}
  (Compare Lemma {\rm7} in {\rm\cite{B}})
  Let $(M,J)$ be a closed tamed almost complex 4-manifold with $h^-_J=b^+-1$.
  Suppose $\varphi\in(\Lambda^{1,1}_\mathbb{R}\otimes L^2(M))_{w}$ and satisfies
  $$
   \int_M\varphi\wedge F\geq 0\,\,\,{\rm and} \,\,\,\int_M\varphi^2\geq 0.
   $$
   For each $\varepsilon>0$ there is a positive $(1,1)$-form $p_\varepsilon$ and a function $f_\varepsilon$
   such that
  $$
  \|\varphi+\widetilde{\mathcal{D}}^+_J(f_\varepsilon)-p_\varepsilon\|_{L^2(M)}<\varepsilon.
  $$
   Moreover, $p_\varepsilon$ and $f_\varepsilon$ can be assumed to be smooth.
   \end{lem}
   \begin{proof}
  Since $h^-_J=b^+-1$, by Lemma \ref{particular case}, we can get $\varphi\in(\Lambda^{1,1}_\mathbb{R}\otimes L^2(M))^0_{w}$.
   If $$\int_M\varphi\wedge F=0,$$ by Lemma \ref{3l7}, it follows that
  \begin{equation}
     \varphi=\beta_{\varphi}+\widetilde{\mathcal{D}}^+_J(f_{\varphi})+d^-_g(\gamma_{\varphi}).
  \end{equation}
   Then
  $$
  0\geq -\|\beta_{\varphi}\|^2_{L^2(M)}-\|d^-_g(\gamma_{\varphi})\|^2_{L^2(M)}
    =\int_M(\varphi-\widetilde{\mathcal{D}}^+_J(f_{\varphi}))^2=\int_M\varphi^2\geq 0,
  $$
  and we can get $\varphi=\widetilde{\mathcal{D}}^+_J(f_{\varphi})$, that is,
  $\varphi$ is $\widetilde{\mathcal{D}}^+_J$ exact.
  In this case the result follows from the denseness of the smooth functions in $L^2_2(M)_0$.

   \vskip 6pt

  We may assume without loss of generality that $$\int_M\varphi\wedge F>0.$$
  After rescaling $\varphi$ if necessary, it can be supposed that
  $$\int_M\varphi\wedge F=1.$$
  Let
   \begin{equation}\label{3eq20}
  \mathcal{P}:=\{p\in\Lambda^{1,1}_\mathbb{R}\otimes L^2(M)\mid  p\geq 0,\,\,\, a.e.,\,\,\, \int_Mp\wedge F=1\};
   \end{equation}
  \begin{equation}\label{3eq21}
  \mathcal{P}_\varepsilon:=\{\rho\in\Lambda^{1,1}_\mathbb{R}\otimes L^2(M)\mid  \|\rho-p\|_{L^2(M)}<\varepsilon \,\,\, for\,\,\, some \,\,\,p\in\mathcal{P}\};
  \end{equation}
  \begin{equation}\label{3eq22}
  \mathcal{H}_\varphi:=\{\varphi+\widetilde{\mathcal{D}}^+_J(f)\mid  f\in L^2(M)_0\}.
  \end{equation}
  Then $\mathcal{P}_\varepsilon$ is an open convex subset of the Hilbert space $H:=\Lambda^{1,1}_\mathbb{R}\otimes L^2(M)$,
   and $\mathcal{H}_\varphi$ is a closed convex subset since $\widetilde{\mathcal{D}}^+_J$ has closed range by Lemma \ref{3l2}.
  If $\mathcal{P}_\varepsilon\cap\mathcal{H}_\varphi=\emptyset$,
  the Hahn-Banach Theorem implies that there exists $\phi\in H$ and a constant $c\in\mathbb{R}$ such that
  \begin{equation}\label{3eq23}
  \int_M\phi\wedge h\leq c, \,\,\,\int_M\phi\wedge p>c,
  \end{equation}
  for every $h\in\mathcal{H}_\varphi$,
   and every $p\in\mathcal{P}_\varepsilon$ (Compare Proof of Theorem I.7 in D. Sullivan \cite{S3} and Proof of Lemma 7 in N. Buchdahl \cite{B}).

   In terms of (\ref{3eq22}) and (\ref{3eq23}), there exists a $f_{\phi}\in L^2_2(M)_0$ such that $h_{\phi}=\varphi+\widetilde{\mathcal{D}}^+_J(f_{\phi})$ and
  $$\int_M\phi\wedge h_{\phi}=c,$$
  since $\mathcal{H}_\varphi$ is a closed space.
  Since $h\in\mathcal{H}_\varphi$, it follows that $h-h_{\phi}$ is in the image of $\widetilde{\mathcal{D}}^+_J$.
  Hence,
  \begin{equation}\label{3eq24}
  \int_M\phi\wedge(h-h_{\phi})\leq 0,\,\,\,\int_M\phi\wedge(h_{\phi}-h)\geq 0.
   \end{equation}
  It follows immediately that $\phi$ is weakly $\widetilde{\mathcal{D}}^+_J$-closed, that is,
  $$\int_M\phi\wedge\widetilde{\mathcal{D}}^+_J(f)=0$$
  for any $f\in L^2_2(M)_0$.
  By Lemma \ref{particular case}, $\phi\in(\Lambda^{1,1}_{\mathbb{R}}\otimes L^2(M))^0_{w}$ since $h^-_J=b^+-1$.

  Let $$\phi_0:=\phi-cF\in(\Lambda^{1,1}_{\mathbb{R}}\otimes L^2(M))^0_{w},$$
  then by (\ref{3eq20}) and (\ref{3eq23}), we have
  \begin{equation}\label{3eq26}
  \int_M\phi_0\wedge\varphi\leq c-c=0
   \end{equation}
  and
  \begin{equation}\label{3eq27}
  \int_M\phi_0\wedge p_0>0
  \end{equation}
   for any $p_0\in\mathcal{P}$.
  So $\phi_0$ is strictly positive almost everywhere.
  Hence
  $$\int_M\phi_0^2>0\,\,\, and \,\,\, \int_M\phi_0\wedge F>0.$$
  It follows from Corollary \ref{3c12} that
  $$\int_M\phi_0\wedge \varphi>0,$$
   giving a contradiction (see (\ref{3eq26})).
  Therefore $\mathcal{P}_\varepsilon\cap\mathcal{H}_\varphi$ can not be empty proving the existence of $p_\varepsilon$ and $f_\varepsilon$.
   The last statement of the lemma follows from denseness of the smooth positive $(1,1)$-forms
    in the $L^2$-positive $(1,1)$-forms and of the smooth functions in $L^2_2(M)_0$. This completes the proof of Lemma \ref{3l13}.
  \end{proof}

  In next section, we will devote to proving main theorem, i.e. Theorem \ref{1t1}.
  The proof of Theorem \ref{1t1} follows mainly Buchdahl's unified proof of the Kodaira conjecture.

\section{The tamed almost complex $4$-manifolds with $h^-_J=b^+-1$}\label{last section}
 \setcounter{equation}{0}
 This section is devoted to proving Theorem \ref{1t1} which follows mainly Buchdahl's unified proof of Kodaira conjecture..
 Throughout this section, we assume that $(M,J)$ is a closed tamed almost complex 4-manifold with
 $h^-_J= b^+-1$.
 Without loss of generality, we may assume that $J$ is tamed by a symplectic form $\omega_1=F+d^-_J(v+\bar{v})$, where $F$ is a fundamental 2-form,
 $$F^2>0,\,\,\, \int_MF^2=2,\,\,\, \int_Md_J^-(v+\bar{v})\wedge d_J^-(v+\bar{v})=2a>0,\,\,\, v\in\Omega^{0,1}_J.$$
 Set $g(\cdot,\cdot)=F(\cdot,J\cdot)$ that is an almost Hermitian metric on $(M,J)$.
 Denote by $d\mu_g$ the volume form defined by $g$.
 Set $\widetilde{\omega}_1=\omega_1-d(v+\bar{v})=F-d^+_J(v+\bar{v})\in\mathcal{Z}^+_J$,
 \begin{equation}\label{4eq1}
 \int_M\omega^2_1=2(1+a)=\int_M\widetilde{\omega}_1^2.
 \end{equation}
 It is easy to see that
 \begin{equation}\label{4eq2}
 \int_Md_J^+(v+\bar{v})\wedge d_J^+(v+\bar{v})=-2a,
 \end{equation}
 \begin{equation}\label{4eq3}
 \int_MF\wedge d_J^+(v+\bar{v})=-2a.
 \end{equation}
 From Section \ref{3}, we know that $\widetilde{\omega}_1$ is in $\mathcal{Z}^+_J$ and cohomologous to $\omega_1$ \,\,satisfying $$\int_M\widetilde{\omega}_1^2=2(1+a),\,\,\, \int_M\widetilde{\omega}_1\wedge F=2(1+a).$$
 By Lemma \ref{particular case}, since $h^-_J=b^+-1$, we have that $\widetilde{\omega}_1\in(\Lambda^{1,1}_\mathbb{R}\otimes L^2(M))^0_{w}$.
 Let $\phi=\widetilde{\omega}_1-(1+a)F$,
 it is easy to see that
 $$
 \int_MP^+_g(\phi)\wedge\omega_1=\int_M\phi\wedge\omega_1=0.
 $$
 Hence $P^+_g(\phi)$ is orthogonal to $\mathcal{H}_g^+(M)$ with respect to the integration since $h^-_J=b^+-1$.
  Moreover, note that both $F$ and $\widetilde{\omega}_1$ are weakly $\widetilde{\mathcal{D}}^+_J$-closed,
  so $\phi$ is weakly $\widetilde{\mathcal{D}}^+_J$-closed.

 For
 $$0<t_0=1+a-\sqrt{(1+a)^2-(1+a)}=(1+\sqrt{\frac{a}{1+a}})^{-1}<1,$$
 the smooth $(1,1)$-form
 $$\varphi=\widetilde{\omega}_1-t_0F=(\sqrt{a(1+a)}-a)F-d^+_J(v+\bar{v})$$
 is still in $(\Lambda^{1,1}_\mathbb{R}\otimes L^2(M))^0_{w}$.
 \begin{eqnarray*}
   \int_M\varphi^2&=& 2(\sqrt{a(1+a)}-a)^2+4(\sqrt{a(1+a)}-a)a-2a \\
   &=&  2a(1+a)-4a\sqrt{a(1+a)}+2a^2+4a\sqrt{a(1+a)}-4a^2-2a\\
 &=&0,
 \end{eqnarray*}
 \begin{eqnarray*}
   \int_MF\wedge\varphi&=& 2(\sqrt{a(1+a)}-a)+2a \\
   &=&  2\sqrt{a(1+a)}>0.
 \end{eqnarray*}
 By Lemma \ref{3l13}, for each $m\in\mathbb{N}$ there is a smooth positive $(1,1)$-form $p_m$ and a smooth function $f_m\in C^\infty(M)_0$ such that
 $$\|\varphi+\widetilde{\mathcal{D}}^+_J(f_m)-p_m\|_{L^2}<\frac{1}{m}.$$
 Since
 \begin{eqnarray*}
   \int_Mp_m\wedge F&=& -\int_M(\varphi+\widetilde{\mathcal{D}}^+_J(f_m)-p_m)\wedge F+\int_M(\varphi+\widetilde{\mathcal{D}}^+_J(f_m))\wedge F \\
   &=&-\int_M(\varphi+\widetilde{\mathcal{D}}^+_J(f_m)-p_m)\wedge F+\int_M\varphi\wedge F\\
 &=&-\int_M(\varphi+\widetilde{\mathcal{D}}^+_J(f_m)-p_m)\wedge F+2\sqrt{a(1+a)}
 \end{eqnarray*}
 and
 \begin{equation}\label{4eq4}
 |-\int_M(\varphi+\widetilde{\mathcal{D}}^+_J(f_m)-p_m)\wedge F|\leq\|\varphi+\widetilde{\mathcal{D}}^+_J(f_m)-p_m\|_{L^2}\|F\|_{L^2}<\frac{\sqrt{2}}{m},
 \end{equation}
 the integral
 $$\int_Mp_m\wedge F$$ is converging to $2\sqrt{a(1+a)}>0$ and
  by Lemma \ref{2l10} and Lemma \ref{3l1},
  the positive functions $(\wedge p_m)^{\frac{1}{2}}$ are uniformly bounded in $L^2$,
  where $\wedge:\Omega_{\mathbb{R}}^{1,1}\longrightarrow\Omega^0_{\mathbb{R}}$ is an algebraic operator in Lefschetz decomposition (cf. \cite{GH}).
 So a subsequence can be found converging weakly in $L^2$.
 The forms $p_m/(\wedge p_m)$ are bounded in $L^\infty$, so subsequence of these forms can also be found converging weakly in $L^4$.
 The sequence $\{\widetilde{\mathcal{D}}^+_J(f_m)\}=\{d\widetilde{\mathcal{W}}(f_m)\}$ is uniformly bounded in $L^1$.
 The uniform $L^1$ bound on $\widetilde{\mathcal{D}}^+_J(f_m)$ does not imply an $L^2$ bound on $f_m$,
 it really needed to find a subsequence converge in the sense of currents.
 Hence, we have the following claim.

 \begin{claim}\label{weak strong}
 Given any $s<\frac{4}{3}$ and $t<2$, there is a subsequence of $\{f_m\}$ that converges weakly in $L^s_1$, and
 strongly in $L^t$ to a limiting function $f_0$.
 \end{claim}
 \begin{proof}
 If $J$ is integrable, $\mathcal{D}^+_J=\sqrt{-1}\partial_J\bar{\partial}_J$.
 Xiaowei Xu \cite{X} pointed out that the uniform $L^1$ bound on $\sqrt{-1}\partial_J\bar{\partial}_J(f_m)$
 does not imply an $L^2$ bound on $f_m$.
 It means that in Buchdahl \cite[p.296]{B} there exists a gap.
 Buchdahl gave a new argument (cf. X. Xu \cite{X}).
 In the follows, we will give a proof of the above claim
 which follows the argument of N. Buchdahl (cf. X. Xu \cite{X}).

 Since $h^-_J=b^+-1$, $J$ is tamed by $\omega_1=F+d^-_J(v+\bar{v})$,
 by Proposition \ref{3p6'},
 $$\widetilde{\mathcal{D}}^+_J(f_m)=d\widetilde{\mathcal{W}}(f_m)=2 dd^*\mathbb{G}(f'_mF)=2 d\mathbb{G}d^*(f'_mF)$$
 and
 $$
 P^+_g\widetilde{\mathcal{D}}^+_J(f_m)=2P^+_gdd^*\mathbb{G}(f'_mF)=\Delta_g\mathbb{G}(f'_mF)=f'_mF,
 $$
 where $f'_m\in L^2(M)_0$ and $\mathbb{G}$ is the Green operator associated to $\Delta_g$ (cf. \cite{Kod}).
 First, take any real number $t'>2$ and let $h$ be any function in $L^{t'}(M)_0$, that is, $$\int_Mhd\mu_g=0$$ and $h\in L^{t'}(M)$,
 so $$hF^2=2P^+_gdd^*\mathbb{G}(f'_mF)\wedge F=  \Delta_g\mathbb{G}(f'_mF)\wedge F$$
 and
 $\mathbb{G}(hF)\in L^{t'}_2$.
 This is standard linear elliptic theory.
 By the Sobolev embedding theorem, the fact $t'>2$ implies that $L^{t'}_2$ is compactly embedded in $C^0$,
 so there is a uniform $C^0$ bound on $\mathbb{G}(hF)$ in terms of its $L^{t'}_2$ norm, and that in turn is uniformly bounded by a
 constant times the $L^{t'}$ norm of $2dd^*\mathbb{G}(hF)$
 by ellipticity and the fact that $hF$ has been chosen to orthogonal to the kernal in $L^2$.
 So the sup norm of $\mathbb{G}(hF)$ is bounded by a fixed constant times the $L^{t'}$ norm of $h$.
 Then
 \begin{eqnarray*}
   \int_Mf'_mhF^2 &=& \int_Mf'_mF\wedge hF \\
    &=& \int_Mf'_mF\wedge  \Delta_g\mathbb{G}(hF)\\
   &=&\int_M \Delta_g\mathbb{G}(f'_mF)\wedge hF \\
    &=&  \int_M 2d\mathbb{G}d^*(f'_mF)\wedge hF \\
    &=& \int_M\widetilde{\mathcal{D}}^+_J(f_m)\wedge hF.
 \end{eqnarray*}
 Since $p_m$ is uniformly bounded in $L^1$ and $\varphi+\widetilde{\mathcal{D}}^+_J(f_m)-p_m$ is converging to $0$
  in $L^2$, it follows that $\widetilde{\mathcal{D}}^+_J(f_m)$ is uniformly bounded in $L^1$.
  Therefore $$|\int_Mf'_mhF^2|\leq Const. \|h\|_{L^{t'}},$$
   which shows that the sequence $\{f'_m\}$ (resp. $\{f_m\}$) is weakly bounded in $L^t$,
  where $\frac{1}{t}+\frac{1}{t'}=1$.
  Since it is weakly bounded, it is bounded, and therefore we can find a subsequence converging weakly in $L^t$.
  We now have to do the same thing with the first derivatives.
  Recall that $$\widetilde{\mathcal{W}}(h)=2\mathbb{G}d^*(hF),\,\,
  \widetilde{\mathcal{D}}^+_J(h)=d\widetilde{\mathcal{W}}(h)=2d\mathbb{G}d^*(hF).$$
 Since
 \begin{eqnarray*}
    \int_M\widetilde{\mathcal{D}}^+_J(h)\wedge f_m\omega_1 &=& \int_Md\widetilde{\mathcal{W}}(h)\wedge f_m\omega_1  \\
    &=&  -\int_M\widetilde{\mathcal{W}}(h)\wedge df_m\wedge\omega_1.
 \end{eqnarray*}
 As done in Lemma \ref{3l2}, we can prove that $\widetilde{\mathcal{W}}(h)$ has closed range.
 This time we take any $\widetilde{\mathcal{W}}(h)$ that lies in $L^{s'}$ where $s'>4$.
 Then, following the same reason as above, we get $\{df_m\}$ uniformly bounded in $L^s$
 for $\frac{1}{s}+\frac{1}{s'}=1$ and therefore $\{df_m\}$ strongly bounded in $L^s$.
 We can then use the compactness part of the Sobolev embedding theorem to pick out a subsequence that
 converges strongly in $L^q$, where $q<2$.
 This completes the proof of the claim.
 \end{proof}
 By Claim \ref{weak strong}, the subsequence of positive $(1,1)$-forms $\{p_m\}$ in the sense of currents to define a positive
 $(1,1)$- current $p=\varphi+\widetilde{\mathcal{D}}^+_J(f_0)$, $f_0\in L^q_2(M)_0$ for some fixed $q\in (1,2)$.
 Note that since $\wedge p\in L^1$ and $p/(\wedge p)\in\Lambda^{1,1}_\mathbb{R}\otimes L^\infty$, the current $$P=p+t_0F=\widetilde{\omega}_1+\widetilde{\mathcal{D}}^+_J(f_0)$$ is a closed $(1,1)$-current which lies in $L^1$ satisfying $P\geq t_0F$.
 Thus, $P$ is called an almost K\"{a}hler current (cf. \cite{D2,HL3,HL4,HL5,LZ2,Pali,RT2,S3,T2}).
 In summary, we have the following proposition:
 \begin{prop}\label{4p1}
 (see Theorem {\rm11} in {\rm\cite{B}} and Lemma {\rm1.7} in {\rm\cite{S3}})
   Suppose that $(M,J)$ is a closed almost complex $4$-manifold with
   $h^-_J=b^+-1$ which is tamed by a symplectic form $\omega_1$.
 As defined the above,
 $$P=p+t_0F=\widetilde{\omega}_1+\widetilde{\mathcal{D}}^+_J(f_0)$$
 is a closed positive almost complex (1,1)-current in $L^1$ (almost K\"{a}hler current) and satisfies $P\geq t_0F$, where
 $f_0\in L^q_2(M)_0$ for some fixed $q\in(1,2)$ and
 $$0<t_0=(1+\sqrt{\frac{a}{1+a}})^{-1}<1.$$
 $P$ is homologous to $\widetilde{\omega}_1$ in the sense of current.
 \end{prop}
 \begin{rem}\label{4r2}
 (1)If $J$ is integrable, which is tamed by $\omega_1$, then $h^-_J=b^+-1$ since $\omega_1\in\mathcal{H}^+_g(M)$.
 By the Dolbeault decomposition (cf. Remark \ref{2r2}, or \cite{BHPV,DLZ1}),
 it is easy to see that $b^1=even$.
 On the other hand, for any compact complex surface, if $b^1=even$, then there exists a symplectic from $\omega$
 by which the integrable complex structure $J$ is tamed.
 Therefore, for any compact complex surface, $b^1=even$ if and only if there exists a symplectic form $\omega$ by which
 the integrable complex structure $J$ is tamed.
 Hence Theorem {\rm\ref{1t1}} is an affirmative answer to the Kodiria conjecture.
 The key ingredients in the unified proof of the Kodaira conjecture by N. Buchdahl in {\rm\cite{B}} are Theorem {\rm11} in {\rm\cite{B}} (i.e., Proposition {\rm\ref{4p1}}), Y.-T. Siu's theorem {\rm\cite{S2}} on the analyticity of the sets associated with the Lelong numbers of closed positive currents, and J.-P. Demailly's result {\rm\cite{D2}} on the smoothing of closed positive (1,1)-currents.

 (2)Taubes studies Donaldson's ``tamed to compatible" question in {\rm\cite{T2}}.
 He constructs an almost K\"{a}hler form in the class $[\omega]$ for a generic almost complex structure tamed by a symplectic form $\omega$ on a 4-manifold $M$ with $b^+=1$.
 To construct the almost K\"{a}hler form, Taubes' strategy is first to construct a closed positive $(1,1)$ current $\Phi_{\mathcal{K}}$ in class $[\omega]$ by irreducible $J$-holomorphic subvarieties.
 This special current satisfies
 $$\mathcal{K}^{-1}t^4<\Phi_{\mathcal{K}}(\sqrt{-1}f_B\sigma\wedge\overline{\sigma})<\mathcal{K}t^4,$$
 where $\mathcal{K}>1$ is a constant, $B$ is a ball of radius $t$, $\sigma$ denotes a unit length section of $\Lambda^{1,0}M\mid_B$ and $f_B$ denotes the characteristic function of $B$ (cf. Proposition {\rm1.3} in {\rm\cite{T2})}.
 To obtain a genuine almost K\"{a}hler form, Taubes smooths currents by a compact supported, closed 4-form on $TM$ which represents the Thom class in the compactly supported cohomology of $TM$ (cf. {\rm\S 1.6} of {\rm\cite{BT}}).
 \end{rem}

 M. Lejmi \cite{L2} shows that any almost complex manifold $(M,J)$ of dimension 4 has the local symplectic property, i.e. $\forall p\in M$, there is a local symplectic form $\omega_p=d{\tau_p}$ compatible with $J$ on a neighborhood, $U_p$, of $p$, where $\tau_p\in\Omega^1_{\mathbb{R}}|_{U_p}$. Note that as a trivial example, any complex manifold has the local symplectic property, hence almost complex manifolds with the local symplectic property can be regarded as a generalization of complex manifold.
 On the other hand, R. Bryant, M. Lejmi \cite{Br1,Br2,L2} showed that the almost complex structure underlying a non-K\"{a}hler, nearly K\"{a}hler 6-manifold ( in particular, the standard almost complex structure of $S^6$) can not be compatible with any symplectic form, even locally.
 Recall that for any closed positive $(1,1)$-current on an analytic variety, one can define Lelong number (cf. \cite{D3,GH,K}).
 By using locally symplectic form $\omega_p$, we will define Lelong number for any closed positive almost complex (1,1)-current on an almost Hermitian 4-manifold $(M,g,J,F)$ in Appendix \ref{Lelong} (cf. \cite{D4,Elk2,HL4,HL5,HL3,LZ2,RT2,Zha}).

 \vskip 6pt

 In the remainder of this section, we will devote to proving our main theorem (Theorem \ref{1t1}).
   To prove Theorem \ref{1t1}, we will study strictly $J$-plurisubharmonic functions,
   closed strictly positive $(1,1)$-current $\widetilde{\mathcal{D}}^+_J(f)$,
    Lelong numbers, the decomposition theorem
    and the regularization of almost K\"{a}hler currents in appendices A, B, C. With the results in appendices,
 we now prove Theorem \ref{1t1} by the similar method in \cite{B}, in particular, by using Proposition \ref{4p1}.

 {\bf Proof of Theorem \ref{1t1}}: By Proposition \ref{4p1}, we have a positive $d$-closed almost complex $(1,1)$-current
 \begin{equation}\label{4eq5}
 P=p+t_0F=\widetilde{\omega}_1+\widetilde{\mathcal{D}}^+_J(f_0)\geq t_0F
 \end{equation}
 on $(M,g,J,F)$ which is tamed by the symplectic form $\omega_1$, it follows that
  $P$ is an almost K\"{a}hler current and $SuppP=M$.
  To complete the proof of Theorem \ref{1t1}, by using the almost K\"{a}hler current $P$ we will construct an almost K\"{a}hler form.
 Let $\nu_1(x,P)$ denote the Lelong number of $P$ at $x$ defined as follows: If $x\in supp~P$, we define $$\nu_1(x,\omega_1,r,P)=\int_{B(x,r)}P\wedge\omega_1,$$
 where $B(x,r):=\{ y\in M | \rho_g(x,y)\leq r \}$, $\rho_g(x,y)$ is the geodesic distance of points $x,y$ with respect to the almost K\"ahler metric $g$.
  And $$\nu_1(x,P)=\lim_{r\rightarrow 0}r^{-2}\nu_1(x,\omega_1,r,P).$$
  For more details, see Definition \ref{Lelong 1} in Appendix \ref{Lelong}.
 For $c>0$, the upperlevel set
 \begin{equation}\label{4eq6}
 E_c(P):=\{x\in M \, |\,\nu_1(x,P)\geq c \}
 \end{equation}
 is a $J$-analytic subset (cf. Appendix \ref{Lelong} or \cite[Definition 2]{Elk2} for the definition) of $M$ of dimension (complex) $\leq 1$
  by the decomposition theorem (will be proven in Appendix \ref{Siu}, see Theorem \ref{Theorem A} and Remark \ref{7.6})
 which is analogous to Siu's Decomposition Formula \cite{S2}.

 \vskip 6pt

  By F. Elkhadhra's result (see Theorem $2$ in \cite{Elk2} or Lemma \ref{lemma2} in Appendix \ref{Lelong}),
  if $D$ is an irreducible $J$-holomorphic curve in $E_c(P)$, $$\nu_0:=inf\{\nu_1(x,P)\, |\, x\in D\},\,\,\, \nu_1(x,P)=\nu_0$$
   for almost all $x\in D$.
 If $D_1,\cdot\cdot\cdot,D_n$ are the irreducible $J$-holomorphic curves in $E_c(P)$ and $$\nu_i:=inf\{\nu_1(x,P)\, |\, x\in D_i\},$$ the $d$-closed $(1,1)$-current
 \begin{equation}\label{decom of current}
  T=P-\Sigma\nu_iT_{D_i}
 \end{equation}
 is positive and the $c$-upperlevel set $E_c(T)$ of this current are isolated singular points by Theorem \ref{Theorem A} and Remark \ref{7.6} in Appendix \ref{Siu}
  as in classical complex analysis.
  Here $T_{D_i}$ are the currents of integration on $D_i$.

  \vskip 6pt

  As done in \cite{B}, it is always possible to approximate the closed positive current $T$ by smooth real currents
  admitting a small negative part and that this negative part can be estimated in terms of the Lelong numbers of $T$
  and the geometry of $(M,g,J,F)$ (cf. Theorem \ref{theorem B} and Remark \ref{last remar} in Appendix \ref{Approximation}).
  Fix a number $K\geq 0$ such that the (1,1) curvature form, $R^{\nabla^1}$, of the second canonical connection $\nabla^1$ with respect to the metric
  $g$ (cf. \cite{G2})
 on $TM$ satisfies $R^{\nabla^1}\geq-KF\otimes Id_{TM}$
 and let $c>0$ be such that $t_0-cK>0$,
   where $R^{\nabla^1}=R^j_{ik\bar{l}}\theta^k\wedge\bar{\theta}^l$, $1\leq i,j,k,l\leq 2$,
   and $\{\theta^1,\theta^2\}$ is a coframe for $\Lambda^{1,0}_J$
    (see \cite{TWY} or Appendix \ref{Exponential}).
    Since the approximation theorem is locally proved, we can consider $J$-pseudoconvex domain.
   Notice that $(M,g,J,F)$ is a closed $\omega_1$-tamed almost Hermitian $4$-manifold which has the local
   symplectic property \cite{L2},
    hence for $\forall x \in M$, there is a neighborhood $U_x$ of $x$ and a $J$-compatible symplectic form $\omega_x$ on $U_x$
    such that $$\omega_x|_x=F|_x,\,\, F|_{U_x}=f_x\omega_x|_{U_x},$$ where $f_x\in C^\infty(U_x)$, $f_x(x)=1$.
    Fix a point $y\in U_x$. We may assume that $r$ is small enough  such that $B(y,r)\subset U_x$.
    On symplectic $4$-manifold $(U_x,\omega_x)$, we can define Lelong number for closed positive $(1,1)$-current on
    $(U_x,\omega_x)$
    $$
    \nu_2(y,\omega_x,r,T)=\frac{2}{r^2}\int_{B(y,r)}T\wedge\omega_x
    $$
    and
    $$
    \nu_2(y,x,T)=\lim_{r\rightarrow 0}\nu_2(y,\omega_x,r,T).
    $$
    Also we may assumed that $U_x$ is very samll and a strictly $J$-pseudoconvex domain,
     hence we can solve $\widetilde{\mathcal{W}},d^-_J$-problem on $U_x$
     (similar to $\bar{\partial}$-problem in classical complex analysis \cite{Hormander}).
     More details, see Appendix \ref{Hormander}.
     Thus, there exists a strictly $J$-plurisubharmonic function $f'_0$ on $U_x$ such that
     $$
      \widetilde{\mathcal{D}}^+_J(f_0)=\widetilde{\mathcal{D}}^+_J(f'_0),
     $$
     where $\widetilde{\mathcal{D}}^+_J(f_0)$ is defined in the equality (\ref{4eq5}),
     the solution $f'_0$ satisfies the above equation with respect to the metric $g_x(\cdot,\cdot)=\omega_x(\cdot,J\cdot)$.
     By Remark \ref{2r6}, $\widetilde{\mathcal{D}}^+_J(f'_0)=\mathcal{D}^+_J(f'_0)$ since $(U_x,g_x,J,\omega_x)$
     is an almost K\"{a}hler $4$-manifold.
     By Theorem \ref{comparison theo} in Appendix \ref{Lelong},
     $\nu_1(y,T)=f_x(y)\nu_2(y,x,T)$, $\forall y \in \sup T\cap U_x$.

     As done in classical complex analysis, using the regularization of almost K\"{a}hler currents (For more details, we refer to Appendix \ref{Demailly}, \ref{Approximation}. Notice that  Theorem \ref{theorem B} in Appendix \ref{Approximation} still holds for  $\widetilde{\mathcal{D}}^+_J(f_0)$
     since the approximation theorem is locally proved, see Remark \ref{last remar} in Appendix \ref{Approximation}.),
     there is a $1$-parameter family $T_{c,\varepsilon}$ of $d$-closed positive
     $(1,1)$-currents in the same homology class as $T=P-\Sigma\nu_iT_{D_i}$ in the sense of currents which weakly converges to $T$
     as $\varepsilon\rightarrow 0^+$, with $T_{c,\varepsilon}$ smooth off $E_c(T)$
     $$
     T_{c,\varepsilon}\geq (t_0-\min\{\lambda_{\varepsilon},c\}K-\delta_\varepsilon)F
     $$
     for some continuous functions $\lambda_\varepsilon$ on $M$ and constants $\delta_\varepsilon$ satisfying
     $\lambda_\varepsilon(x)\searrow\nu_1(x,T)$ for each $x\in M$ and $\delta_\varepsilon\searrow 0$ (see Buchdahl \cite[P.296]{B} or Appendix \ref{Demailly approximation}).
     Moreover, $\nu_1(x, T_{c,\varepsilon})=(\nu_1(x, T)-c)_+$ at each point $x$.
     For $\varepsilon$ sufficiently small therefore, $T_{c,\varepsilon}\geq t_1F$ for some $t_1>0$,
     where $t_1$ can be chosen arbitrarily close to $t_0$ if $c$ and $\varepsilon$ are small enough (see Buchdahl \cite[P.296]{B} or Appendix \ref{Demailly approximation}).

   The current $T_{c,\varepsilon}$ is smooth off the zero-dimensional singular set $E_c(T)$, that is, off a finite set of points since $M$ is compact.
   More details, see Appendix \ref{Siu decomposition}.
   Without loss of generality, we may assume that $E_c(T)=\{p_0\}$.
   There is a neighbourhood, $U_{p_0}$, of $p_0$ and a locally symplectic form $\omega_{p_0}=d\tau_{p_0}$ on $U_{p_0}$ that
     is compatible with $J|_{U_{p_0}}$, where $\tau_{p_0}\in\Omega^1_{\mathbb{R}}|_{U_{p_0}}$.
   Without loss of generality, we may assume that $U_{p_0}$ is $\omega_{p_0}$-convex which is also called $J$-pseudoconvex
   (for the definition of $J$-pseudoconvex we refer to Appendix \ref{psh}, and for more details,
      please see \cite{EG,HL1,Pali}).
   Moreover, we assume that $U_{p_0}$ is a strictly $J$-pseudoconvex domain in the almost complex $4$-manifold $(\mathbb{R}^4,J)$ (also see Appendix \ref{Hormander}).
  By Lemma \ref{current app} (which solves $\widetilde{\mathcal{W}},d^-_J$-problem),
    there exists a strictly $J$-plurisubharmonic function $f$ such that $T_{c,\varepsilon}=d\widetilde{\mathcal{W}}(f)=\widetilde{\mathcal{D}}^+_J(f)$ since $T_{c,\varepsilon}|_{U_{p_0}}$
    is a closed positive $(1,1)$-current.
    Also we have the following estimate (see Theorem \ref{app 1} in Appendix \ref{Hormander}):
   \begin{equation}\label{estimate f}
  \|f\|_{L^2(U_{p_0},\varphi)}\leq\frac{1}{\sqrt{c}}\|\widetilde{\mathcal{W}}(f)\|_{L^2(U_{p_0},\varphi)},
   \end{equation}
   where $\varphi$ is a strictly $J$-plurisubharmonic function
   satisfying $$\sum_{i,j}(\partial_{J_i}\bar{\partial}_{J_j}\varphi)\xi^i\bar{\xi}^j\geq c\sum_i|\xi^i|^2,$$ $\xi\in\mathbb{C}^2$.
    Note that when $U_{p_0}$ is very small, we can choose $\varphi=|z|^2=|z_1|^2+|z_2|^2$, $(z_1,z_2)\in\mathbb{C}^2$
   which is the Darboux coordinate chart on $(U_{p_0},\omega_{p_0})$ (see Proposition 6.4 in \cite{HL3}).
   Using a standard modifying function as in \cite[p.147]{GT},
    $f$ can be smoothed in a neighbourhood of $p_0$ to a family $f_t$ of smooth strictly $J$-plurisubharmonic functions converging to $f$.

    Recently, F.R. Harvey, H.B. Lawson JR. and S. Pli\'{s} got a result in \cite{HLP}  (see Theorem 4.1 in \cite{HLP} or Proposition \ref{plurisubharmonic prop}): Suppose $(X,J)$ is an almost complex manifold which is $J$-pseudoconvex, and let $f$ be a $J$-plurisubharmonic function on $(X,J)$.
    Then there exists a decreasing sequence $f_j$ of smooth strictly $J$-plurisubharmonic
         functions point-wise decreasing down to $f$.

    On an annular region surrounding $p_0$ the convergence of this sequence is uniform in $C^k$ for any $k$
   with respect to the almost K\"{a}hler metric $g'_J(\cdot,\cdot):=\omega_{p_0}(\cdot,J\cdot)$.
         (by Lemma 4.1 and the accompanying discussing in \cite{GT}).
       Choose two small neighbourhoods, $U'_{p_0}$ and $U''_{p_0}$ of $p_0$ satisfying $p_0\in U'_{p_0}\subset\subset U''_{p_0}\subset\subset U_{p_0}$.
   Construct a cut-off function:
    \begin{equation}
 \rho(x)=\left\{
  \begin{array}{ll}
  1 & ~ x\in M\setminus U_{p_0}'',  \\
  0 & ~ x\in \overline{U}_{p_0}'.
  \end{array}
 \right.
 \end{equation}
  It is clear that $\rho f+(1-\rho)f_t$ is a smooth strictly $J$-plurisubharmonic function for $t$ sufficiently small which agrees with $f$
    outside the annulus.
   Construct a smooth closed strictly positive $(1,1)$ form $\tau_{c,\varepsilon}$ for $\varepsilon>0$ as follows:
    \begin{equation}
 \tau_{c,\varepsilon}=\left\{
  \begin{array}{ll}
  T_{c,\varepsilon} &  {\rm on}~ M\setminus U_{p_0},  \\
  d\widetilde{\mathcal{W}}(\rho f+(1-\rho)f_t) & {\rm on}~ \overline{U}_{p_0}.
  \end{array}
 \right.
 \end{equation}
  Hence the current $T_{c,\varepsilon}$ is $\widetilde{\mathcal{D}}^+_J$-homologous to the smooth closed strictly positive $(1,1)$-form $\tau_{c,\varepsilon}$.
    Moreover, for $0<t_1<t_0$, there is some $c$ and $\varepsilon$ such that $\tau_{c,\varepsilon}\geq t_1F$ (see Buchdahl \cite[P.296]{B}).
  Thus, $\tau_{c,\varepsilon}$ is a smooth almost K\"{a}hler form on $(M,J)$.
    This completes the proof of Theorem \ref{1t1}. ~~~~$\Box$

   \vskip 6pt

  In the following three appendixes, we will discuss $J$-plurisubharmonic functions as in classical complex analysis,
  minimal principle for $J$-plurisubharmonic functions,
  Lelong numbers of closed positive $(1,1)$-currents on almost complex $4$-manifold,
  Siu's decomposition theorem for closed positive $(1,1)$-currents on tamed almost complex $4$-manifold and Demailly's regularization theorem for closed positive $(1,1)$-currents on tamed almost complex $4$-manifold.
  These notations and results extend various foundational notations and results from pluripotential theory,
   used in the main argument in Section \ref{last section}, to the almost-complex case.

\begin{appendices}

\section{Elementary pluripotential theory}
\renewcommand{\theequation}{A.\arabic{equation}}
\setcounter{equation}{0}
  This appendix is devoted to discussing $J$-plurisubharmonic functions,
   minimal principle for $J$-plurisubharmonic functions,
    $\mathcal{\widetilde{W}}$, $d^-_J$-problem on tamed almost complex $4$-manifolds,
    and the singularities of $J$-plurisubharmonic functions.

  \medskip

\subsection{$J$-plurisubharmonic functions on almost complex manifolds}\label{psh}

  In this subsection, we will discuss $J$-plurisubharmonic functions on almost complex manifolds as done in classical complex analysis.
 We will adopt classical notations from geometric measure theory \cite{DS,Elk,Elk2,HL3,HL4,HL5,IR,Pali,Rosay,Rosay2}.

 Let $(M,J)$ be an almost complex manifold of real dimension $2n$.
 We let $\mathcal{D}^{p,q}(M)$ denote the space of $C^\infty$ $(p,q)$-forms on $M$ with compact support
 and let $\mathcal{D}'^{p,q}(M)=\mathcal{D}^{n-p,n-q}(M)'$ be the space $(p,q)$-currents on $(M,J)$.
 We also let $\mathcal{E}^{p,q}(M)$ be the space of $C^\infty$ $(p,q)$-forms on $(M,J)$ and $\mathcal{E}'^{p,q}(M)=\mathcal{E}^{n-p,n-q}(M)'$ denote the space of compactly supported $(p,q)$-currents on $(M,J)$.
 Suppose $T\in\mathcal{D}'^{p,q}(M)$.
 We let Singsupp$T$ denote the smallest closed subset $A$ of $M$ such that $T$ is a smooth current on $M\setminus A$.
 For $\varphi\in\mathcal{D}^{n-p,n-q}(M)$, we let $(T,\varphi)=T(\varphi)$ denote the pairing of $T$ and $\varphi$.
 We note that if $M$ is a closed manifold and $T$, $\varphi$ is closed, then $(T,\varphi)=(T\cdot\varphi)$, where $(T\cdot\varphi)$ is the intersection number given by the cup-product (cf. \cite{BT,HL3,HL4,HL5,GH,Pali}).
 \begin{defi}\label{5A2}
   (cf. {\rm\cite{JS,Elk2}})
 (1) A real $(p,p)$-form on $(M,J)$ is strictly positive (positive) if it is strictly positive (positive) at each point.
          A real $(p,p)$-current $T$ on $M$ is positive if $(T,\varphi)$ is positive for all test strictly positive $(n-p,n-p)$-forms $\varphi$ on $(M,J)$.\\
 (2) A real $(p,p)$-current $T$ on $(M,J)$ is strictly positive if there is a strictly positive $(1,1)$-form $F$ on $(M,J)$such that
     $T-F^p$ is positive;
     $T$ is said to be strictly positive at a point $x\in M$ if there is a neighborhood $U$ of $x$ such that $T|_U$ is a strictly positive current on $U$.
 \end{defi}

 Note that $T$ is strictly positive on $(M,J)$ if and only if $T$ is strictly positive at each point of $M$.
 By the definition above, a smooth form is strictly positive (positive) as a form if and only if it is strictly positive (positive) as a current.
 If a $(p,p)$-current $T$ is strictly positive (positive), we write $T>0$ $(T\geq 0)$.
 We also write $S>T$ $(S\geq T)$ if $S-T>0$ $(S-T\geq0)$, for $(p,p)$-currents $S$, $T$.
 A strictly positive $(1,1)$-current on an almost complex manifold is called an almost K\"{a}hler current \cite{T2,Zha}
 (Since a strictly positive $(1,1)$-current on a complex manifold $(M,J)$ is called K\"{a}hler current \cite{D3,GH}.).

  \vskip 6pt

 In fact, for any real-valued $C^\infty$-function $u$ we have $\partial_J\bar{\partial}_Ju=-\bar{\partial}_J\partial_Ju$ (see (\ref{2eq9})).
  We can define the complex Hessian operator (cf. Harvey-Lawson \cite{HL3})
  $$\mathcal{H}: C^\infty(M)\rightarrow \Gamma(M,\Lambda^{1,1}_J)$$
  by $\mathcal{H}(u)(X,Y):=(\partial_J\bar{\partial}_Ju)(X,\overline{Y})$ for $X,Y\in TM^{1,0}$.
  The real form $H(u)$ of the complex Hessian $\mathcal{H}$ is given by the polarization of the real quadratic form
  $$
  H(u)(X,Y):={\rm Re}\mathcal{H}(u)(X-\sqrt{-1}JX,Y-\sqrt{-1}JY),
  $$
  where $X,Y\in TM$.
 Of course, it is enough to define the quadratic form
 $$
  H(u)(X,X):={\rm Re}\mathcal{H}(u)(X-\sqrt{-1}JX,X-\sqrt{-1}JX)
  $$
  for all real vector fields $X$ and it is a real-valued form.
  By a simple calculation (\cite[Lemma 4.1]{HL3}), we can obtain that $H(u)$ is given by
  $$
  H(u)(X,X)=\{XX+(JX)(JX)+J([X,JX])\}u
  $$
  defined for all $X\in TM$ (see \cite[Lemma 4.1]{HL3}).
  \begin{defi}
  (cf. {\rm Harvey-Lawson \cite{HL3}})
  A smooth function $u$ on $(M,J)$ is called
   $J$-plurisubharmonic if $H_x(u)\geq 0$ for each $x\in M$.
  \end{defi}
   This notion extends directly to the space of distributions by requiring $\sqrt{-1}\partial_J\bar{\partial}_Ju$ to be positive.
  The definition of $J$-plurisubharmonic function could be broadened to the space of upper semi-continuous functions on $M$ takinng values in
  $[-\infty,\infty)$.
  Denote by ${\rm USC(M)}$ the space of upper semi-continuous functions on $M$.
  A function $\varphi$ which is $C^2$ in a neighborhood of $x\in M$ is called a test function for $u\in {\rm USC(M)}$ at $x$
  if $u-\varphi\leq 0$ near $x$ and $u=\varphi$ at $x$.
  A function $u\in {\rm USC(M)}$ is called $J$-plurisubharmonic on $M$
  if for each $x\in M$ and each test function $\varphi$ for $u$ at $x$ we have $H_x(\varphi)\geq 0$.
  On the other hand, an upper semi-continuous function $u$ on $(M,J)$ is said to be $J$-plurisubharmonic in the standard sense
  if its restriction to each $J$-holomorphic curve in $(M,J)$ is subharmonic (for detials, see \cite{HL3,Pali}).
  If the function $u$ is of class $C^2$, there is a simple characterization.
  For any tangent vector field $X\in TM$ one must have
  \begin{equation}\label{defi psh}
     dd^c_Ju(X,JX)\geq 0,
  \end{equation}
  where the twisted exterior differential $d^c_J=(-1)^pJdJ$ acting on $p$-forms, in particular $d^c_Ju(X)=-du(JX)$.
  We say that a function $u$ of class $C^2$ is strictly $J$-plurisubharmonic if $dd^c_Ju(X,JX)> 0$.
  The manifold $(M,J)$ is said to be (strictly) $J$-pseudoconvex if it admits a smooth exhaustion function $\phi:M\rightarrow \mathbb{R}$
   which is (strictly) $J$-plurisubharmonic.
  If $J=J_{st}$ is the standard complex structure on $\mathbb{C}^n$
  , $d^c_{J_{st}}=d^c$.
  Moreover, we have the following integration by parts formula.

  \begin{prop}\label{integration by part}
    (cf. {\rm Demailly \cite[Formula 3.1 in Chapter 3]{D3}})
  Let $(M,J)$ be a closed almost complex $2n$-manifold and let $\alpha,\beta$ be smooth forms of pure bidegrees $(p,p)$ and $(q,q)$ with $p+q=n-1$.
  Then
  $$
  \int_M\alpha\wedge dd^c_J\beta-dd^c_J\alpha\wedge \beta=0.
  $$
  \end{prop}
  \begin{proof}
  Note that
  $$
  d(\alpha\wedge d^c_J\beta-d^c_J\alpha\wedge \beta)=\alpha\wedge dd^c_J\beta-dd^c_J\alpha\wedge \beta+(d\alpha\wedge d^c_J\beta+d^c_J\alpha\wedge d\beta).
  $$
  Hence, by Stokes' theorem, we get
  $$
  \int_M\alpha\wedge dd^c_J\beta-dd^c_J\alpha\wedge \beta=-\int_Md\alpha\wedge d^c_J\beta+d^c_J\alpha\wedge d\beta.
  $$
  As all forms of total degree $2n$ and bidegree$\neq(n,n)$ are zero, we have
  $$
  d\alpha\wedge d^c_J\beta=-\sqrt{-1}\cdot(\partial_J\alpha\wedge\bar{\partial}_J\beta-\bar{\partial}_J\alpha\wedge\partial_J\beta
  +A_J\alpha\wedge \bar{A_J}\beta-\bar{A_J}\alpha\wedge A_J\beta)
  $$
  and
   $$
 d^c_J\alpha\wedge d\beta=\sqrt{-1}\cdot(\partial_J\alpha\wedge\bar{\partial}_J\beta-\bar{\partial}_J\alpha\wedge\partial_J\beta
  +A_J\alpha\wedge \bar{A_J}\beta-\bar{A_J}\alpha\wedge A_J\beta),
  $$
  where $A_J$ and $\bar{A_J}$ are defined in Section $2$ (cf. (\ref{2eq4})).
  Therefore,
  $ d\alpha\wedge d^c_J\beta=-d^c_J\alpha\wedge d\beta$.
  \end{proof}

  By a simple calculation, we get
  $$dd^c_Ju=2\sqrt{-1}\partial_J\bar{\partial}_Ju+\sqrt{-1}(\bar{A_J}\bar{\partial}_Ju-\partial^2_Ju)+\sqrt{-1}(\bar{\partial}^2_Ju-A_J\partial_Ju)$$
  and
   $$d^c_Jdu=-2\sqrt{-1}\partial_J\bar{\partial}_Ju+\sqrt{-1}(\bar{A_J}\bar{\partial}_Ju-\partial^2_Ju)+\sqrt{-1}(\bar{\partial}^2_Ju-A_J\partial_Ju).$$
  Hence, a $C^2$ function $u$ is $J$-plurisubharmonic if and only if the $(1,1)$ part of $dd^c_Ju$ is positive.
  Harvey and Lawson have proven that the notion of $J$-plurisubharmonic is equivalent to the $J$-plurisubharmonic in the standard sense
   (cf. Harvey-Lawson \cite[Theorem 6.2]{HL3}).
   Harvey and Lawson also introduce the notion of Hermitian plurisubharmonic on an almost Hermitian manifold $(M,g,J)$.
   Denote the Riemannian Hessian operator by $$({\rm Hess}~u)(X,Y):= XYu-(\nabla_XY)u$$ for $X,Y\in TM$,
   where $\nabla$ is the Levi-Civita connection.
   A function $u\in C^\infty(M)$ is then defined to be Hermitian plurisubharmonic if ${\rm Hess^C}u\geq 0$,
   where
   $$
   ({\rm Hess^C}u)(X,Y):=({\rm Hess}~u)(X,Y)+({\rm Hess}~u)(JX,JY).
   $$
   In general, Hermitian plurisubharmonic does not agree with the standard $J$-plurisubharmonic (cf. Harvey-Lawson \cite[Section 9]{HL3}).
   But we have the following proposition proved by Harvey and Lawson:
   \begin{prop}\label{lawson}
   (cf. {\rm Harvey-Lawson \cite[Theorem 9.1]{HL3}})
   Let $(M,g,J)$ be an almost Hermitian manifold.
   If the associated K\"{a}hler form $\omega(\cdot,\cdot)=g(J\cdot,\cdot)$ is closed, that is, $(M,g,J,\omega)$ is almost K\"{a}hler,
   then the notion of Hermitian plurisubharmonic coincides with the notion of $J$-plurisubharmonic.
   \end{prop}

    Let $(M,g,J,\omega)$ be an almost K\"{a}hler manifold of (complex) dimension $n$.
   For any $p\in M$, assume $T_pM\cong\mathbb{C}^n$.
   Let $$
  B_1(p,\varepsilon_1):=\{\xi\in T_pM \mid |\xi|\leq \varepsilon_1\}
  $$
  and
   $$
  S_1(p,\varepsilon_1):=\{\xi\in T_pM \mid |\xi|=\varepsilon_1\}.
  $$
   Suppose that $\rho_g(p,q)$ is the geodesic distance of points $p$, $q$ with respect to $g$ (for details, see Chavel \cite{Cha}).
   Denote by
   $$
  B(p,\varepsilon_1):=\{q\in M \mid \rho_g(p,q)\leq \varepsilon_1\}
  $$
  and
  $$
  S(p,\varepsilon_1):=\{q\in M \mid  \rho_g(p,q)= \varepsilon_1\}.
  $$
   It is well known that for each $p\in M$, there exists $\varepsilon_2>0$ and a neighborhood $U$ of $p$ in $M$ such that
   for each $q\in U$, $\exp_q$ maps $B_1(p,\varepsilon_2)$ diffeomorphically onto an open set in $M$.
  Hence, for $\varepsilon_1<\varepsilon_2$, we have
  $$
  B(p,\varepsilon_1)=\exp B_1(p,\varepsilon_1)
  $$
  and
  $$
  S(p,\varepsilon_1)=\exp S_1(p,\varepsilon_1).
  $$
  Let $inj M$ be the injectivity radius of $M$ (for the detailed definition, we refer to Chavel \cite[Chapter III]{Cha}).
  \begin{prop}
  (cf. {\rm Chavel \cite[Theorem IX.6.1]{Cha}})
  Let $(M,g,J,\omega)$ be an almost K\"{a}hler manifold.
  Assume that the sectional curvature $K\leq \delta$ on $M$. Set $r=\min\{\frac{inj M}{2},\frac{\pi}{2\sqrt{\delta}}\}$,
  then $B(p,r)$ is strictly convex.
  \end{prop}

  Therefore, on an almost K\"{a}hler manifold with bounded geometry (cf. \cite{Cha}), a small geodesic ball is strictly convex.
  It is well known that one of the fundamental results of classical complex analysis establishes the equivalence between the holomorphic
  disc convexity of a domain in an affine complex space, the Levi convexity of its boundary
  and existence of a strictly plurisubharmonic exhaustion function.
  On the other hand, in the works of K. Diederich-A. Sukhov, Y. Eliashberg-M. Gromov, F.R. Harvey-H.B. Lawson, Jr \cite{DS,EG,HL4,HL5}
  and other authors, the convexity properties of strictly $J$-pseudoconvex domains in almost complex manifolds are substantially used give rise to many
  interesting results.
  Concerning symplectic structure, K. Diederich and A. Sukhov \cite[Theorem 5.4]{DS} obtained a characterization of
   $J$-pseudoconvex domain in almost complex manifolds similar to the classical results of complex analysis.
  Hence fix a point $p$, $\rho_g(p,q)$ is a strictly subharmonic function on $\{q \mid  \rho_g(p,q)<r\}$.

  \begin{claim}\label{claim 1}
   Let $(M,g,J,\omega)$ be an almost K\"{a}hler manifold of (complex) dimension $n$.
   For any $p\in M$,
 $\log \rho_g(p,q)$ is $J$-plurisubharmonic if $\rho_g(p,q)$ is small enough.
   \end{claim}

  We will prove the above claim later. Note that when we identify $\mathbb{R}^{2n}$ with $\mathbb{C}^n$.
  Chirka (unpublished) observed that if the almost complex structure $J$ defined
  in a neighborhood of $0$ coincides with the standard complex structure at $0$,
  then for $A>0$ large enough the function $z\rightarrow \log|z|+A|z|$ is $J$-plurisubharmonic near $0$, with
  $z=(z_1,\cdot\cdot\cdot,z_n)$ and $|z|=(|z_1|^2+\cdot\cdot\cdot+|z_n|^2)^{\frac{1}{2}}$.
  One should of course not expect the function $\log|z|$ to be $J$-plurisubharmonic,
  since it is not strictly plurisubharmonic for the standard complex structure, and
 hence even a small change of complex structure will not preserve plurisubharmonicity.
 The term $A|z|$ is a needed correction term.
 The computation is made in detail in Ivashkovich-Rosay \cite[Lemma 1.4]{IR}.
 Note that $J$-holomorphic curves are $-\infty$ sets of $J$-plurisubharmonic functions, with a singularity of $\log\log$ type (cf. Rosay \cite{Rosay}),
 but it is shown that in general they are not $-\infty$ set of $J$-plurisubharmonic functions with logarithmic singularity (cf. Rosay \cite{Rosay2}).

   \vskip 6pt

    Suppose that $(M,g,J)$  is an almost Hermitian $2n$-manifold.
   Let $\nabla^1$ be the second canonical connection satisfying $\nabla^1g=0$ and $\nabla^1J=0$ \cite{G2}.
   There exists a unique second canonical connection on almost Hermitian manifold $(M,g,J)$ whose torsion has everywhere vanishing $(1,1)$ part (cf.\cite{G2,TWY}).
   This connection was first introduced by Ehresmann and Libermann (cf. \cite{EL}).
   It is also sometimes referred to as the Chern connection,
   since when $J$ is integrable it coincides with the connection defined in \cite{Che}.
   Choose a local unitary frame $\{e_1,\cdot\cdot\cdot,e_n\}$ for $TM^{1,0}$ with respect to
   the Hermitian inner product $h=g-\sqrt{-1}\omega$, where $\omega(\cdot,\cdot)=g(J\cdot,\cdot)$,
   and let $\{\theta^1,\cdot\cdot\cdot,\theta^n\}$ be a dual coframe.
   The metric $h$ can be written as
  $$
  h=\theta^i\otimes\overline{\theta^i}+\overline{\theta^i}\otimes\theta^i.
  $$
    Let $\Theta$ be the torsion of the canonical almost Hermitian connection $\nabla^1$.
     Define functions $N^i_{\bar{j}\bar{k}}$ and $T^i_{jk}$ (cf. \cite{TWY}) by
  $$(\Theta^i)^{(0,2)}=N^i_{\bar{j}\bar{k}}\overline{\theta^j}\wedge \overline{\theta^k},$$
  $$(\Theta^i)^{(2,0)}=T^i_{jk}\theta^j\wedge \theta^k$$
  with $N^i_{\bar{j}\bar{k}}=-N^i_{\bar{k}\bar{j}}$ and  $T^i_{jk}=-T^i_{kj}$.

  It is not hard to obtain the following lemma:
  \begin{lem}
  (cf. {\rm \cite{G1,TWY,WZ2}})
   The $(0,2)$ part of the torsion is independent of the choice of metric.
   \end{lem}
   Consider the real $(1,1)$ form $\omega(\cdot,\cdot)=g(J\cdot,\cdot)$,
   $$\omega=\sqrt{-1}\sum^n_{i=1}\theta^i\wedge\overline{\theta^i}.$$
   We say that $(M,J,g,\omega)$ is almost K\"{a}hler if $d\omega=0$, and it is quasi K\"{a}hler
  if $(d\omega)^{(1,2)}=0$.
  An almost K\"{a}hler or quasi K\"{a}hler manifold with $J$ integrable is a K\"{a}hler manifold.
  \begin{lem}
  (cf. {\rm \cite{G1,TWY}})
  An almost Hermitian manifold $(M,g,J,\omega)$ is almost K\"{a}hler if and only if
  $$T^i_{kj}=0$$
  and
  $$N_{\bar{i}\bar{j}\bar{k}}+N_{\bar{j}\bar{k}\bar{i}}+N_{\bar{k}\bar{i}\bar{j}}=0,$$
  where $N_{\bar{i}\bar{j}\bar{k}}=N^i_{\bar{j}\bar{k}}$.
  $(M,g,J,\omega)$ is quasi K\"{a}hler if and only if
  $$T^i_{kj}=0.$$
  \end{lem}
  Notice that if $(M,g,J,\omega)$ is almost K\"{a}hler, then $(M,g,J,\omega)$ is quasi K\"{a}hler.

  Let $f$ be a smooth function on $M$.
  We define the canonical Laplacian $\Delta^1$ of $f$ by
  $$
  \Delta^1f=\sum_i(\nabla^1\nabla^1f(e_i,\bar{e}_i)+\nabla^1\nabla^1f(\bar{e}_i,e_i)).
  $$
  This expression is independent of the choice of unitary frame.
  By Lemma 2.5 in \cite{TWY},
  $$
  \Delta^1f=\sqrt{-1}\sum_i(dd_J^cf)^{(1,1)}(e_i,\bar{e}_i).
  $$
  \begin{lem}\label{Yau equivalence}
  (cf. {\rm \cite{G1,TWY}})
  If the metric $g$ is quasi-K\"{a}hler then the canonical Laplacian $\Delta^1$ is equal to the
  usual Laplacian, $\Delta_g$, of the Levi-Civita connection of $g$.
  \end{lem}

   \vskip 6pt

  Let us return to the proof of the above claim.

  \vskip 6pt

  {\bf Proof of Claim \ref{claim 1}}
  To verity that $\log \rho_g(p,q)$ is $J$-plurisubharmonic on almost K\"ahler manifold $(M,g,J,\omega)$, we introduce geodesic spherical coordinates about $p$ by defining
  $$
  V:[0,\varepsilon)\times T_pM\longrightarrow M
  $$
  by $V(s,X)=\exp sX$.
  For any $\xi\in S_p=S_1(p,1)$, denote by
  $$\xi^{\bot}:=\{\eta\in T_pM \mid  \langle\eta,\xi\rangle=0\}.$$
   Then the map $\eta\mapsto sF\eta$ is an isomorphism of $\xi^{\bot}$ onto $S_1(p,s)_{s\xi}$,
   where $F:T_pM\rightarrow (T_pM)_{s\xi}$ is the canonical isomorphism.
  Hence for any point $q'$ which lies in a small neighborhood of $p$,
  $q'$ could be written as
  $$q'=\exp s(\xi+\sum^{2n-1}_{i=1}\theta_ie_i),$$
  where $e_1,\cdot\cdot\cdot,e_{2n-1},\xi=e_{2n}\in T_pM$ is a local unitary orthogonal frame,
  and $Je_{2i-1}=e_{2i}$, $1\leq i\leq n$.
  Therefore, $$\rho_g(p,q')=\sqrt{s^2(1+\sum^{2n-1}_{i=1}\theta^2_i)}.$$
  Hence, when $s=t$, $\theta_i=0$, $i=1,2,\cdot\cdot\cdot,2n-1$
  \begin{eqnarray}\label{positive}
    \Delta_g\log \rho_g(p,q')|_q &=&\frac{1}{2}(\frac{\partial^2}{\partial s^2}+\frac{1}{t^2}\sum^{2n-1}_{i=1}\frac{\partial^2}{\partial\theta^2_i})\log s^2(1+\sum^{2n-1}_{i=1}\theta^2_i)|_{s=t,\theta_i=0} \nonumber\\
    &=& -\frac{1}{t^2}+\sum^{2n-1}_{i=1}\frac{1}{t^2}=\frac{2n-2}{t^2} .
  \end{eqnarray}
  Since we mainly consider it on almost K\"{a}hler $2n$-manifold, especially, on almost K\"{a}hler $4$-manifold,
  $\Delta_g\log \rho_g(p,q)\geq 0$.
  By Lemma \ref{Yau equivalence}, notice that an almost K\"{a}hler manifold is a quasi K\"{a}hler manifold, we have
  $$\Delta^1\log \rho_g(p,q)=\Delta_g\log \rho_g(p,q)\geq 0.$$
  Define $l_j$ to be the $J$-holomorphic curves spanned by $\{e_{2j-1},Je_{2j-1}\}$, $1\leq j\leq n$.
  Then, we have
  \begin{eqnarray}
    \Delta_g|_{l_n}\log \rho_g(p,q')|_q &=&\frac{-1}{s^2(1+\sum^{2n-1}_{i=1}\theta^2_i)}|_{s=t,\theta_i=0}\nonumber\\
      &&+
   [\frac{1}{s^2(1+\sum^{2n-1}_{i=1}\theta^2_i)}+\frac{-2\theta^2_{2n-1}}{s^2(1+\sum^{2n-1}_{i=1}\theta^2_i)^2}]|_{s=t,\theta_i=0} \nonumber\\
    &=& -\frac{1}{t^2}+\frac{1}{t^2}=0
  \end{eqnarray}
  and
   \begin{eqnarray}\label{j curve}
    \Delta_g|_{l_j}\log \rho_g(p,q')|_q &=&
     \frac{(1+\sum^{2n-1}_{i=1}\theta^2_i)-2\theta^2_{2j-1}}{(1+\sum^{2n-1}_{i=1}\theta^2_i)^2}|_{s=t,\theta_i=0}\nonumber\\
      &&+
   \frac{(1+\sum^{2n-1}_{i=1}\theta^2_i)-2\theta^2_{2j}}{(1+\sum^{2n-1}_{i=1}\theta^2_i)^2}|_{s=t,\theta_i=0} \nonumber\\
    &=& \frac{1}{t^2}+\frac{1}{t^2}=\frac{2}{t^2},
  \end{eqnarray}
  where $1\leq j\leq n-1$.
  Hence, for any $J$-holomorphic curve $l=\sum^n_{j=1}a_jl_j$ spanned by $\{X,JX\}$,
  \begin{eqnarray*}
   dd^c_J\log \rho_g(p,q')(X,JX)|_q&=&\sum^n_{j=1} a^2_j\Delta_g|_{l_j}\log \rho_g(p,q')|_q \\
   &=&\frac{2}{t^2}\sum^{n-1}_{j=1} a^2_j\geq 0,
   \end{eqnarray*}
   which means that $\log \rho_g(p,q)$ is $J$-plurisubharmonic if $\rho_g(p,q)<\varepsilon$.
  This completes the proof of the claim.  \qed

    \vskip 6pt

   In the remainder of this subsection, we will discuss the basic properties of $J$-plurisubharmonic functions on almost K\"{a}hler manifolds.
   In fact, a number of the results established in complex analysis via plurisubharmonic functions have
  been extended to almost complex manifolds (cf. \cite{HL4,HL5,HL3,HLP,Sukhov}).
    Let $(M,J)$ be an almost complex manifold and ${\rm PSH}(M,J)$ the set of $J$-plurisubharmonic functions on $(M,J)$.
  We have the following facts (cf. \cite{HL4,HL5,HL3,HLP,Sukhov}):
  \begin{prop}\label{plurisubharmonic prop}

    ~

    1) Suppose $(M,J)$ is an almost complex manifold which is $J$-pseudoconvex, and let $u\in{\rm PSH}(M,J)$ be a $J$-plurisubharmonic function.
            Then there exists a decreasing sequence $\{u_j\}\subset C^\infty(M)$ of smooth strictly $J$-plurisubharmonic functions such that
               $u_j(x)\downarrow u(x)$ at each $x\in M$.

   2) (Maximum property) If $u,v\in{\rm PSH}(M,J)$, then $w=max\{u,v\}\in{\rm PSH}(M,J)$.

   3) (Coherence property) If $u\in{\rm PSH}(M,J)$ is twice differentiable at $x\in M$, then $Hess_xu$ is positive.

   4) Let $u_1$ and $u_2$ be smooth strictly $J$-plurisubharmonic functions on $(M,J)$.
      Then for every $\varepsilon>0$ and every relatively compact domain $\Omega\subset M$ there exists a smooth and strictly $J$-plurisubharmonic
   function $u$ in $\Omega$ such that
   $\max\{u_1,u_2\}\leq u\leq\max\{u_1,u_2\}+\varepsilon$ on $\Omega$.

   5) If $\psi$ is convex non-decreasing function, then $\psi\circ u\in{\rm PSH}(M,J)$ for each $u\in{\rm PSH}(M,J)$.

   6) (Decreasing sequence property) If $\{u_j\}$ is a decreasing ($u_j\geq u_{j+1}$) sequence of functions with all $u_j\in{\rm PSH}(M,J)$,
      then the limit $u=\lim_{j\rightarrow\infty}u_j\in{\rm PSH}(M,J)$.

   7) (Uniform limit property) Suppose $\{u_j\}\subset{\rm PSH}(M,J)$ is a sequence which converges to $u$ uniformly on compact subsets on $M$,
    then $u\in{\rm PSH}(M,J)$.

    8) (Families locally bounded above) Suppose $\mathcal{F}\subset{\rm PSH}(M,J)$ is a family of functions which are locally uniformly bounded above.
    Then the upper envelope $v=\sup_{f\in\mathcal{F}}f$ has upper semi-continuous regularization $v^*\in{\rm PSH}(M,J)$ and $v^*=v$ a.e..
    Moreover, there exists a sequence $\{u_j\}\subset\mathcal{F}$ with $v^j=max\{u_1,\cdot\cdot\cdot,u_j\}$ converging to $v^*$ in $L^1_{loc}(M)$.
  \end{prop}

  For an almost K\"{a}hler $4$-manifold, we use Theorem \ref{app 1} for $\mathcal{\widetilde{W}}, d^-_J$-problem
   in Appendix \ref{Hormander} to establish the following result:
  \begin{lem}\label{current app}
 Let $(M,g,J,\omega)$ be an almost K\"{a}hler $4$-manifold, and let $T$ be a strictly positive closed $(1,1)$-current on $M$
  with $L^q$ coefficients for some fixed $q\in(1,2)$.
  Then, $T$ can be written as $T=d\mathcal{\widetilde{W}}(f_T)$ locally, where $f_T$ is in $L^q_2(M)$ and strictly $J$-plurisubharmonic.
 \end{lem}
 \begin{proof}
 It is often convenient to work with smooth forms and then prove statements about
  currents by using an approximation of a given current by smooth forms (cf. \cite{GH,S1}).
  For any point $p\in M$, we choose a neighborhood $U_p$ of $p$.
 We may assume without loss of generality that $U_p$ is a star shaped strictly $J$-pseudoconvex open set, by Poincar\'{e} Lemma,
  $T=dA$ on $U_p$ since $T|_{U_p}$ is a strictly positive closed $(1,1)$-current.
  Note that $T$ is $(1,1)$ type, so $d^-_J(A)=0$.
  Then applying Theorem \ref{app 1} in Appendix \ref{Hormander} ($\mathcal{\widetilde{W}}, d^-_J$-problem),
    there exists a smooth function $f_T$ such that $T=d\mathcal{\widetilde{W}}(f_T)$ on $U_p$.
    Since $(M,g,J,\omega)$ is an almost K\"{a}hler $4$-manifold, $\mathcal{\widetilde{W}}(f_T)=\mathcal{W}(f_T)$ (see Section \ref{2}),
    hence $T=d\mathcal{W}(f_T)$ locally.
    When $U_p$ is very small, on $U_p$ there exists Darboux coordinate chart $(z_1,z_2)$ (cf. \cite{A2,MS})
    with standard complex structure $J_0=J(p)$.
    Since $d\mathcal{W}(f_T)=\mathcal{D}^+_J(f_T)$
    is smooth and strictly  positive $(1,1)$-form,
    $\mathcal{D}^+_J(f_T)$ can be regarded as a local symplectic form on $U_p$.
    Hence, the complex coordinate $(z_1,z_2)$ is also Darboux coordinate on $U_p$ for $\mathcal{D}^+_J(f_T)$,
    that is,  $\mathcal{D}^+_J(f_T)$ are $J$ and $J_0(=J(p))$ compatible.
    Hence $\mathcal{D}^+_J(f_T)=2\sqrt{-1}\partial_{J(p)}\bar{\partial}_{J(p)}f_T$, i.e., $f_T=|z_1|^2+|z_2|^2$.
    It is easy to see that $\sqrt{-1}\partial_J\bar{\partial}_Jf_T>0$ on $U_p$.
    Therefore $f_T$ is also strictly $J$-plurisubharmonic.
    By Proposition \ref{plurisubharmonic prop}, when $f_T\in L^q_2(U_p)$ for some fixed $q\in(1,2)$, the above conclusion also holds
    since there exists a sequence $\{f_{T,k}\}$ of smooth $J$-plurisubharmonic functions on $U_p$ such that $f_{T,k}$ converges to $f_T$
    in norm $ L^q_2$.
    This completes the proof of Lemma \ref{current app}.
  \end{proof}

   In classical complex analysis case, we have Poincar\'{e}-Lelong equation ({\rm\cite{GH}}).
   If the holomorphic function $f$ has divisor the analytic hypersurface $Z$, then the equation of currents
    $$
   \frac{\sqrt{-1}}{2\pi}\partial\bar{\partial}\log|f|^2=T_Z
    $$
   is valid.
  In \cite{Elk2}, Elkhadhra extended Poincar\'{e}-Lelong equation to the almost complex category.
   Let $\Omega$ be an open set of $\mathbb{R}^{2n}$ equipped with an almost complex structure $J$.
    Given a submanifold $Z$ of $\Omega$ of codimension $2p$ if $J(TZ)=TZ$, that is,
    $TZ$ is $J$-invariant, then $J$ is also an almost complex structure on $TZ$,
    it means that $Z$ is an almost complex submanifold of dimension $2n-2p$.
  Let $U$ be an open subset of  $\Omega$ such that $Z$ is defined on $U$ by $f_i=0$, $1\leq i\leq p$,
  where the $f_i$ are of smooth functions on $U$, $\bar{\partial}_Jf_i=0$ on $Z\cap U$
 and $\partial_Jf_1\wedge\cdot\cdot\cdot\wedge\partial_Jf_p\neq 0$ on $U$.
  With these notations, Elkhadhra obtained a generalized Poincar\'{e}-Lelong formula:
  $$
  (\frac{\sqrt{-1}}{2\pi}\partial_J\bar{\partial}_J\log|f|^2)^p=T_Z+R_J(f),
  $$
  where $f=(f_i)_{1\leq i\leq p}$, $|f|^2=\sum^p_{i=1}|f_i|^2$ and
 $R_J(f)$ is a $(p,p)$-current which has $L^{\alpha}_{loc}$ integrable as coefficients, $\alpha<1+\frac{1}{2p-1}$.
  Moreover, $R_J(f)=0$ when the structure $J$ is integrable.
    Our Lemma {\rm\ref{current app}} can be viewed as a generalized Poincar\'{e}-Lelong equation of closed positive
   $(1,1)$-currents on almost K\"{a}hler 4-manifold.

   \subsection{Kiselman's minimal principle for $J$-plurisubharmonic functions}\label{Kiselman}

   This subsection is devoted to studying Kiselman's minimal principle for $J$-plurisubharmonic functions.
   A linear image of a convex set is convex, but in spite of far reaching analogy between convexity and pseudoconvexity the corresponding result is
   not true in the complex domain, the projection in $\mathbb{C}^2$ of a pseudoconvex set in $\mathbb{C}^3$ may fail to be pseudoconvex.
   C. O. Kiselman \cite{Kise} exhibited, in classical complex analysis, a class of pseudoconvex sets which admit pseudoconvex projections and
   studied an associated functional transformation, the partial Legendre transformation.
   This transformation can be used to study the local behavior of plurisubharmonic functions in classical complex analysis.
   In this subsection, we use this method to study the local behavior of $J$-plurisubharmonic functions.

  Let $(\mathbb{R}^{2n},\omega_0)$ be the standard symplectic vector space, where
   $\omega_0=\sum^n_{i=1}dx_i\wedge dy_i$.
   Here $(x_1,y_1,\cdot\cdot\cdot,x_n,y_n)$ is the global coordinate of $\mathbb{R}^{2n}$.
  As in classical complex analysis \cite{JP}, we have the following definition.
  \begin{defi}
   (cf. {\rm Jarnicki-Pflug \cite[Definition 1.1.1]{JP}})
    A pair $(X,\pi)$ is called a symplectic Riemann region over the symplectic vector space $(\mathbb{R}^{2n},\omega_0)$ if:

   (1) $X$ is a topological Hausdorff space;

   (2) $\pi:X\longrightarrow (\mathbb{R}^{2n},\omega_0)$ is a local homeomorphism.

  Moreover, if $X$ is connected, then we say that $(X,\pi)$ is a symplectic Riemann domain over $(\mathbb{R}^{2n},\omega_0)$.
  The mapping $\pi$ is called the projection. $\forall z\in \pi(X)$, $\pi^{-1}(z)$ is called the stalk over $z$.
    A subset $A\subset X$ is said to be univalent if $\pi|_A:A\rightarrow \pi(A)$ is homeomorphic.
  \end{defi}

  \begin{rem}
   (cf. {\rm Jarnicki-Pflug \cite{JP}})
   (1) If we replace $(\mathbb{R}^{2n},\omega_0)$ in the above definition by a (connected) $2n$-dim symplectic manifold $(M,\omega)$,
    then we get the notion of a Riemann region (domain) over $(M,\omega)$.

   (2) $\omega_0$ can be pulled back to $X$ so that $(X,\omega=\pi^*\omega_0)$ is a symplectic manifold.
    It is well known that there exists an $\omega$-compatible almost complex structure $J$ on $X$, that is, $\omega(J\cdot,J\cdot)=\omega(\cdot,\cdot)$.
      Let $g(\cdot,\cdot):=\omega(\cdot,J\cdot)$ be an almost K\"{a}hler metric on $X$.
  Then $(X,g,J,\omega)$ is an almost K\"{a}hler manifold (cf. {\rm \cite{MS}}).
  Let $J_0:= J_{st}$ be the standard complex structure on $\mathbb{R}^{2n}$, $g_0(\cdot,\cdot):=\omega_0(\cdot,J_0\cdot)$,
   then $(\mathbb{R}^{2n},g_0,J_0,\omega_0)=\mathbb{C}^n$.

   (3) If $\Omega\subset(\mathbb{R}^{2n},g_0,J_0,\omega_0)$ is a domain,
  then $(\Omega,\omega_0)$ is a (symplectic) Riemann domain over $\mathbb{C}^n$.

   (4) If $(X,\pi,\omega)$ is a symplectic Riemann domain over $(\mathbb{R}^{2n},\omega_0)$,
  then $\pi$ is an open mapping.
   Hence, $\pi(X)$ is a domain over $(\mathbb{R}^{2n},\omega_0)$ and the stalk $\pi^{-1}(p)$ is discrete for all $p\in\pi(X)$.

   (5) Let $(X,\pi,\omega)$ be a symplectic Riemann domain over $(\mathbb{R}^{2n},\omega_0)$, and let $Y$ be a univalent subset such that
    $\pi(Y)=\pi(X)$, then $Y=X$.

   (6) Evidently, not all connected symplectic $2n$-dimensional manifolds are symplectic Riemann domains,
  e.g., a compact symplectic manifold cannot be a symplectic Riemann domain.
  In the category of non-compact connected symplectc manifolds the situation is as follows:
  If $n=1$, then any complex (symplectic $2$-dimensional) manifold is a symplectic Riemann domain over $\mathbb{C}$ $((\mathbb{R}^2,\omega_0))$
   with suitable projection $\pi$;
   If $n\geq2$, then there exist very regular non-compact connected symplectic manifolds
   which are not symplectic Riemann domains over $(\mathbb{R}^{2n},\omega_0)$.

  (7) If $(X,\pi,\omega)$ is a symplectic Riemann domain over $(\mathbb{R}^2,\omega_0)$, then
      $(Y,\pi|_Y,\omega|_Y)$ is a symplectic Riemann domain over $(\mathbb{R}^2,\omega_0)$ for any domain $Y\subset X$.

    (8) If $(X,\pi^1,\omega^1)$ and $(Y,\pi^2,\omega^2)$ are symplectic Riemann domains
       over $(\mathbb{R}^{2n},\omega^1_0)$ and $(\mathbb{R}^{2m},\omega^2_0)$, respectively,
  then $(X\times Y,\pi^1\times\pi^2,\omega^1\oplus\omega^2)$ is a symplectic Riemann domain over
     $(\mathbb{R}^{2n}\times\mathbb{R}^{2m},\omega^1_0\oplus\omega^2_0)$.
  \end{rem}

   \begin{exa}
   (1) Let $(\mathbb{R}^{2n},\pi=id_{\mathbb{R}^{2n}},\omega^1_0)$ be a symplectic vector space,
   where $$\mathbb{R}^{2n}:=\{(x_1,\cdot\cdot\cdot,x_{2n})\,|\, x_i\in\mathbb{R},1\leq i\leq 2n\},$$
     $\omega^1_0=dx_1\wedge dx_2+\cdot\cdot\cdot+dx_{2n-1}\wedge dx_{2n}$.
   Suppose $J$ is an $\omega^1_0$-compatible almost complex structure on $\mathbb{R}^{2n}$.
    Let $g_J(\cdot,\cdot)=\omega^1_0(\cdot,J\cdot)$, then $E:=(\mathbb{R}^{2n},g_J,J,\omega^1_0)$
  is an almost K\"{a}hler manifold and also a topological vector space.

  (2) Let $(\mathbb{R}^{2m},\pi=id_{\mathbb{R}^{2m}},\omega^2_0)$ be a symplectic vector space,
   where $$\mathbb{R}^{2m}=\{(y_1,\cdot\cdot\cdot,y_{2m})\,|\,y_i\in\mathbb{R},1\leq i\leq 2m\},$$
     $\omega^2_0=dy_1\wedge dy_2+\cdot\cdot\cdot+dy_{2m-1}\wedge dy_{2m}$.
  Let $J_0$ be the standard complex structure on $\mathbb{R}^{2m}$.
  It is easy to see that $J_0$ is $\omega^2_0$-compatible. Then
   $(\mathbb{R}^{2m},J_0,\omega^2_0)=\mathbb{C}^{m}=\mathbb{R}^{m}+\sqrt{-1}\mathbb{R}^{m}$.
   \end{exa}

    \begin{defi}
    A domain $\Omega\subset\mathbb{C}^n$ is called a tube domain if $\Omega=\Omega+\sqrt{-1}\mathbb{R}^n$.
     \end{defi}

  In classical complex analysis, one has the following theorem (cf. \cite{D3,Hormander,Kise2}):
   \begin{theo}\label{cla com the}
  (1) Let $\Omega\subset\mathbb{C}^n$ be a domain, $u$ a $(J_0)$-plurisubharmonic function which is locally indenpendent of
  the imaginary part of $z$, i.e., for any $z\in\Omega$, $u(z')=u(z)$
  if $z'$ is sufficientlly close to $z$ and ${\rm Re}z'={\rm Re}z$.
  Then $u$ is locally convex in $\Omega$ (thus convex if $\Omega$ is convex).

  (2) Any $(J_0)$-pseudoconvex tube domain $\Omega\subset\mathbb{C}^n$ is of the form $\Omega=\Omega_1+\sqrt{-1}\mathbb{R}^n$,
   where $\Omega_1$ is a convex subdomain of $\mathbb{R}^n$.
   \end{theo}

    The main goal of this subsection is to prove a minimum principle for $J$-plurisubharmonic function as in classical complex analysis
    (cf. Kiselman \cite{Kise}).
     \begin{theo}\label{sym the}
     (minimal principle for $J$-plurisubharmonic functions)

   Let $E=(\mathbb{R}^{2n-2k},g_J,J,\omega^1_0)$ be an almost K\"{a}hler manifold
     which is also a topological vector space with the induced topology from the metric $g_J$.
    Let $J_1:= J\oplus J_0$ be an $\omega^1_0\oplus\omega^2_0$-compatible almost complex structure on
    $(\mathbb{R}^{2n-2k},\omega^1_0)\times(\mathbb{R}^{2k},\omega^2_0)$.
    Suppose that $\Omega$ is a $J_1$-pseudoconvex subdomain of $E\times\mathbb{C}^k$ such that for each $x\in E$, the fiber
   $$
    \Omega_x:=\{z\in\mathbb{C}^k \mid  (x,z)\in\Omega\}
   $$
   is a non-empty connected tube domain.
   Let $u$ be a $J_1$-plurisubharmonic function on $\Omega$.
   Then the function
    $$
    f:\pi(\Omega)\rightarrow [-\infty,+\infty),\,\,\,\pi: E\times\mathbb{C}^k\rightarrow E
    $$
     \begin{equation}
   f(x):=\inf\{u(x,z) \mid  z\in\Omega_x\},\,\,\, x\in\pi(\Omega)
     \end{equation}
  is $J$-plurisubharmonic.
      \end{theo}

    \begin{rem}
   (1) $\pi(\Omega)\subset E$ is $J$-pseudoconvex (cf. {\rm Kiselman \cite{Kise}}).

   (2) If the fibres are tubular but not necessarily connected (they must consist of convex components),
   then the function $f$ is not defined on $E$ but on a symplectic Riemann domain over $(\mathbb{R}^{2n-2k},\omega^1_0)$.
    For more details see {\rm\cite[Proposition 2.1]{Kise}}.
   \end{rem}

   The similar proof as in classical complex analysis we will present here is taken from Kiselman \cite{Kise2} and Jarnicki-Pflug \cite{JP}.
    We need the following technical lemmas:

  \begin{lem}\label{matrix dec1}
  Let $L$ be a positive semidefinite Hermitian $(n\times n)$-matrix.
  Then there exists a Hermitian $(n\times n)$-matrix $M$ with $LML=L$.
  \end{lem}
  \begin{proof}
   There exists $P\in U(n)$ such that
  $$PL\bar{P}^T=\left(\begin{array}{cccccc}
                                          \lambda_1 &  &  &  &  &  \\
                                           & \cdot\cdot\cdot &  &  &  &  \\
                                           &  &  \lambda_m &  & &  \\
                                           &  &  & 0 &  &  \\
                                           &  &  &  & \cdot\cdot\cdot &  \\
                                           &  &  &  &  & 0 \\
                                        \end{array}
                                      \right) =: \Lambda, \,\,\,m\leq n,$$
   since $L$ is a positive semidefinite Hermitian $(n\times n)$-matrix.
  Let $$\Lambda^-:=\left(\begin{array}{cccccc}
                                          \frac{1}{\lambda_1} &  &  &  &  &  \\
                                           & \cdot\cdot\cdot &  &  &  &  \\
                                           &  &  \frac{1}{\lambda_m }&  & &  \\
                                           &  &  & 0 &  &  \\
                                           &  &  &  & \cdot\cdot\cdot &  \\
                                           &  &  &  &  & 0 \\
                                        \end{array}
                                      \right),$$
   and take $M=\bar{P}^T\Lambda^-P$, then $LML=(\bar{P}^T\Lambda P)(\bar{P}^T\Lambda^-P)(\bar{P}^T\Lambda P)=L$.
    \end{proof}

    Such matrix $M$ is called a Hermitian quasi-inverse of $L$.

  \begin{lem}\label{matrix dec2}
  Let $F:\mathbb{C}^n\rightarrow\mathbb{R}$, $$F(z):=\sum^n_{i,j=1}L_{ij}z_i\bar{z_j}+2Re(\sum^n_{j=1}b_jz_j)$$ be bounded from below,
  where $L=(L_{ij})_{n\times n}$ is a positive semidefinite Hermitian matrix and $b=(b_1,\cdot\cdot\cdot,b_n)\in\mathbb{C}^n$.
   If $M$ is a Hermitian quasi-inverse of $L$, then $LMb^T=b^T$ and $$F(z)\geq -\bar{b}Mb^T=F(-(\bar{M}\bar{b}^T)^T),\, z\in\mathbb{C}^n.$$
 \end{lem}
  \begin{proof}
  For a detailed proof of this lemma, we refer to \cite[Lemma 2.3.6]{JP}.
  \end{proof}

      By using Lemma \ref{matrix dec1}, \ref{matrix dec2}, we can prove the following lemma.

   \begin{lem}\label{matrix dec3}
   Let $\Omega$ be a domain in $\mathbb{C}_z\times\mathbb{C}^n_w$ and let $u\in PSH(\Omega)\cap C^2(\Omega)$.
  Moreover, let $M(z,w)$ denote a quasi-inverse of $$L(z,w)=(\frac{\partial^2u}{\partial w_i\partial \bar{w}_j}(z,w))_{1\leq i,j\leq n},\,(z,w)\in \Omega.$$
    Then $u_{z\bar{z}}\geq \bar{b}Mb^T$ on $\Omega$,
     where $b=(b_1,\cdot\cdot\cdot,b_n)=(\frac{\partial^2u}{\partial\bar{z}\partial w_1}
          ,\cdot\cdot\cdot,\frac{\partial^2u}{\partial\bar{z}\partial w_n}):\Omega\rightarrow \mathbb{C}^n$.
  \end{lem}
  \begin{proof}
  For a detailed proof of this lemma, we refer to \cite[Lemma 2.3.7]{JP}.
  \end{proof}

  Let $U\subset \mathbb{C}$ be an open set, and let $y:U\rightarrow\mathbb{C}^n$ be a $C^1$-function such that
   \begin{equation}
 (z,y(z))\in \Omega,\,\,\, \frac{\partial u}{\partial w_j}(z,y(z))=0,\,\,\, 1\leq j\leq n,\,\,\, z\in U,
   \end{equation}
   where $u$ and $\Omega$ are the same as in the above lemma.
   Define $g:U\rightarrow \mathbb{R}$, $g(z):= u(z,y(z))$.
   Differentiation of $g$ with respect to $z$ and $\bar{z}$ leads to
   \begin{equation}
   g_{z\bar{z}}(z)=u_{z\bar{z}}(z,y(z))+\sum^n_{j=1}u_{zw_j}(z,y(z))y_{j\bar{z}}(z)+\sum^n_{j=1}u_{z\bar{w}_j}(z,y(z))\bar{y}_{j\bar{z}}(z).
   \end{equation}
  Since $u_{w_k}(z,y(z))=0$, $k=1,\cdot\cdot\cdot,n$, we differentiate the equations with respect to $z$ and $\bar{z}$,
  then
    $$
   0=a_k(z,y(z))+\sum^n_{j=1}H_{kj}(z)\alpha_j(z)+\sum^n_{j=1}L_{kj}(z)\bar{\beta}_j(z),\,\,\,1\leq k\leq n,
    $$
   and
    $$
   0=b_k(z,y(z))+\sum^n_{j=1}H_{kj}(z)\beta_j(z)+\sum^n_{j=1}L_{kj}(z)\bar{\alpha}_j(z),\,\,\,1\leq k\leq n,
    $$
  where $$\alpha=(y_{1z},\cdot\cdot\cdot,y_{nz}),\,\,\, \beta=(y_{1\bar{z}},\cdot\cdot\cdot,y_{n\bar{z}}),$$
 $$a=(a_1,\cdot\cdot\cdot,a_n)=(u_{zw_1},\cdot\cdot\cdot,u_{zw_n}),\,\,\,
   b=(b_1,\cdot\cdot\cdot,b_n)=(u_{\bar{z}\bar{w}_1},\cdot\cdot\cdot,u_{\bar{z}\bar{w}_n}),$$
  \begin{equation}
  H(z)=(H_{kj}(z))=(\frac{\partial^2u}{\partial w_k\partial w_j}(z,y(z)),\,\,\,
    L(z)=(L_{kj}(z))=(\frac{\partial^2u}{\partial w_k\partial \bar{w}_j}(z,y(z)), \,\,\,z\in U.
  \end{equation}
  Summarizing, the following identities hold for $z\in U$:
        $$
       a(z,y(z))=-\alpha(z)H(z)-\bar{\beta}(z)L^T(z),
        $$
         \begin{equation}
  b(z,y(z))=-\beta(z)H(z)-\bar{\alpha}(z)L^T(z).
         \end{equation}
    \begin{prop}
   Let $M$ be a matrix-valued function on $U$ such that for all $z\in U$ the matrix $M(z)$ is a Hermitian quasi-inverse of $L(z)$.
    Then
    $$
     g_{z\bar{z}}(z)\geq(\beta(HM^T\bar{H}-L)\bar{\beta}^T)(z), \,\,\, z\in U.
     $$
  In particular, $g$ is subharmonic on $U$, if the right-hand side of this inequality is never negative on $U$.
  \end{prop}
  \begin{proof}
   Lemma \ref{matrix dec3} shows $\forall z\in U$, $u_{z\bar{z}}(z,y(z))\geq (\bar{b}Mb^{T})(z,y(z))$ and using $LMb^T=b^T$,
   \begin{eqnarray*}
   g_{z\bar{z}}(z)&=&u_{z\bar{z}}(z,y(z))+a(z,y(z)) \beta(z)+\overline{b(z,y(z))\alpha(z)}\\
   &\geq& \bar{b}Mb^T+a\beta^T+\bar{b}\bar{\alpha}^T\\
   &=& \bar{\beta}\bar{H}MH\beta^T+\bar{\beta}\bar{H}ML\bar{\alpha}^T+\alpha H^T\beta^T+\alpha H^T\beta^T   \\
    &&+\alpha L\bar{\alpha}^T-\alpha H\beta^T-\bar{\beta}L^T\beta^T-\bar{\beta}\bar{H}\bar{\alpha}^T-\alpha\bar{L}^T\bar{\alpha}^T\\
   &=&\bar{\beta}(\bar{H}MH-L^T)\beta^{T}+\bar{\beta}\bar{H}(ML-I_n)\bar{\alpha}^T\\
   &=&\beta(HM^T\bar{H}-L)\bar{\beta}^T+(-\bar{b}-\alpha\bar{L}^T)(ML-I_n)\bar{\alpha}^T\\
   &=&\beta(HM^T\bar{H}-L)\bar{\beta}^T.
   \end{eqnarray*}
  \end{proof}

  \begin{col}\label{usefull col}
  Under the assumptions of the above proposition,
  moreover, assume that the following properties are fulfilled: if $z\in U$ and $t\in \mathbb{R}^n$, then $(z,w+\sqrt{-1}t)\in U$ and $u(z,w)=u(z,w+\sqrt{-1}t)$.
  Then $g:U\rightarrow \mathbb{R}$ is subharmonic on $U$.
  \end{col}

   By the above lemmas, proposition and corollary, we return to prove Theorem \ref{sym the}.

   {\it Proof} of Theorem \ref{sym the}: Suppose that
    $$(\mathbb{R}^{2n-2k},g_J,J,\omega^1_0)\times(\mathbb{R}^{2k},g_{J_0},J_0,\omega^2_0)=(\mathbb{R}^{2n-2k},g_J,J,\omega^1_0)\times\mathbb{C}^k$$
   is an almost K\"{a}hler manifold, where $J$ is an $\omega^1_0$-compatible almost complex structure on $\mathbb{R}^{2n-2k}$,
    $g_J(\cdot,\cdot):= \omega^1_0(\cdot,J\cdot)$, $J_0=J_{st}$ is the standard complex structure on $(\mathbb{R}^{2k},\omega^2_0)\cong\mathbb{C}^k$,
       $g_{J_0}(\cdot,\cdot):= \omega^2_0(\cdot,J_0\cdot)$.
    Let $$\Omega\subset(\mathbb{R}^{2n-2k},g_J,J,\omega^1_0)\times\mathbb{C}^k$$ be a $J_1$-pseudoconvex domain, where $J_1:=  J\oplus J_0$
    is an $\omega^1_0\oplus\omega^2_0$-compatible almost complex structure on $\mathbb{R}^{2n-2k}\times\mathbb{R}^{2k}$.
    Suppose that $u(x,w)$ is a $J_1$-plurisubharmonic function on $\Omega$,
         where $$(x,w)\in\Omega\subset(\mathbb{R}^{2n-2k},g_J,J,\omega^1_0)\times\mathbb{C}^k.$$
   Let $$\pi:(\mathbb{R}^{2n-2k},g_J,J,\omega^1_0)\times\mathbb{C}^k\rightarrow(\mathbb{R}^{2n-2k},g_J,J,\omega^1_0),
       \pi(x,w)=x\in(\mathbb{R}^{2n-2k},g_J,J,\omega^1_0).$$
    Define a function on $\pi(\Omega)$ as follows:
        Let $$\Omega_x:= \{w\in\mathbb{C}^k\mid  (x,w)\in \Omega\},\,\,
    g(x):= inf\{u(x,w)\mid  w\in\Omega_x\},\,\, x\in\pi(\Omega).$$

   To complete the proof of Theorem \ref{sym the}, we must prove that $g:\pi(\Omega)\rightarrow [-\infty,+\infty)$ is a
   $J$-plurisubharmonic function on $\pi(\Omega)$.
        It is well know that a $J$-plurisubharmonic function is $J$-plurisubharmonic in the standard sense (cf. \cite{HL3}),
   that is, its restriction to each $J$-holomorphic curve $\Sigma$ in $(\pi(\Omega),J)$ is subharmonic.
   Hence, without loss of generality, we may assume $k=n-1$, that is , $\Omega\subset(\mathbb{R}^2,g_J,J,\omega^1_0)\times\mathbb{C}^{n-1}$.
   Note that $(\mathbb{R}^2,g_J,J,\omega^1_0)$ is a Riemann surface (cf. \cite{GH}) since $J$ on $\mathbb{R}^2$ is integrable.
      Hence $\Omega\subset(\mathbb{R}^2,g_J,J,\omega^1_0)\times\mathbb{C}^{n-1}$ is a K\"{a}hler manifold
     which is also a Riemann domain over $\mathbb{C}^n$ in classical complex analysis.
   By using Theorem \ref{cla com the} and Corollary \ref{usefull col},
         similar to the proof of Theorem $2.3.2$ in \cite{JP},
    we can prove that $g(x):\pi(\Omega)\rightarrow [-\infty,+\infty)$ is a subharmonic function on $\pi(\Omega)$.
    For details, we refer to \cite[proof of Theorem 2.3.2]{JP}.
   This completes the proof of Theorem \ref{sym the}. $\Box$

\subsection{H\"{o}rmander's $L^2$ estimates on tamed almost complex $4$-manifolds}\label{Hormander}

  In this subsection, we devote to considering $\mathcal{\widetilde{W}}$, $d^-_J$-problem (as $\bar{\partial}$-problem in classical complex analysis, cf. H\"{o}rmander \cite{Hormander0,Hormander}).
   In Stein manifold, the $L^2$-method for the $\bar{\partial}$ operator has many applications, for example,
      using $L^2$-method we can prove the theorem of Siu \cite{S2} on the Lelong numbers of plurisubharmonic functions (cf. \cite{D3}).
  In this subsection, we extend H\"{o}rmander's $L^2$ estimates \cite{Hormander0,Hormander} to tamed almost complex $4$-manifold.

   Suppose that $J$ is an almost complex structure on $\mathbb{R}^4$ which is tamed by a symplectic $2$-form $\omega_1=F+d^-_J(v+\bar{v})$,
    where $F$ is a fundamental form on $\mathbb{R}^4$ and
  $v\in\Lambda^{0,1}_J\otimes L^2_1(\mathbb{R}^4)$.
   Let $g_J(\cdot,\cdot)=F(\cdot,J\cdot)$ be an almost Hermitian metric and $d\mu_{g_J}$ the volume form.
   Let $(\Omega,J)$ be a bounded open set in $(\mathbb{R}^4,J)$, $A=u+\bar{u}\in\Lambda^1_\mathbb{R}\otimes L^2_1(\Omega)$
   and satisfy $d^-_J(A)=0$, where $u\in\Lambda^{0,1}_J\otimes L^2_1(\Omega)$.
   Let $L^2_2(\Omega)_0$ be the completion of the space of smooth functions with compact support in $\Omega$ under the $L^2_2$ norm.
    Since $d^-_Jd^*:\Omega^-_J(\Omega)\rightarrow\Omega^-_J(\Omega)$ is a strongly elliptic linear operator (see Section \ref{2} or \cite{L3}),
    where $d^*=-*_{g_J}d*_{g_J}$,
    we define a linear operator $\mathcal{\widetilde{W}}$ as in Section $2$,
  $\mathcal{\widetilde{W}}: L^2_2(\Omega)_0\longrightarrow\Lambda^1_\mathbb{R}\otimes L^2_1(\Omega)$,
  where $L^2_2(\Omega)_0$ is the completion of the space of smooth functions with compact support in $\Omega$ under the $L^2_2$ norm,
  $$\mathcal{\widetilde{W}}(f)=Jdf+d^*(\eta^1_f+\overline{\eta}^1_f)-*_{g_J}(df\wedge d^-_J(v+\bar{v}))+d^*(\eta^2_f+\overline{\eta}^2_f),\,\,\,
   \eta^1_f,\eta^2_f\in\Lambda^{0,2}_J\otimes L^2_2(\Omega),$$
  satisfying
  $$d^*\mathcal{\widetilde{W}}(f)=0,$$
   $$d^-_JJdf+d^-_Jd^*(\eta^1_f+\overline{\eta}^1_f)=0,$$
  and
 $$-d^-_J*_{g_J}(df\wedge d^-_J(v+\bar{v}))+d^-_Jd^*(\eta^2_f+\overline{\eta}^2_f)=0,$$
  where
 $$ \eta^1_f|_{\partial\Omega}=0, \,\,\,\eta^2_f|_{\partial\Omega}=0.$$
 Notice that $C^\infty_0(\Omega)$ (which is the space of smooth functions with compact support in $\Omega$) is dense in $L^2_2(\Omega)_0$.
   The question with our relationship is whether $\mathcal{\widetilde{W}}(f)=A$ has a solution.
   Note that $d^-_J\circ\mathcal{\widetilde{W}}=0$.
   If we use the theory of Hilbert space, considing
  \begin{equation}\label{}
     L^2_2(\Omega)_0\stackrel{\mathcal{\widetilde{W}}}{\longrightarrow}\Lambda^1_\mathbb{R}\otimes L^2_1(\Omega)\stackrel{d^-_J}{\longrightarrow}
   \Lambda^-_J\otimes L^2(\Omega),
  \end{equation}
   then the above problem is equivalent to: Whether the kernel of $d^-_J$ is equal to the image of $\mathcal{\widetilde{W}}$.
  As the $\bar{\partial}$-problem in classical complex analysis,
  we call this problem the $\mathcal{\widetilde{W}}$, $d^-_J$-problem.

  \vskip 6pt

  Our approach is along the lines used by L. H\"{o}rmander to present the method of $L^2$ estimates for the $\bar{\partial}$-problem
    in \cite{Hormander0}.
  We summarize the above discussion in terms of the model of Hilbert spaces below:
   $$
  H_1\stackrel{T}{\longrightarrow}H_2\stackrel{S}{\longrightarrow}
  H_3,
   $$
   where $H_1,H_2,H_3$ are all Hilbert spaces, and $T, S$ are linear, closed and densely defined operators.
  Assume $ST=0$, the problem is whether, $\forall g\in\ker S$, a solution to
  $$
   Tf=g
  $$
  exists.
  First, note a simple fact that $Tf=g$ is equivalent to
   \begin{equation}\label{equivalent formular}
  (Tf,h)_{H_2}=(g,h)_{H_2},\,\,\,\forall h\in {\rm some \,\, dense\,\, subset}
   \end{equation}
  because $(Tf-g,h)_{H_2}=0$, $\forall h\in$ some dense subset $\Longleftrightarrow$ $(Tf-g,H_2)_{H_2}=0$ $\Longleftrightarrow$ $Tf=g$.

    \vskip 12pt

  Let $T^*$ be an adjoint operator of $T$ in the sense of distributions.
  By the theory of functional analysis, $T^*$ is a closed operator, and $(T^*)^*=T$ if and only if $T$ is closed.

  From (\ref{equivalent formular}), $(Tf,h)_{H_2}=(g,h)_{H_2}$, $\forall h\in$ some dense subset.
  If this dense subset is contained in $D_{T^*}$, then, noticing $(Tf,h)_{H_2}=(f,T^*h)_{H_1}$,
 \begin{eqnarray}\label{equivalent formular2}
   Tf=g &\Longleftrightarrow& (Tf,h)_{H_2}=(g,h)_{H_2} \nonumber \\
     &\Longleftrightarrow& (f,T^*h)_{H_1}=(g,h)_{H_2}, \,\,\,\forall h\in {\rm some \,\, dense\,\, subset\,\,in\,\,D_{T^*}}.
  \end{eqnarray}
  Let $T^*h\longrightarrow(g,h)_{H_2}$ be a linear functional defined on a subset of $H_1$ (that is, \{$T^*g$ $|$ $g\in$ some dense subset in $D_{T^*}$\}).
  If we can extend the above functional to a bounded linear functional on the entire $H_1$,
  then an application of Riesz Representation theorem to (\ref{equivalent formular2}) will thus show that
    the problem $Tf=g$ is solved.
   Recall that the Riesz Representation theorem states that if $\lambda: H\rightarrow \mathbb{C}$ is a bounded linear functional on a Hilbert space $H$,
   then there exists $g\in H$ such that $\lambda(x)=(x,g)_H$ $\forall x\in H$.
    Hence the main step is whether we can extend $T^*h\longrightarrow(g,h)_{H_2}$ to a bounded linear functional on the entire $H_1$
   (for details, see \cite{Hormander0,Hormander}).

As in classical complex analysis, we have the following lemmas:

     \begin{lem}\label{Siu lemma 1}
      (cf. {\rm\cite[Theorem 1.1.1]{Hormander0}})
     If there exists a constant $c_g$ depending only on $g$ such that
    \begin{equation}\label{Siu equ}
    |(g,h)_{H_2}|\leq c_g\|T^*h\|_{H_1},
   \end{equation}
    then $T^*h\longrightarrow(g,h)_{H_2}$ can be extended to a bounded linear functional on $H_1$.
      \end{lem}

      In the above discussion, we used only the front half of
       $$
  H_1\stackrel{T}{\longrightarrow}H_2\stackrel{S}{\longrightarrow}
  H_3.
   $$
   However, since we only need to solve the equation $Tf=g$ or $(T^*h,f)=(h,g)$ for $g\in\ker S$,
   it is unnecessary to prove (\ref{Siu equ}) for $g\in H_2$, rather we just need to prove (\ref{Siu equ})
   for $g\in\ker S$. In this case, we hope that $h$ in (\ref{Siu equ}) belongs to some dense subset in $D_{T^*}$.

 The method of proving $$|(g,h)_{H_2}|\leq c_g\|T^*h\|_{H_1}$$ is through proving a more general inequality:
   $$
      \|h\|^2_{H_2}\leq c(\|T^*h\|^2_{H_1}+\|Sh\|^2_{H_3}),\,\,\, h\in D_{T^*}\cap D_S.
   $$
 First we note, in our problem, $D_{T^*}$ and $D_S$ contain $C^\infty(\Omega)_0$
 which is the space of smooth functions on $\Omega$ with compact support,
  hence $D_{T^*}\cap D_S$ is dense on both $D_{T^*}$ and $H_2$.
  Notice that $T,S$ are linear, closed densely defined operators, and $ST=0$.
  Now we need
  \begin{lem}\label{Siu lemma 2}
  (cf. {\rm\cite[Theorem 1.1.2]{Hormander0}})
  If \begin{equation}\label{inequvi3}
      \|h\|^2_{H_2}\leq c(\|T^*h\|^2_{H_1}+\|Sh\|^2_{H_3})\,\,\, h\in D_{T^*}\cap D_S,
   \end{equation}
 then
  \begin{equation}\label{inequvi4}
    |(g,h)_{H_2}|\leq c^{\frac{1}{2}}\|g\|_{H_2}\|T^*h\|_{H_1}\,\,\,\forall g\in\ker S,\,\,\, h\in D_{T^*}\cap D_S.
  \end{equation}
   \end{lem}

  Applying Lemma \ref{Siu lemma 2}, we have that if $$\|h\|^2_{H_2}\leq c(\|T^*h\|^2_{H_1}+\|Sh\|^2_{H_3})$$ for all $h\in D_{T^*}\cap D_S$,
 then $$|(g,h)_{H_2}|\leq c^{\frac{1}{2}}\|g\|_{H_2}\|T^*h\|_{H_1}\,\,\,\forall g\in\ker S,\,\, h\in D_{T^*}\cap D_S.$$
  Hence, by Lemma \ref{Siu lemma 1}, $T^*h\longrightarrow(g,h)_{H_2}$ can be extended to a bounded linear functional on $H_1$,
 whose bound is $c^{\frac{1}{2}}\|g\|_{H_2}$.
  By Riesz Representation theorem, there exists $f\in H_1$ such that $$(T^*h,f)_{H_1}=(h,g)_{H_2}, \,\,\forall h\in D_{T^*}\cap D_S.$$
  Since $D_{T^*}\cap D_S$ is dense in $H_2$, we have $$(h,Tf)_{H_2}=(h,g)_{H_2}, \,\,\forall h\in H_2.$$
  By (\ref{equivalent formular2}), the equation $Tf=g$ has a solution.
   In addition, from the Riesz Representation theorem, we have
   $$ \|f\|_{H_1}\leq c^{\frac{1}{2}}\|g\|_{H_2},\,\,\,f\in(\ker T)^{\bot}.$$
   In fact, $$\|f\|_{H_1}\leq c^{\frac{1}{2}}\|g\|_{H_2}$$ is the direct consequence of  Riesz Representation theorem.
   To show $f\in(\ker T)^{\bot}$, note that, according to the way that $T^*h\rightarrow(h,g)_{H_2}$
   is extended to a bounded linear functional on the entire $H_1$,
   this functional vanishes on the orthogonal complement of $\overline{\{T^*h \,|\, h\in D_{T^*}\}}$,
   thus $f\in \overline{\{T^*h \,|\, h\in D_{T^*}\}}$.
   If $f\in \lim_{k\rightarrow\infty}T^*h_k$, then for every $X\in\ker T$, we have
   $$
   (X,f)_{H_1}=\lim_{k\rightarrow\infty}(X,T^*h_k)_{H_1}=\lim_{k\rightarrow\infty}(TX,h_k)_{H_2}=0,
   $$
   hence, $f\in(\ker T)^{\bot}$.

   In general, the solution of $Tf=g$ is not unique, since $f_1\in\ker T$, then
   \begin{eqnarray*}
     (T^*h,f+f_1)_{H_1} &=& (T^*h,f)_{H_1}+ (T^*h,f_1)_{H_1}\\
      &=& (T^*h,f)_{H_1}+ (Th,Tf_1)_{H_2}\\
      &=&  (T^*h,f)_{H_1},
   \end{eqnarray*}
   and $f,f+f_1$ are both the solutions of $Tf=g$.
   However, $f\in(\ker T)^{\bot}$ is the condition to assure that the above solution to $Tf=g$ is unique.

   From the above discussion, we have
  \begin{lem}\label{Siu lemma 3}
  (cf. {\rm\cite[Theorem 1.1.4]{Hormander0}})
  If $$\|h\|^2_{H_2}\leq c(\|T^*h\|^2_{H_1}+\|Sh\|^2_{H_3}),$$ then $Tf=g$ has a solution to $g\in \ker S$.
  This solution $f$ satisfies the estimate
  \begin{equation}\label{inequvi5}
    \|f\|_{H_1}\leq c^{\frac{1}{2}}\|g\|_{H_2},\,\,\,f\in(\ker T)^{\bot}.
  \end{equation}
   \end{lem}

    We now return to the  $\mathcal{\widetilde{W}}$, $d^-_J$-problem discussed above.
   If $\varphi$ is a continuous function in $\Omega$, we denote by $L^2(\Omega,\varphi)$ the space of
   functions in $\Omega$ which are square integrable with respect to the measure $e^{-\varphi}d\mu_{g_J}$.
   This is a subspace of the space $L^2(\Omega,loc)$ of functions in $\Omega$ which are locally square integrable with respect to the
   Lebesgue measure, and it is clear that every function in $L^2(\Omega,loc)$
   belongs to $L^2(\Omega,\varphi)$ for some $\varphi$.
   By $\Lambda^k\otimes L^2(\Omega,\varphi)$ we denote the space of $k$-forms with coefficients in $L^2(\Omega,\varphi)$.
   We set
   $$
   \|f\|^2=\int_\Omega|f|^2e^{-\varphi}d\mu_{g_J}.
   $$
   It is clear that $L^2(\Omega,\varphi)$ is a Hilbert space with this norm.

   In our application of the above lemmas, the spaces $H_1$, $H_2$ and $H_3$ will be
   $ L^2_2(\Omega,\varphi)_0$, $\Lambda^1_\mathbb{R}\otimes L^2_1(\Omega,\varphi)$ and $\Lambda^-_J\otimes L^2(\Omega,\varphi)$, respectively,
   $T$ the operator between these space defined as explained above by the $\mathcal{\widetilde{W}}$ operator, and let $G$ be the set of all
   $A\in\Lambda^1_\mathbb{R}\otimes L^2_1(\Omega,\varphi)$ with $d^-_J(A)=0$.
   Let $S$ be the operator from $\Lambda^1_\mathbb{R}\otimes L^2_1(\Omega,\varphi)$
   to $\Lambda^-_J\otimes L^2(\Omega,\varphi)$ defined by $d^-_J$.
   Then $G$ is the null space of $S$, and to prove (\ref{inequvi3}) it will be sufficient to show that
   \begin{equation}\label{inequvi2}
      \|A\|^2_{H_2}\leq C^2(\|T^*A\|^2_{H_1}+\|SA\|^2_{H_3}),\,\,\, A\in D_{T^*}\cap D_S.
   \end{equation}
  To prove this basic inequality, we require the following set steps:

    \vskip 6pt

    {\it Step 1.} The formally adjoint operator, $\mathcal{\widetilde{W}}^{\ast}$, of $T=\mathcal{\widetilde{W}}$ (for $\bar{\partial}$-operator
    cf. L. H\"{o}rmander \cite{Hormander0}).

    First, we calculate it in the non weighted space.
    For all $f\in C^\infty(\bar{\Omega})\subset D_\mathcal{\widetilde{W}}$ (where $ C^\infty(\bar{\Omega})$ is the set of infinitely
    differentiable functions on some neighborhood of $\bar{\Omega}$), we have
    $$
    (\mathcal{\widetilde{W}}(f),A)=(f, \mathcal{\widetilde{W}}^*A).
    $$
  If $supp f\subset\Omega$, $A=u+\bar{u}\in\Omega^1_\mathbb{R}(\bar{\Omega})$ (where $\Omega^1_\mathbb{R}(\bar{\Omega})$ is the set of infinitely differentiable real $1$-forms on some neighborhood of $\bar{\Omega}$),
  $u\in\Omega^{0,1}_J(\bar{\Omega})$  (where $\Omega^{0,1}_J(\bar{\Omega})$ is the set of infinitely differentiable real $(0,1)$-forms with respect
  to the almost complex structure $J$ on some neighborhood of $\bar{\Omega}$) and $d^-_J(A)=0$,
  the above equality becomes
  \begin{eqnarray*}
    (\mathcal{\widetilde{W}}(f),A) &=& -\int_\Omega A\wedge d[f\omega_1+(\eta^1_f+\eta^2_f+\overline{\eta}^1_f+\overline{\eta}^2_f)]\\
     &=& -\int_\Omega d(A)\wedge[f\omega_1+(\eta^1_f+\eta^2_f+\overline{\eta}^1_f+\overline{\eta}^2_f)]  \\
     &=& -\int_\Omega d^+_J(A)\wedge[f\omega_1+(\eta^1_f+\eta^2_f+\overline{\eta}^1_f+\overline{\eta}^2_f)] \\
     &~& -\int_\Omega d^-_J(A)\wedge[f\omega_1+(\eta^1_f+\eta^2_f+\overline{\eta}^1_f+\overline{\eta}^2_f)] \\
     &=& -\int_\Omega d^+_J(A)\wedge fF\\
     &=& (f, \mathcal{\widetilde{W}}^*A)
  \end{eqnarray*}
  is valid to all $f\in C^\infty(\bar{\Omega})_0$.
  Thus, the formally adjoint operator of $\mathcal{\widetilde{W}}$ is
    $$\mathcal{\widetilde{W}}^*A=\frac{-2F\wedge d^+_J(A)}{F^2}.$$
  Then we define $\mathcal{\widetilde{W}}^*$ in weighted space by
  \begin{equation}\label{representation with weighted}
  \mathcal{\widetilde{W}}^*A=\frac{-2F\wedge d^+_J(e^{-\varphi}A)}{F^2}\cdot e^{\varphi}.
  \end{equation}

  {\it Step 2.} Computing $\|\mathcal{\widetilde{W}}^*A\|^2_{H_1}+\|d^-_JA\|^2_{H_3}$,
  as $A\in D_{\mathcal{\widetilde{W}}^*}\cap D_{d^-_J}\cap\Omega^1_\mathbb{R}(\bar{\Omega})$ (for $\bar{\partial}$-operator
    cf. L. H\"{o}rmander \cite{Hormander0}).

  Using the second canonical connection $\nabla^1$ with respect to metric $g_J$ (cf. Appendix \ref{psh} or \cite{G2}),
   for $p\in\Omega$, choose a local moving unitary frame $\{e^1,e^2\}$ for $T^{1,0}(\Omega)$
   and local complex coordinate $\{z^1,z^2\}$ in a neighborhood of $p$ satisfying $e^i(p)=\frac{\partial}{\partial z^i}|_p$
   with
  respect to the Hermitian inner product $h=g_J-\sqrt{-1}F$ (cf. \cite{Cha}).
 Denote $\{\theta_1,\theta_2\}$ by the dual frame of $\{e^1,e^2\}$.
  Hence $$h=g_J-\sqrt{-1}F=\theta_1\otimes\bar{\theta}_1+\theta_2\otimes\bar{\theta}_2$$ and
  $$F=\theta_1\wedge\bar{\theta}_1+\theta_2\wedge\bar{\theta}_2.$$
   By a direct calculation,
  \begin{eqnarray}\label{formular cup F}
   d^+_J(e^{-\varphi}A)\wedge F&=& [\partial_J(e^{-\varphi}u)+\bar{\partial}_J(e^{-\varphi}\bar{u})]\wedge F \nonumber \\
   &=&-e^{-\varphi}(\frac{\partial\varphi}{\partial z^1}\theta_1+\frac{\partial\varphi}{\partial z^2}\theta_2)\wedge(u^1\bar{\theta}_1+u^2\bar{\theta}_2)\wedge F+e^{-\varphi}\partial_J(u^1\bar{\theta}_1+u^2\bar{\theta}_2)\wedge F  \nonumber \\
   &~&+\bar{\partial}_J(e^{-\varphi}\bar{u})\wedge F    \nonumber \\
   &=&-e^{-\varphi}(\frac{\partial\varphi}{\partial z^1}u^1\theta_1\wedge\bar{\theta}_1
   +\frac{\partial\varphi}{\partial z^2}u^2\theta_2\wedge\bar{\theta}_2)\wedge F  \nonumber \\
    &~&+e^{-\varphi}(\frac{\partial u^1}{\partial z^1}\theta_1\wedge\bar{\theta}_1+\frac{\partial u^2}{\partial z^2}\theta_2\wedge\bar{\theta}_2)\wedge F
     +\bar{\partial}_J(e^{-\varphi}\bar{u})\wedge F     \nonumber \\
     &=&-\frac{1}{2}e^{-\varphi}(\frac{\partial\varphi}{\partial z^1}u^1+\frac{\partial\varphi}{\partial z^2}u^2)F^2
     +\frac{1}{2}e^{-\varphi}(\frac{\partial u^1}{\partial z^1}+\frac{\partial u^2}{\partial z^2})F^2 +\bar{\partial}_J(e^{-\varphi}\bar{u})\wedge F, \nonumber \\
     &~& \end{eqnarray}
     where $u=u^1\bar{\theta}_1+u^2\bar{\theta}_2$, $A=u+\bar{u}$.
         Thus, by (\ref{representation with weighted}) and (\ref{formular cup F}),
         \begin{equation}\label{formal adjoint}
           \mathcal{\widetilde{W}}^*A=\frac{\partial\varphi}{\partial z^1}u^1+\frac{\partial\varphi}{\partial z^2}u^2
         -\frac{\partial u^1}{\partial z^1}-\frac{\partial u^2}{\partial z^2}
         +\frac{\partial\varphi}{\partial \bar{z}^1}\bar{u}^1+\frac{\partial\varphi}{\partial \bar{z}^2}\bar{u}^2
         -\frac{\partial \bar{u}^1}{\partial \bar{z}^1}-\frac{\partial \bar{u}^2}{\partial \bar{z}^2}.
         \end{equation}

      Now computing
      $$\|\mathcal{\widetilde{W}}^*A\|^2_{H_1}=\int_\Omega|\sum_i\delta_iu^i|^2e^{-\varphi}
      =\sum_{i,j}\int_\Omega(\delta_iu^i)\overline{(\delta_ju^j)}e^{-\varphi},$$
      where $\delta_iu^i=\frac{\partial u^i}{\partial z^i}-\frac{\partial\varphi}{\partial z^i}u^i$.
      \begin{eqnarray}\label{d J minuous}
         d^-_J(A) &=& d^-_J(u+\bar{u})  \nonumber \\
         &=&\bar{\partial}_Ju+\bar{A}_Ju+\partial_J\bar{u}+A_J\bar{u } \nonumber\\
         &=& (\frac{\partial \bar{u}^2}{\partial z^1}-\frac{\partial \bar{u}^1}{\partial z^2})\theta_1\wedge\theta_2
         +(\frac{\partial u^2}{\partial \bar{z}^1}-\frac{\partial u^1}{\partial \bar{z}^2})\bar{\theta}_1\wedge\bar{\theta}_2 \nonumber\\
         &~&+(A_{J 2}\bar{u}^2-A_{J 1}\bar{u}^1)\theta_1\wedge\theta_2+(\bar{A}_{J 2}u^2-\bar{A}_{J 1}u^1)\bar{\theta}_1\wedge\bar{\theta}_2,
      \end{eqnarray}
  where $A_{J i}$ are the coefficients of $A_J$ which is the linear operator defined in Section $2$.
   So
      \begin{eqnarray*}
        \|d^-_JA\|^2_{H_3} &=&\int_\Omega\sum_{i<j}(|\frac{\partial u^j}{\partial \bar{z}^i}-\frac{\partial u^i}{\partial \bar{z}^j}|^2
        +|A_{J j}\bar{u}^j-A_{J i}\bar{u}^i|^2)e^{-\varphi}  \\
         &=& \sum_{i,j}\int_\Omega(|\frac{\partial u^j}{\partial \bar{z}^i}|^2-\frac{\partial u^j}{\partial \bar{z}^i}\frac{\partial \bar{u}^i}{\partial z^j})e^{-\varphi}
        +\int_\Omega\sum_{i<j}|A_{J j}\bar{u}^j-A_{J i}\bar{u}^i|^2e^{-\varphi}.
      \end{eqnarray*}
    Hence,
      \begin{eqnarray}
          \|\mathcal{\widetilde{W}}^*A\|^2_{H_1}+\|d^-_JA\|^2_{H_3}&=&\sum_{i,j}\int_\Omega|\frac{\partial u^j}{\partial \bar{z}^i}|^2e^{-\varphi}
          +\int_\Omega\sum_{i<j}|A_{J j}\bar{u}^j-A_{J i}\bar{u}^i|^2e^{-\varphi} \nonumber\\
          &~&+\sum_{i,j}\int_\Omega((\delta_iu^i)\overline{(\delta_ju^j)}
          -\frac{\partial u^j}{\partial \bar{z}^i}\frac{\partial \bar{u}^i}{\partial z^j})e^{-\varphi}.
      \end{eqnarray}

      Before continuing discussing, we need a formula which is basically the divergence theorem.
  \begin{prop}\label{divergence theorem}
 (for $\bar{\partial}$ operator, see{\rm\cite[ChapterII]{Hormander0}}{\rm\cite[ChapterIV]{Hormander}})
   If the boundary $\partial\Omega=\{r=0\}$ of a bounded domain $\Omega=\{r<0\}\subset(\mathbb{R}^4,J)$ is differentiable, $|dr|=1$ on $\partial\Omega$
   with respect to the metric $g_J$,
   and $L=\sum_ia_i\frac{\partial}{\partial x^i}$ is a differentiable operator of $1$-order with constant coefficients,
   then $$\int_\Omega Lf=\int_{\partial\Omega}(Lr)f. $$
   \end{prop}

   By the above proposition, we can get
   \begin{eqnarray*}
      \sum\int_\Omega f\overline{\frac{\partial(u^ie^{-\varphi})}{\partial z^i}}&=&-\sum\int_\Omega\frac{\partial f}{\partial\bar{z}^i}\bar{u}^ie^{-\varphi}
            +\sum\int_\Omega\frac{\partial(f\bar{u}^ie^{-\varphi})}{\partial\bar{z}^i} \\
      &=&-\sum\int_\Omega\frac{\partial f}{\partial\bar{z}^i}\bar{u}^ie^{-\varphi}
            +\sum\int_{\partial\Omega}\frac{\partial r}{\partial\bar{z}^i}(f\bar{u}^ie^{-\varphi}).
   \end{eqnarray*}
   We can reduce the deduced formula above to
   \begin{equation}\label{dual}
     (f,\delta_ig)=-(\bar{\partial}_if,g)+((\bar{\partial}_ir)f,g)_{\partial\Omega},
   \end{equation}
   where $f,g\in C^\infty(\bar{\Omega})$, and $(\cdot,\cdot)_{\partial\Omega}$ indicates the integral on $\partial\Omega$
   relative to the weight factor $e^{-\varphi}$.
   By (\ref{dual}),
   $$
   \int_\Omega(\delta_iu^i)\overline{(\delta_ju^j)}e^{-\varphi}=-(\bar{\partial}_j\delta_iu^i,u^j)+((\bar{\partial}_jr)\delta_iu^i,u^j)_{\partial\Omega},
   $$
   $$
   \int_\Omega(\bar{\partial}_iu^j)\overline{(\bar{\partial}_ju^i)}e^{-\varphi}=-(u^j,\delta_i\bar{\partial}_ju^i)+((\bar{\partial}_ir)u^j,\bar{\partial}_ju^i)_{\partial\Omega}.
   $$
   Then,
    \begin{eqnarray}
          \|\mathcal{\widetilde{W}}^*A\|^2_{H_1}+\|d^-_JA\|^2_{H_3}&=&\sum_{i,j}\int_\Omega|\frac{\partial u^j}{\partial \bar{z}^i}|^2e^{-\varphi}
          +\int_\Omega\sum_{i<j}|A_{J j}\bar{u}^j-A_{J i}\bar{u}^i|^2e^{-\varphi} \nonumber\\
          &~&+\sum_{i,j}((\delta_i\bar{\partial}_j-\bar{\partial}_j\delta_i)u^i,u^j)+\sum_{i,j}\int_{\partial\Omega}(\bar{\partial}_ir)(\delta_iu^j)\bar{u}^ie^{-\varphi} \nonumber\\
           &~&- \sum_{i,j}\int_{\partial\Omega}(\partial_ir)\bar{u}^j(\overline{\partial_j\bar{u}^i})e^{-\varphi} \nonumber\\
           &=&\sum_{i,j}\int_\Omega|\frac{\partial u^j}{\partial \bar{z}^i}|^2e^{-\varphi}
          +\int_\Omega\sum_{i<j}|A_{J j}\bar{u}^j-A_{J i}\bar{u}^i|^2e^{-\varphi} \nonumber\\
          &~&+\sum_{i,j}\int_\Omega(\bar{\partial}_j\partial_i\varphi)u^i\bar{u}^je^{-\varphi}
          +\sum_j\int_{\partial\Omega}(\delta_iu^j)\sum_i(\bar{\partial}_ir)\bar{u}^ie^{-\varphi}     \nonumber\\
           &~&- \sum_{i,j}\int_{\partial\Omega}(\partial_ir)\bar{u}^j(\overline{\partial_j\bar{u}^i})e^{-\varphi}.
      \end{eqnarray}
      If we add conditions
    \begin{equation}\label{boundary conditions}
   \sum_i(\partial_ir)u^i|_{\partial\Omega}=0
    \end{equation}
    to $A=u+\bar{u}$,
     then
     \begin{eqnarray*}
       \|\mathcal{\widetilde{W}}^*A\|^2_{H_1}+\|d^-_JA\|^2_{H_3} &=& \sum_{i,j}\int_\Omega|\frac{\partial u^j}{\partial \bar{z}^i}|^2e^{-\varphi}
          +\int_\Omega\sum_{i<j}|A_{J j}\bar{u}^j-A_{J i}\bar{u}^i|^2e^{-\varphi} \\
        &~&+\sum_{i,j}\int_\Omega(\bar{\partial}_j\partial_i\varphi)u^i\bar{u}^je^{-\varphi}
           - \sum_{i,j}\int_{\partial\Omega}(\partial_ir)\bar{u}^j(\overline{\partial_j\bar{u}^i})e^{-\varphi}.
     \end{eqnarray*}

     {\it Step 3.} The domination of the boundary term--Morrey's trick (cf. Morrey \cite{Mo} or H\"{o}rmander \cite[Chapter II]{Hormander0}).

     The method is: Let $A\in D_{\mathcal{\widetilde{W}}^*}\cap\Omega^1_\mathbb{R}(\bar{\Omega})$, $r=0$ define the boundary of $\Omega$,
     and the defining function $r$ be differentiable.
          Thus $\sum_i(\partial_ir)u^i$ are local functions, differentiable at every point.
       By (\ref{boundary conditions}), these functions vanish at $r=0$, i.e. on $\partial\Omega$.
     By Taylor expansion, it can be written as
      $$
        \sum_i(\partial_ir)u^i=\lambda r,
      $$
    where $\lambda$ is some differentiable function.
   Taking $\bar{\partial}_j$ to both sides to yield
  $$
  \sum_i(\bar{\partial}_j\partial_ir)u^i+\sum_i(\partial_ir)(\bar{\partial}_ju^i)=(\bar{\partial}_j\lambda)r+\lambda\bar{\partial}_jr.
  $$
  Multiplying $\bar{u}^j$ and summing up for $j$,
  $$
    \sum_{i,j}(\bar{\partial}_j\partial_ir)u^i\bar{u}^j+\sum_{i,j}(\partial_ir)(\bar{\partial}_ju^i)\bar{u}^j
   =\sum_jr(\bar{\partial}_j\lambda)\bar{u}^j+\sum_j\lambda(\bar{\partial}_jr)\bar{u}^j.
  $$
  Integrating on $\partial\Omega$, noting $r=0$ on $\partial\Omega$, $\sum_i(\partial_ir)u^i|_{\partial\Omega}=0$, to get
  $$
  -\sum_{i,j}\int_{\partial\Omega}(\partial_ir)(\bar{\partial}_ju^i)\bar{u}^je^{-\varphi}
   =\sum_{i,j}\int_{\partial\Omega}(\bar{\partial}_j\partial_ir)u^i\bar{u}^je^{-\varphi}.
  $$
  Then we get
  \begin{eqnarray}\label{normal}
        \|\mathcal{\widetilde{W}}^*A\|^2_{H_1}+\|d^-_JA\|^2_{H_3} &=& \sum_{i,j}\int_\Omega|\frac{\partial u^j}{\partial \bar{z}^i}|^2e^{-\varphi}
          +\int_\Omega\sum_{i<j}|A_{J j}\bar{u}^j-A_{J i}\bar{u}^i|^2e^{-\varphi} \nonumber\\
           &&+\sum_{i,j}\int_\Omega(\bar{\partial}_j\partial_i\varphi)u^i\bar{u}^je^{-\varphi} \nonumber\\
           && +\sum_{i,j}\int_{\partial\Omega}(\bar{\partial}_j\partial_ir)u^i\bar{u}^je^{-\varphi}.
     \end{eqnarray}
  Note that we have not made any special restrictions to the choice of $\varphi$ so far.
   Now we assume

  (1) $\Omega$ is a compact $J$-pseudoconvex domain, i.e.
   $$
    \sum_{i,j}(\bar{\partial}_j\partial_ir)\xi^i\bar{\xi}^j\geq 0,\,\,\,\forall \sum_i(\partial_ir)\xi^i=0;
   $$

  (2) $\varphi$ satisfies that complex Hessian is strictly positive-definite
         (i.e. $\varphi$ is a strictly $J$-plurisubharmonic function (cf. Harvey-Lawson \cite{HL3} or Appendix \ref{psh})),
   that is,
   there exists $c>0$ such that
   $$
  \sum_{i,j}(\partial_i\bar{\partial}_j\varphi)\xi^i\bar{\xi}^j\geq c\sum_i|\xi^i|^2.
   $$
  Under the two assumptions above, we have proved the following theorem:
  \begin{prop}\label{equi estimate}
   (for $\bar{\partial}$-problem see {\rm\cite{Hormander0,Hormander}})
   Let $\Omega$ be a compact $J$-pseudoconvex domain. Given a real valued function $\varphi\in C^\infty(\bar{\Omega})$
  satisfying $\sum_{i,j}(\partial_i\bar{\partial}_j\varphi)\xi^i\bar{\xi}^j\geq c\sum_i|\xi^i|^2$, $c>0$,
  then for $A\in D_{\mathcal{\widetilde{W}}^*}\cap D_{d^-_J}\cap\Omega^1_\mathbb{R}(\bar{\Omega})$, we have
  $$
    c\|A\|^2_{H_2}\leq\|\mathcal{\widetilde{W}}^*A\|^2_{H_1}+\|d^-_JA\|^2_{H_3}.
  $$
 \end{prop}

  Recall that in the previous discussion, if for all $A\in D_{\mathcal{\widetilde{W}}^*}\cap D_{d^-_J}$,
 we have
 $$c\|A\|^2_{H_2}\leq\|\mathcal{\widetilde{W}}^*A\|^2_{H_1}+\|d^-_JA\|^2_{H_3},$$ then the
  $\mathcal{\widetilde{W}}, d^-_J$-problem of a $J$-pseudoconvex domain has a solution (which is similar to the $\bar{\partial}$-problem in \cite{Hormander0,Hormander}).
  However, Proposition \ref{equi estimate} implies that $$c\|A\|^2_{H_2}\leq\|\mathcal{\widetilde{W}}^*A\|^2_{H_1}+\|d^-_JA\|^2_{H_3}$$
 holds for all infinitely differentiable functions in $D_{\mathcal{\widetilde{W}}^*}\cap D_{d^-_J}$.
  To prove that this estimate holds for all $A$ in $D_{\mathcal{\widetilde{W}}^*}\cap D_{d^-_J}$,
   it suffices to show that, $\forall A\in D_{\mathcal{\widetilde{W}}^*}\cap D_{d^-_J}$ there exists a sequence $A_\nu\in D_{\mathcal{\widetilde{W}}^*}\cap D_{d^-_J}\cap\Omega^1_\mathbb{R}(\bar{\Omega})$
  such that
  $$
 A_\nu\rightarrow A,\,\,\, \mathcal{\widetilde{W}}^*A_\nu\rightarrow \mathcal{\widetilde{W}}^*A,\,\,\, d^-_JA_\nu\rightarrow d^-_JA.
  $$
  Note that it is important to prove that this convergence holds at the same time.
  It is easy to prove that the first and the third hold.
   The question becomes to show that the second holds at the same time.
  The method presented below is called the regularization method of K. Friedrichs,
   first due to K. Friedrichs \cite{Fr} in 1944, and later further developed by L. H\"{o}rmander \cite{Hormander0} in 1965.

   Let  a domain $\Omega\subset\mathbb{R}^n$, $L$ be a linear differential operator
  $$
   L:C^\infty(\bar{\Omega})\longrightarrow C^\infty(\bar{\Omega}).
  $$
  We want to extend $L$ to $L_1$,
  $$
   L_1:L^2(\Omega)\longrightarrow L^2(\Omega).
  $$
   There are two ways to do the extension (cf. L. H\"{o}rmander \cite{Hormander0,Hormander}):

     1. The strict extension. $L_1$ is the closed extension of $L$, that is, $L_1=\bar{L}$.
    The definition is : $L_1f=g$ is equivalent to that there exists $f_\nu\in C^\infty(\bar{\Omega})$ such that
  $f_\nu\rightarrow f$, $Lf_\nu\rightarrow g$ (the convergence in the sense of $L^2$).

      2. The weak extension. The extension is in the sense of distributions, i.e. as $f,g\in L^2$.
   The definition of $Lf=g$ is:
   $$
    (g,\varphi)=(f,L^*\varphi)
   $$
    to every $\varphi\in C^\infty(\Omega)_0$.

   \begin{theo}\label{Friedrichs}
   ({\bf Friedrichs} )
  If L is a differential operator of first-order,
    the weak extension is equivalent to the strict extension (that is, the weak extension implies the strict extension).
   \end{theo}

  \begin{rem}
  It is enough to require that $\varphi$ is a strictly $J$-plurisubharmonic function.
    If $J$ is integrable, then $\mathcal{\widetilde{W}}, d^-_J$-problem becomes $\bar{\partial}$-problem, hence
    Proposition \ref{equi estimate} is a generalization of Theorem 4.2.2 in {\rm\cite{Hormander}}.
  \end{rem}

  Now we return to prove the iequality
  $$
 c\|A\|^2_{H_2}\leq\|\mathcal{\widetilde{W}}^*A\|^2_{H_1}+\|d^-_JA\|^2_{H_3},\,\,A\in D_{\widetilde{W}^*}\cap D_{d^-_J}.
  $$
  We have proved the case for $A\in\Omega^1_{\mathbb{R}}(\bar{\Omega})$.
  For $A\in D_{\widetilde{W}^*}\cap D_{d^-_J}$, we need to find $A_\nu\in\Omega^1_{\mathbb{R}}(\bar{\Omega})$ so that
  $$
  A_\nu\rightarrow A,\,\,\mathcal{\widetilde{W}}^*A_\nu\rightarrow\mathcal{\widetilde{W}}^*A,\,\,d^-_JA_\nu\rightarrow d^-_JA.
  $$
  We can do that by using the smoothing method of K. Friedrichs.
  Since $A\in D_{\widetilde{W}^*}\cap D_{d^-_J}$,
  $\mathcal{\widetilde{W}}^*A$ and $d^-_JA$ exists.
  Note by the definition of  $\mathcal{\widetilde{W}}^*$,
  $\mathcal{\widetilde{W}}^*A=f$ is in the sense of weak extension, and $d^-_J$ is a closed operator,
   $d^-_JA$ is in the sense of strict extension.
   Obviously, sttict extension implies weak one, so, in the sense of distributions ( recall (\ref{formal adjoint})-(\ref{d J minuous})),
   we have
   \begin{equation}\label{long eq1}
     \mathcal{\widetilde{W}}^*A=\frac{\partial\varphi}{\partial z^1}u^1+\frac{\partial\varphi}{\partial z^2}u^2+\frac{\partial\varphi}{\partial \bar{z}^1}\bar{u}^1+\frac{\partial\varphi}{\partial \bar{z}^2}\bar{u}^2-\frac{\partial u^1}{\partial z^1}-\frac{\partial u^2}{\partial z^2}
     -\frac{\partial \bar{u}^1}{\partial \bar{z}^1}-\frac{\partial \bar{u}^2}{\partial \bar{z}^2},
   \end{equation}
   where $A=u+\bar{u}$, $u=u^1\bar{\theta}_1+u^2\bar{\theta}_2\in\Omega^{0,1}_J(\bar{\Omega})$,
   $\{\theta_1,\theta_2\}$ is the dual frame of the local moving unitary frame $\{e^1,e^2\}$ for $T^{1,0}(\bar{\Omega})$;
   \begin{eqnarray}\label{long eq2}
      d^-_JA&=& (\frac{\partial \bar{u}^2}{\partial z^1}-\frac{\partial \bar{u}^1}{\partial z^2})\theta_1\wedge\theta_2
     +(\frac{\partial u^2}{\partial \bar{z}^1}-\frac{\partial u^1}{\partial \bar{z}^2})\bar{\theta}_1\wedge\bar{\theta}_2 \nonumber \\
      & +&(A_{J_2}\bar{u}^2-A_{J_1}\bar{u}^1)\theta_1\wedge\theta_2 +(\bar{A}_{J_2}u^2-\bar{A}_{J_1}u^1)\bar{\theta}_1\wedge\bar{\theta}_2,
   \end{eqnarray}
   where $$A_J:\Omega^{1,0}_J(\bar{\Omega})-\Omega^{0,2}_J(\bar{\Omega}), \,\,\bar{A}_J:\Omega^{0,1}_J(\bar{\Omega})-\Omega^{2,0}_J(\bar{\Omega}),$$
   are linear operators depending on $J$ (if $J$ is integrable, $A_J=0=\bar{A}_J$),
   $A_{J_i}$, $i=1,2$, are the coefficients of $A_J$ (more details, see Section \ref{2}).
  There are linear differential equations of first order.
  By the smoothing method of Friedrichs (Friedrichs theorem holds for first-order differential operator),
  setting $A_\varepsilon=A\ast\chi_\varepsilon$ (where $A\ast\chi_\varepsilon$ is the convolution of
   $A$ with respect to mean value function $\chi_\varepsilon$, cf. \cite{Hormander0,Hormander}),
   then $$\mathcal{\widetilde{W}}^*A_\varepsilon\rightarrow\mathcal{\widetilde{W}}^*A,
   d^-_JA_\varepsilon\rightarrow d^-_JA, A_\varepsilon\rightarrow A.$$
   Note that $A_\varepsilon$ which is obtained by quoting Friedrichs regularization method directly,
   is contained in $\Omega^1_{\mathbb{R}}(\bar{\Omega})$.
   However, it is not clear whether it is in $D_{\mathcal{\widetilde{W}}^*}$,
   since that $A_\varepsilon\in D_{\mathcal{\widetilde{W}}^*}\cap \Omega^1_{\mathbb{R}}(\bar{\Omega})$
   has to satisfy the boundary condition (cf. (\ref{boundary conditions})):
   \begin{equation}\label{boundary con}
    \sum^2_{i=1}(\partial_ir)u^i_\varepsilon|_{\partial\Omega}=0,\,\,\,A_\varepsilon=u_\varepsilon+\bar{u}_\varepsilon.
   \end{equation}

   How do all $A_\varepsilon$ satisfy (\ref{boundary con}) at the same time?
   In 1965, L. H\"{o}rmander \cite{Hormander} further extended Friedrichs regularization method to satisfy the given boundary conditions.

   Assume $\Omega=\{r<0\}\subset\mathbb{R}^N$, we consider differential equations system (in the sense of distribution) on $\Omega$:
   \begin{equation}\label{sumsum}
    \sum^N_{i=1}\sum^I_{j=1}b^k_{ij}D_iu_j+\sum^I_{j=1}c^k_ju_j=f_k, \,\,\,1\leq k\leq I,
   \end{equation}
   where $D_i=\frac{\partial}{\partial x_i}$, $i=1,\cdot\cdot\cdot,N$, $b^k_{ij},c^k_j\in C^\infty(\bar{\Omega})$.
   We write them in a matrix form:
   \begin{equation}\label{}
    Bu+Cu=f
   \end{equation}
   where $u=(u_1,\cdot\cdot\cdot,u_I)^T$, $f=(f_1,\cdot\cdot\cdot,f_I)^T$.
   The actual situation over here is
   \begin{eqnarray*}
     f=\left(
     \begin{array}{c}
       T^*u   \\
       S^*u  \\
     \end{array}
   \right).
   \end{eqnarray*}
 We set the former $K^0$ equations of (\ref{sumsum}) by
 \begin{equation}\label{}
     B^0u+C^0u=f^0.
 \end{equation}

 Next we see how to describe the boundary conditions.
 For $u\in L^2(\Omega)$, we denote its null extension by $\tilde{u}$
 \begin{eqnarray*}
  u\rightarrow\tilde{u}\in L^2(\mathbb{R}^N),
 \end{eqnarray*}
    \begin{equation}
  \tilde{u}(x)=\left\{
    \begin{array}{ll}
    u(x), & ~x\in\Omega,\\
      &   \\
    0, & ~ x\in\mathbb{R}^N\setminus\Omega.
   \end{array}
  \right.
    \end{equation}
  We know that $u\in D_{T^*}\Leftrightarrow (T\varphi,u)=(\varphi,T^*u)$,
  $\forall \varphi\in D_T$.
  That is $$
  \int_\Omega(T\varphi)u=\int_\Omega\varphi(T^*u).
  $$
  In particular, it is true for a $C^\infty$ function $\varphi$ with a compact support in $\mathbb{R}^N$, but
  $$
  \int_\Omega(T\varphi)u=\int_\Omega\varphi(T^*u)=\int_{\mathbb{R}^N}\varphi\widetilde{(T^*u)},
  $$
  while
  $$
  \int_\Omega(T\varphi)u=\int_{\mathbb{R}^N}(T\varphi)\tilde{u},
  $$
  so
  $$
  \int_{\mathbb{R}^N}(T\varphi)\tilde{u}=\int_{\mathbb{R}^N}\varphi\widetilde{(T^*u)}.
  $$
  It is true for each $C^\infty$ function $\varphi$ with its support in $\mathbb{R}^N$,
  thus
  \begin{equation}\label{}
    T^\ast\tilde{u}=\widetilde{(T^*u)}.
  \end{equation}
  So we consider that the equations and their boundary conditions are
   \begin{equation}
  \left\{
    \begin{array}{ll}
   (B+C)u=f, \\
      &   \\
    (B^0+C^0)\tilde{u}=\tilde{f^0}.
   \end{array}
  \right.
    \end{equation}
  We have the following Friedrichs-H\"{o}rmander Theorem (cf. L. H\"{o}rmander \cite[Proposition 1.2.4]{Hormander}):
   Let $u,f\in L^2(\Omega)$ satisfy (in the sense of distributions) equations
    \begin{equation}
  \tilde{u}(x)=\left\{
    \begin{array}{ll}
   (B+C)u=f, & ~B=\left(
                    \begin{array}{c}
                      B^0 \\
                     \ast \\
                    \end{array}
                  \right)_{K\times I}
   ,~C=\left(
                    \begin{array}{c}
                      C^0 \\
                     \ast \\
                    \end{array}
                  \right)_{K\times I},\\
      &   \\
    (B^0+C^0)\tilde{u}=\tilde{f^0}, & ~ f=\left(
                    \begin{array}{c}
                      f^0 \\
                     \ast \\
                    \end{array}
                  \right)_{I\times 1},
   \end{array}
  \right.
    \end{equation}
  where $\Omega=\{r<0\}\subset\subset \mathbb{R}^N$.
  If the ranks of $B^0(r)$ at each point in $\partial\Omega$ are constants,
  there is a sequence of $u_\nu\in C^\infty(\Omega)$ such that
  \begin{eqnarray*}
   \left\{
    \begin{array}{ll}
    u_\nu\rightarrow u;\\
      &   \\
    Bu_\nu+Cu_\nu\rightarrow f;\\
     &   \\
      B^0\tilde{u}_\nu+C^0\tilde{u}_\nu\rightarrow  \widetilde{B^0u_\nu+C^0u_\nu}.\\
   \end{array}
  \right.
  \end{eqnarray*}

  Now we return to $\mathcal{\widetilde{W}},d^-_J$-problem.
  In our discussed situations,
  $\Omega=\{r<1\}\subset\subset\mathbb{R}^4$,
  $T^*=\mathcal{\widetilde{W}}^*$, $S=d_J^-$.
  For $A\in D_{\mathcal{\widetilde{W}}^*}\cap D_{d^-_J}$,
  $$ f=\left(
    \begin{array}{c}
      \mathcal{\widetilde{W}}^*A \\
      d_J^-A \\
    \end{array}
  \right).
  $$
  In terms of local moving unitary dual frame $\{\theta_1,\theta_2\}$,
  $$ A=u+\bar{u}=u^1\bar{\theta}_1+u^2\bar{\theta}_2+\bar{u}^1\theta_1+\bar{u}^2\theta_2. $$
  By (\ref{long eq1}) and (\ref{long eq2})
  $$
   \mathcal{\widetilde{W}}^*A=\frac{\partial\varphi}{\partial z^1}u^1+\frac{\partial\varphi}{\partial z^2}u^2+\frac{\partial\varphi}{\partial \bar{z}^1}\bar{u}^1+\frac{\partial\varphi}{\partial \bar{z}^2}\bar{u}^2-\frac{\partial u^1}{\partial z^1}-\frac{\partial u^2}{\partial z^2}
     -\frac{\partial \bar{u}^1}{\partial \bar{z}^1}-\frac{\partial \bar{u}^2}{\partial \bar{z}^2},
  $$
  \begin{eqnarray*}
      d^-_JA&=& (\frac{\partial \bar{u}^2}{\partial z^1}-\frac{\partial \bar{u}^1}{\partial z^2})\theta_1\wedge\theta_2
     +(\frac{\partial u^2}{\partial \bar{z}^1}-\frac{\partial u^1}{\partial \bar{z}^2})\bar{\theta}_1\wedge\bar{\theta}_2  \\
      & +&(A_{J_2}\bar{u}^2-A_{J_1}\bar{u}^1)\theta_1\wedge\theta_2 +(\bar{A}_{J_2}u^2-\bar{A}_{J_1}u^1)\bar{\theta}_1\wedge\bar{\theta}_2.
   \end{eqnarray*}
 The $1$-form $A$ can be written as a vector: $A_1=(u^1,u^2,\bar{u}^1,\bar{u}^2)^T$.
 Hence we have a matrix equation
 $$f_1=\left(
     \begin{array}{c}
       B^0A_1+C^0A_1 \\
       DA_1+EA_1 \\
     \end{array}
   \right),
   $$
   which is equivalent to $$ f=\left(
    \begin{array}{c}
      \mathcal{\widetilde{W}}^*A \\
      d_J^-A \\
    \end{array}
  \right).
  $$
  It is easy to see that
  $$
  B^0=(-\frac{\partial}{\partial z^1}\,\,-\frac{\partial}{\partial z^2}\,\,-\frac{\partial}{\partial \bar{z}^1}\,\,-\frac{\partial}{\partial \bar{z}^2}),
  $$
   $$
  C^0=(\frac{\partial\varphi}{\partial z^1}\,\,\,\frac{\partial\varphi}{\partial z^2}\,\,\,\frac{\partial\varphi}{\partial \bar{z}^1}\,\,\,\frac{\partial\varphi}{\partial \bar{z}^2}),\,\,\,K^0=1,
  $$
  $$
  D=\left(
     \begin{array}{cccc}
       0 & 0 & -\frac{\partial}{\partial z_2} & \frac{\partial}{\partial z_1}  \\
      -\frac{\partial}{\partial \bar{z}_2}  & \frac{\partial}{\partial \bar{z}_1} & 0 & 0 \\
     \end{array}
   \right),\,\,
   E=\left(
     \begin{array}{cccc}
       0 & 0 & -A_{J_1} & A_{J_2} \\
      -\bar{A}_{J_1}   & \bar{A}_{J_2} & 0 & 0 \\
     \end{array}
   \right).
  $$

  By Friedrichs-H\"{o}rmander Theorem, having proved that for a $J$-pseudoconvex domain $\Omega$ in a tamed almost complex $4$-manifold $(M,J)$,
 if $\varphi\in C^\infty(\bar{\Omega})$ satisfies $$\sum_{i,j}(\partial_i\bar{\partial}_j\varphi)\xi^i\bar{\xi}^j\geq c\sum_i|\xi^i|^2,\,\, c>0,$$
  then for $A\in D_{\mathcal{\widetilde{W}}^*}\cap D_{d^-_J}$, we have
  $$
    c\|A\|^2_{H_2}\leq\|\mathcal{\widetilde{W}}^*A\|^2_{H_1}+\|d^-_JA\|^2_{H_3}.
  $$
   Combining the former part of this subsection, we solved the $\mathcal{\widetilde{W}}$,
   $d^-_J$-problem (as the $\bar{\partial}$-problem in classical complex analysis) of
  $J$-pseudoconvex domain in the sense of distribution (for $\bar{\partial}$-problem see \cite{Hormander0,Hormander}).
   \begin{theo}\label{app 1}
  Let $\Omega$ be a compact $J$-pseudoconvex domain in a tamed almost complex $4$-manifold.
   Given a real valued function $\varphi\in C^\infty(\bar{\Omega})$
  satisfies $$\sum_{i,j}(\partial_i\bar{\partial}_j\varphi)\xi^i\bar{\xi}^j\geq c\sum_i|\xi^i|^2,\,\,c>0,$$
  then for all $A\in\Lambda^1_\mathbb{R}\otimes L^2_1(\Omega,\varphi)$
   and satisfy $d^-_J(A)=0$, then there exists $f\in L^2_2(\Omega,\varphi)_0$ such that
    $$
    \mathcal{\widetilde{W}}(f)=A,\,\,\, \|f\|_{H_1}\leq\frac{1}{\sqrt{c}}\|A\|_{H_2}.
    $$
  \end{theo}

    \begin{rem}
    1. As in classical complex analysis, there is the regularity properties of the solution, i.e.,
    when $A$ has enough differentiability, the solution $f$ to $\mathcal{\widetilde{W}}(f)=A$ must have appropriate
    differentiability (for $\bar{\partial}$-problem, see J. J. Kohn \cite{Kohn1,Kohn2}).
    A stronger result is:
     For a strictly pseudoconvex domain $\Omega$, $\mathcal{\widetilde{W}}(f)=A$.
     If $A\in \Omega^1_{\mathbb{R}}(\bar{\Omega})$, then $f\in C^\infty(\bar{\Omega})$.

  2. It is well known that $\bar{\partial}$-problem in classical complex analysis is for any dimension.
     It is natural to ask that could we consider $\mathcal{\widetilde{W}}$, $d^-_J$-problem for higher dimensional almost K\"{a}hler manifolds.
    \end{rem}

     \subsection{The singularities of $J$-plurisubharmonic functions on tamed almost complex $4$-manifolds}\label{singularity}

   The goal of this subsection is to study singularities of $J$-plurisubharmonic functions on tamed almost complex $4$-manifolds
   as in classical complex analysis.
      F. Elkhadhra had the following result (cf. \cite[Proposition 1]{Elk}):

    Let $\Omega$ be an open set of $\mathbb{R}^{2n}$ equipped with an almost complex structure $J$ of class $C^1$.
    Let $N$ be a $C^2$ submanifold of codimension $2k$ such that $J(TN)=TN$.
  Then for every $x_0\in N$ there exists an open neighborhood $U$ of $x_0$ and functions $f_1,\cdot\cdot\cdot,f_k$ of class $C^2$
  on $U$ such that
 $$
  N\cap U=\{x\in U\mid  f_1(x)=\cdot\cdot\cdot=f_k(x)=0,\,\,\,\bar{\partial}_Jf_j=0
$$
$$
   \,\,\, {\rm on}\,\,\, N\cap U,\,\,\,{\rm and}\,\,\,\partial_Jf_1\wedge\cdot\cdot\cdot\wedge\partial_Jf_k\neq 0 \,\,\, {\rm on}\,\,\, U\}.
$$
  Moreover there exists a $J$-plurisubharmonic function $u$ on $U$ of class $C^2$ on $U\backslash N$ such that $ N\cap U=\{u=-\infty\}$.

   In fact, if $(M,J)$ is an almost complex manifold, and $f$ a $J$-holomorphic function at some point $p\in M$.
   Then, for all vector fields $X,Y$, $df(\mathcal{N}_J(X,Y))=0$ at $p$, where $\mathcal{N}_J$ is the Nijenhuis tensor (cf. Lemma 3.2 in Wang-Zhu \cite{WZ2}).
   Note that if there exist $n$ $J$-holomorphic functions on a real $2n$-dimensional almost Hermitian manifold $(M,g,J)$ which are independent at some point $p\in M$,
     then the Nijenhuis tensor $\mathcal{N}_J$  identically vanishes at $p$.
   This means that an integrable complex structure is one with many holomorphic functions.
  It is a hard theorem (Newlander-Nirenberg integrability theorem for almost complex structures) that the converse is also true.
    In general, an almost complex manifold has no holomorphic functions at all.
 On the other hand, it has a lot of $J$-holomorphic curves (i.e., maps $u:\mathbb{C}\rightarrow (M,g,J)$ such that $df\circ i=J\circ df$)
   (cf. M. Gromov \cite{G3}).

  As done in Theorems 4.4.2-4.4.5 of L. H\"{o}rmander {\rm\cite{Hormander}}, we study a $J$-plurisubharmonic function $\varphi$ which is not identically $-\infty$
  on a connected $J$-pseudoconvex open set $\Omega$, then $e^{-\varphi}$ is locally integrable in a dense open subset of $\Omega$.
  Therefore we have the following theorem:
   \begin{theo}\label{HBS2}
   Suppose that $(M,J)$ is an almost complex $4$-manifold which is tamed by symplectic form $\omega_1=F+d^-_J(v+\bar{v})$,
   where $F$ is the fundamental $2$-form on $M$.
   $g_J(\cdot,\cdot):= F(\cdot,J\cdot)$ is an almost Hermitian metric on $M$.
    Let $\varphi$ be a strictly $J$-plurisubharmonic function on a $J$-pseudoconvex open set $\Omega\subset M$.
   If $p\in\Omega$, there exists a neighborhood of $p$ such that the set of points of which $e^{-\varphi}$ is not integrable in this neighborhood
  is a $J$-analytic subset of $\Omega$ of dimension (complex) $\leq 1$.
    \end{theo}

   \begin{rem}
   According to Gromov's fundamental theory of $J$-holomorphic curves {\rm \cite{G3}},
   almost complex submanifolds of complex dimension one always exist locally in a given almost complex manifold (there are no local obstructions).
 These curves can be realized globally as images of Riemann surfaces under $J$-holomorphic maps.
   In higher dimension, even through the existence of almost submanifolds can be obstructed.
   Donaldson {\rm \cite{D4}} has shown that every compact symplectic manifold admits symplectic submanifolds which is done by approximating a compatible
   almost complex structure.
   It is natural to ask the following question:
 Could one generalize Theorem \ref{HBS2} to higher dimensional symplectic manifolds for closed positive $(1,1)$-currents or $(n-1,n-1)$-currents ($n>2$).
   \end{rem}

   {\it Proof} of Theorem \ref{HBS2}:
   Since any almost complex $4$-manifold has the local symplectic property (cf. \cite{L2}), there exists an open set $U_p\subset\Omega$ and a symplectic form $\omega_p$ on $U_p$
   such that $F|_p=\omega_p|_p$.
   Hence we choose a Darboux coordinate chart $$\{(z_1,z_2)\mid  z_1(p)=z_2(p)=0\}$$ for the symplectic form $\omega_p$.
   Without loss of generality, we may assume that $U_p$ is the Darboux coordinate chart (see \cite{A2}).
   Let  $$g'_J(\cdot,\cdot):=\omega_p(\cdot,J\cdot), \,\,g_0(\cdot,\cdot):= \omega_p(\cdot,J_{st}\cdot),$$
  then $g'_J(p)=g_0(p)=g_J(p)$.
  Since $$dd^c_{J_{st}}(|z_1|^2+|z_2|^2)=2\sqrt{-1}(dz_1\wedge d\bar{z}_1+dz_2\wedge d\bar{z}_2),$$
   $|z_1|^2+|z_2|^2$ is a strictly plurisubharmonic function in classical sense on the Darboux coordinate chart.
   Let $$B_r(p):=\{|z_1|^2+|z_2|^2<r\}\subset U_p$$ and $B_r(p)$ is a strictly pseudoconvex domain.
  $\|J-J_{st}\|$ is small on $B_r(p)$ when $r$ is small enough (cf. \cite{DS,D4,HL3,TWZ}).
  Indeed, we can get
   \begin{equation}
  g'_J|_{B_r(p)}=g_0|_{B_r(p)}\cdot e^h,
   \end{equation}
   where $h$ is a symmetric $J$-anti-invariant $(2,0)$ tensor (cf. Kim \cite{Kim}, also see Tan-Wang-Zhou \cite{TWZ}) and
    $g_0e^h$ is defined by $g_0e^h(X,Y)=g_0(X,e^{g'^{-1}_Jh}Y)$.
   Here $g'^{-1}_Jh$ is the lifted $(1,1)$ tensor of $h$ with respect to $g'_J$
 and $e^{g'^{-1}_Jh}$ is identity at point $p$.
   Hence, when $r$ is small enough, $\varphi+\log(1+|z|^2)^2$ is a strictly plurisubharmonic function in classical sense on $B_r(p)$.
   Without loss of generality, we may assume that $r=1$.

    To complete the proof of Theorem \ref{HBS2}, we need the following propositions:
    \begin{prop}\label{Hor443}
    (cf. {\rm H\"{o}rmander \cite[Theorem 4.4.3]{Hormander}})
    Let $\psi$ be a plurisubharmonic function in classical sense on $B_1(p)$ such that
    $$
   |\psi(z)-\psi(z')|<c,\,\,\, z,z'\in B_1(p)
    $$
  for some constant $c$.
  Let $V$ be a complex linear subspace of $\mathbb{C}^2$ of codimension $k$, $k=0,1,2$.
  For every holomorphic function $g$ on $V\cap B_1(p)$ such that
    $$
   \int_{V\cap B_1(p)}|g|^2e^{-\psi}d\lambda <\infty,
   $$
  where $d\lambda$ denotes the volume form of $V$,
     there exists a holomorphic function $f$ on $B_1(p)$ such that $f|_{V\cap B_1(p)}=g$ and
     \begin{equation}
   \int_{B_1(p)}|f|^2e^{-\psi}(1+|z|^2)^{-3k}d\mu_{g'_J}\leq 9^k\pi^ke^{kc}\int_{V\cap B_1(p)}|g|^2e^{-\psi}d\lambda.
     \end{equation}
   Note that $d\mu_{g'_J}=d\mu_{g_0}=\omega_p^2/2$ is the volume form on $B_1(p)$ since $J$ and $J_{st}$ are $\omega_p$-compatible;
   and on $B_1(p)$, for any $q\in B_1(p)$, $F(q)=L_p(q)\omega_p(q)$, where $L_p(q)$ is a positive function on $B_1(p)$, $L_p(p)=1$.
    \end{prop}
   By Proposition \ref{Hor443}, we have the following proposition:
   \begin{prop}\label{Hor444}
    (cf. {\rm H\"{o}rmander \cite[Theorem 4.4.4]{Hormander}})
    Let $\psi$ be a plurisubharmonic function in classical sense on $B_1(p)$.
    If $z^0\in B_1(p)$ and $e^{-\psi}$ is integrable in a neighborhood of $z^0$ one can find a holomorphic function $f$ in $B_1(p)$
    such that $f(z^0)=1$ and
    $$
    \int_{B_1(p)}|f(z)|^2e^{-\psi}(1+|z|^2)^{-6}d\mu_{g'_J}<\infty.
    $$
    \end{prop}

     Let $(\Sigma,j_\Sigma)$ be a compact Riemann surface.
   A smooth map $u:(\Sigma,j_\Sigma)\rightarrow (M,J)$ is called a $J$-holomorphic curve if
       the differential $du$ is a complex linear map with respect to $j_\Sigma$ and $J$:
   \begin{equation}
   J\circ du=du\circ j_\Sigma.
    \end{equation}
   Hence $$\bar{\partial}_Ju(X)=\frac{1}{2}[du(X)+J(u)du(j_\Sigma X)]=0$$ if $u$ is a $J$-holomorphic curve.
   Recall that the energy of a smooth map $u:\Sigma\longrightarrow (B_1(p),g'_J,J)$
     is defined as the $L^2$-norm of the $1$-form $du\in\Omega^1(\Sigma, u^*TM)$:
        $$
      E_J(u):=\frac{1}{2}\int_\Sigma |du|_J^2dvol_\Sigma.
        $$
   Here the norm of the (real) linear map $$L:= du(z):T_z\Sigma\rightarrow T_{u(z)}B_1(p)$$ is defined by
   \begin{equation}
   |L|_J:=\xi|^{-1}\sqrt{|L(\xi)|^2_J+|L(j_\Sigma\xi)|^2_J}
   \end{equation}
   for  $0\neq\xi\in T_z\Sigma$, where $|L(\xi)|^2_J=g'_J(\xi,\xi)$.
     By Lemma 2.2.1 in McDuff-Salamon \cite{MS},
   \begin{equation}
  E_J(u)=\int_\Sigma|\bar{\partial}_Ju|_J^2dvol_\Sigma+\int_\Sigma u^*\omega_p.
   \end{equation}
   Hence a $J$-holomorphic curve $u:\Sigma\longrightarrow (B_1(p),g'_J,J)$ is a minimal surface with respect to the metric $g'_J$.
  Note that a smooth map $u:\Sigma\longrightarrow (M,g,J)$ (an almost Hermitian manifold)
    is a $J$-holomorphic curve if and only if it is conformal with respect to $g$,
   i.e. its differential preserves angles or, equivalently, it preserves inner products up to a common positive factor.
   In our case, $g_J$ and $g'_J$ are in the same conformal class since $F|_{B_1(p)}$ and $\omega_p$ are in the same conformal class
   since for any $q\in B_1(p)$, $F(q)=L_p(q)\omega_p(q)$, where $L_p(q)$ is a positive function on $B_1(p)$, $L_p(p)=1$.
  Therefore, a $J$-holomorphic curve $u:\Sigma\longrightarrow (B_1(p),g'_J,J)$ is also a minimal surface with respect to the almost Hermitian metric $g_J$.

  \vskip 6pt

    We now return to the proof of Theorem \ref{HBS2}.
   The set of non integrability points of $e^{-\varphi}$ is the intersection of
  all hypersurfaces $f^{-1}(0)$ defined by holomorphic functions such that
    \begin{equation}
  \int_{B_1(p)}|f|^2(1+|z|^2)^{-6}e^{-\varphi}d\mu_{g'_J}<\infty.
    \end{equation}
  Indeed $f$ must vanish at any non integrability point, and on the other hand Proposition \ref{Hor444} shows that one can choose $f(z^0)=1$
   at any integrability point $z^0$.
   Suppose that $z^0\in f^{-1}(0)$, where $f$ is a holomorphic function on $B_1(p)$.
   Then there exists a holomorphic curve $u_f:\Sigma\longrightarrow (B_1(p),g_0,J_{st})$ passing through point $z^0$.
  Nijenhuis and Woolf (cf. \cite[Theorem III]{NW}) proved the following result:
  Let $J$ be an almost-complex structure on a manifold $X$ of real dimension $2n$,
  of class $C^{k,\lambda}$ ($k\geq 0$ is integer, $0<\lambda<1$).
  Then for every point $x$ of $X$ and every complex tangent vector $v$, there is a $J$-holomorphic curve of class $C^{1,\lambda}$
  passing through $x$ with tangent vector $v$ at $x$.
  Every such curve is actually of class $C^{k+1,\lambda}$.

  Hence, there exists a $J$-holomorphic curve
   $u'_f:\Sigma'\rightarrow B_1(p)$ passing through $z^0\in B_1(p)$
    which is contact $u_f:\Sigma\rightarrow B_1(p)$ at $z^0$, that is, $T_{z^0}u'_f(\Sigma')=T_{z^0}u_f(\Sigma)$.
    In fact, one can obtain a bijective corresponding between small enough $J$-holomorphic discs and usual holomorphic discs
    (see Diederich-Sukhov \cite[p.334]{DS} for details).

  Therefore, the set of non integrability points of $e^{-\varphi}$  is the intersection of all $J$-holomorphic curves
           $u'_f:\Sigma'\rightarrow (B_1(p),J)$ which are minimal surfaces with respect to the almost Hermitian metric $g_J$.
   Thus, the set of points in the neighborhood of which $e^{-\varphi}$  is not integrable is a $J$-analytic
   subset of $\Omega$ of dimension (complex)$\leq 1$.
   This completes the proof of Theorem \ref{HBS2}. ~~~~~~~~~~~~~~~~$\Box$

  \section{Siu's decomposition theorem on tamed almost complex $4$-manifolds}\label{Siu decomposition}
\renewcommand{\theequation}{B.\arabic{equation}}
\setcounter{equation}{0}
  As done in classical complex analysis, we define Lelong number for closed,
  positive almost complex $(1,1)$-currents (almost K\"{a}hler currents).
  We will discuss basic properties of almost K\"{a}hler currents and prove Siu's decomposition theorem on tamed almost complex $4$-manifolds.
  Our argument follows J.-P. Deamilly \cite{D3}.
 \medskip

   \subsection{Lelong numbers of closed positive $(1,1)$-currents on tamed almost complex $4$-manifolds}\label{Lelong}

   In this subsection, we will study closed, positive almost complex $(1,1)$-currents on almost complex $4$-manifolds.
   Note that any almost complex $4$-manifold $(M,J)$ has the local symplectic property \cite{L2}, that is, $\forall p\in M$,
  there are a neighborhood $U_p$ of $p$ and a closed $J$-compatible $2$-form
  $\omega_p$ on $U_p$ such that $d\omega_p=0$ and $\omega_p\wedge\omega_p>0$ on $U_p$.
  We may assume without loss of generality that $U_p$ is a star shaped strictly $J$-pseudoconvex open set,
   by Poincar\'{e} Lemma, there is a vector field $\xi_p$ on $U_p$
  such that $i_{\xi_p}\omega_p=\alpha_p$ and $\omega_p=d\alpha_p$.
  The fundamental theorem of Darboux \cite{A2,EG} shows that there are a neighborhood $U'_p\subset\subset U_p$ of $p$ and diffeomorphism $\Phi_p$
  from $U'_p$ onto $\Phi_p(U'_p)\subset\mathbb{C}^2\cong \mathbb{R}^4$ such that $\omega_p|_{U'_p}=\Phi_p^*\omega_0$, where $\Phi_p(p)=0\in\mathbb{C}^2$.
  Since the concepts we are going to study mostly concern the behaviour of currents or $J$-plurisubharmonic functions in a neighbordhood of a point
   on an almost complex $4$-manifold $(M,J)$,
   {\it we may assume that $(M,g_J,J,\omega)$ is an almost K\"{a}hler $4$-manifold, where $g_J(\cdot,\cdot)=\omega(\cdot,J\cdot)$}.
   Moreover, without loss of generality, we may assume that $M$ is an open subset of $\mathbb{C}^2$.
    Then the $J$-plurisubharmonic, standard plurisubharmonic and Hermitian plurisubharmonic on $M$ are equivalent.
   Let $\phi: M\rightarrow [-\infty,\infty)$ be a continuous $J$-plurisubharmonic function
   (our continuity assumption means that $e^{\phi}$ is continuous).
   We say that a $J$-plurisubharmonic function $\phi$ is semi-exhaustive if there exists a real number $c$ such that $B_{c,\phi}\subset\subset M$,
   where $$B_{c,\phi}:=\{x\in M\,|\, \phi(x)<c\}.$$
   Similarly, $\phi$ is said to be semi-exhaustive on a closed subset $A\subset M$ if there exists $c$ such that $A\cap B_{c,\phi}\subset\subset M$.
   We are interested especially in the set of poles $\{\phi=-\infty\}$.
   Let $T$ be a closed positive current of bidimension $(1,1)$ on $M$.
   Assume that $\phi$  is semi-exhaustive on $SuppT$ and that $B_{c,\phi}\cap SuppT\subset\subset M$.
   \begin{defi}\label{Leong number defi}
    (cf. {\rm Demailly \cite[Definition (5.4) in Chapter 3]{D3}})
   Let $(M,g_J,J,\omega)$ be an almost K\"ahler 4-manifold. If $\phi$ is semi-exhaustive on $SuppT$ and $B_{c,\phi}\cap SuppT\subset\subset M$, we set
   for $r\in(-\infty,c)$
   $$
   \nu(\phi,r,T)=\int_{B_{r,\phi}}T\wedge(dd^c_J\phi)
   $$
   and
   $$
   \nu(\phi,T)=\lim_{r\rightarrow-\infty} \nu(\phi,r,T).
   $$
   The number $\nu(\phi,T)$ will be called the generalized Lelong number of $T$ with respect to the weight $\phi$.
   \end{defi}

   As in cassical complex analysis (cf. \cite{D3,GH}),
   the above limit exists because $\nu(\phi,r,T)$ is a monotone increasing function of $r$.

  \begin{prop}\label{formula}
  (cf. {\rm Demailly \cite[Formula (5.5) in Chapter 3]{D3}})
  For any convex increasing function $\chi:\mathbb{R}\rightarrow\mathbb{R}$ we have
  $$
  \int_{B_{r,\phi}}T\wedge(dd^c_J\chi\circ\phi)=\chi'(r-0)\nu(\phi,r,T)
  $$
  where $\chi'(r-0)$ denotes the left derivative of $\chi$ at $r$.
  \end{prop}
  \begin{proof}
  For a detailed proof of the above Proposition, we refer to Formula $(5.5)$ in Chapter $3$ of \cite{D3}.
  \end{proof}

   We get in particular $$\int_{B_{r,\phi}}T\wedge(dd^c_Je^{2\phi})=2e^{2r}\nu(\phi,r,T),$$ whence the formula
  \begin{equation}\label{formula 2}
    \nu(\phi,r,T)=e^{-2r}\int_{B_{r,\phi}}T\wedge(\frac{1}{2}dd^c_Je^{2\phi}).
  \end{equation}
  Suppose $p\in SuppT$, then we define the Lelong number of $T$ with respect to the weight function $\varphi=\log \rho_g(p,q)$,
  $$\nu(\varphi,r,T)=\int_{B_{r,\varphi}}T\wedge(dd^c_J\varphi)$$
  and
  $$\nu(p,T)=\lim_{r\rightarrow -\infty}\nu(\varphi,r,T).$$
  The number $\nu(p,T)$ will be called the {\it Lelong number of $T$ at point $p$}.
  Then Formula (\ref{formula 2}) gives
  \begin{eqnarray*}
    \nu(\varphi,\log r,T) &=&  r^{-2}\int_{\rho_g(p,q)<r}T\wedge\frac{1}{2}dd^c_J\rho^2_g(p,q)\\
     &=&  r^{-2}\int_{\rho_g(p,q)<r}T\wedge\sqrt{-1}\partial_J\bar{\partial}_J\rho^2_g(p,q).
  \end{eqnarray*}
  The positive measure $\sigma_{T}=T\wedge\sqrt{-1}\partial_J\bar{\partial}_J\rho^2_g(p,q)$ is called the {\it trace measure} of $T$ (cf. Demailly \cite{D3}).
  We get
  \begin{equation}\label{formula 3}
   \nu(\varphi,\log r,T)=\frac{\sigma_{T}(B(p,r))}{r^2}
  \end{equation}
  and $\nu(p,T)$ is the limit of this ratio as $r\rightarrow 0$.
  The ratio $\frac{\sigma_{T}(B(p,r))}{r^2}$ is an increasing function of $r$.
  If $T$ is smooth at $p$, then $\sigma_{T}(B(p,r))$ is bounded near the point $p$ and $\sigma_{T}(B(p,r))=O(r^4)$.
  Hence, $$\nu(p,T)=\lim_{r\rightarrow 0}\frac{\sigma_{T}(B(p,r))}{r^2}=\lim_{r\rightarrow 0}O(r^2)=0.$$
 It is similar to the case of $J$ being integrable (cf. {\rm\cite{D3,GH,K,S2}})
  that $\nu(p,T)\geq 0$ and is identically equal to zero in case $T$ is a smooth current.
  Also, as in classical complex analysis (cf. \cite{D3,GH}), we have the following proposition
  \begin{prop}
  According to the above definition, we have
  \begin{equation}\label{Lelong number equ}
    \nu(p,T)=\lim_{r\rightarrow 0}\frac{2}{r^2}\int_{\rho_g(p,q)<r}T\wedge\omega.
  \end{equation}
  \end{prop}
   \begin{proof}
  We have the result of K. Diederich and A. Sukhov (cf. Lemma $2.1$ in \cite{DS}):
  Let $(M,J)$ be an almost complex manifold.
  Then for every point $p\in M$, every $\alpha\geq 0$ and $\lambda_0>0$ there exists a neighborhood $U$ of $p$ and
  a coordinate diffeomorphism $z:U\rightarrow \mathbb{B}$ such that $z(p)=0$,
  $dz(p)\circ J(p)\circ dz^{-1}(0)=J_{st}$ and the direct image $z_*(J)=dz\circ J\circ dz^{-1}$
  satisfies $\parallel z_*(J)-J_{st}\parallel_{C^{\alpha}(\bar{\mathbb{B}})}\leq \lambda_0$.

  Now, let $(M,g_J,J,\omega)$ be an almost K\"ahler 4-manifold.
  For any $p\in M$, there exists a Darboux coordinate $\{z_1,z_2\}$ on a small neighborhood $U_p$ of $p$
  such that
  $$
  \omega=\frac{\sqrt{-1}}{2}(dz_1\wedge d\bar{z}_1+dz_2\wedge d\bar{z}_2)=\frac{\sqrt{-1}}{2}\partial_{J_{st}}\bar{\partial}_{J_{st}}|z|^2
  =\frac{\sqrt{-1}}{2}\partial_{J_{st}}\bar{\partial}_{J_{st}}(z_1\bar{z}_1+z_2\bar{z}_2).
  $$
  Choose $\alpha=1$, $\lambda_0=1$.
  When $r$ is small, for $$\forall z\in B(0,r):=\{z\in U_p\,|\, \rho_{g_J}(0,z)<r\},$$
  we have
  $\parallel z_*(J)-J_{st}\parallel_{C^1}\leq 1$ and
  \begin{eqnarray*}
    (dd^c_J-dd^c)|z|^2 &=& d(J_{st}-J)d|z|^2 \\
     &=& d(J_{st}-J)(z_1\cdot d\bar{z}_1+dz_1\cdot \bar{z}_1+z_2\cdot d\bar{z}_2+dz_2\cdot \bar{z}_2).
  \end{eqnarray*}
  Hence $$|(dd^c_J-dd^c)|z|^2|\leq c|z|,$$ where $c$ is a positive constant.
 Then
 \begin{eqnarray*}
        \frac{1}{r^2}\int_{\rho_{g_J}(0,z)<r}T\wedge \sqrt{-1}\partial_{J_{st}}\bar{\partial}_{J_{st}}|z|^2&=& \frac{1}{r^2}\int_{\rho_{g_J}(0,z)<r}T\wedge \sqrt{-1}\partial_J\bar{\partial}_J|z|^2 \\
         && +O(r)\cdot\frac{1}{r^2}\int_{\rho_{g_J}(0,z)<r}T\wedge \sqrt{-1}\partial_J\bar{\partial}_J|z|^2.
      \end{eqnarray*}
      Therefore
  \begin{equation}\label{estimating equation1}
      \lim_{r\rightarrow 0}\frac{1}{r^2}\int_{\rho_{g_J}(0,z)<r}T\wedge \sqrt{-1}\partial_{J_{st}}\bar{\partial}_{J_{st}}|z|^2= \lim_{r\rightarrow 0}\frac{1}{r^2}\int_{\rho_{g_J}(0,z)<r}T\wedge \sqrt{-1}\partial_J\bar{\partial}_J|z|^2.
  \end{equation}

  On the other hand, let $(x^1,\cdot\cdot\cdot,x^4)$ be the normal coordinates of $g_J$ in a neighborhood $U$ of the point $p$.
 Then $g_{J,kl}$ have the following Taylor expansion (cf. Schone-Yau \cite{SY}):
   $$
   g_{J,kl}(x)=\delta_{kl}+\frac{1}{3}R_{kijl}x^ix^j+\frac{1}{6}R_{kijl,s}x^ix^jx^s+O(r^4),
   $$
   where all the curvatures and their covariant derivatives are evaluated at $p$.
   If $q\in U$,
   $$\rho_{g_J}(p,q)=\int^1_0|\gamma'(t)|_{g_J(\gamma(t))}dt,$$
    where $\gamma$ is the geodesic connecting points $p$ and $q$.
    Hence,
    \begin{eqnarray*}
      \rho_{g_J}(p,q) &=&\int^1_0 \sqrt{g_J(\gamma(t))(\gamma'(t),\gamma'(t))} dt  \\
       &=&\int^1_0 \sqrt{g_{J,kl}(tx)x^kx^l} dt  \\
       &=& \int^1_0 \sqrt{[\delta_{kl}+\frac{1}{3}R_{kijl}tx^itx^j+O(r^3)]x^kx^l} dt  \\
       &=& \int^1_0 \sqrt{|x|^2+\frac{t^2}{3}R_{kijl}x^ix^jx^kx^l+O(r^5)} dt \\
       &=&  \int^1_0|x|\sqrt{1+\frac{\frac{t^2}{3}R_{kijl}x^ix^jx^kx^l+O(r^5)}{|x|^2}} dt \\
       &=& \int^1_0[|x|+\frac{t^2R_{kijl}x^ix^jx^kx^l}{6|x|}+O(r^4)]dt\\
       &=& |x|+\frac{R_{kijl}x^ix^jx^kx^l}{18|x|}+O(r^4).
    \end{eqnarray*}
   Therefore,
    $$\rho^2_{g_J}(p,q)=|x|^2+\frac{1}{9}R_{kijl}x^ix^jx^kx^l+O(r^5),$$
    and
    $$
    \rho^2_{g_J}(p,q)-|x|^2=\frac{1}{9}R_{kijl}x^ix^jx^kx^l+O(r^5)=O(r^4).
    $$
    In fact, $\rho^2_{g_J}(p,q)$ is strictly $J$-plurisubharmonic near $p$ (cf. Ivashkovich-Rosay \cite[Lemma 1.3]{IR}).
   Then we can get
   \begin{eqnarray*}
     \frac{1}{r^2}\int_{\rho_{g_J}(0,z)<r}T\wedge\sqrt{-1}\partial_J\bar{\partial}_J\rho^2_{g_J}(p,q) &=&
       \frac{1}{r^2}\int_{\rho_{g_J}(0,z)<r}T\wedge \sqrt{-1}\partial_J\bar{\partial}_J|z|^2\\
         && +O(r^2)\cdot\frac{1}{r^2}\int_{\rho_{g_J}(0,z)<r}T\wedge  \sqrt{-1}\partial_J\bar{\partial}_J|z|^2,
      \end{eqnarray*}
      and
      \begin{equation}\label{estimating equation2}
      \lim_{r\rightarrow 0}\frac{1}{r^2}\int_{\rho_{g_J}(0,z)<r}T\wedge\sqrt{-1}\partial_J\bar{\partial}_J\rho^2_{g_J}(p,q)= \lim_{r\rightarrow 0}\frac{1}{r^2}\int_{\rho_{g_J}(0,z)<r}T\wedge \sqrt{-1}\partial_J\bar{\partial}_J|z|^2.
  \end{equation}
      At last, by (\ref{estimating equation1}) and (\ref{estimating equation2}),
      $$
       \lim_{r\rightarrow 0}\frac{1}{r^2}\int_{\rho_{g_J}(0,z)<r}T\wedge\sqrt{-1}\partial_J\bar{\partial}_J\rho^2_{g_J}(p,q)
       =  \lim_{r\rightarrow 0}\frac{1}{r^2}\int_{\rho_{g_J}(0,z)<r}T\wedge \sqrt{-1}\partial_{J_{st}}\bar{\partial}_{J_{st}}|z|^2.
      $$
      This completes the proof of the proposition.
    \end{proof}

   All these results are particularly interesting when $T_\Sigma$ is the current of integration over a $J$-holomorphic curve.
   Then $\sigma_{T}(B(p,r))$ is the Euclidean area of $\Sigma\cap B(p,r)$, while $\pi r^2$ is the area of a disc of radius $r$.
   Then it is immediate to check that
  $$
 \nu(p,T_\Sigma)=\left\{
  \begin{array}{ll}
   0 & \textrm{if}~ p\notin \Sigma,\\
   1 & \textrm{if}~ p \in \Sigma.
  \end{array}
 \right.
 $$

  \vskip 6pt

  In \cite{Elk2}, Elkhadhra has studied the Lelong number of a positive current
   $T$ of bidimension $(p,p)$ defined on an almost complex manifold.
   In particular, he has proven that the Lelong numbers of a positive current are independent on the coordinate systems
   (cf. Elkhadhra \cite[Theorem 3]{Elk2}). Thus, we have the following proposition:
 \begin{prop}\label{5A6}
  (cf. {\rm \cite{D3,Elk2,S2}})
 The Lelong number, $\nu(\phi,T)$, is independent of the choice of local coordinates.
 \end{prop}

   \vskip 6pt

  We are going to introduce the notions of $J$-pluripolar subset and $J$-analytic subset in an almost complex $2n$-manifold $(X,J)$.
  Such subsets should be considered as almost complex analogues of ``classical" complex case.
  In general, $J$-pluripolar subsets are the sets of $-\infty$ poles of $J$-plurisubharmonic functions.
   \begin{defi}\label{pluripolar}
   (cf. {\rm\cite{D3,Elk}})
      A subset $A$ of an almost complex $2n$-manifold $(X,J)$ is said to be $J$-pluripolar if for every point $x\in X$
   there exist a connected neighborhood $U$ of $x$ and $u\in PSH(X,J)$, $u\not\equiv-\infty$,
   such that $A\cap U\subset\{y\in U\mid  u(y)=-\infty\}$.

   A subset $A\subset X$ is said to be complete $J$-pluripolar in $X$ if for every point $x\in X$
   there exist a neighborhood $U$ of $x$ and $u\in PSH(X,J)\cap L^1_{loc}(U)$ such that $A\cap U\subset\{y\in U\mid  u(y)=-\infty\}$.
   $A$ is said to be regular complete $J$-pluripolar if there exists a $J$-plurisubharmonic function $u$ on $X$,
    of class $C^2$ on $X\setminus u^{-1}(-\infty)$ such that $A=u^{-1}(-\infty)$.
   \end{defi}

  \begin{rem}
   In the case when the structure $J$ is integrable, El Mir {\rm \cite{ElMir}} proved that every complete ($J$-)pluripolar subset is regular.
   \end{rem}

    Let $(X,J)$ be an almost complex manifold, $A$ a closed subset of $X$ and $T$ a current of order zero on $X\setminus A$.
    One says that $T$ admits a trivial extension $\tilde{T}$ on $X$ if $T$ has a locally finite mass in the neighborhood of every point of $A$,
    in which case $\tilde{T}$ can be defined by putting $\tilde{T}=0$ on $A$;
   the existence of some extension $T'$ is in any case equivalent to the local finiteness of the mass of $T$ near $A$.
    In \cite{Elk}, F. Elkhadhra presented a generalization of El Mir's theorem \cite{ElMir}
      on the extension of positive currents across a complete $J$-pluripolar subset,
   in the almost complex setting.
    For a detailed description of the almost complex version of El Mir's theorem, we refer to Theorem $1$ in \cite{Elk}.
    Here, we mainly want to apply its corollary, hence, we have the following proposition:
   \begin{prop}\label{extention}
  (cf. {\rm  Elkhadhra \cite[Corollary 1]{Elk}})
  Let $T$ is a closed positive current of bidimension $(1,1)$.
   If $A\subset X$ is a closed regular complete $J$-pluripolar set and $id_A$ is its characteristic function,
  then $id_AT$ is a closed positive current.
   \end{prop}

  It is well known that if $J$ is integrable,
  every ($J$-)analytic subset is a regular complete ($J$-)pluripolar set.
   But this is not yet established in the non-integrable case. As a generalization of classical complex analysis, we have the following definition:

  \begin{defi}\label{Elk defi}
  (cf. {\rm  Elkhadhra \cite{Elk2}})
  We say that $A$ is a $J$-analytic subset of an almost complex $2n$-manifold $(X,J)$ of dimension $p$
  if there exists a finite sequence of closed subsets
  $$
  \emptyset=A_{-1}\subset A_0\subset\cdot\cdot\cdot\subset A_p=A,
  $$
  where $A_j\setminus A_{j-1}$ is a smooth almost complex submanifold of $X\setminus A_{j-1}$,
  of complex dimension $j$ and has a locally finite $2j$-Hausdorff measure in the neighborhood of every point of $X$.
  We say that $A$ is of pure complex dimension $p$ if moreover we have $A_{j-1}\subset \overline{A_j\setminus A_{j-1}}$,
  for $j=0,1,2,\cdot\cdot\cdot, p$.
  If the $p$-dimensional strata $A_p\setminus A_{p-1}$ are connected we say that $A$ is irreducible.
  \end{defi}

   Notice that the definition for the almost complex setting does coincide with the usual analytic subsets in the integrable case.
   In order to justify the above definition let us recall that every closed $J$-holomorphic curve $A$ of $(X,J)$ is $J$-analytic.
   Indeed, we write $\emptyset=A_{-1}\subset A_0\subset A_1=A$,
   where $A_0$ is the singular part of $A$ which is discrete.
   More generally, every almost complex submanifold is a $J$-analytic subset. As in classical complex analysis, we have the following lemma:

   \begin{lem}\label{lemma2}
  (cf. {\rm  Demailly \cite[Lemma 8.15 in Chapter 3]{D3}})
  If $T$ is a closed positive current of bidimension $(1,1)$ on a almost K\"{a}hler $4$-manifold $(X,g_J,J,\omega)$
  and let $A$ be an irreducible $J$-analytic set,
  we set
  $$
  m_A:=\inf\{\nu(x,T)\mid x\in A\}.
  $$
  Then $\nu(x,T)=m_A$ for $x\in A\setminus \cup A_j$, where  $(A_j)$ is a countable family of proper $J$-analytic subsets of $A$.
  We say that $m_A$ is the generic Leong number of $T$ along $A$.
  \end{lem}
  \begin{proof}
   The upperlevel sets of the Lelong number is defined by
  $$
  E_c(T):=\{x\in X\mid  \nu(x,T)\geq c \}.
  $$
  By definition of $m_A$ and $E_c(T)$, we have $\nu(x,T)\geq m_A$ for every $x\in A$ and
  $$
  \nu(x,T)=m_A
  $$
  on $A\setminus \bigcup_{c\in\mathbb{Q}, c>m_A}A\cap E_c(T)$.
  However, for $c>m_A$, the intersection $A\cap E_c(T)$ is a proper $J$-analytic subset of $A$.
  \end{proof}

  According to Definition \ref{Elk defi}, this enables us to deduce without difficulty that
  every $J$-analytic subset $A$ is a locally regular complete $J$-pluripolar subset away from the singular part of $A$.
  Obviously, a natural question arises here: Is every $J$-analytic subset a (locally) regular complete $J$-pluripolar set?
   What would happen if closed positive currents are restricted to $J$-analytic subsets?
     Although this is a well-known result when $J$ is integrable.
  Our next result concerns the restriction of closed positive currents on $J$-analytic subsets.
   First, recall that in terms of currents, if $A$ is a $J$-analytic subset of complex dimension $p$
   then $T_A$ defines a closed positive $(p,p)$-current by integrating $(p,p)$ test forms on the components of $A$ of dimension $2p$.
   More precisely, assume that $$\emptyset=A_{-1}\subset A_0\subset\cdot\cdot\cdot\subset A_p=A$$ is a sequence as in Definition \ref{Elk defi}
  and let $Y=A_p\backslash A_{p-1}$.
  Since $Y$ is a smooth almost complex submanifold of $X\backslash A_{p-1}$,
   then the integration on $Y$ defines a positive closed current on $X\backslash A_{p-1}$.
  When $A$ is a $J$-analytic subset of complex dimension $p$, we obtain the following proposition.

   \begin{prop}\label{Elk lemma 1}
    (cf. {\rm  Elkhadhra \cite[Lemma 1]{Elk2}})
  Assume that $T$ is a positive closed current of bidimension $(p,p)$ on almost complex manifold $(X,J)$,
      and $A$ is a $J$-analytic subset of complex dimension $p$,
   then the cut-off $id_AT$ is also a positive and closed current supported by $A$.
  \end{prop}

  Notice also that by the same idea of Proposition \ref{Elk lemma 1},
   we can easily see that the current of integration $T_A$ on a $J$-analytic subset is positive and closed.

   \begin{prop}\label{Elk theorem 2}
  (cf. {\rm  Elkhadhra \cite[Theorem 2]{Elk2}})
    Let $T$ be a closed positive current of bidimension $(p,p)$ on an almost K\"{a}hler manifold $(X,J)$.
    Let $A$ be a $J$-analytic subset of $(X,J)$ of dimension $p$. Then, we have
  $$
    id_AT=m_AT_A,
  $$
   in particular $T-m_AT_A$ is positive.
   \end{prop}

   \begin{rem}
   Elkhadhra proved the above proposition on the almost complex manifold in {\rm \cite{Elk2}}.
   Since our Lelong number is defined on the almost K\"{a}hler manifold in this paper,
   we describe Elkhadhra's result on the almost K\"{a}hler manifold.
   \end{rem}

   \vskip 6pt

  The purpose of the remainder of this subsection is to give two other definitions of Lelong number on tamed closed almost complex $4$-manifolds.
   Suppose that $(M,J)$ is an almost complex $4$-manifold tamed by a symplectic $2$-form $\omega_1=F+d^-_J(v+\bar{v})$,
    where $v\in \Omega^{0,1}_J$ and $F$ is a fundamental $2$-form.
    Let $g_J(\cdot,\cdot)=F(\cdot,J\cdot)$ be an almost Hermitian metric and $d\mu_{g_J}$ the volume form.
   Suppose that $\rho_{g_J}(p,q)$ is the geodesic distance of points $p$, $q$ with respect to $g_J$ (cf. Chavel \cite{Cha}).
 Denote by
   $$
  B(p,r):=\{q\in M\mid \rho_{g_J}(p,q)\leq r\}.
  $$

  \begin{defi}\label{Lelong 1}
 If $p\in SuppT$, $T$ is a closed positive $(1,1)$-current on a closed almost complex $4$-manifold tamed by a symplectic form $\omega_1=F+d_{J}^{-}(v+\bar{v}),~v\in\Omega_{J}^{0,l}$,
 we define the Lelong number as follows
   $$
   \nu_1(p,\omega_1,r,T)=\frac{2}{r^2}\int_{B(p,r)}T\wedge\omega_1
   $$
  and
   $$
   \nu_1(p,T)=\lim_{r\rightarrow 0}\nu_1(p,\omega_1,r,T).
   $$
  \end{defi}
  Notice that as in the almost K\"{a}hler case, $\nu_1(p,\omega_1,r,T)$ is an increasing function of $r$.
  On the other hand, any almost complex $4$-manifold $(M,J)$ has the local symplectic property \cite{L2}, that is, $\forall p\in M$,
  there is a neighborhood $U_p$ of $p$ and a $J$-compatible symplectic form
  $\omega_p$ on $U_p$ such that $\omega_p|_p=F|_p$ and $F=f_p\omega_p$, $f_p\in C^\infty(U_p)$.
  Fix a point $q\in U_p$.
  Moreover, we assume that $r$ is small enough such that $B(q,r)\subset U_p$.
 It is similar to Definition \ref{Leong number defi}, in particular (\ref{Lelong number equ}),
  on symplectic $4$-manifold $(U_p,\omega_p)$,
  we can define Lelong number as follows,
  \begin{defi}\label{Lelong 2}
 If $p\in SuppT$, $T$ is a closed positive $(1,1)$-current on a closed almost complex $4$-manifold,
 we define
   $$
   \nu_2(q,\omega_p,r,T)=\frac{2}{r^2}\int_{B(q,r)}T\wedge\omega_p,
   $$
  and
   $$
   \nu_2(q,p,T)=\lim_{r\rightarrow 0}\nu_2(q,\omega_p,r,T).
   $$
  \end{defi}

   Note that $$\nu_1(q,\omega_1,r,T)=\frac{2}{r^2}\int_{B(q,r)}T\wedge\omega_1=\frac{2}{ r^2}\int_{B(q,r)}T\wedge F=\frac{2}{ r^2}\int_{B(q,r)}f_pT\wedge\omega_p,$$
    we will get the following comparison theorem:

   \begin{theo}\label{comparison theo}
    Let $T$ be a closed positive $(1,1)$-current on a closed almost complex $4$-manifold tamed by symplectic form $\omega_1$.
    If $p\in SuppT$,
   then $\nu_1(q,T)=f_p(q)\nu_2(q,p,T)$ for any $q$ which is very close to $p$. Moreover, there exists a constant $c>1$ depending on $\omega_1$ such that $c^{-1}\nu_2(q,p,T)\leq\nu_1(q,T)\leq c\nu_2(q,p,T)$, $\forall q\in SuppT\cap U_p\subseteq M$.
   \end{theo}
  \begin{proof}
    Since $f_p$ is smooth on $U_p$,
   $f_p$ can achieve the maximum and minimum values on $\overline{B(q,r)}$.
   Assume that $M_r$ and $m_r$ are the maximum and minimum values of $f_p$ on $\overline{B(q,r)}$, respectively.
    Thus,
  $$
          m_r\frac{2}{ r^2}\int_{B(q,r)}T\wedge\omega_p\leq\nu_1(q,\omega_1,r,T)=\frac{2}{ r^2}\int_{B(q,r)}f_pT\wedge\omega_p\leq M_r\frac{2}{ r^2}\int_{B(q,r)}T\wedge\omega_p.
  $$
 It is easy to see that $\lim_{r\rightarrow 0}M_r=\lim_{r\rightarrow 0}m_r=f_p(q)$.
 Taking the limit of both sides of the above inequality, for $q\in SuppT\cap U_p$, we can get
  $$
  f_p(q)\nu_2(q,p,T)\leq \nu_1(q,T)\leq f_p(q)\nu_2(q,p,T).
   $$
 Hence, we obtain $\nu_1(q,T)=f_p(q)\nu_2(q,p,T)$,
  in particular $\nu_1(p,T)=\nu_2(p,p,T)$,
 since $f_p(p)=1$. Note that $M$ is a closed almost complex 4-manifold which has local symplectic property, so we can find a finite open symplectic covering $\{ (U_{p_1},\omega_{p_1}),\cdots,(U_{p_k},\omega_{p_k}) \}$ of $M$.
   \end{proof}

 \begin{rem}
 (1) Let $T$ be a closed positive $(n-1,n-1)$-current on a closed almost complex $2n$-manifold tamed by a symplectic form $\omega$.
  If $ p\in Supp T$, we define
   $$
   \nu_1(p,\omega,r,T)=\frac{2}{r^2}\int_{B(p,r)}T\wedge\omega,
   $$
   and $\nu_1(p,T)=\lim_{r\rightarrow 0}\nu_1(p,\omega,r,T)$.

  (2) Let $T$ be a closed positive $(p,p)$-current on a closed almost K\"{a}hler $2n$-manifold
   $(M,g,J,\omega)$. If $q\in Supp T$, we define
    $$
   \nu(q,\omega,r,T)=\frac{2}{ r^{2n-2p}}\int_{B(q,r)}T\wedge\omega^{n-p}
   $$
  and
   $\displaystyle{
   \nu(q,T)=\lim_{r\rightarrow 0}\nu(q,\omega,r,T)}.
   $
  \end{rem}

   \subsection{Siu's decomposition formula of closed positive $(1,1)$-currents on tamed almost complex $4$-manifolds}\label{Siu}

   T. Rivi\`{e}re and G. Tian \cite{RT2} have obtained a very important result on the singular set
    of $(1,1)$ integral currents on almost complex manifolds with the local symplectic property. The regularity question for almost complex cycles is embedded into the problem of calibrated current and hence the theory of area-minimizing rectifiable 2-cycles. Their result appears to be a consequence of the ``Big Regularity Paper" of F. Almgren \cite{Alm} combined with the Ph.D thesis of his student S. Chang \cite{Ch}. This subsection is devoted to considering regularity of closed $(1,1)$-currents on tamed closed almost complex 4-manifolds.
   It is natural to generalize Siu's semicontinuity theorem \cite{S2} of closed positive $(1,1)$-currents on almost complex manifolds with local symplectic property.
  Note that any almost complex $4$-manifold $(M,J)$ has the local symplectic property \cite{L2} and
  the concepts we are gonging to study mostly concern the behaviour of currents or $J$-plurisubharmonic
     function in a neighbordhood of a point on an almost complex $4$-manifold $(M,J)$,
   we may assume that $(M,g,J,\omega)$ is an almost K\"{a}hler $4$-manifold throughout this section.
   Moreover, without loss of generality, we may assume that $M$ is an open subset of $\mathbb{C}^2$.
Suppose that $\nu_1(p,T)$ is the Lelong number defined on the closed almost Hermitian 4-manifold $(M,g_J,J,F)$ tamed by a symplectic form $\omega_1=F+d^{-}_J(v+\bar{v})$, where $v\in \Omega^{1,0}_J$. Since Lelong number is locally defined,
we first consider properties of Lelong number on an open almost K\"ahler 4-manifold.

  \begin{lem}\label{lemma}
  (cf. {\rm  Demailly \cite[The first and second steps of the proof of Theorem 8.4 in Chapter 3]{D3}})
  If $T$ is a closed positive current of bidimension $(1,1)$ on an open almost K\"{a}hler $4$-manifold $(M,g,J,\omega)$,
  the upperlevel sets
  $$
  E_c(T)=\{p\in M\mid  \nu(p,T)\geq c \}
  $$
   of the usual Lelong number are complete $J$-pluripolar subsets of $M$.
 \end{lem}
\begin{proof}
 Suppose $(M,g,J,\omega)$ is an open almost K\"ahler 4-manifold, where $M\subset\subset \mathbb{C}^2$.
 Let $\varphi(x,y)=\log \rho_g(x,y):M\times M\rightarrow[-\infty,+\infty)$ be a continuous $J$-plurisubharmonic function (see Claim \ref{claim 1}),
 where $\rho_g(x,y)$ is the geodesic distance of points $x$, $y$ with respect to $g$.
 Let $\chi\in C^\infty(\mathbb{R},\mathbb{R})$ be an increasing function such that $\chi(t)=t$ for $t\leq -1$
 and $\chi(t)=0$ for $t\geq0$.
 We consider the half-plane $H=\{z\in\mathbb{C}\mid  {\rm Re} \,z<-1\}$ and associate with $T$ the potential function $V$
 on $M\times H$ defined by
 $$
 V(y,z)=-\int^0_{{\rm Re}\,z}\nu(\varphi_y,t,T)\chi'(t)dt.
 $$
 For every $t>{\rm Re}\,z$, Stokes' formula gives
\begin{eqnarray*}
   \nu(\varphi_y,t,T)&=& \int_{\varphi(x,y)<t}T(x)\wedge dd^c_{J,x}\tilde{\varphi}(x,y,z)
\end{eqnarray*}
 with $$\tilde{\varphi}(x,y,z):= \max\{\varphi(x,y)\mid {\rm Re}\,z\}.$$
 By Fubini theorem, we obtain
 \begin{eqnarray*}
   V(y,z) &=& -\int_{x\in M,\varphi(x,y)<t,{\rm Re}\,z<t<0}T(x)\wedge(dd^c_{J,x}\tilde{\varphi}(x,y,z))\chi'(t)dt \\
    &=&  \int_{x\in M}T(x)\wedge\chi(\tilde{\varphi}(x,y,z))dd^c_{J,x}\tilde{\varphi}(x,y,z),
 \end{eqnarray*}
 where $dd^c_{J,x}\tilde{\varphi}(x,y,z)=dJ(x)d\tilde{\varphi}(x,y,z)$.
 For any smooth $(2,2)$-form $\alpha$ with compact support in $M\times H$, by Proposition \ref{integration by part},
 we get
 \begin{eqnarray*}
   <dd^c_JV,\alpha> &=&  <V,d^c_Jd\alpha> \\
   &=&  \int_{M\times M\times H}T(x)\wedge\chi(\tilde{\varphi}(x,y,z))dd^c_J\tilde{\varphi}(x,y,z)\wedge d^c_Jd\alpha(y,z)\\
    &=&  -\int_{M\times M\times H}T(x)\wedge\chi(\tilde{\varphi}(x,y,z))dd^c_J\tilde{\varphi}(x,y,z)\wedge dd^c_J\alpha(y,z)\\
    &=&  -\int_{M\times M\times H} dd^c_J[T(x)\wedge\chi(\tilde{\varphi}(x,y,z))\wedge dd^c_J\tilde{\varphi}(x,y,z)]\wedge\alpha(y,z)\\
     &=&  \int_{M\times M\times H} T(x)\wedge dd^c_J\chi(\tilde{\varphi}(x,y,z))\wedge dd^c_J\tilde{\varphi}(x,y,z)\wedge\alpha(y,z).
 \end{eqnarray*}
 Observe that the replacement of $dd^c_{J,x}$ by the total differentiation $dd^c_J$ does not modify the integrand,
 because the terms in $dx$, $d\bar{x}$ must have total bidegree.
 On $\{-1\leq\varphi(x,y)\leq 0\}$ we have $\tilde{\varphi}(x,y,z)=\varphi(x,y)$, whereas for $\varphi(x,y)<-1$ we get $\tilde{\varphi}<-1$
 and $\chi(\tilde{\varphi})=\tilde{\varphi}$.
 We see that $dd^c_JV(y,z)$ is the sum of $(1,1)$-form
 \begin{equation}\label{form 1}
   \int_{\{x\in M\,|\,-1\leq\varphi(x,y)\leq 0\}}T\wedge dd^c_J(\chi\circ\varphi)\wedge (dd^c_J\varphi),
 \end{equation}
 and
 \begin{equation}\label{form 2}
   \int_{\{x\in M\,|\,\varphi(x,y)<-1\}}T\wedge (dd^c_J\tilde{\varphi})^2.
 \end{equation}
 As $\varphi$ is smooth outside $\varphi^{-1}(-\infty)$, this form (\ref{form 1}) has locally bounded coefficients.
 Hence $dd^c_JV(y,z)\geq 0$ except perhaps for locally bounded  terms.
 In addition, $V$ is continuous on $M\times H$ because $T\wedge(dd^c_J\tilde{\varphi})^2$ is weakly continuous in the
 variables $(y,z)$ by Corollary 3.6 in \cite{D3}.
  Therefore, there exists a positive $J$-plurisubharmonic function $\rho\in C^{\infty}(M)$ such that
 $\rho(y)+V(y,z)$ is $J$-plurisubharmonic on $M\times H$.
 If we let ${\rm Re}z$ tend to $-\infty$, we see that the function
 $$
 U_0(y)=\rho(y)+V(y,-\infty)=\rho(y)-\int^0_{-\infty}\nu(\varphi_y,t,T)\chi'(t)dt
 $$
 is locally $J$-plurisubharmonic or identically $-\infty$ on $M$.
 Moreover, it is clear that $ U_0(y)=-\infty$ at every point $y$ such that $\nu(\varphi_y,T)>0$.
 If $M$ is connected and $U_0\not\equiv-\infty$, we already conclude that the density set $\cup_{c>0}E_c$ is pluripolar in $M$.

   \vskip 6pt

 Let $a\geq 0$ be arbitrary.
 The function
 $\rho(y)+V(y,z)-a{\rm Re}z$ is $J$-plurisubharmonic and independent of ${\rm Im}z$.
 By Kiselman's minimal principle \cite{Kise} which also holds on almost K\"{a}hler manifolds (see Theorem \ref{sym the} in Appendix \ref{Kiselman}),
 the partial Legendre transform $$U_a(y):=\inf_{r<-1}\{\rho(y)+V(y,r)-ar\}$$ is locally $J$-plurisubharmonic or $\equiv-\infty$ on $M$.
 Let $y_0\in M$ be a given point.
 We claim that:

  ${\bf(a)}$ If $a>\nu(\varphi_{y_0},T)$, then $U_a$ is bounded below on a neighborhood of $y_0$.

  ${\bf(b)}$ If $a<\nu(\varphi_{y_0},T)$, then $U_a(y_0)=-\infty$.

  By the definition of $V$ we have
  $$
  V(y,r)\leq-\nu(\varphi_y,r,T)\int^0_r\chi'(t)dt=r\nu(\varphi_y,r,T)\leq r\nu(\varphi_y,T).
  $$
 Then clearly $U_a(y_0)=-\infty$ if $a<\nu(\varphi_{y_0},T)$.
 On the other hand, if $a>\nu(\varphi_{y_0},T)$, there exists $t_0<0$ such that $\nu(\varphi_{y_0},t_0,T)<a$.
 Fix $r_0<t_0$.
 The semi-continuity property (Demailly \cite[Proposition 5.13]{D3}) shows that there exists a neighborhood $\varpi$ of $y_0$
 such that $\sup_{y\in\varpi}\nu(\varphi_y,r_0,T)<a$.
 For all $y\in\varpi$, we get
 $$
 V(y,r)\geq-C-a\int^0_r\chi'(t)dt=-C+a(r-r_0),
 $$
  and this implies $U_a(y)\geq-C-ar_0$.
  We complete the proof of the claim above.

   \vskip 6pt
   Now return to the proof of Lemma \ref{lemma}.
  Note that the family $\{U_a\}$ is increasing in $a$, that $U_a=-\infty$ on $E_c$ for all $a<c$
  and that $\sup_{a<c}U_a(y)>-\infty$ if $y\in M\setminus E_c$ (apply the above claim).
  For any integer $k\geq 1$, let $f_k\in C^\infty(M)$ be a $J$-plurisubharmonic regularization of $U_{c-\frac{1}{k}}$
  such that $f_k\geq U_{c-\frac{1}{k}}$ on $M$ and $f_k\leq-2^k$ on $E_c\cap M_k$ where
  $$M_k=\{y\in M\mid  d_{g_J}(y,\partial M)\geq\frac{1}{k}\}.$$
  Then the above claim shows that the family $(f_k)$ is uniformly bounded below on every compact subset of $M\setminus E_c$.
  We can also choose $(f_k)$ uniformly bounded above on every compact
 subset of $M$ because $U_{c-\frac{1}{k}}\leq U_c$.
 The function
 $$
 f=\sum^{+\infty}_{k=1}2^{-k}f_k
 $$
 is a continuous $J$-plurisubharmonic function $f:M\rightarrow [-\infty,+\infty)$ such that $$E_c=f^{-1}(-\infty).$$
 Hence $E_c$ is a complete $J$-pluripolar subset of $M$ and has zero Lebesgue measure.
 \end{proof}

   To prove the $J$-analyticity of $E_c$, we need the following estimation
 \begin{lem}\label{estimation lemma}
 (cf. {\rm Demailly \cite[The third step of the proof of Theorem 8.4 in Chapter 3]{D3}})
 Let $y_0\in M$ be a given point, $L$ a compact neighborhood of $y_0$, $K\subset M$ a compact subset and $r_0$ a real number$<-1$ such that
 $$
 \{(x,y)\in M\times L\mid  \varphi(x,y)\leq r_0\}\subset K\times L,
 $$
 where $$\varphi(x,y)=\log \rho_g(x,y):M\times M\rightarrow[-\infty,+\infty)$$ is a continuous $J$-plurisubharmonic function.
 Assume that $e^{\varphi(x,y)}$ is locally H\"{o}lder continuous in $y$ and that
 $$
 |e^{\varphi(x,y_1)}-e^{\varphi(x,y_2)}|\leq C\rho_g(y_1,y_2)^\gamma
 $$
 for all $(x,y_1,y_2)\in K\times L\times L$.
 Then for all $\varepsilon\in(0,1)$, there exists a real number $\eta(\varepsilon)>0$ such that all $y\in M$ with $\rho_g(y,y_0)<\eta(\varepsilon)$
 satisfy
 $$
 U_a(y)\leq\rho(y)+((1-\varepsilon)\nu(\varphi_{y_0},T)-a)(\gamma\log \rho_g(y,y_0)+\log\frac{2eC}{\varepsilon}).
 $$
 \end{lem}
  \begin{proof}
 For a detailed proof of this lemma, we refer to Demailly \cite[The third step of the proof of Theorem 8.4 in Chapter 3]{D3}.
 \end{proof}

By Lemma \ref{estimation lemma}, \ref{lemma}, as in classical complex analysis, we have the following theorem:

  \begin{theo}\label{analytic theo}
  (cf. {\rm Demailly \cite[Theorem 8.4 and Corollary 8.5 in Chapter 3]{D3}})
  If $T$ is a closed positive current of bidimension $(1,1)$ on an almost K\"{a}hler $4$-manifold $(M,g,J,\omega)$,
  the upperlevel sets
  $$
  E_c(T)=\{p\in M\mid  \nu(p,T)\geq c \}
  $$
   of the usual Lelong number are $J$-analytic subsets of dimension$\leq1$.
 \end{theo}
 \begin{proof}
 For $a,b>0$, we let $Z_{a,b}$ be the set of points in a neighborhood of which $e^{-U_a/b}$ is not integrable.
  Then $Z_{a,b}$ is $J$-analytic by Theorem \ref{HBS2} in Appendix \ref{singularity},
   and as the family $\{U_a\}$ is increasing in $a$, we have $Z_{a',b'}\supset Z_{a'',b''}$
  if $a'\leq a'', b'\leq b''$.

  Let $y_0\in M$ be a given point. If $y_0\notin E_c$, then $\nu(\varphi_{y_0},T)<c$ by definition of $E_c$.
  Choose $a$ such that $\nu(\varphi_{y_0},T)<a<c$.
  The claim {\bf(a)} in Lemma \ref{lemma} implies that $U_a$ is bounded below in a neighborhood of $y_0$,
  thus $e^{-U_a/b}$ is integrable and $y_0\notin Z_{a,b}$ for $b>0$.

  On the other hand, if $y_0\in E_c$ and if $a<c$, then Lemma \ref{estimation lemma} implies for all $\varepsilon>0$ that
  $$
  U_a(y)\leq(1-\varepsilon)(c-a)\gamma \log \rho_g(y,y_0)+C(\varepsilon)
  $$
   on a neighborhood of $y_0$.
   Hence $e^{-U_a/b}$ is non integrable at $y_0$ as soon as $b<(c-a)\gamma/4$.
   We obtain therefore
   $$
   E_c=\bigcap_{a<c,b<(c-a)\gamma/4}Z_{a,b}.
   $$
  This proves that $E_c$ is a $J$-analytic subset of $M$.
 \end{proof}

  \begin{rem}
  1) For an almost complex 4-manifold $(M,J)$, it has the local symplectic property {\rm \cite{Lejmi}}.
   For any $p\in M$, there exists a locally symplectic form $\omega_p$ on small neighborhood $U_p$.
    Hence on $U_p$ we can define Lelong number $\nu_2(q,p,T)$, see Definition \ref{Lelong 2} in Appendix \ref{Lelong}.
    Thus, we have Theorem \ref{analytic theo} in \ref{Siu} for $(U_p, g_p, J, \omega_p)$, $g_p(\cdot,\cdot)=\omega_p(\cdot,J\cdot)$.
  By Theorem {\rm\ref{comparison theo}} in Appendix \ref{Lelong},
  it is also true for Lelong number $\nu_1(p,T)$ (see Definition \ref{Lelong 1} in Appendix \ref{Lelong}) defined on tamed almost complex $4$-manifold, that is,
    the upper level sets
  $$
  E_c(T)=\{p\in M\mid  \nu_1(p,T)\geq c \}
  $$
   are $J$-analytic subsets of complex dimension$\leq1$ on a closed almost complex $4$-manifold $(M,J)$ which is tamed by a symplectic form $\omega_1$.

 2) It is natural to ask that for bidegree $(1,1)$ or bidegree $(n-1,n-1)$ closed positive currents on the higher dimensional almost K\"{a}hler manifolds,
          could one extend the above theorem?
  \end{rem}

As in classical complex analysis, we have Siu's decomposition formula of closed positive (1,1) currents on almost K\"ahler 4-manifolds.

 \begin{theo}\label{Theorem A}
 If $T$ is a closed positive almost complex $(1,1)$-current on an almost K\"{a}hler 4-manifold $(M,g,J,\omega)$, there is a unique decomposition of
 $T$ as a (possibly finite) weakly convergent series
 $$T=\Sigma_{j\geq 1}\lambda_jT_{\Sigma_j}+R,\,\,\, \lambda_j>0, $$
 where $T_{\Sigma_j}$ is the current of integration over an irreducible $1$-dimensional $J$-analytic set $\Sigma_j\subset M$ and where $R$ is a closed positive almost complex $(1,1)$-current with the property that $dim_{\mathbb{C}}E_c(R)<1$ for every $c>0$.
 \end{theo}
 \begin{proof}
 {\bf Uniqueness.} If $T$ has such a decomposition, the $1$-dimensional components of $E_c(T)$ are $(\Sigma_j)_{\lambda_j>c}$,
 for $$\nu(p,T)=\Sigma_{j\geq 1}\lambda_j\nu(p,T_{\Sigma_j})+\nu(p,R)$$ is non zero only on $\bigcup\Sigma_j\cup\bigcup E_c(R)$,
 and is equal to $\lambda_j$ generically on $\Sigma_j$
  (more precisely, $\nu(p,T)=\lambda_j$ at every regular point of $\Sigma_j$ which does not belong to any intersection $\Sigma_j\cap\Sigma_k$,
  $k\neq j$ or to $\bigcup E_c(R)$). In particular $\Sigma_j$ and $\lambda_j$ are unique.

  {\bf Existence.} By Theorem \ref{analytic theo}, $E_c(T)$ is a $J$-analytic subset of dimension$\leq 1$.
  For any $p\in M$, by Theorem \ref{HBS2}, there are $1$-dimensional components $(\Sigma_j)_{\lambda_j>c}$ of $E_c(T)$ passing through $p$.
   Let $(\Sigma_j)_{j\geq 1}$ be the countable collection of $1$-dimensional components occurring in one of the sets
  $E_c(T)$, $c\in\mathbb{Q}^*_+$, and let $\lambda_j>0$ be the generic Lelong number of $T$ along $\Sigma_j$.
  Then Proposition \ref{Elk theorem 2} shows by induction on $N$ that
   $$R_N=T-\sum_{1\leq j\leq N}\lambda_jT_{\Sigma_j}$$
   is positive.
   As $R_N$ is a decreasing sequence, there must be a limit $R=\lim_{N\rightarrow+\infty}R_N$ in the weak topology.
   Thus we have the asserted decomposition. By construction,
   $R$ has zero generic Lelong number along $\Sigma_j$, so $dim_{\mathbb{C}}E_c(R)<1$ for every $c>0$.
  \end{proof}

  \begin{rem}\label{7.6}
  Similarly, by Theorem {\rm\ref{comparison theo}}, it is also true for closed positive almost complex $(1,1)$-current
    $T$ on a closed almost complex $4$-manifold $(M,J)$ which is tamed by a symplectic form $\omega_1$.
  \end{rem}

\section{Demailly's approximation theorem on tamed almost complex $4$-manifolds}\label{Demailly approximation}
\renewcommand{\theequation}{C.\arabic{equation}}
\setcounter{equation}{0}
 Let $(M,J)$ be a closed almost complex $4$-manifold and let $T$ be a closed positive current of bidegree $(1,1)$ on $(M,J)$.
   In general $T$ can not be approximated by smooth closed positive currents.
   However, as done in classical complex analysis, we shall see that it is always possible to approximate a closed positive
   current $T$ of type $(1,1)$ by smooth closed real currents admitting a small negative part and that this negative part can be
   estimated in terms of the Lelong numbers of $T$ and the geometry (for complex analysis, see Demailly \cite{D1,D2}).

   In this appendix, we will give a Demailly's approximation theorem on tamed almost complex $4$-manifolds.
   Our approach is along the lines used by Demailly to give a proof of Theorem $1.1$ in \cite{D2}.

\medskip

\subsection{Exponential map associated to the second canonical connection}\label{Exponential}

  In this subsection, we study exponential map associated to the second canonical connection on almost Hermitian manifolds.
   Suppose $(M,g_J,J,F)$ is an almost Hermitian $2n$-manifold.
   Choose a complex coordinate $\{z_i=x_i+\sqrt{-1}y_i\}^n_{i=1}$ around $p\in M$ such that
   $\{\frac{\partial}{\partial z_i}|_p\}^n_{i=1}\subseteq T^{1,0}_pM$ is orthonormal at $p$ with respect to the almost Hermitian metric $h=g_J-\sqrt{-1}F$.
   Let $\{e_i\}^n_{i=1}$ be a unitary frame around $p$ such that $e_i(p)=\frac{\partial}{\partial z_i}|_p$.
   Let $\nabla^1$ be the second canonical connection satisfying $\nabla^1g_J=0$ and $\nabla^1J=0$, hence $
   \nabla^1 F=0$ and $\nabla^1 h=0$ (P. Gauduchon \cite{G2}). In particular, note that if $J$ is integrable, that is, $(M,J)$ is a complex manifold, $\nabla^1$ is Chern connection; if $(M,g_J,J,F)$ is a K\"ahler manifold, $\nabla^1$ is Levi-Civita connection  (P. Gauduchon \cite{Gauduchon}).
   Then locally there exists a matrix of valued $1$-forms $\{\theta^j_i\}$, called the connection $1$-forms, such that
   \begin{equation}\label{unitary frame}
     \nabla^1e_i=\theta^j_ie_j, \,\,\,\theta^j_i(p)=0.
   \end{equation}
    Let $\{\theta^1,\cdot\cdot\cdot,\theta^n\}$ be the dual coframe of $\{e_1,\cdot\cdot\cdot,e_n\}$.
   Then we have $\theta^i(p)=dz_i(p)$ by the choice of $\{z_i\}^n_{i=1}$.
   There holds the following Maurer-Cartan equations \cite{Cha,Gauduchon}:
    \begin{equation}\label{Maurer-Cartan}
   \left\{
  \begin{array}{ll}
  d\theta^i=-\theta^i_j\wedge\theta^j+\Theta^i,  \\

  d\theta^i_j=-\theta^i_k\wedge\theta^k_j+\Psi^i_j,
  \end{array}
  \right.
  \end{equation}
  where
  \begin{equation}\label{tausion}
    \Theta^i=(\Theta^i )^{(2,0)}+(\Theta^i )^{(0,2)}=T^i_{jk}\theta^j\wedge\theta^k+N^i_{\bar{j}\bar{k}}\bar{\theta}^j\wedge\bar{\theta}^k
  \end{equation}
  is the torsion form with vanishing $(1,1)$ part
  and $\Psi^i_j$ is the curvature form (see Tosatti-Weinkove-Yau \cite{TWY}).
   Take exterior derivative of (\ref{Maurer-Cartan}) to get
  \begin{eqnarray*}
    0 &=& -d\theta^i_j\wedge\theta^j+\theta^i_j\wedge d\theta^j+d\Theta^i\\
     &=& -d\theta^i_j\wedge\theta^j-\theta^i_j\wedge\theta^j_k\wedge\theta^k+d\Theta^i+\theta^i_j\wedge\Theta^j\\
     &=&-(d\theta^i_j+\theta^i_k\wedge\theta^k_j)\wedge\theta^j+d\Theta^i+\theta^i_j\wedge\Theta^j  \\
     &=& -\Psi^i_j\wedge\theta^j+d\Theta^i+\theta^i_j\wedge\Theta^j.
  \end{eqnarray*}
   Hence $d\Theta^i=\Psi^i_j\wedge\theta^j-\theta^i_j\wedge\Theta^j$.
   Define $R^j_{ik\bar{l}},~K^i_{jkl}$ and $K^i_{j\bar{k}\bar{l}}$ (see Tosatti-Weinkove-Yau \cite{TWY}) by
   \begin{eqnarray}\label{curvature equ}
   && (\Psi^j_i)^{(1,1)}=R^j_{ik\bar{l}}\theta^k\wedge\bar{\theta}^l,\nonumber\\
   && (\Psi^j_i)^{(2,0)}=K^i_{jkl}\theta^k\wedge\theta^l,\nonumber\\
   && (\Psi^j_i)^{(0,2)}=K^i_{j\bar{k}\bar{l}}\bar{\theta}^k\wedge\bar{\theta}^l,
   \end{eqnarray}
   with $K^i_{jkl}=-K^i_{jlk},~K^i_{j\bar{k}\bar{l}}=-K^i_{j\bar{l}\bar{k}},~K^i_{jkl}=\overline{K^j_{i\bar{l}\bar{k}}},~\delta^{s\bar{j}}\delta_{t\bar{i}}R^t_{sk\bar{l}}=\overline{R^j_{il\bar{k}}}$, where
   \begin{eqnarray}\label{D3}
   K^i_{j\bar{k}\bar{l}}=2T^i_{pj}N^{p}_{\bar{j}\bar{l}}+N^{i}_{\bar{k}\bar{l},j},~K^i_{jkl}=\overline{K^j_{i\bar{l}\bar{k}}},
   \end{eqnarray}
   and
   $\delta^{s\bar{j}}$ is the Kronecker delta and $\delta_{t\bar{i}}$ is its inverse.

    For a local complex frame $$\{ \frac{\partial }{\partial z_1},\frac{\partial }{\partial z_2},\cdots,\frac{\partial }{\partial z_n} \}\subseteq T^{1,0}M,\,\,\,
   \{ \frac{\partial }{\partial \bar{z}_1},\frac{\partial }{\partial \bar{z}_2},\cdots,\frac{\partial }{\partial \bar{z}_n} \}\subseteq T^{0,1}M.$$
   Denote by $\frac{\partial}{\partial z_{\bar{i}}}=\frac{\partial }{\partial \bar{z}_i}$,
       and define $\Gamma_{AB}^C$ as
   \begin{equation}
   \nabla^1_{\frac{\partial }{\partial z_A}}\frac{\partial }{\partial z_B}
           =\Gamma_{AB}^C\frac{\partial}{\partial z_C}, A,B,C\in\{1,2,\cdots,n,\bar{1},\bar{2},\cdots,\bar{n}\}.
   \end{equation}
   Hence, $\overline{\Gamma_{AB}^C}=\Gamma_{\bar{A}\bar{B}}^{\bar{C}}$, $\Gamma_{AB}^C=\Gamma_{BA}^C$.
   Let $h:= g_J-\sqrt{-1}F=\sum_i\theta^i\otimes\bar{\theta}^i$,
   then $h_{ij}=h(\partial/\partial z_i,\partial/\partial z_j)$.

  \begin{lem}\label{Christoffol symbols}
   The only non-vanishing Christoffel symbols are $\Gamma_{ij}^{k}, \Gamma_{\bar{i}\bar{j}}^{\bar{k}}$,
   where $$\displaystyle{\Gamma_{ij}^{k}=\sum_{l=1}^nh^{k\bar{l}}\frac{\partial h_{j\bar{l}}}{\partial z_i}}.$$
  \end{lem}
  \begin{proof}
   There hold
   $$\displaystyle{\nabla^1_{\frac{\partial}{\partial z_{i}}} \frac{\partial}{\partial z_{j}}
   =\sum_{k}\Gamma_{ij}^k\frac{\partial}{\partial z_{k}}+\sum_{k}\Gamma_{ij}^{\bar{k}}\frac{\partial}{\partial \bar{z}_{k}}},$$
   and
  $$\displaystyle{\nabla^1_{\frac{\partial}{\partial \bar{z}_{i}}} \frac{\partial}{\partial z_{j}}
    =\sum_{k}\Gamma_{i\bar{j}}^k\frac{\partial}{\partial z_{k}}+\sum_{k}\Gamma_{i\bar{j}}^{\bar{k}}\frac{\partial}{\partial \bar{z}_{k}}}.$$
   Since $\nabla^1 J=0$, and $J$ acts on $T^{1,0} M$ being by multiplying $\sqrt{-1}$ and acts on $T^{0,1} M$ by $-\sqrt{-1}$,
      we have $$\sqrt{-1}\nabla^1_{\frac{\partial}{\partial z_{i}}} \frac{\partial}{\partial z_{j}}=\nabla^1_{\frac{\partial}{\partial z_{i}}}
   (J\frac{\partial}{\partial z_{j}})=J(\nabla^1_{\frac{\partial}{\partial z_{i}}} \frac{\partial}{\partial z_{j}}).$$
    Then $$\sqrt{-1}(\sum_k\Gamma_{ij}^k \frac{\partial}{\partial z_k}+\sqrt{-1}\sum_{k}\Gamma_{ij}^{\bar{k}} \frac{\partial}{\partial
   \bar{z}_k})=\sqrt{-1}(\sum_k\Gamma_{ij}^k \frac{\partial}{\partial z_k}-\sqrt{-1}\sum_{k}\Gamma_{ij}^{\bar{k}} \frac{\partial}{\partial \bar{z}_k}),$$
   which implies that $\Gamma_{ij}^{\bar{k}}=0$. Similarly, $\Gamma_{i\bar{j}}^{\bar{k}}$, $\Gamma_{i\bar{j}}^{k}$ vanish.
    Nonzero ones are only $\Gamma_{ij}^{k}$, $\Gamma_{\bar{i}\bar{j}}^{\bar{k}}$.
  Moreover, $$\displaystyle{\frac{\partial}{\partial z_i}h(\frac{\partial}{\partial z_j},
        \frac{\partial}{\partial \bar{z}_j})=h(\sum_l\Gamma_{ij}^l \frac{\partial}{\partial z_l},
       \frac{\partial}{\partial \bar{z}_k})=\sum_l\Gamma_{ij}^l h_{l\bar{k}}}.$$
     Hence, $\displaystyle{\Gamma_{ij}^{k}=\sum_{l=1}^nh^{k\bar{l}}\frac{\partial h_{j\bar{l}}}{\partial z_i}}$.
   \end{proof}
  By (\ref{unitary frame}) and Lemma \ref{Christoffol symbols},
    we have
    \begin{eqnarray}\label{Taylor expansion1}
      e_i &=& e_i(p)+\frac{1}{2}(\nabla^1)^2e_i+O(|z|^3) \nonumber\\
       &=& \frac{\partial}{\partial z_i}|_p+\sum_{j,l,m}(b''_{jilm}z_lz_m+\bar{b}''_{jilm}\bar{z}_l\bar{z}_m
       +c''_{jilm}z_l\bar{z}_m)\frac{\partial}{\partial z_j}+O(|z|^3) .
    \end{eqnarray}
  Without loss of generality, we may assume that $b''_{jilm}=b''_{jiml}$,
   otherwise, if $b''_{jilm}=-b''_{jiml}$ then $\sum_{l,m}b''_{jilm}z_lz_m=0$.
  Also, by (\ref{curvature equ}), the skew symmetric part of $(\nabla^1)^2e_i$ is $(\Psi^i_j)^{(1,1)}=R^i_{jl\bar{m}}\theta^l\wedge\bar{\theta}^m$.
  Hence
  \begin{equation}
    c''_{jilm}=\frac{1}{2}R^i_{jl\bar{m}}.
  \end{equation}
  By (\ref{tausion}), the skew symmetric part is $\Theta^i=T^i_{jk}\theta^j\wedge\theta^k+N^i_{\bar{j}\bar{k}}\bar{\theta}^j\wedge\bar{\theta}^k$.
  Hence,
  \begin{eqnarray}\label{Taylor expansion2}
    \theta^i &=&\theta^i(p)+\nabla^1\theta^i+O(|z|^2)  \nonumber\\
     &=&\sum_j\delta_{ij}dz^j+\sum_{j,l}(a'_{jil}z_ldz_j+\bar{a}''_{jil}\bar{z}_ld\bar{z}_j)+O(|z|^2).
  \end{eqnarray}
  By (\ref{Taylor expansion1}) and (\ref{Taylor expansion2}),
  we can expand $h_{ij}(z)=h(\frac{\partial}{\partial z_i},\frac{\partial}{\partial z_j})$ as follows:
  \begin{eqnarray*}
    h_{ij}(z) &=& \delta_{ij}+\sum_l(a_{jil}z_l+\bar{a}_{jil}\bar{z}_l)+\sum_{l,m}(b'_{jilm}z_lz_m+\bar{b}'_{jilm}\bar{z}_l\bar{z}_m) \\
     &&+  \sum_{l,m}c'_{jilm}z_l\bar{z}_m+O(|z|^3),
  \end{eqnarray*}
   where $a_{jil}=a'_{jil}+a''_{jil}$, $b'_{jilm}=b'_{jiml}$.
   We may always arrange that skew symmetry relation $a_{jil}=-a_{lij}$ holds;
   otherwise the change of variables
   $z_i=z'_i-\frac{1}{4}\sum_{j,l}(a_{jil}+a_{lij})z'_jz'_l$
   yields coordinates $(z'_l)$ with this property.
   By the definition of $a_{jil}$ and
   $$
   \nabla^1\theta^i|_p=d\theta^i|_p=(-\theta^i_j\wedge\theta^j+\Theta^i)|_p
   =T^i_{jl}\theta^j\wedge\theta^l+N^i_{\bar{j}\bar{l}}\bar{\theta}^j\wedge\bar{\theta}^l,
   $$
    it is easy to see that $a'_{jil}=T^i_{jl}$, $\bar{a}''_{jil}=\overline{N^i_{\bar{j}\bar{l}}}$.
  If $h$ is K\"{a}hler, then $a_{jil}=0$; in that case $b'_{jilm}$ is also symmetric in $j,l,m$ and a new change of variables
  $z_i=z'_i-\frac{1}{3}\sum_{j,l,m}b'_{jilm}z'_jz'_lz'_m$ gives $b'_{jilm}=0$ likewise.

  The complex frame of $T^{1,0}_pM$ defined by
  $$\tilde{e}_s=\partial/\partial z_s-\sum_j(a_{jsk}z_j+\sum_m b'_{jskm}z_jz_m)\partial/\partial z_k$$
  satisfies
 \begin{equation}
    <\tilde{e}_s,\tilde{e}_t>_h=\delta_{st}-\sum_{j,k}c_{tsjk}z_j\bar{z}_k+O(|z|^3),
 \end{equation}
 \begin{equation}\label{Taylor expansion4}
   \partial/\partial z_s=\tilde{e}_s+\sum_l(\sum_ja_{jsl}z_j+\sum_{j,k}b_{jslk}z_jz_k+O(z^3))\tilde{e}_l
 \end{equation}
 with $c_{tsjk}=-c'_{tsjk}-\sum a_{jsl}\bar{a}_{ktl}$ and
 $b_{jskl}=b'_{jskl}+\sum a_{lsm}a_{jmk}$.
 Hence, in the K\"{a}hler case, $a_{jsl}=0$ and $b_{jslk}=0$.
 The formula $\partial_{\frac{\partial}{\partial z_j}}<\tilde{e}_s,\tilde{e}_t>_h=<\nabla^1_{\frac{\partial}{\partial z_j}}\tilde{e}_s,\tilde{e}_t>_h$
 with respect to $J(p)$ easily gives the following
 $$
 \nabla^1\tilde{e}_s=-\sum_{t,j,k} c_{tsjk}\bar{z}_kdz_j\otimes\tilde{e}_t+O(|z|^2),
 $$
 \begin{equation}
 (\tilde{\Psi})^{(1,1)}|_p=\sum_{s,t,j,k}c_{tsjk}dz_j\wedge d\bar{z}_k\otimes\tilde{\theta}^s\otimes\tilde{e}_t,
 \end{equation}
 where $\tilde{\theta}^s$ is the dual frame of $\tilde{e}_s$.
 Hence $c_{tsjk}=R^s_{tjk}$.
 \begin{rem}
 If $M$ is a complex manifold, then $N^i_{\bar{j}\bar{k}}=0$.
 By (\ref{D5}), $K^i_{j\bar{k}\bar{l}}=0$, thus $(\tilde{\Psi}^i_j)^{(1,1)}=\tilde{\Psi}^i_j$.
 \end{rem}
 Given a vector field $\zeta=\sum_l\zeta_l\partial/\partial z_l$ in $T^{1,0}M$,
  we denote by $(\xi_m)$ the components of $\zeta$ with respect to the basis $(\tilde{e}_m)$,
  thus $\zeta=\sum_m\xi_m\tilde{e}_m$ in $T^{1,0}M$.
  By (\ref{Taylor expansion4}), we have
  \begin{equation}\label{Taylor expansion6}
    \xi_m=\zeta_m+\sum_{j,l}a_{jml}z_j\zeta_l+\sum_{j,k,l}b_{jmlk}z_jz_k\zeta_l.
  \end{equation}
 By a direct calculation, we have
 \begin{eqnarray*}
    \nabla^1(\partial/\partial z_l)&=& -\sum_{j,k,m} c_{mljk}\bar{z}_kdz_j\otimes \tilde{e}_m+\sum_{j,m}a_{mlj}dz_j\otimes\tilde{e}_m\\
    &&+2\sum_{j,k,m}b_{mljk}z_kdz_j\otimes\tilde{e}_m+O(|z|^2)dz\\
    &=&-\sum_{j,k,m}(c_{mljk}\bar{z}_k-2b_{mljk}z_k)dz_j\otimes\frac{\partial}{\partial z_m}      \\
    &&+\sum_{j,m}(a_{mlj}-\sum_{k,i}a_{ilj}a_{imk}z_k)dz_j\otimes\frac{\partial}{\partial z_m}+O(|z|^2)dz.
 \end{eqnarray*}
       Hence, as in classical complex analysis (cf. (2.5) in Demailly \cite{D2}), we have
       \begin{eqnarray}\label{Taylor expansion7}
          \nabla^1\zeta&=&\sum_md\zeta_m\otimes \frac{\partial}{\partial z_m}
          -\sum_{j,k,l,m}(c_{lmjk}\bar{z}_k-2b_{lmjk}z_k)\zeta_mdz_j\otimes\frac{\partial}{\partial z_l} \nonumber\\
          &&+\sum_{j,l,m} (a_{lmj}-\sum_{k,i}a_{lij}a_{imk}z_k)\zeta_mdz_j\otimes\frac{\partial}{\partial z_l}
       + O(|z|^2)\zeta dz.
       \end{eqnarray}
       Consider a curve $t\rightarrow u(t)$.
       By a substitution of variables $z_j=u_j(t)$, $\zeta_l=\frac{du_l}{dt}$
       in formula (\ref{Taylor expansion7}), the equation $\nabla^1(\frac{du}{dt})=0$ becomes
       \begin{equation}\label{Taylor expansion8}
                 \frac{d^2u_s}{dt^2}=\sum_{j,k,l}(c_{lsjk}\bar{u}_k(t)-2b_{lsjk}u_k(t))\frac{du_j}{dt}\frac{du_l}{dt}
                 + O(|u(t)|^2)(\frac{du}{dt})^2.
        \end{equation}
        Notice that the contribution of the terms $\sum a_{j\bullet l}\zeta_ldz_j$ is zero by the skew symmetry relation.
        The initial condition $u(0)=z$, $u'(0)=\zeta$ gives $u_s(t)=z_s+t\zeta_s+O(t^2|\zeta|^2)$.
         Hence,
         \begin{eqnarray*}
           u_s(t) &=&z_s+t\zeta_s+\sum_{i,j,k}c_{isjk}(\frac{t^2}{2}\bar{z}_k+\frac{t^3}{6}\bar{\zeta}_k)\zeta_i\zeta_j  \\
            &&-2b_{isjk}(\frac{t^2}{2}z_k+\frac{t^3}{6}\zeta_k)\zeta_i\zeta_j+O(t^2|\zeta|^2(|z|+|\zeta|)^2).
         \end{eqnarray*}
  An iteration of this procedure (substitution in (\ref{Taylor expansion8}) followed by an integration) easily shows that
  all terms but the first two in the Taylor expansion of $u_s(t)$ contain $\mathbb{C}$-quadratic factors of the form $\zeta_j\zeta_l$.
   Let us substitute $\zeta_j$ by its expression in terms of $z$, $\xi$ deduced from (\ref{Taylor expansion6}).
   We find that $\exp_z(\zeta)=u(1)$ has a third order expansion
  \begin{eqnarray}\label{D14}
   \exp_z(\zeta)_s&=& K_{p,s}(z,\xi)  \nonumber\\
   &&+\sum_{j,k,l} c_{lsjk}(\frac{1}{2}\bar{z}_k+\frac{1}{6}\bar{\xi}_k)\xi_j\xi_l+O(|\xi|^2(|z|+|\xi|)^2),
   \end{eqnarray}
   where
   \begin{eqnarray}\label{Taylor expansion9}
      K_{p,s}(z,\xi)&=&z_s+\xi_s-\sum_{j,l}a_{jsl}z_j\xi_l+\sum_{i,j,k,l}a_{jil}a_{ksi}z_jz_k\xi_l  \nonumber \\
      &&-\sum_{j,k,l}b_{lsjk}(z_jz_k\xi_l+z_k\xi_j\xi_l+\frac{1}{3}\xi_j\xi_k\xi_l)
   \end{eqnarray}
   is a holomorphic polynomial of degree $3$ in $z$, $\xi$ with respect to complex structure $J(p)$.
   In the K\"{a}hler case we simply have $\xi_l=\zeta_l$ and $K_{p,s}(z,\xi)=z_s+\xi_s$.

   \begin{rem}
  \begin{enumerate}[1]
    \item When $M$ is a complex manifold,
    $$
    N^s_{\bar{i}\bar{j}}=0,~ a_{isj}=T^s_{ij},~ c_{lsij}=(\Psi^s_l)^{(1,1)}=R^s_{li\bar{j}}.
    $$
   \item When $M$ is a quasi-K\"ahler (or almost K\"ahler) manifold,
       $$
       T^s_{ij}=0,~ a_{isj}=\overline{N^s_{\bar{i}\bar{j}}},~c_{lsij}=(\Psi^s_l)^{(1,1)}=R^s_{li\bar{j}}.
       $$
   \item When $M$ is a K\"ahler manifold,
   $$
   a_{isj}=0,~ b_{lsij}=0,~ c_{lsij}=(\Psi^s_l)^{(1,1)}=R^s_{li\bar{j}}.
   $$
  \end{enumerate}
    \end{rem}

   The exponential map is unfortunately non-holomorphic for $z$ fixed with respect to $J(p)\cong J_{st}$.
   However, as done in classical complex analysis, we make it quasi-holomorphic with respect to $\zeta\in T^{1,0}_z M$ as follows: for $z$, $J(p)$ fixed, we consider the formal power series obtained by eliminating all monomials in the Taylor expansion of $\zeta\mapsto \exp_z(\zeta)$ at the origin which are not holomorphic with respect to $\zeta$.
    This defines in a unique way a jet of infinite order along the zero section of $T^{1,0}_zM$.
   There is a smooth map
  \begin{eqnarray*}
  T^{1,0}_zM\rightarrow M,~(z,\zeta)\mapsto {\rm exph}_{z}(\zeta),
   \end{eqnarray*}
   such that its jet at $\zeta=0$ coincides with the ``$J(p)(\cong J_{st})$-holomorphic" part of $\zeta\mapsto \exp_z(\zeta)$.
   Moreover, (\ref{D14}) and (\ref{Taylor expansion9}) imply that
   \begin{eqnarray}\label{D15}
    {\rm exph}_z(\zeta)_s= K_{p,s}(z,\xi)+\frac{1}{2}\sum_{i,j,k}c_{jsik} \bar{z}_k\xi_i\xi_j+O(|\xi|^2(|z|+|\xi|)^2).
    \end{eqnarray}
  By including in $K_{p,s}$ all holomorphic monomials of partial degree at most $2$ in $z$ and $N$ in $\xi$ ($N\geq 2$ being a given integer),
   we get holomorphic polynomials $L_{p,s}(z,\xi)$ of linear part $z_s+\xi_s$ and total degree $N+2$, such that
   \begin{equation}\label{Taylor expansion10}
      {\rm exph}_z(\zeta)_s= K_{p,s}(z,\xi)+O(\bar{z},z\bar{z},\bar{z}\bar{z},|z|^3,\xi^{N-1})\xi^2.
   \end{equation}
  Here a notation as $O(\bar{z},z\bar{z},\bar{z}\bar{z},|z|^3,\xi^{N-1})\xi^2$ indicates an arbitrary function in
  the ideal of $C^\infty$ functions generated by monomials of the form $\bar{z}_k\xi_l\xi_m$, $z_i\bar{z}_j\xi_l\xi_m$,
  $\bar{z}_i\bar{z}_j\xi_l\xi_m$, $z^{\alpha}\bar{z}^{\beta}\xi_l\xi_m$ and $\xi^{\gamma}$,
  for all multi-indices $|\alpha|+|\beta|=3$ and $|\gamma|=N+1$.
  By the implicit function theorem applied to the mapping $L_p=(L_{p,m})_{1\leq m\leq n}$ we thus get (cf. Proposition $2.9$ in Demailly \cite{D2})
  \begin{prop}\label{D prop}
  Suppose $(M,g_J,J,F)$ is an almost Hermitian manifold.
  Let $h=g_J-\sqrt{-1}F$ be an almost Hermitian metric on $T^{1,0}M$.
  There exists a $C^{\infty}$ map
   $$T^{1,0}_pM\rightarrow M,~(p,\zeta)\mapsto {\rm exph}_p(\zeta)$$
   with the following properties:\\
   (1). For every $p\in M$, ${\rm exph}_p(0)=p$ and $d_{\zeta}{\rm exph}_p(0)=Id_{T^{1,0}_pM}$.\\
   (2). For every $p\in M$, the map $\zeta\rightarrow {\rm exph}_p(\zeta)$ has a quasi-holomorphic Taylor expansion at $\zeta=0$ with respect to fixed almost complex structure $J(p)$ on small neighborhood. Moreover, with respect to an almost Hermitian structure $(g_J,J,F)$, there are local normal complex coordinates $(z_1,z_2,\cdots,z_n)$ on $M$ centered at $p$, $z_i(p)=0,~i=1,2,\cdots,n$, and holomorphic normal complex coordinates $(\zeta_j)$ on the fibers of $T^{1,0}M$ near $p$ with respect to the fixed complex structure $J(p)$ such that
   $${\rm exph}_z(\xi)=L_p(z,\rho_p(z,\xi)),$$
   where $L_p(z,\xi)$ is a holomorphic polynomial map of degree $2$ in $z$ and of degree $N$ in $\xi$,
   and where $\rho_p:\mathbb{C}^n\times\mathbb{C}^n\rightarrow\mathbb{C}^n$ is a smooth map such that
   \begin{eqnarray}\label{pro exp1}
   L_{p,m}(z,\xi)&=&z_m+\xi_m-\sum_{j,l} a_{jml}z_j\xi_l+\sum_{i,j,k,l}a_{lmi}a_{jik}z_jz_k\xi_l   \nonumber\\
   &&-\sum_{j,k,l} b_{lmjk}(z_jz_k\xi_l+z_k\xi_j\xi_l+\frac{1}{3}\xi_j\xi_k\xi_l)+O((|z|+|\xi|)^4),
   \end{eqnarray}
   \begin{eqnarray}\label{pro exp2}
   \rho_{p,m}(z,\xi)&=&\xi_m+\sum_{2\leq|\alpha|\leq N}(\sum_k d_{\alpha mk}\xi^{\alpha}\bar{z}_k+\sum_{i,k} e_{\alpha mik}\xi^{\alpha}z_i\bar{z}_k) \nonumber\\
         &&+O(\bar{z}^2,|z|^3,\xi^{N-1})\xi^2.
     \end{eqnarray}
        (3). For $\alpha=(0,\cdots,1_l,\cdots,1_j,\cdots,0)$ of degree 2,
         we have $$d_{\alpha mk}=\frac{1}{2}c_{lmjk},\,\,e_{\alpha mik}=\frac{1}{2}\sum_sa_{lms}c_{jsik}z_s,$$
          where $c_{lmjk}$ is the curvature tensor $R^m_{lj\bar{k}}$, $a_{lmj}=T^m_{lj}+\overline{N^m_{\bar{l}\bar{j}}}$.
       \end{prop}
          \begin{proof}
        The argument is similar to the proof of Proposition $2.9$ in Demailly \cite{D2}.
             \end{proof}

\begin{rem}\label{D5}
Suppose that $(M,g_J,J,F)$ is an almost Hermitian 4-manifold.
For any $p\in M$, there exists a $J$-compatible local symplectic form $\omega_p$ on a small neighborhood $U_p$ such that $F=f_p\omega_p$,
 where $f_p>0$ on $U_p$ and $f_p(p)=1$ (cf. {\rm Lejmi \cite{Lejmi}}). On $U_p$, by Darboux's theorem (cf. {\rm McDuff-Salamon \cite{MS}}), there is a coordinate chart $(V_p,\phi_p)$, where $V_p\subseteq U_p$ is a neighborhood of $p$, $\phi_p:~V_p\rightarrow \phi_p(V_p)\subset \mathbb{R}^4$ is a homeomorphism such that $\phi^{\ast}\omega_0=\omega_p$, and $$\omega_0=\sum_{i=1}^{2}dx_i\wedge dy_i$$
is the standard symplectic form on $\mathbb{R}^4$. Let $J_{st}$ be the standard complex structure on $\mathbb{C}^2\cong\mathbb{R}^4$ with complex coordinates $z_i=x_i+\sqrt{-1}y_i,~i=1,2$, and $J_p=\phi^{\ast}J_{st}$ the induced complex structure on $V_p$. Set $g_p(\cdot,\cdot)=F(\cdot,J\cdot)$. So we can get $g_J=g_pe^D$ on $V_p$, where $D$ is a symplectic $J$-anti-invariant (2,0) tensor (for details, see {\rm Tan-Wang-Zhou \cite{TWZ}}). Therefore, for the almost Hermitian 4-manifold $(M,g_J,J,F)$, any $p\in M$, there exists a small neighborhood $V_p$ such that on $V_p$ there is $F$-compatible complex structure $J_p$, that is, any almost complex 4-manifold has locally complex structure. Let $g_{J_{st}}(\cdot,\cdot)=F(\cdot,J_{st}\cdot)$ on $V_p$, then $g_{J_{st}}(p)=g_J(p)$, $g_{J_{st}}$ is a Hermitian metric on $V_p$.
\end{rem}

  \subsection{Regularization of quasi-J-plurisubharmonic functions on tamed almost Hermitian $4$-manifolds}\label{Regularization}

    In this subsection, we consider regularization of quasi-J-plurisubharmonic functions on almost Hermitian $2n$-manifolds.
   Let $(M,g_J,J,F)$ be an almost Hermitian $2n$-manifold.
   Suppose $\phi$ is a quasi-J-plurisubharmonic function, that is,
   a function which is locally the sum of $\phi_1$ and $\phi_2$ where $\phi_1$ is a smooth function and $\phi_2$ is
   a $J$-plurisubharmonic function.
    In this section, as done in Section $3$ of Demailly's article \cite{D2},
    we consider regularization of quasi-J-plurisubharmonic functions  in almost Hermitian $2n$-manifolds tamed by $\omega_1=F+d_J^{-}(v+\bar{v})$.

   For any $p\in(M,g_J,J,F)$, choose a complex coordinate $$U_p=\{ z_i=x_i+\sqrt{-1} y_i,~i=1,\cdot\cdot\cdot,n \}$$
      around $p$ such that $\{ \frac{\partial}{\partial z_i}|_{p} \}_{i=1,2,\cdot\cdot\cdot,n}\subset T^{1,0}_p M$
      is orthonormal at $p$ with respect to almost Hermitian metric $h=g_J-\sqrt{-1}F$.
      Consider the exponential map: $$T^{1,0}_z M\rightarrow M,~(z,\zeta)\mapsto \exp_z(\zeta),~z\in U_p,~(z,\zeta)\in T^{1,0}_zM.$$

    By (\ref{D14}), we have Taylor expansion of exponential map,
       \begin{eqnarray}\label{E1}
             \textrm{exp}_z(\zeta)_s&=&K_{p,s}(z,\xi)+\sum_{1\leq i,j,k\leq n} c_{jsik}(\frac{1}{2}\bar{z}_k+\frac{1}{6}\bar{\xi}_k)\xi_i\xi_j\nonumber\\
             && +O(|\xi|^2(|z|+|\xi|)^2),
     \end{eqnarray}
   where
      \begin{eqnarray}\label{E2}
         K_{p,s}(z,\xi)&=&z_s+\xi_s-\sum_{1\leq i,j\leq n}a_{isj}z_i\xi_j+\sum_{1\leq i,j,k,l\leq n}a_{ksl}a_{ilj}z_iz_j\xi_k\nonumber\\
            &&-\sum_{1\leq i,j,k\leq n}b_{jski}(z_iz_j\xi_k+z_i\xi_j\xi_k+\frac{1}{3}\xi_i\xi_j\xi_k).
          \end{eqnarray}
   Here $a_{ijs}, b_{iksj}$ and $c_{ijks}$ are given in Appendix \ref{Exponential}.
  However, we make this map quasi-holomorphic as follows:
  \begin{eqnarray}\label{E3}
   \textrm{exph}_z(\zeta)_s&=&K_{p,s}(z,\xi)+\frac{1}{2}\sum_{1\leq i,j,k\leq n} c_{jsik}\bar{z}_k\xi_i\xi_j+O(|\xi|^2(|z|+|\xi|)^2).
   \end{eqnarray}
   Here, for fixed $z\in M$, ${\rm exph}_z(\zeta)$ is holomorphic for $\zeta\in T^{1,0}_z M$

    For a fixed point $p\in M$ and use the coordinate $(p,e_1,\cdot\cdot\cdot,e_n)$ for $T_p^{1,0}M$,
    where $(e_1,\cdot\cdot\cdot,e_n)$ is orthernormal.
    Suppose $(\theta^1,\cdot\cdot\cdot,\theta^n)$ is the dual coframe of $(e_1,\cdot\cdot\cdot,e_n)$.
    As in Appendix \ref{Exponential}, $\zeta\in T^{1,0}_z(M)$, $\zeta=\sum\zeta_i\frac{\partial}{\partial z_i}=\sum\xi_i\tilde{e}_i$,
    \begin{equation}\label{}
      |\zeta|^2=\sum_m|\xi_m|^2-\sum_{j,k,l,m}c_{lmjk}z_j\bar{z}_k\xi_l\bar{\xi}_m+O(|z|^3)|\xi|^2.
    \end{equation}
   The volume form
   \begin{eqnarray}
      \textrm{d}\lambda(\zeta)&=& \frac{1}{2^nn!}(\sqrt{-1}\partial_{J(p)}\bar{\partial}_{J(p)}|\zeta|^2)^n \nonumber\\
      &=&  (1-\sum_{j,k,l}c_{lljk}z_j\bar{z}_k+O(|z|^3))\frac{\sqrt{-1}}{2}d\xi_1\wedge d\bar{\xi}_1
      \wedge\cdot\cdot\cdot\wedge\frac{\sqrt{-1}}{2}d\xi_n\wedge d\bar{\xi}_n.
   \end{eqnarray}

   Choose a smooth cut-off function $\chi:~\mathbb{R}\rightarrow\mathbb{R}$ satisfying
   $$\chi(t)\left\{
    \begin{array}{l}
    >0, ~t<1\\
    =0, ~t\geq 1,
    \end{array}\right.
    ~~\int_{v\in\mathbb{C}^{n}}\chi(|v|^2)~\textrm{d}\lambda(v)=1.$$

   Set
   $$\phi_{\varepsilon}(z)=\frac{1}{\varepsilon^{2n}}\int_{\zeta\in T_z^{1,0}M} \phi(\textrm{exph}_z(\zeta))\cdot\chi(\frac{|\zeta|^2}{\varepsilon^2})~\textrm{d}\lambda(\zeta),\,\,\,\,\,\,\varepsilon>0.$$

    \begin{eqnarray}\label{smoothing}
    \Phi(z,w)=\int_{\zeta\in T_z^{1,0}M}\phi(\textrm{exph}_z(w\zeta))\cdot\chi(|\zeta|^2)~\textrm{d}\lambda(\zeta),
    \end{eqnarray}
        which is smooth on $M\times\{w\in\mathbb{C}\mid 0<|w|<\varepsilon_0\}$ for some $\varepsilon_0>0$.
        Then for $w\in\mathbb{C}$ with $|w|=\varepsilon$, we have $\phi_{\varepsilon}(z)=\Phi(z,w)$.
   In the following,
    we need to compute $(dJd\Phi)^{(1,1)}$ over the set $M\times\{0<|w|<\varepsilon_0\}$ and estimate the negative part when $|w|$ is small.

     In (\ref{smoothing}), we make the change of variables $s=w^{-1}\rho(p,w\zeta)$,
     hence we can write $\rm{exph}_p(w\zeta)=L_p(z,ws)$.
       By (\ref{pro exp1}) and (\ref{pro exp2}), we get
      \begin{eqnarray}\label{sm}
        s_m&=&\xi_m+\sum_{2\leq|\alpha|\leq N} \left( \sum_{k}d_{\alpha mk}w^{|\alpha|-1}\xi^{\alpha}\bar{z}_k+\sum_{j,k}e_{\alpha mjk}w^{|\alpha|-1}\xi^{\alpha}z_j\bar{z}_k \right) \nonumber\\
        && +O(\bar{z}^2,|z|^3,w^{N-1}\xi^{N-1})w\xi^2.
          \end{eqnarray}
      Hence,
     \begin{eqnarray}\label{xim}
          \xi_m&=&s_m-\sum_{2\leq|\alpha|\leq N} \left( \sum_{k}d_{\alpha mk}w^{|\alpha|-1}s^{\alpha}\bar{z}_k+\sum_{j,k}e_{\alpha jkm}w^{|\alpha|-1}s^{\alpha}z_j\bar{z}_k \right)\nonumber\\
           && +O(\bar{z}^2,|z|^3,w^{N-1}s^{N-1})ws^2,
           \end{eqnarray}
           and $\xi=s+O(w^Ns^{N+1})$ for $z=0$.
      Plugging into (\ref{smoothing}), we get
     \begin{eqnarray}\label{smoothing1}
         \Phi(z,w)=\int_{\mathbb{C}^{n}}\phi(L_p(z,ws))\chi(A(z,w,s))B(z,w,s)d\lambda(s).
          \end{eqnarray}
            where
    \begin{eqnarray*}
        &&A(z,w,s)\\
  &=& \sum_{1\leq m\leq n} |s_m|^2-\sum_{1\leq j,k,l,m\leq n} c_{lmjk}z_j\bar{z}_ks_l\bar{s}_m  \\
   && -2Re\sum_{\alpha,k,m} d_{\alpha mk}w^{|\alpha|-1}s^{\alpha}\bar{s}_m\bar{z}_k-2Re\sum_{\alpha,j,k,m}
   e_{\alpha mjk}w^{|\alpha|-1}s^{\alpha}\bar{s}_mz_j\bar{z}_k \\
         &&     +\sum_{\alpha,\beta,j,k,m} d_{\alpha mk}\overline{d_{\beta mj}}w^{|\alpha|-1}w^{|\beta|-1}s^{\alpha}\bar{s}^{\beta}z_j\bar{z}_k  \\
         && +O(z^2,\bar{z}^2,|z|^3,|w|^{N-1}|s|^{N-1})|w||s|^3,
          \end{eqnarray*}
    \begin{eqnarray*}
   &&B(z,w,s)\\
   &=& 1-\sum_{1\leq j,k,l\leq n}c_{lljk}z_j\bar{z}_k\\
   &&-2Re\sum_{\alpha,k,m}d_{\alpha mk}w^{|\alpha|-1}\alpha_ms^{\alpha-1_m}\bar{z}_k  \\
       && -2Re\sum_{\alpha,j,k,m}e_{\alpha mjk}w^{|\alpha|-1}\alpha_ms^{\alpha-1_m}z_j\bar{z}_k \\
      && +\sum_{\alpha,\beta,j,k,l,m} d_{\alpha mk}\overline{d_{\beta lj}}w^{|\beta|-1}\alpha_m\beta_ls^{\alpha-1_m}\bar{s}^{\beta-1_l}z_j\bar{z}_k\\
             && +O(z^2,\bar{z}^2,|z|^3,|w|^{N-1}|s|^{N-1})|w||s|,
           \end{eqnarray*}
   here $(1_m)_{1\leq m\leq n}$ denotes the standard basis of $\mathbb{Z}^n$, hence $s^{1_m}=s_m$.

   Let $(M,g_J,J,F)$ be a $2n$-dimensional almost Hermitian manifold.
   We have the following lemma (cf. Wang-Zhu \cite{WZ2})
   \begin{lem}\label{dJd}
   Suppose $f$ is a smooth function on $M$, then
   \begin{eqnarray*}
     dJdf &=& (dJdf)^{(1,1)}+(dJdf)^{(2,0)+(0,2)} \\
      &=&  2\sqrt{-1}f_{i\bar{j}}\theta^i\wedge\bar{\theta}^j-
      2\sqrt{-1}(\overline{N^k_{\bar{i}\bar{j}}}\bar{f}_k\theta^i\wedge\theta^j+N^k_{\bar{i}\bar{j}}f_k\bar{\theta}^i\wedge\bar{\theta}^j),
   \end{eqnarray*}
   where $\partial_J f=\sum f_k\theta^k$, $\bar{\partial}_J f=\sum \bar{f}_k\bar{\theta}^k$,
    $N^k_{\bar{i}\bar{j}}$ is the Nijenhuis tensor $J$ which is independent of the choice of a metric.
   \end{lem}

    By Lemma 2.1 of Diederich-Sukhov \cite{DS}, for any $p\in M$,
    there exists a neighborhood $U$ of $p$ and a coordinate map $z: U\rightarrow\mathbb{B}$ such that $z(p)=0$ and $dz(p)\circ J(p)\circ dz^{-1}(0)=J_{st}$.
    Moreover, $z_{\ast}(J):=dz\circ J\circ dz^{-1}$ satisfies $||z_{\ast}(J)-J_{st}||_{C^{\alpha}(\bar{\mathbb{B}})}\leq\lambda_0$ for every $\alpha\geq 0$ and $\lambda_0>0$, where $\mathbb{B}$ is the unit ball in $\mathbb{C}^n$.
     It is easy to see that
     $$\partial_Jf|_p=\partial_{J_{st}}f|_p,\,\,
      \bar{\partial}_Jf|_p=\bar{\partial}_{J_{st}}f|_p,$$
      and
      $$dJdf|_p=2\sqrt{-1}\partial_J\bar{\partial}_Jf|_p=2\sqrt{-1}\partial_{J_{st}}\bar{\partial}_{J_{st}}f|_p.$$
     For more details, please see Diederich-Sukhov \cite{DS}.
      Fix a point $p\in M$, choose a complex coordinate chart $U_p=\{(z_1,\cdot\cdot\cdot,z_n)\in\mathbb{C}^n\}$ around $p$.
      Define two almost complex structures on $U_p\times\mathbb{C}$ as follows:
      $$
     \tilde{J}(z)=J(z)\oplus J_{st},\,\,\,\tilde{J}_0=\tilde{J}(0)=J(0)\oplus J_{st}.
      $$
      It is easy to see that $\tilde{J}_0$ is integrable.
      Return to (\ref{smoothing}),
      $$
      \Phi(z,w)=\int_{\zeta\in T_z^{1,0}M}\phi(\textrm{exph}_z(w\zeta))\cdot\chi(|\zeta|^2)~\textrm{d}\lambda(\zeta).
      $$
      The change of variable $y={\rm exph}_z(w\zeta)$ expresses $w\zeta$ as a smooth function of $y,z$ in neighborhood of the diagonal in $M\times M$.
      Hence $\Phi$ is a smooth over $M\times\{0<|w|<\varepsilon_0\}$ for some $\varepsilon_0>0$.
      By (\ref{smoothing1}), we are going to compute $\partial_{\tilde{J}}\Phi$, $\bar{\partial}_{\tilde{J}}\Phi$ and $\partial_{\tilde{J}}\bar{\partial}_{\tilde{J}}\Phi$.
      Note that
      $$
      (d\tilde{J}d\Phi(z,w))^{(1,1)}|_{(0,w)}= (d\tilde{J}_0d\Phi(z,w))^{(1,1)}|_{(0,w)}
      =2\sqrt{-1}\partial_{\tilde{J}_0}\tilde{\partial}_{\tilde{J}_0}\Phi(z,w))|_{(0,w)},
      $$
      and
      \begin{eqnarray*}
         (d\tilde{J}d\Phi(z,w))^{(2,0)+(0,2)}|_{(0,w)}
         &=& -2\sqrt{-1}(\overline{N^k_{\bar{i}\bar{j}}}\frac{\partial}{\partial \bar{z}_k}\Phi(z,w)|_{(0,w)}dz_i\wedge dz_j \\
         && +N^k_{\bar{i}\bar{j}}\frac{\partial}{\partial z_k}\Phi(z,w)|_{(0,w)}d\bar{z}_i\wedge d\bar{z}_j).
      \end{eqnarray*}
      By Lemma \ref{dJd}, we have
      \begin{eqnarray}
        d\tilde{J}d\Phi(z,w))^{(1,1)}|_{(0,w)} &=&  d\tilde{J}_0d\Phi(z,w))^{(1,1)}|_{(0,w)}  \nonumber\\
         &=& 2\sqrt{-1}\partial_{\tilde{J}_0}\tilde{\partial}_{\tilde{J}_0}\Phi(z,w))|_{(0,w)} \nonumber\\
         &=& 2\sqrt{-1}(\frac{\partial^2}{\partial z_i\partial \bar{z}_j}\Phi(z,w)|_{(0,w)}dz_i\wedge d\bar{z}_j\nonumber\\
         &&+\frac{\partial^2}{\partial z_i\partial \bar{w}}\Phi(z,w)|_{(0,w)}dz_i\wedge d\bar{w} \nonumber\\
         && +\frac{\partial^2}{\partial w\partial \bar{z_j}}\Phi(z,w)|_{(0,w)}dw\wedge d\bar{z}_j \nonumber\\
         &&  +\frac{\partial^2}{\partial w\partial \bar{w}}\Phi(z,w)|_{(0,w)}dw\wedge d\bar{w},
      \end{eqnarray}
      and
      \begin{eqnarray}
        (d\tilde{J}d\Phi(z,w))^{(2,0)+(0,2)}|_{(0,w)} &=&(d\tilde{J}_0d\Phi(z,w))^{(2,0)+(0,2)}|_{(0,w)}  \nonumber\\
         &=& (dJ(p)d\Phi(z,w))^{(2,0)+(0,2)}|_{(0,w)} \nonumber\\
         &=& -2\sqrt{-1}(\overline{N^k_{\bar{i}\bar{j}}}\frac{\partial}{\partial \bar{z}_k}\Phi(z,w)|_{(0,w)}dz_i\wedge dz_j \nonumber\\
                &&+N^k_{\bar{i}\bar{j}}\frac{\partial}{\partial z_k}\Phi(z,w)|_{(0,w)}d\bar{z}_i\wedge d\bar{z}_j).
      \end{eqnarray}
     By the above observation, Proposition $3.8$ of Demailly \cite{D2} can be generalized to almost Hermitian $2n$-manifolds as follows
          \begin{prop}\label{lem e1}
           For any integer $N\geq 2$ and any $(\varrho,\eta)\in T^{1,0}_zU_p\times\mathbb{C}$,
           at $(z,w)\in U_p\times\mathbb{C}$ we have the following estimates

          (1)
           \begin{eqnarray*}
             \partial_{\tilde{J}_0}\Phi_{(p,w)}\cdot(\varrho,\eta) &=& \int_{\zeta\in T^{1,0}_pM}\partial_{\tilde{J}_0}\phi_{(\rm{exph}_z(w\zeta))}\cdot\tau
                       \chi(|\zeta|^2)~\textrm{d}\lambda(\zeta)
                        +O(|w|^N)(\varrho,\eta) , \\
           \end{eqnarray*}

           (2)
            \begin{eqnarray*}
           \partial_{\tilde{J}}\bar{\partial}_{\tilde{J}}\Phi_{(p,w)}(\varrho\wedge\bar{\varrho},\eta\wedge\bar{\eta})
         &=&\partial_{\tilde{J}_0}\bar{\partial}_{\tilde{J}_0}\Phi_{(p,w)}(\varrho\wedge\bar{\varrho},\eta\wedge\bar{\eta})  \\
        &=&\int_{\zeta\in T^{1,0}_pM} \partial_{\tilde{J}_0}\bar{\partial}_{\tilde{J}_0}\phi\cdot(\tau\wedge\bar{\tau}+|w|^2V)_{{\rm exph}_p(w\zeta)}\chi(|\zeta|^2)~\textrm{d}\lambda(\zeta)\\
        && +O(|w|^{N-1})(\varrho\wedge\bar{\varrho},\eta\wedge\bar{\eta}),
           \end{eqnarray*}
        where $\tau$ is a vector field over $TM^{1,0}$,
        $V$ is a $(1,1)-$vector field, both depending smoothly on the parameters $p,w$ and linearly or quadratically on $\varrho,\eta$.
        The vector fields $\tau, V$ are given at $y={\rm exph}_p(w\zeta)$ by
       \begin{eqnarray*}
       &&\tau_y=\partial_{J(p)} {\rm exph}_{(p,w\zeta)}(\varrho^h+\eta\zeta^v+|w|^2\Xi^v_y),\\
         &&V_y=\partial_{J(p)}  {\rm exph}_{(p,w\zeta)}(U^v-|w|^2\Xi^v\wedge\overline{\Xi^v})_y,
            \end{eqnarray*}
     where $\varrho^{h},\zeta^v\in T(TM)_{(p,w\zeta)}$ are respectively the horizontal lifting of $\varrho$ with respect to the Chern connection $\nabla$ with respect to $h$ and $J(p)$, and the vertical vector associated to $\zeta$, and where $\epsilon$ can be arbitrarily small. Here, $\Xi, U$ is defined by
       \begin{eqnarray*}
      &&\Xi_y(\zeta)=\sum_{\alpha,j,l,m}\frac{1}{\chi(|\zeta|^2)}
      \frac{\partial}{\partial\bar{\zeta}_l}(\chi_1(|\zeta|^2)\bar{\zeta}^{\alpha-1_m})
      \overline{d_{\alpha lj}}\frac{\alpha_m}{|\alpha|}\bar{w}^{|\alpha|-2}\varrho_j\frac{\partial}{\partial z_m},\\
           &&U_y(\zeta)=\sum_{l,m}\frac{1}{2}(U_{m,l}(\zeta)+\overline{U_{l,m}}(\zeta))\frac{\partial}{\partial z_m}\wedge\frac{\partial}{\partial \bar{z}_l},\\
      &&U_{m,l}(\zeta)=-\frac{\chi_1(|\zeta|^2)}{\chi(|\zeta|^2)}\{\sum_{j,k}c_{lmjk}\varrho_j\bar{\varrho}_k+2\sum_{\alpha,j,k}e_{\alpha mjk}w^{|\alpha|-1}\frac{\alpha_l}{|\alpha|}\zeta^{\alpha-1_t}\varrho_j\bar{\varrho}_k\\
        &&+2\sum_{\alpha,k}d_{\alpha mk}(|\alpha|-1)w^{|\alpha|-2}\frac{\alpha_l}{|\alpha|}\zeta^{\alpha-1_t}\eta\bar{\varrho}_k
        +\sum_{\alpha,\beta,j,k}d_{\alpha mk}\overline{d_{\beta lj}}w^{|\alpha|-2}\bar{w}^{|\beta|-2}\zeta^{\alpha}\bar{\zeta}^{\beta}\varrho_j\bar{\varrho}_k\}.
           \end{eqnarray*}
         Here, $$\chi_1(t)=\int_{+\infty}^{t}\chi(u)du,$$ and $c_{lmjk}$, $d_{\beta lj}, e_{\alpha mjk}$ are defined in Appendix \ref{Exponential}.
          Moreover, $\alpha,\beta\in\mathbb{N}^n$ run over all multi-indices such that $2\leq|\alpha|,|\beta|\leq N$.
    \end{prop}

      \begin{proof}
               Our approach is similar to the proof of Proposition 3.8 in Demailly \cite{D2}.
               A brute force differentiation of (\ref{smoothing1}) gives
               \begin{eqnarray}\label{partial equ}
                  \partial_{\tilde{J}_0} \Phi_{(p,w)}\cdot(\varrho,\eta)&=&\int_{\mathbb{C}^n} \partial_{\tilde{J}_0}(\phi\circ L_p)_{(0,ws)}\cdot(\varrho,\eta)
                  \chi(A(0,w,s))B(0,w,s) d\lambda(s)       \nonumber\\
                   && - \int_{\mathbb{C}^n}(\phi\circ L_p)_{(0,ws)}E_{(w,s)}\cdot(\varrho,\eta)d\lambda(s) ,
               \end{eqnarray}
    where
    $$
    E_{(w,s)}=-\partial_{\tilde{J}_0}(\chi(A(z,w,s))B(z,w,s))_{(z,w)}.
    $$
    We find
    \begin{eqnarray}\label{E equ}
       E_{(w,s)}\cdot(\varrho,\eta) &=& \sum_{l,m}\frac{\partial^2}{\partial \bar{s}_l\partial s_m}(\chi(|s|^2)
       \sum_{\alpha,j}\overline{d_{\alpha lj}}\bar{w}^{|\alpha|-1}\frac{\alpha_m}{|\alpha|}\bar{s}^{\alpha-1_m}\varrho_j ) \nonumber \\
       &&  +O(|w|^{N-1}|s|^N)\cdot(\varrho,\eta),
    \end{eqnarray}
     \begin{eqnarray}\label{zeta equ}
        \partial_{\tilde{J}_0}\bar{\partial}_{\tilde{J}_0} \Phi_{(p,w)}\cdot(\varrho\wedge\bar{\varrho},\eta\wedge\bar{\eta})
        &=& \int_{\mathbb{C}^n} \partial_{\tilde{J}_0}\bar{\partial}_{\tilde{J}_0} (\phi\circ L_p)_{(0,ws)}
        \cdot(\varrho\wedge\bar{\varrho},\eta s\wedge\overline{\eta s}) \nonumber \\
          &&    \,\,\,\,\,\,\,\,\,\,    \,\,\,\,\,\,\,\,\,\, \cdot\chi(A(0,w,s))B(0,w,s) d\lambda(s) \nonumber\\
        &&- \int_{\mathbb{C}^n}\bar{\partial}_{\tilde{J}_0}(\phi\circ L_p)_{(0,ws)}\cdot(\bar{\varrho},\overline{\eta s}) \cdot E_{(w,s)}\cdot(\varrho,\eta s) d\lambda(s)  \nonumber\\
        && - \int_{\mathbb{C}^n} \partial_{\tilde{J}_0}(\phi\circ L_p)_{(0,ws)}\cdot(\varrho,\eta s)  \cdot \overline{E_{(w,s)}}\cdot (\bar{\varrho},\overline{\eta s})  d\lambda(s)  \nonumber\\
        &&-\int_{\mathbb{C}^n} (\phi\circ L_p)_{(0,ws)}
        \cdot F_{(w,s)}\cdot(\varrho\wedge\bar{\varrho},\eta s\wedge\overline{\eta s}) d\lambda(s), \nonumber\\
        &&
     \end{eqnarray}
     where
     \begin{equation}
       F_{(w,s)}= -\partial_{\tilde{J}_0}\bar{\partial}_{\tilde{J}_0}(\chi(A(z,w,s))B(z,w,s))_{(z,w)}.
     \end{equation}
     We find
     \begin{eqnarray}\label{F equ}
       && F_{(w,s)}\cdot(\varrho\wedge\bar{\varrho},\eta s\wedge\overline{\eta s}) \nonumber\\
        &&~~~= \sum_{l,m}\frac{\partial^2}{\partial \bar{s}_l\partial s_m}(\chi_1(|s|^2)\sum_{j,k}c_{lmjk}\varrho_j\bar{\varrho}_k ) \nonumber\\
        &&~~~+2Re\{\sum_{l,m}\frac{\partial^2}{\partial \bar{s}_l\partial s_m}(\chi_1(|s|^2)
        \sum_{\alpha,j,k}e_{\alpha mjk}w^{|\alpha|-1}\frac{\alpha_l}{|\alpha|}s^{\alpha-1_l}\varrho_j\bar{\varrho}_k )  \nonumber\\
        &&~~~+\sum_{l,m}\frac{\partial^2}{\partial s_l\partial \bar{s}_m}(\chi_1(|s|^2)
        \sum_{\alpha,k}d_{\alpha mk}(|\alpha|-1)w^{|\alpha|-2}\frac{\alpha_l}{|\alpha|}\bar{s}^{\alpha-1_l}\eta\bar{\varrho}_k )\}   \nonumber\\
        &&~~~-\sum_{l,m}\frac{\partial^2}{\partial \bar{s}_l\partial s_m}(\chi_1(|s|^2)
        \sum_{\alpha,k}d_{\alpha mk}\overline{d_{\beta lj}}w^{|\alpha|-1}\bar{w}^{|\beta|-1}s^{\alpha}\bar{s}^{\beta}\varrho_j\bar{\varrho}_k )  \nonumber\\
        &&~~~  +O(|w|^{N-2}|s|^N)(\varrho\wedge\bar{\varrho},\eta s\wedge\overline{\eta s}) .
     \end{eqnarray}
  In all these expansions, the remainder terms $O(\cdot)$ involve uniform constants when
       the origin x of coordinates belongs to a compact subset of a coordinate patch.
       When $U_p$ is very small, without loss of generality, we may assume that $\phi$ is strictly $J$-convex (and $J(p)$-convex).
       By the mean value properties of plurisubharmonic functions (cf. L. Simon \cite{S1}), we have
           $$
           \int_{|s|<1}|\phi(p+ws)|d\lambda(s)\leq C(1+\log|w|)
           $$
   locally uniformly in $p$.
   An integration by parts with compact supports yields
   $$
    \int_{|s|<1}\partial_{\tilde{J}_0}(\phi\circ L_p)_{(0,ws)}O(|w|)d\lambda(s)= \int_{|s|<1}\phi\circ L_p(0,ws)d\lambda(s)=O(\log|w|).
   $$
   Hence, the remainder term $O(|w|^{N-1})$ in $E_{(w,s)}$
    gives contributions of order at most $O(|w|^{N-1}\log|w|)$ in $\partial_{\tilde{J}_0}\Phi$
   as $|w|$ tends to 0;
   the remainder terms $O(|w|^{N-1})$ in $E_{(w,s)}$ and $O(|w|^{N-2})$ in $F_{(w,s)}$
    give contributions of order at most $O(|w|^{N-2}\log|w|)$ in $\partial_{\tilde{J}_0}\bar{\partial}_{\tilde{J}_0}\Phi$
   as $|w|$ tends to 0.

   By (\ref{E equ}), an integration by parts in (\ref{partial equ}) gives
   \begin{eqnarray}\label{par J_0 equ}
     \partial_{\tilde{J}_0}\Phi_{(p,w)}\cdot(\varrho,\eta) &=&\int_{\mathbb{C}^n}\partial_{J(p)}(\phi\circ L_p)\{(\varrho,\eta s)+|w|^2(0,\Xi) \} \nonumber \\
      &&\chi(A(0,w,s))B(0,w,s) d\lambda(s)\nonumber \\
      &&  +O(|w|^{N-1}\log|w|)\cdot(\varrho,\eta),
   \end{eqnarray}
    with
   $$
   \Xi(\zeta)=\sum_{\alpha,j,l,m}\frac{1}{\chi(|s|^2)}
      \frac{\partial}{\partial\bar{s}_l}(\chi_1(|s|^2)\bar{s}^{\alpha-1_m})
      \overline{d_{\alpha lj}}\frac{\alpha_m}{|\alpha|}\bar{w}^{|\alpha|-2}\varrho_j\frac{\partial}{\partial z_m}.
   $$
  The choice $\chi(t)=\frac{C}{(1-t)^2}\exp(\frac{1}{t-1})$ for $t<1$ gives $\chi_1(t)=-C\exp(\frac{1}{t-1})$,
  so $$\chi_1(t)/\chi(t)=(1-t)^2$$ is smooth and bounded, and our vector field $\Xi(\zeta)$ is smooth. We can write
  $$
  \tau=dL_p(0,ws)(\varrho,\eta s+|w|^2\Xi(\zeta)).
  $$
   Since $${\rm exph}_z(\zeta)=L_p(z,\rho_p(z,\xi)),\,\,\,\rho_p(0,\xi)=\xi+O(\xi^{N+1}),$$
   and $$\partial_{J(p)}\rho_p(0,\xi)=d\xi+O(\xi^{N})d\xi$$ by Proposition \ref{D prop},
   we infer that the $(1,0)$-differential of $\rm{exph}$ at $(p,\zeta)\in T^{1,0}M$ is
   $$\partial_{J(p)}\rm{exph}_{(p,\zeta)}=dL_p(0,\xi)+O(\xi^{N})d\xi$$
    modulo the identification of the tangent spaces $T(T^{1,0}M)_{(p,\xi)}$
    and $T(T\mathbb{C}^n)_{(0,\xi)}$ given by the coordinates $(z,\xi)$ on $T^{1,0}M$.
    However, these coordinates are precisely those which realize the splitting
    $$T(T^{1,0}M)_{(p,\xi)}=(T^{1,0}_pM)^h\oplus(T^{1,0}_pM)^v$$ with respect to the
    Chern connection on $U_p$.
    Since $s=\xi+O(w^N\xi^{N+1})$ and $\xi=\zeta$ at $z=0$,
   we get $$\tau=\partial_{J(p)}{\rm exph}_{(p,w\zeta)}(\varrho^h+\eta\zeta^v+|w|^2\Xi(\zeta)^v)+O(|w|^N|\zeta|^{N}).$$
    We can drop the terms $O(|w|^N)$ in $\tau$ because
    \begin{eqnarray}\label{int equ}
      \int_{|\zeta|<1}\partial_{J(p)}\phi({\rm exph}_p(w\zeta))d\lambda(\zeta) &=& \frac{1}{|w|^{2n}}\int_{|\zeta|<|w|}\partial_{J(p)}\phi({\rm exph}_p(\zeta))d\lambda(\zeta)  \nonumber\\
       &=& O(|w|^{-1}).
    \end{eqnarray}
   By (\ref{E equ}) and (\ref{F equ}), an integration by parts in (\ref{zeta equ}) gives
   \begin{eqnarray}\label{par par equ}
     \partial_{\tilde{J}_0}\bar{\partial}_{\tilde{J}_0} \Phi_{(p,w)}(\varrho,\eta)\wedge\overline{(\varrho,\eta)} &=&
      \int_{\mathbb{C}^n}\partial_{\tilde{J}_0}\bar{\partial}_{\tilde{J}_0}(\phi\circ L_p)_{(0,ws)}\cdot\{(\varrho,\eta s)\wedge\overline{(\varrho,\eta s)} \nonumber \\
      &&+|w|^2(0,\Xi(\zeta))\wedge\overline{(\varrho,\eta s)}+|w|^2(\varrho,\eta s)\wedge\overline{(0,\Xi(\zeta))}   \nonumber \\
      &&+|w|(0,U)\}\chi(A(0,w,s))B(0,w,s)d\lambda(s)   \nonumber \\
      &&+O(|w|^{N-2}\log|w|)(\varrho,\eta) \wedge \overline{(\varrho,\eta) },
   \end{eqnarray}
   where
   $$U(\zeta)=\sum_{l,m}\frac{1}{2}(U_{m,l}+\overline{U_{l,m}})\frac{\partial}{\partial z_m}\wedge \frac{\partial}{\partial z_l}$$ is smooth,
   \begin{eqnarray*}
     U_{m,l}(\zeta) &=&-\frac{\chi_1(|s|^2)}{\chi(|s|^2)}\cdot\{\sum_{j,k}c_{lmjk}\varrho_j\bar{\varrho}_k
     +2\sum_{\alpha,j,k}e_{\alpha mjk}w^{|\alpha|-1}\frac{\alpha_l}{|\alpha|}s^{\alpha-1_l}\varrho_j\bar{\varrho}_k  \nonumber\\
      && +2\sum_{\alpha,k}d_{\alpha mk}(|\alpha|-1)w^{|\alpha|-2}\frac{\alpha_l}{|\alpha|}s^{\alpha-1_l}\eta\bar{\varrho}_k \} \nonumber\\
      && +\sum_{\alpha,\beta,j,k}d_{\alpha mk}\overline{d_{\beta lj}}w^{|\alpha|-1}\overline{w}^{|\beta|-1}s^{\alpha}\bar{s}^{\beta}\varrho_j\bar{\varrho}_k.
   \end{eqnarray*}
   We can write
   \begin{eqnarray*}
      &&(\varrho,\eta s)\wedge \overline{(\varrho,\eta s)}+|w|^2(0,\Xi(\zeta))\wedge \overline{(\varrho,\eta s)}
      +|w|^2(\varrho,\eta s)\wedge \overline{(0,\Xi(\zeta))}+|w|(0,U) \\
      &=&  (\varrho,\eta s+|w|^2\Xi(\zeta)\wedge\overline{(\varrho,\eta s+|w|^2\Xi(\zeta))}+(0,U-|w|^2\Xi(\zeta)\wedge\overline{\Xi(\zeta)})).
   \end{eqnarray*}
  Therefore (\ref{par par equ}) implies the formula in Proposition \ref{lem e1} with
  $$
  V=dL_{p(0,ws)}(0,U-|w|^2\Xi\wedge\overline{\Xi}).
  $$
  Finally, we get
    $$
     V=\partial_{\tilde{J}_0}{\rm exph}_{(p,w\zeta)}(U^v-|w|^2\Xi^v\wedge \overline{\Xi^v})+O(|w|^N|\zeta|^N).
     $$
   Also, we can get
   \begin{eqnarray}\label{par par equ}
      \int_{|\zeta|<1} \partial_{\tilde{J}_0}\bar{\partial}_{\tilde{J}_0}{\rm exph}_p(w\zeta)d\lambda(\zeta)&=&
       \frac{1}{|w|^{2n}} \int_{|\zeta|<|w|} \partial_{\tilde{J}_0}\bar{\partial}_{\tilde{J}_0}{\rm exph}_p(\zeta)d\lambda(\zeta) \nonumber\\
      &=&  O(|w|^{-2}).
   \end{eqnarray}
   After substituting $\zeta$ to $s$ in the formal expression of $\Xi$ and $U$,
   we get precisely the formula given in Proposition \ref{lem e1}.
   As done in the proof of Proposition $3.8$ in \cite{D2},
   the remainder term $O(|w|^{N-1}\log|w|)$ in (\ref{par J_0 equ}) (resp. $O(|w|^{N-2}\log|w|)$ in (\ref{par par equ}) )
   is in fact of the type  $O(|w|^N)$ (resp. $O(|w|^{N-1})$).
   To see this, we increase $N$ by two units and estimate the additional terms in the expansions,
    due to the contribution of all multi-indices $\alpha$ with $|\alpha|=N+1$ or $N+2$.
    It is easily seen that the additional terms in $\Xi$ and $U$ are $O(|w|^{N-1})$,
    so they are $O(|w|^{N+1})$ in $\tau$ and $|w|^2V$.
    The contribution of these terms to $\partial_{J(p)}\Phi_{(p,w)}$ and $\partial_{J(p)}\bar{\partial}_{J(p)}\Phi_{(p,w)}$
    are thus of the forms
    $$
    \int_{|\zeta|<1}\partial_{J(p)}\phi({\rm exph}_p(w\zeta))O(|w|^{N+1})d\lambda(\zeta)=O(|w|^N),
    $$
    $$
    \int_{|\zeta|<1}\partial_{J(p)}\bar{\partial}_{J(p)}\phi({\rm exph}_p(w\zeta))O(|w|^{N+1})d\lambda(\zeta)=O(|w|^{N-1}).
    $$
    This completes the proof of Proposition \ref{lem e1}.
      \end{proof}

      By Lemma \ref{dJd}, (\ref{par J_0 equ}) and (\ref{int equ}), we have
      \begin{col}
     Let $N=2$, we have
     \begin{eqnarray*}
       ( \frac{1}{2}d\tilde{J}d\Phi(z,w)_{(0,w)})^{(0,2)}(\bar{\varrho},0)\wedge (\bar{\varrho},0)&=&
       \sqrt{-1}\int_{\zeta\in T^{1,0}_pM}-\sum_k\frac{\partial}{\partial z_k}(\phi\circ L_p(z,w))N^k(p) \\
        && \{[(\bar{\varrho},0)+|w|^2(0,\overline{\Xi})]\wedge[(\bar{\varrho},0)+|w|^2(0,\overline{\Xi})]\}_{(0,w)} \\
        && +O(|w|^2)\\
        &=& \sqrt{-1}\int_{\zeta\in T^{1,0}_pM}-\sum_{k,i,j}\frac{\partial}{\partial z_k}(\phi\circ L_p(z,w))N^k_{\bar{i}\bar{j}}
        \bar{\varrho}_i\wedge\bar{\varrho}_j \\
        &&+O(|w|).
     \end{eqnarray*}
      \end{col}

   \subsection {Regularization of closed positive $(1,1)$-currents on tamed almost complex $4$-manifolds}\label{Demailly}

   In this subsection, we devote to studying regularization of closed positive (1,1) currents on tamed almost complex 4-manifolds. It is similar to J.-P. Demailly's result \cite{D1,D2} that we will see that it is always possible to approximate a closed positive almost complex
 $(1,1)$ current $T$ on almost Hermitian 4-manifold $(M,g_J,J,F)$ by smooth closed real currents admitting a small negative part, and that this negative part
 can be estimated in terms of the Lelong numbers of $T$ and geometry of $M$.
 Let $(M,g_J,J,F)$ be an almost Hermitian 4-manifold tamed by a symplectic form $\omega_1=F+d^{-}_J(v+\bar{v})$.
 In general, $\partial_J\bar{\partial}_Jf$ is not $d$-closed since $J$ is not integrable.
   In Section \ref{2}, we have defined an operator
     \begin{equation}
  \mathcal{D}^+_J: C^\infty(M)\longrightarrow\Omega^+_J(M).
     \end{equation}
     For any $f\in C^\infty(M)$, $\mathcal{D}^+_J(f)\in\Omega^+_J(M)$ is $d$-closed.
   Let $T$ be a closed strictly positive current of bidegree $(1,1)$ on $(M,g_J,J,F)$ tamed by $\omega_1$.
   Let $\widetilde{\omega}$ be a smooth closed $(1,1)$-form representing the same $\mathcal{D}^+_J$-cohomology class as $T$
    and let $\psi=\mathcal{D}^+_J(f)$ be a quasi-$J$-positive $(1,1)$-current (that is, a $(1,1)$-form which is
   locally the sum of a positive $(1,1)$-current and a smooth  $(1,1)$-form)
    such that $T=\widetilde{\omega}+\mathcal{D}^+_J(f)$. Such a function $f$, is called a quasi-$J$-plurisubharmonic function.
    Such a decomposition exists since we can always find an open covering $(\Omega_k)$ where $\Omega_k$ are $J$-pseudoconvex domains
      such that $T=\mathcal{D}^+_J(f_k)$
   over $\Omega_k$ (see Lemma \ref{current app} or Theorem \ref{app 1} in Appendix A),
   and costruct a global $f=\sum\varsigma_kf_k$ by means of a partion of unity $(\varsigma_k)$ (note that $f-f_k$ is smooth on $\Omega_k$).
   Notice that for any $p\in M$, there exists a $J$-compatible symplectic form $\omega_p$ on a small neighborhood
   $U_p$ which is $J$-pseudoconvex.
   By the construction of $\omega_p$ (cf. Lejmi \cite{L2}),
   there exists real $1$-form $\alpha$ on $U_p$ such that $\omega_p=d\alpha$.
   Hence, by Lemma \ref{current app} (that is Theorem \ref{app 1} in Appendix \ref{Hormander}), there is a real function
   $f_p$ on $U_p$ which is strictly $J$-plurisubharmonic
   such that $\omega_p=\widetilde{\mathcal{D}}^+_J(f_p)=d\mathcal{\widetilde{W}}(f_p)$ with respect to metric $g_p(\cdot,\cdot)=\omega_p(\cdot,J\cdot)$.
   Since $(U_p,\omega_p)$ is a symplectic $4$-manifold, thus $\mathcal{\widetilde{W}}(f_p)=\mathcal{W}(f_p)$ (see Section \ref{2}),
   \begin{equation}\label{local rep}
     \omega_p=d\mathcal{W}(f_p)=\mathcal{D}^+_J(f_p).
   \end{equation}
   Therefore, we have the following lemma,
   \begin{lem}\label{exist lem}
   Suppose that $(M,J)$ is an almost complex $4$-manifold.
   For any $p\in M$, there exist a small neighborhood $U_p$ and a smooth strictly $J$-plurisubharmonic function $f_p$ on $U_p$ such that $\mathcal{D}^+_J(f_p)$
   is a strictly positive closed $(1,1)$-form on $U_p$.
   \end{lem}

   Now suppose that $(M,g_J,J,F)$ is an almost Hermitian $4$-manifold tamed by $\omega_1=F+d^-_J(v+\bar{v})$
   where $v\in \Omega^{0,2}_J(M)$.
   Let $T=\widetilde{\omega}+\mathcal{D}^+_J(\phi)$ be a closed $(1,1)$-current on $M$, where $\widetilde{\omega}$ is a smooth closed $(1,1)$-form
   on $M$ and $\phi\in L^q_2(M)$ for some fixed $q\in(1,2)$.
    It is easy to see that
    \begin{equation}\label{lelong lelong}
      \nu_1(T,p)=\nu_1(\mathcal{D}^+_J(\phi),p),\,\,\,p\in M,
    \end{equation}
   where $\nu_1$ is the Lelong number defined in Appendix \ref{Lelong} (cf. Definition \ref{Lelong 1}).

  As done in Appendix \ref{Exponential}, for almost Hermitian 4-manifold $(M,g_J,J,F)$,
  we choose the second canonical connection $\nabla^1$ with respect to the almost Hermitian structure $(g_J,J,F)$.
  Then, for the coframe $\{ \theta^1,\theta^2 \}$ of the metric $g=g_J-\sqrt{-1}F$ on $M$,
  the curvature form of $\nabla^1$ is given by
   \begin{eqnarray*}
        && (\Psi^j_i)^{(1,1)}=R^j_{ik\bar{l}}\theta^k\wedge\bar{\theta}^l,~1\leq i,j,k,l \leq 2,\\
     && (\Psi^j_i)^{(2,0)}=K^i_{ikl}\theta^k\wedge\theta^l,~1\leq i,j,k,l \leq 2,\\
     && (\Psi^j_i)^{(0,2)}=K^i_{j\bar{k}\bar{l}}\bar{\theta}^k\wedge\bar{\theta}^l,~1\leq i,j,k,l \leq 2,
       \end{eqnarray*}
      with $K^i_{jkl}=-K^i_{jlk},~K^i_{j\bar{k}\bar{l}}=-K^i_{j\bar{l}\bar{k}}$ and $R^i_{jk\bar{l}}=-R^j_{il\bar{k}}$.
          Denote by $R^{\nabla^1}$ the (1,1) part of the curvature form $\Psi$ of $\nabla^1$, hence $R^{\nabla^1}=R^j_{ik\bar{l}}\theta^k\wedge\bar{\theta}^l,~1\leq i,j,k,l\leq 2$.
           Using Taylor expansion of exponential map (cf Appendix \ref{Exponential}),
           we can make regularization of quasi-$J$-plurisubharmonic functions.
           Suppose that $(M,g_J,J,F)$ is an almost Hermitian $4$-manifold tamed by a symplectic form $\omega_1=F+d^-_J(v+\bar{v})$,
           $v\in \Lambda^{0,1}(M)$.
           Let $\phi\in L^q_2(M)$ for some fixed $q\in(1,2)$ be a quasi-$J$-plurisubharmonic function,
           then $d^{1,1}_J(\phi)\in\Lambda^{1,1}_{\mathbb{R}}(M)\otimes L^q$ is a closed $(1,1)$-current.
           As done in Appendix \ref{Exponential}, $\forall p\in M$, choose a strictly $J$-pseudoconvex neighborhood
            $U_p=\{(z_1,z_2)\in \mathbb{C}^2\mid \,\, z_i(p)=0,\,\,i=1,2\}$ of $p$.
            Then
            $$
            \phi_{\varepsilon}(z)=\frac{1}{\varepsilon^4}\int_{\zeta\in T^{1,0}_zM}\phi({\rm exph}_z(\zeta))\chi(\frac{|\zeta|}{\varepsilon^2})d\lambda(\zeta),\,\,\,\varepsilon>0,
            $$
            $$
            \Phi(z,w)=\int_{\zeta\in T^{1,0}_zM}\phi({\rm exph}_z(w\zeta))\chi(|\zeta|^2)d\lambda(\zeta).
            $$
            Here $d\lambda$ denotes the Lebesgue measure on $\mathbb{C}^2$.
            The change of variable $y={\rm exph}_z(w\zeta)$ expresses $ws$ as a smooth function of $y,z$ in a neighborhood of the diagonal in $M\times M$.
            Hence $\Phi$ is smooth over $M\times \{0<|w|<\varepsilon_0\}$ for some $\varepsilon_0>0$.
            Let $\tilde{J}=J\oplus J_{st}$, $\tilde{J}_0=J(p)\oplus J_{st}$ on $U_p\times\mathbb{C}$,
            as done in Appendix \ref{Regularization}, we have the following formula:
            \begin{eqnarray}
              \mathcal{D}^+_{\tilde{J}}(\phi)|_{(p,w)}(\zeta\wedge\bar{\zeta},\eta\wedge\bar{\eta}) &=&
                \int_{\zeta\in T^{1,0}_pM} \mathcal{D}^+_{\tilde{J}_0}\phi(\tau\wedge\bar{\tau}+|w|^2V)_{{\rm exph}_p(w\zeta)}
                \chi(|\zeta|^2)d\lambda(\zeta)\nonumber \\
               && +O(|w|^{N-1})(\zeta\wedge\bar{\zeta},\eta\wedge\bar{\eta}).
            \end{eqnarray}
           Where at $y={\rm exph}_p(w\zeta)$,
           $$\tau_y=\partial_{J(p)}{\rm exph}_{(p,w\zeta)}(\varrho^h+\eta\zeta^v+|w|^2\Xi^v_y),$$
           $$V_y=\partial_{J(p)}{\rm exph}_{(p,w\zeta)}(U^v-|w|^2\Xi^v\wedge \overline{\Xi^v})_y.$$
         For more details, see Appendix \ref{Regularization}.
            The following theorem is similar to Theorem $4.1$ in Demailly \cite{D2}.

     \begin{theo}\label{theorem F1}
       Let $(M,g_J,J,F)$ be an almost Hermitian $4$-dimensional manifold tamed by the symplectic form $\omega_1=F+d^{-}_J(v+\bar{v})$,
        $\nabla^{1}$ the second canonical connection on $TM$.
        Fix a smooth semipositive $(1,1)-$form $u$ on $M$ such that the (1,1) curvature form $R^{\nabla^1}$ of $\nabla^{1}$ satisfies
         $$(R^{\nabla^1}+u\otimes Id_{TM})(\varrho\otimes\xi,\varrho\otimes\xi)\geq 0$$
        $\forall \varrho,\xi\in TM^{1,0}$ such that $\langle \varrho,\xi\rangle=0$.
        Let $T=\widetilde{\omega}+\mathcal{D}^+_J(\phi)$ be a closed real current where $\widetilde{\omega}$ is a smooth closed real $(1,1)-$form and $\phi$ is quasi-J-plurisubharmonic.
         Suppose that $T\geq\gamma$ for some real $(1,1)-$form $\gamma$ with continuous coefficients.
         As $w$ tends to $0$ and $p$ runs over $M$, there is a uniform lower bound
     $$\widetilde{\omega}_p(\zeta\wedge\bar{\zeta})+\mathcal{D}^+_J\Phi_{(p,w)}(\varrho\wedge\bar{\varrho},\eta\wedge\bar{\eta})\geq \gamma_p(\varrho\wedge\bar{\varrho})-\lambda(p,|w|)u_p(\varrho\wedge\bar{\varrho})-\delta(|w|)|\varrho|^2-\frac{1}{\pi}K(|\varrho||\eta|+|\eta|^2),$$
       where $(\varrho,\eta)\in TM^{1,0}\times\mathbb{C}$, $K>0$
       is a sufficiently large constant, $\delta(t)$ a continuous increasing function with $\displaystyle{\lim_{t\rightarrow 0}\delta(t)=0}$,
        and $$\lambda(p,t)=t\frac{\partial}{\partial t}(\Phi(p,t)+Kt^2),$$ where
        $$\Phi(p,w)=\int_{s\in T^{1,0}_pM}\phi(exph_p(ws))\cdot\chi(|s|^2)d\lambda(s).$$
        The above derivative $\lambda(p,t)$ is a nonnegative continuous function on $M\times (0,\varepsilon_0)$ which is increasing in $t$ and such that
      $$\displaystyle{\lim_{t\rightarrow 0}\lambda(p,t)=\nu_1(p,T)}.$$
          In particular, the currents $T_{\varepsilon}=\widetilde{\omega}+\mathcal{D}^+_J(\Phi(\cdot,\varepsilon))$ are smooth closed real currents converging weakly to $T$ as $\varepsilon$ tends to $0$, such that $$T_{\varepsilon}\geq\gamma-\lambda(\cdot,\varepsilon)u-\delta(\varepsilon)F.$$
           \end{theo}

    \begin{proof}
  Our approach is along the lines used by Demailly to give a proof of Theorem 4.1 in Demailly \cite{D2}
  by replacing $\sqrt{-1}\partial\bar{\partial}\phi$ with $\mathcal{D}^+_J(\phi)$ .
  It suffices to prove the estimate for $|w|<\varepsilon(\delta)$, with $\delta>0$ fixed in place $\delta(|w|)$.
  Also, the estimates are local on $M$.
  For any $p\in M$, choose a small neighborhood $U_p$ which is strictly $J$-pseudoconvex, and there exists a symplectic form $\omega_p$ on $U_p$.
  We may assume that $U_p$ is very small, hence on $U_p$ there exists Darboux coordinate $(z_1,z_2)$, $z_i(p)=0$, $i=1,2$, for $\omega_p$.
  If we change $\phi$ into $\phi+\phi_p$ with a small function $\phi_p$ such that $\mathcal{D}^+_J(\phi_p)$ is strictly positive (or negative) on
  $U_p$ due to Lemma \ref{exist lem}, then $\widetilde{\omega}$ is changed into $\widetilde{\omega}-\mathcal{D}^+_J(\phi_p)$
  and $\Phi$ into $\Phi+\Phi_p$, where $\Phi_p$ is a smooth function on $U_p\times\mathbb{C}$
  such that $\Phi_p(z,w)=\phi_p(z)+O(|w|^2)$.
  It follows that the estimate remains unchanged up to a term $O(1)|\eta|^2$.
  We can thus work on a small coordinate open set $\Omega\subset U_p\subset M$
   and choose $\phi_p$ such that $\gamma-(\widetilde{\omega}-\mathcal{D}^+_J(\phi_p))$ is positive definite and small at $p$,
   say equal to $\frac{\delta}{4}F_p$.
   After shrinking $\Omega$ and making $\phi\mapsto\phi+\phi_p$, we may in fact suppose that $T=\widetilde{\omega}+\mathcal{D}^+_J(\phi)$
  on $\Omega_{p,\delta}\subset\Omega$ where $\Omega$ satisfies $\gamma_p-\widetilde{\omega}_p=\frac{\delta}{4}F_p$
  and $\gamma-\frac{\delta}{2}F\leq\widetilde{\omega}\leq\gamma$ on $\Omega_{p,\delta}$.
   In particular, $\mathcal{D}^+_J(\phi)\geq\gamma-\alpha$, $\mathcal{D}^+_J(\phi)$ is strictly positive on $\Omega_{p,\delta}$
   and also $\phi$ is a strictly $J$-plurisubharmonic function (cf. Lemma \ref{current app}).
    As done in classical complex analysis (cf. Demailly \cite{D2}), all we have to show is
  $$\mathcal{D}^+_J(\Phi_{(p,w)})(\varrho\wedge\bar{\varrho},
  \eta\wedge\bar{\eta})\geq-\lambda(p,|w|)u_p(\varrho\wedge\bar{\varrho})
  -\frac{\delta}{2}|\varrho|^2-K(|\varrho||\eta|+|\eta|^2),
  $$
for $|w|<w_0(\delta)$ small.
Let $$\chi_1(t)=\int_{+\infty}^t \chi(t),$$
we apply Proposition \ref{lem e1} at order $N=2$, $|\alpha|=2$.
 Similar to the argument in Appendix \ref{Regularization} (cf. (\ref{par par equ})),
 we have
 \begin{eqnarray}\label{equ small}
      \int_{|\zeta|<1} \mathcal{D}^+_{J(p)}\phi({\rm exph}_p(w\zeta))d\lambda(\zeta)&=&
       \frac{1}{|w|^{4}} \int_{|\zeta|<|w|}\mathcal{D}^+_{J(p)}\phi(\rm{exph}_p(\zeta))d\lambda(\zeta) \nonumber\\
      &=&  O(|w|^{-2}).
   \end{eqnarray}
   Notice that $0\leq-\chi_1\leq\chi$. As done in the proof of Theorem $4.1$ in \cite{D2},
   we use the fact that $\tau=\varrho+\eta\zeta+O(|w|)$.
 Consider $J_{st}, \partial_{st}$ and $\overline{\partial}_{st}$,
 by (\ref{equ small}), we can neglect all terms of the form $\mathcal{D}^+_{J(p)}(\phi)(\tau\wedge\bar{\tau}+|w|^2V)_{{\rm exph}_p(w\zeta)}O(|w|^3)$
 under the integral sign.
 Up to such terms, in terms of Proposition \ref{D prop},
 $\mathcal{D}^+_{J(p)}(\phi)(\tau\wedge\bar{\tau}+|w|^2V)_{{\rm exph}_p(w\zeta)}\chi(|\zeta|^2)$
 is equal to
 \begin{eqnarray*}
&& -|w|^2\chi_1(|\zeta|^2)Re\sum_{l,m}\mathcal{D}^+_{J(p)}(\phi)_{\bar{l}m}
\{\frac{\chi(|\zeta|^2)}{-|w|^2\chi_1(|\zeta|^2)}\bar{\tau}_l\tau_m+\sum_{j,k}c_{jklm}\varrho_j\bar{\varrho}_k  \nonumber\\
&&\,\,\,\,\,\,\,\,\,\,\,\,\,\,\,\,\,\,\,\,\,\,\,\,\,\,\,\,\,\,\,\,\,\,\,\,\,\,\,\,\,\,\,\,\,\,\,\,\,\,\,\,\,\,
+2\sum_{|\alpha|=2,k}d_{\alpha km}(|\alpha|-1)w^{|\alpha|-2}\frac{\alpha_l}{|\alpha|}\zeta^{\alpha-1_l}\eta\bar{\varrho}_k \}\nonumber\\
&&\geq-|w|^2\chi_1(|\zeta|^2)\sum_{l,m}\mathcal{D}^+_{J(p)}(\phi)_{\bar{l}m}\{\frac{1}{|w|^2}\bar{\tau}_l\tau_m
     +\sum_{j,k}c_{jklm}(\varrho_j\bar{\varrho}_k+\frac{1}{2}\zeta_j\eta\bar{\varrho}_k+\frac{1}{2}\bar{\zeta}_k\varrho_j\bar{\eta})\}  \nonumber\\
     &&=-|w|^2\chi_1(|\zeta|^2)\sum_{l,m}\mathcal{D}^+_{J(p)}(\phi)_{\bar{l}m}\{\frac{1}{|w|^2}\bar{\tau}_l\tau_m
     +\sum_{j,k}c_{jklm}\tau_j\bar{\tau}_k\nonumber\\
    &&\,\,\,\,\,\,\,\,\,\,\,\,\,\,\,\,\,\,\,\,\,\,\,\,\,\,\,\,\,\,\,\,\,\,\,\,\,\,\,\,\,\,\,\,\,\,\,\,\,\,\,\,\,\,\,\,\,\,\,\,\,\,\,\,\,\,\,\,\,\,\,\,
    -\sum_{j,k}c_{jklm}(\frac{1}{2}\zeta_j\eta\bar{\varrho}_k+\frac{1}{2}\bar{\zeta}_k\varrho_j\bar{\eta}+\zeta_j\bar{\zeta}_k\eta\bar{\eta})\},
\end{eqnarray*}
 where $\mathcal{D}^+_{J(p)}(\phi)_{\bar{l}m}=\mathcal{D}^+_{J(p)}(\phi)(\frac{\partial}{\partial \bar{z}_l}\wedge\frac{\partial}{\partial z_m})$.
 By (\ref{equ small}), the mixed terms $\varrho_j\bar{\eta}$, $\eta\bar{\varrho_k}$ give rise to contributions bounded below by
  $-K'(|\varrho||\eta|+|\eta|^2)$.
  Hence, we get the estimate (cf. ($4.3)$ in Demailly \cite{D2})
\begin{eqnarray}\label{f1}
&&\mathcal{D}^+_J(\Phi_{(p,w)})(\varrho\wedge\bar{\varrho},\eta\wedge\bar{\eta})\nonumber\\
&\geq& |w|^2\int_{\mathbb{C}^{2}}-\chi_1(|\zeta|^2)\sum_{j,k,l,m}\mathcal{D}^+_{J(p)}(exph_p(w\zeta))_{\bar{l}m}
(c_{jklm}+\frac{1}{|w|^2}\delta_{jm}\delta_{kl})\tau_j\bar{\tau}_k~d\lambda(\zeta) \nonumber\\
&&-K'(|\varrho||\eta|+|\eta|^2),
\end{eqnarray}
where $c_{jklm}$ is the curvature of $\nabla^1$ with respect to the metric $g_J$.
Similar to the argument of Lemma 4.4 in Demailly \cite{D2}, since $\mathcal{D}^+_{J(p)}(\phi)$ is strictly positive,
we have
\begin{eqnarray*}
\sum_{j,k,l,m}\mathcal{D}^+_{J(p)}(\phi)_{\bar{l}m}(c_{jklm}+M_{\varepsilon}\delta_{jm}\delta_{kl})\tau_j\bar{\tau}_k+
\sum_{l}\mathcal{D}^+_{J(p)}(\phi)_{l\bar{l}}(u(\tau\wedge\bar{\tau})+\varepsilon|\tau|^2)\geq 0,
\end{eqnarray*}
for a constant $M_{\varepsilon}>0$.
Combining this with (\ref{f1}) for $|w|^2<\frac{1}{M_{\varepsilon}}$, we have
\begin{eqnarray*}
&&\mathcal{D}^+_J(\Phi_{(p,w)})(\varrho\wedge\bar{\varrho},\eta\wedge\bar{\eta})\\
&\geq&-\left[2|w|^2\int_{\mathbb{C}^{2}}-\chi(|\zeta|^2)\sum_{l}\mathcal{D}^+_{J(p)}(\phi)_{l\bar{l}}
(exph_p(w\zeta))~d\lambda(\zeta)\right](u_p(\varrho\wedge\bar{\varrho})+\varepsilon|\varrho|^2)\\
&&-K''(|\varrho||\eta|+|\eta|^2).
\end{eqnarray*}
Change variables $\zeta\rightarrow s$ defined by $exph_p(w\zeta)=p+ws$, and choose $\varepsilon\ll\delta$, we get
$$\mathcal{D}^+_J(\Phi_{(p,w)})(\varrho\wedge\bar{\varrho},\eta\wedge\bar{\eta})\geq-\lambda_{\Omega}(p,|w|)u_p(\varrho\wedge\bar{\varrho})-\frac{\delta}{3}|\varrho|^2
-K(|\varrho||\eta|+|\eta|^2),$$
where
$$\lambda_{\Omega}(p,|w|)=2|w|^2\int_{\mathbb{C}^{2}}-\chi_1(s^2)\sum_{l}\mathcal{D}^+_{J(p)}(\phi)_{l\bar{l}}(p+ws)~d\lambda(s).$$
More details, see the proof of Theorem $4.1$ in Demailly \cite{D2}.

Recall that the Lelong number $\nu_1(p,T)=\lim_{r\rightarrow 0}\nu_1(p,\omega_1,r,T)$, where $T=\tilde{\omega}+\mathcal{D}^+_J(\phi)$, $\tilde{\omega}$ is smooth closed $(1,1)$-form
$$\nu_1(p,\omega_1,r,T)=\int_{B(p,r)} T\wedge\omega_1.$$
More details, see Definition \ref{Lelong 1} in Appendix \ref{Lelong}.

Hence $$\nu_1(p,T)=\lim_{r\rightarrow 0}\nu_1(p,\omega_1,r,T)=\lim_{r\rightarrow 0}\nu_1(p,F,r,\mathcal{D}^+_J(\phi)).$$
By remark \ref{D5} and Theorem \ref{comparison theo}, we have
\begin{eqnarray*}
\lim_{|w|\rightarrow 0} \nu_1(p,F,r,\mathcal{D}^+_J\phi)&=&
\lim_{r\rightarrow 0}\frac{2}{ r^2}\int_{B(p,r)} \sum_{1\leq l\leq 2}\mathcal{D}^+_{J(p)}(\phi)_{l\bar{l}}(p+ws)  d\lambda(s)\\
&=& \lim_{r\rightarrow 0} \nu'_1(p,r,\mathcal{D}^+_J(\phi)),
\end{eqnarray*}
where
$$\nu'_1(p,r,\mathcal{D}^+_J(\phi))=\frac{2}{r^2}|w|^2\int_{|s|<r} \sum_{1\leq l\leq 2} \mathcal{D}^+_{J(p)}(\phi)_{l\bar{l}}(p+ws) d\lambda(s).$$
Since $$-\chi_1(|s|^2)=2\int^{\infty}_{|s|} \chi(r^2) r dr,$$ by Fubini formula
$$\lambda_{\Omega}(p,|w|)=\int_{0}^{1} \nu'_1(p,|w|r,\mathcal{D}^+_J(\phi))\chi(r^2)r dr,$$
$$\lambda_{\Omega}(p,t)=\int_{\mathbb{R}^4} \nu'_1(p,t|s|,\mathcal{D}^+_J(\phi))\chi(|s|^2) d\lambda(s).$$
Hence $\lambda_{\Omega}(p,t)$ is smooth, increasing in $t$ and
$$\lim_{t\rightarrow 0}\lambda_{\Omega}(p,t)=\nu_1(p,\mathcal{D}^+_J(\phi))=\nu_1(p,T).$$
Recall that, in Theorem \ref{theorem F1},
$$\lambda(p,t)=\frac{\partial}{\partial\log t}(\Phi(p,t)+Kt^2)$$
is a nonnegative increasing function of $t$, since $\Phi(p,t)+Kt^2$ is plurisubharmonic in $t$.

Putting $\varrho=0$, Proposition \ref{lem e1} gives
\begin{eqnarray*}
\frac{\partial^2\Phi}{\partial w\partial\bar{w}}(p,w)=\int_{\mathbb{C}^{n}}\partial_{st}\bar{\partial}_{st}\phi_{{\rm exph}_p(w\zeta)}
(\zeta\wedge\bar{\zeta})\chi(|\zeta|^2)~d\lambda(\zeta)+O(1).
\end{eqnarray*}
Change coordinates so that ${\rm exph}_p(w\zeta)=p+ws$ where $\zeta=s+O(w^2s^3)$.
 Similar to Equality $(4.5)$ in Demailly \cite{D2}, since $\frac{\partial^2}{\partial w\partial\bar{w}}=t^{-1}\frac{\partial}{\partial t}(t\frac{\partial}{\partial t})$ for a function of $w$ depending only on $t=|w|$, a multiplication by $t$ followed by an integration implies
\begin{eqnarray}
t\frac{\partial\Phi(p,t)}{\partial t}=\int_{\mathbb{C}^{2}}\nu_1(p,t|s|,\mathcal{D}^+_J(\phi))\chi(|s|^2)d\lambda(s)+O(t^2)=\lambda_{\Omega}(p,t)+O(t^2).
\end{eqnarray}

Hence, $\lambda_{\Omega}(p,t)-\lambda(p,t)=O(t^2)$ and the first estimate in Theorem \ref{theorem F1}.
$\phi_{\varepsilon}$ converges to $\phi$ in $L_{loc}^1$, so $T_{\varepsilon}$ converges weakly to $T$.
Also, $\phi_{\varepsilon}+K\varepsilon^2$ is increasing in $\varepsilon$ by the above arguments.
We may assume that $(M,g_J,J,F)$ be a closed almost Hermitian $4$-manifold tamed by $\omega_1=F+d^-_J(v+\bar{v})$.
Hence $\lambda(p,|w|),~\delta(t)$ is well-defined on the whole $M$ when $|w|$ is very small. Then, $\lim_{t\rightarrow 0}\delta(t)=0,~\lim_{t\rightarrow 0}\lambda(p,t)=0,~\forall p\in M$.
The proof is completed.
\end{proof}

\begin{rem}\label{rem 8}
The estimates obtained in Theorem \ref{theorem F1} can be improved by setting
$$\tilde{\Phi}(p,w)=\Phi(p,w)+|w|,~\tilde{\lambda}(p,t)=t\frac{\partial}{\partial t}(\tilde{\Phi}(p,t)).$$ Similar to Remark $4.7$ in Demailly {\rm\cite{D2}}, we have
\begin{eqnarray}\label{rem 8.1}
\widetilde{\omega}_p(\varrho\wedge\bar{\varrho})+\mathcal{D}^+_J\tilde{\Phi}_{(p,w)}
(\varrho\wedge\bar{\varrho},\eta\wedge\bar{\eta})\geq\gamma_{p}(\varrho\wedge\bar{\varrho})
-\tilde{\lambda}(p,|w|)u_p(\varrho\wedge\bar{\varrho})-\tilde{\delta}(|w|)|\varrho|^2,
\end{eqnarray}
where $\displaystyle{ \lim_{t\rightarrow 0}\tilde{\lambda}(p,t)=\nu_1(p,T) }$, and $\displaystyle{ \lim_{t\rightarrow 0}\tilde{\delta}(t)=0 }$, $\tilde{\delta}$ being continuous and increasing.
\end{rem}

   \subsection{Approximation theorem on tamed almost complex four manifolds}\label{Approximation}

   This subsection is devoted to proving approximation theorem on tamed closed almost complex 4-manifolds.
   If $T$ is a closed positive or almost positive current on a tamed almost complex manifold $M$,
   we denote by $E_c(T)$ the $c$-upper level set of Lelong numbers:
   $$E_c(T)=\{ p\in M \mid  \nu_1(p,T)\geq c \},~c>0.$$

  As done in classical complex analysis, we have the following theorem:
 \begin{theo}\label{theorem B}
 (see Theorem $6.1$ in Demailly {\rm \cite{D2}})
 Let $T$ be a closed positive almost complex $(1,1)$ current on closed almost Hermitian 4-manifold $(M,g_J,J,F)$
 tamed by a symplectic form $\omega_1=F+d^{-}_J(v+\bar{v})$ and let $\widetilde{\omega}$ be a smooth real (1,1)-form in the same $\mathcal{D}^+_J$-cohomology class as $T$, that is, $T=\widetilde{\omega}+\mathcal{D}^+_J(\phi)$ where $\phi$ is in $L^q_2(M)_0$ for some fixed $q\in(1,2)$. Let $\gamma$ be a continuous real $(1,1)$-form such that $T\geq\gamma$. Let $\nabla^1$ be the second canonical connection on $TM$ with respect to the metric $g_J$ such that the corresponding (1,1) curvature form $R^{\nabla^1}$ of $\nabla^1$ satisfies
 $$(R^{\nabla^1}+u\otimes Id_{TM})(\varrho\otimes\xi,\varrho\otimes\xi)\geq 0,\,\,\,\forall \varrho,\xi\in TM^{1,0}$$
 with $<\varrho,\xi>_{g_J}=0$ for some continuous $(1,1)$-form $u$ on $M$. Then there is a family of closed positive almost complex $(1,1)$ currents
 $T_\varepsilon=\widetilde{\omega}+\mathcal{D}^+_J(\phi_\varepsilon), \varepsilon\in(0,\varepsilon_0)$ such that $\phi_\varepsilon$ is smooth over $M$, increases with $\varepsilon$, and converges to $\phi$ as $\varepsilon$ tends to zero (in particular, $T_\varepsilon$ is smooth and converges weakly to $T$ on $M$), and such that
\begin{enumerate}[1)]
 \item $T_\varepsilon\geq \gamma-\lambda_\varepsilon u-\delta_\varepsilon F$ where:
 \item $\lambda_\varepsilon(p)$ is an increasing family of continuous function on $M$ such that  $lim_{\varepsilon\rightarrow 0}\lambda_\varepsilon(p)=\nu_1(p,T)$ at every point $p\in M$,
 \item $\delta_\varepsilon$ is an increasing family of positive constants such that $lim_{\varepsilon\rightarrow 0}\delta_\varepsilon=0$.
\end{enumerate}
 \end{theo}

\begin{proof}
Our approach is along lines used by Demailly to give a proof of Theorem 6.1 in \cite{D2}.
As done in Theorem \ref{theorem F1} and Remark \ref{rem 8},
for a quasi-$J$-plurisubharmonic function $\phi$ on $M$, we have $\phi_{\varepsilon}$ defined on a small neighborhood of the diagonal
 of $M\times M$ and $\Phi$ on $M\times\{0<|w|<\varepsilon_0\}$.
Let $\phi_{c,\varepsilon}$ be the Legendre transform
$$\phi_{c,\varepsilon}=\inf_{|w|<1}(\widetilde{\Phi}(p,\varepsilon w)+\frac{\varepsilon}{1-|w|^2}-c\log|w|),$$
where $\widetilde{\Phi}(p,w)=\Phi(p,w)+|w|$. The sequence $\phi_{c,\varepsilon}$ is increasing in $\varepsilon$ and
$$\lim_{\varepsilon\rightarrow 0_{+}}\phi_{c,\varepsilon}(p)=\widetilde{\Phi}(p,0_{+})=\Phi(p,0_{+})=\phi(p),$$
where $\varepsilon\rightarrow 0_{+}$ means the limit from the right at $0$.
Moreover, as $\widetilde{\Phi}(p,w)$ is convex and increasing in $t=\log|w|$, the function
$$\Phi_{c,\varepsilon}(p,t):=\widetilde{\Phi}(p,\varepsilon t)+\frac{\varepsilon}{1-t^2}-c\log t$$
is strictly convex in $\log t$ and tends to $+\infty$ as $t$ tends to 1. Then the infimum is attained for $t=t_0(x)\in [0,1)$ given either by the zero of the $\frac{\partial}{\partial \log t}$ derivative:
$$\tilde{\lambda}(x,\varepsilon t)+\frac{2\varepsilon t^2}{(1-t^2)^2}-c=0$$
when $\displaystyle{\nu_1(p,T)=\lim_{t\rightarrow 0_{+}}\tilde{\lambda}(p,t)<c}$, or by $t_0(p)=0$ when $\nu_1(p,T)\geq c$.

Since the $\frac{\partial}{\partial \log t}$ derivative is itself strictly increasing in $t$, the implicit function theorem shows that $t_0(p)$ depends smoothly on $p$ on $M\backslash E_c(T)=\{ \nu_1(p,T)<c \},$ hence $\phi_{c,\varepsilon}=\Phi_{c,\varepsilon}(p,t_0(p))$ is smooth on $M\backslash E_c(T)$.

Fix a point $p\in M\backslash E_c(T)$ and $t_1>t_0(p)$. For all $z$ in a neighborhood $V$ of $p$ we still have $t_0(z)<t_1$, hence on $V$, we have
$$\phi_{c,\varepsilon}(z)=\inf_{|w|<t_1}(\widetilde{\Phi}(z,\varepsilon w)+\frac{\varepsilon}{1-|w|^2})-c\log|w|.$$

By (\ref{rem 8.1}), all functions involved in that infimum have a complex Hessian in $(z,w)$ bounded below by
$$\gamma_z-\widetilde{\omega}-\tilde{\lambda}(z,\varepsilon t_1)u_z-\tilde{\delta}(\varepsilon t_1)w_z.$$
By taking $t_1$ arbitrarily close to $t_0(p)$ and by shrinking $V$, the lower bound comes arbitrarily close to
$$\gamma_p-\widetilde{\omega}_p-\tilde{\lambda}(p,\varepsilon t_0(x))u_p-\tilde{\delta}(\varepsilon t_0(p))w_p\geq \gamma_p-\widetilde{\omega}_p-\min\{ \tilde{\lambda}(p,\varepsilon),c \}u_p-\tilde{\delta}(\varepsilon)w_p,$$
since $$\tilde{\lambda}(p,\varepsilon t_0(p))=c-2\varepsilon t_0(p)^2/(1-t_0(p)^2)^2\leq c, $$
and $\tilde{\lambda}(p,t),~\tilde{\delta}(t)$ are increasing in $t$. Hence we have
$$\widetilde{\omega}+\mathcal{D}^+_J\phi_{c,\varepsilon}\geq \gamma-\min\{ \tilde{\lambda}(\cdot,\varepsilon),c \}u-\tilde{\delta}(\varepsilon)w$$
on $M\backslash E_c(T)$. However, as the lower bound is a continuous $(1,1)$-form and $\phi_{c,\varepsilon}$ is quasi-$J$-plurisubharmonic, the lower bound extends to $M$ by continuity and $M$ is closed. Hence, 1), 2), 3) are proved. This completes the proof of Theorem \ref{theorem B}.
\end{proof}

\begin{rem}\label{last remar}
In Section \ref{last section}, we consider closed positive current
  $T=\widetilde{\omega}+\widetilde{\mathcal{D}}^+_J(\phi)$ on closed Hermitian $4$-manifold $(M,g_J,J,F)$ tamed by
  $\omega_1=F+d^-_J(v+\bar{v})$, $v\in\Omega^{0,1}_J(M)$.
  Here $\widetilde{\omega}$ is a closed smooth $(1,1)$-form, $\widetilde{\mathcal{D}}^+_J$ is defined in Section \ref{2},
  $\phi\in L^q_2(M)$ for some fixed $q\in(1,2)$.
  We would like point out that Theorem \ref{theorem B} also holds for $\widetilde{\mathcal{D}}^+_J$.
  In fact, the approximation theorem is locally proved.
  For $\forall p\in M$, there exists a symplectic $\omega_p$ on a strictly $J$-pseudoconvex domain $U_p$.
  Notice that it is often convenient to work with smooth forms and then prove statements about
  currents by using an approximation of a given current by smooth forms (cf.{\rm \cite{GH, S1}}).
  By Lemma \ref{current app} or Theorem \ref{app 1} in Appendix A,
   we can solve $\mathcal{\widetilde{W}}$, $d^-_J$-problem on strictly $J$-pseudoconvex symplectic domain $(U_p,\omega_p)$.
   Hence there is a $\phi_p\in L^2_2(U_p)$ such that $\mathcal{\widetilde{W}}(\phi)|_{U_p}=\mathcal{W}(\phi_p)$ and
   $\widetilde{\mathcal{D}}^+_J(\phi)|_{U_p}=\mathcal{D}^+_J(\phi_p)$ since $d\omega_p=0$ (cf. Remark \ref{2r6}).
   \end{rem}

\end{appendices}

 \vspace{3mm}\par\noindent {\bf{Acknowledgements.}} The first author would like to thank Professor Xiaojun Huang for his support.
   The second author would like to thank Fudan University and Professor Jiaxing Hong for hosting his visit in the autumn semester in 2013,
   East China Normal University and Professor Qing Zhou for hosting his visit in the spring semester in 2014,
        Wuhan University and Professor Hua Chen for hosting his visit in the autumn semester in 2015.
        The authors would like to thank Xiaowei Xu for useful conversations.
        The authors also had illuminating conversations with their colleague Shouwen Fang.
        The authors are also grateful to the referees for their valuable comments and suggestions.

  \vskip 24pt

 \noindent Qiang Tan\\
 Faculty of Science, Jiangsu University, Zhenjiang, Jiangsu 212013, China\\
 e-mail: tanqiang@ujs.edu.cn\\

 \vskip 6pt

 \noindent Hongyu Wang\\
 School of Mathematical Sciences, Yangzhou University, Yangzhou, Jiangsu 225002, China\\
 e-mail: hywang@yzu.edu.cn\\

 \vskip 6pt

 \noindent Jiuru Zhou\\
 School of Mathematical Sciences, Yangzhou University, Yangzhou, Jiangsu 225002, China\\
 e-mail: zhoujr1982@hotmail.com

 \vskip 6pt

 \noindent Peng Zhu\\
 School of Mathematics and Physics, Jiangsu University of Technology, Changzhou, Jiangsu 213001, China\\
 e-mail: zhupeng@jsut.edu.cn


\begin{thebibliography}{10}
 \addtolength{\itemsep}{-0.6em}
 \bibliographystyle{siam}
 \bibitem{Alm} F. Almgren, {\it Almgren's Big Regularity Paper}, World Sci. Mono. Series in Math. (V. Scheffer and J. Taylor, eds.), World Scientific, River Edge NJ, 2000.
 \bibitem{A2} M. Audin, {\it Symplectic and almost complex manifolds, with an appendix by P.Gauduchon}, Holomorphic Curves in Symplectic Geometry, 41-74,
   Progress in Math., {\bf 117},  Birkh\"{a}user, Basel, 1994.
 \bibitem{BHPV} W. Barth, K. Hulek, C. Peters and A. Van de Ven, {\it Compact Complex Surfaces}, Springer-Verlag, Berlin, 2004.
 \bibitem{BT} R. Bott and L. Tu, {\it Differential Forms in Algebraic Topology}, Springer, 1982.
 \bibitem{Br1} R. Bryant, {\it Submanifolds and special structures on the octonians}, J. Diff. Geom., {\bf 17} (1982), 185-232.
 \bibitem{Br2} R. Bryant, {\it Remarks on the geometry of almost complex 6-manifolds}, Asian J. Math., {\bf 10} (2006), 561-606.
 \bibitem{B} N. Buchdahl, {\it On compact K\"{a}hler surfaces}, Ann. Inst. Fourier, {\bf 49} (1999), 287-302.
 \bibitem{Ch} S. X.-D. Chang, {\it Two-dimensional area minimizing integral currents are classical minimal surface}, J. Amer. Math. Soc. {\bf 1},  (1988), 699-778.
  \bibitem{Cha} I. Chavel, {\it Riemannian geometry: a modern introduction}, Cambridge Studies in Advanced Mathematics, 98.
   Cambridge University Press, Cambridge, 2006.
 \bibitem{Che} S.-S. Chern, {\it Characteristic classes of Hermitian manifolds}, Ann. Math., {\bf 47} (1946), 85-121.
 \bibitem{D1} J.-P. Demailly, {\it Regularization of closed positive currents and intersection theory}, J. Alg. Geom., {\bf 1} (1992), 361-409.
 \bibitem{D2} J.-P. Demailly, {\it Regularization of closed positive currents of type (1,1) by the flow of a Chern connection}, in: Contributions to complex analysis and analytic geometry: dedicated to Pierre Dolbeault, ed. H. Skoda and J. M. Tr\'{e}preau, Wiesbaden, Vieweg, 1994.
 \bibitem{D3} J.-P. Demailly, {\it Complex Analytic and Differential Geometry}, Universit\'{e} de Grenoble I Institut Fourier, UMR 5582 du CNRS 38402 Saint-Martin d¡¯H\`{e}res, France, 2012.
     \bibitem{DS} K. Diederich and A. Sukhov, {\it Plurisubharmonic exhaustion functions and almost complex Stein structures},
      Michigan Math. J., {\bf 56} (2008), 331-355.
 \bibitem{D4} S. K. Donaldson, {\it Symplectic submanifolds and almost-complex geometry}, J. Diff. Geom., {\bf 44} (1996), 666-705.
 \bibitem{D6} S. K. Donaldson, {\it Two forms on four manifolds and elliptic equations}, Nankai Tracts Math., {\bf 11}, Inspired by S. S. Chern, 153-172, World Sci. Publ., Hackensack, N.J., 2006.
 \bibitem{DK} S. K. Donaldson and P. B. Kronheimer, {\it The Geometry of Four-Manifolds}, Oxford Mathematical Monographs, Oxford Science Publications, New York, 1990.
 \bibitem{DLZ1} T. Draghici, T.-J. Li and W. Zhang, {\it Symplectic forms and cohomology decomoposition of almost complex four-manifolds}, Int. Math. Res. Not., (2010), no. 1, 1-17.
 \bibitem{DLZ2} T. Draghici, T.-J. Li and W. Zhang, {\it On the J-anti-invariant cohomology of almost complex 4-manifolds},
          Quart. J. Math., {\bf 64} (2013), 83-111.
  \bibitem{ElMir} H. El Mir, {\it Sur le prolongement des courants positifs ferm\'{e}s}, Acta Math., {\bf 153} (1984), 1-45.
 \bibitem{EL} C. Ehresmann and P. Libermann, {\it Sur les structures presque hermitiennes isotropes}, C. R. Math. Acad. Sci. Paris., {\bf 232} (1951), 1281-1283.
 \bibitem{EG} Y. Eliashberg and M. Gromov, {\it Convex symplectic manifolds}, Proc. of Symp. in Pure Math., {\bf 52}, Part 2, Several Complex Variables and Complex Geometry, 135-162, Amer. Math. Soc., Providence, RI, 1991.
  \bibitem{Elk} F. Elkhadhra, {\it J-pluripolar subsets and currents on almost complex manifolds}, Math. Z., {\bf 264} (2010), 399-422.
   \bibitem{Elk2} F. Elkhadhra, {\it Poincar\'{e}-Lelong formula, J-analytic subsets and Lelong numbers of currents on almost complex manifolds},
   Bull. Sci. Math. {\bf 138} (2014), 393-405.
 \bibitem{F} H. Federer, {\it Geometric Measure Theory}, Springer-Verlag, 1969.
 \bibitem{Fr} K. Friedrichs, {\it The identity of weak and strong extensions of differential operators}, Trans. A.M.S., {\bf 55} (1944), 132-151.
 \bibitem{G1} P. Gauduchon, {\it Le th\'{e}or\`{e}me de l'excetricit$\acute{e}$ nulle}, C. R. Acad. Sci. Paris. S\'{e}rie A, {\bf 285}
 (1977), 387-390.
 \bibitem{G2} P. Gauduchon, {\it Hermitian connections and Dirac operators}, Boll. Un. Mat. Ital. B, {\bf 11} (1997), suppl., 257-288.
  \bibitem{Gauduchon} P. Gauduchon, {\it Calabi's extremal Kaehler metrics: An elementary introduction}, book in preparation (2011).
 \bibitem{GT} D. Gilbarg and N. Trudinger, {\it Elliptic partial differential equations of second order},
              2nd ed., Berlin-Heidelberg-New York, Springer, 1983.
 \bibitem{GH} P. A. Griffiths and J. Harries, {\it Principles of Algebraic Geometry}, New York, Wiley, 1978.
 \bibitem{G3} M. Gromov, {\it Pseudoholomorphic curves in symplectic manifolds}, Invent. Math., {\bf 82} (1985), 307-347.
 \bibitem{HL1} R. Harvey and H. B. Lawson, Jr., {\it  Calibrated geometries}, Acta Math., {\bf 148} (1982), 47-157.
 \bibitem{HL2} R. Harvey and H. B. Lawson, Jr., {\it An intrinsic characterization of K\"{a}hler manifolds}, Invent. Math., {\bf 74} (1983), 169-198.
 \bibitem{HL4} R. Harvey and H. B. Lawson, Jr., {\it An introduction to potential theory in calibrated geometry},
             Amer. J. Math., Vol. 131, {\bf 4} (2009), 893-944.
  \bibitem{HL5} R. Harvey and H. B. Lawson, Jr., {\it Geometric plurisubharmonicity and convexity: An introduction}, Adv. Math., {\bf 230} (2012), 2428-2456.
 \bibitem{HL3} R. Harvey and H. B. Lawson, Jr., {\it Potential theory on almost complex manifolds}, Ann. Inst. Fourier, {\bf 65} (2015), 171-210.
  \bibitem{HLP} R. Harvey, H. B. Lawson, Jr. and S. Pli\'{s}, {\it Smooth approximation of plurisubharmonic functions on almost complex manifolds},
         Math. Ann., {\bf 366} (2016), 929-940.
           \bibitem{Hormander0} L. H\"{o}rmander, {\it $L^2$  estimates and existence theorems for the $\bar{\partial}$ operator},
           Acta Math., {\bf 113} (1965), 89-152.
    \bibitem{Hormander} L. H\"{o}rmander, {\it An introduction to complex analysis in several variables}, third edition (revised),
     D. Van Nostrand Co., Inc., Princeton, N.J.-Toronto, Ont.-London, 1990.
 \bibitem{IR} S. Ivashkovich and J.-P. Rosay, {\it Schwarz-type Lemmas for solutions of $\partial$-inequalities and complete
    hyperbolicity of almost complex manifolds}, Ann. Inst. Fourier {\bf 54} (2004), 2387-2435.
 \bibitem{JS} S. Ji and B. Shiffman, {\it Properties of compact complex manifolds carrying closed positive currents}, J. Geom. Anal., {\bf 3} (1993), 37-61.
      \bibitem{JP} M. Jarnicki and P. Pflug, {\it Extension of holomorphic functions},
     Walter de Gruyter, 2000.
   \bibitem{Kim} J. Kim, {\it A closed symplectic four-manifold has almost K\"{a}hler metrics of negative scalar curvature},
                Ann. Glob. Anal. Geom., {\bf 33} (2008), 115-136.
 \bibitem{K} J. R. King, {\it The currents defined by analytic varieties}, Acta Math., {\bf 127} (1971), 185-220.
   \bibitem{Kise} C. O. Kiselman, {\it The partial Legendre transformation for plurisubharmonic functions}, Invent. Math., {\bf 49} (1978), 137-148.
   \bibitem{Kise2} C. O. Kiselman, {\it Plurisubharmonic functions and their singularities},
          NATO Adv. Sci. Inst. Ser. C Math. Phys. Sci., {\bf 439},
       Kluwer Acad. Publ., Dordrecht, 1994.
 \bibitem{KN} S. Kobayashi and K. Nomizu, {\it Foundations of Differential Geometry},  Vol. II, Inc., New York, 1996.
 \bibitem{Kod} K. Kodaira, {\it Complex manifolds and deformation of complex structures}, Springer-Verlag, 2005.
 \bibitem{KM} K. Kodaira and J. Morrow, {\it Complex Manifolds}, Holt, Rinehart and Winston, New York, 1971.
  \bibitem{Kohn1} J. J. Kohn, {\it Harmonic integrals on strongly pseudo-convex manifolds I}, Ann. Math., {\bf 78} (1963), 206-213.
   \bibitem{Kohn2} J. J. Kohn, {\it Harmonic integrals on strongly pseudo-convex manifolds II}, Ann. Math., {\bf 79} (1964), 450-472.
 \bibitem{L1} A. Lamari, {\it Courants kaehleriens et surfaces compactes}, Ann. Inst. Fourier, {\bf 49} (1999), 263-285.
 \bibitem{L2} M. Lejmi, {\it Strictly nearly K\"{a}hler 6-manifolds are not compatible with symplectic forms}, C. R. Math. Acad. Sci. Paris, {\bf 34} (2006), 759-762.
  \bibitem{Lejmi} M. Lejmi, {\it Extremal almost-K\"{a}hler metrics}, International J. Math., {\bf 21} (2010), 1639-1662.
 \bibitem{L3} M. Lejmi, {\it Stability under deformations of extremal almost-K\"{a}hler metrics in dimension 4}, Math. Res. Lett., {\bf 17} (2010), 601-612.
 \bibitem{L4} M. Lejmi, {\it Stability under deformations of Hermitian-Einstein almost-K\"{a}hler metrics}, Ann. Inst. Fourier, {\bf 64} (2014), 2251-2263.
 \bibitem{LZ} T.-J. Li and W. Zhang, {\it Comparing tamed and compatible symplectic cones and cohomological properties of almost complex manifolds}, Comm. Anal. Geom., {\bf 17} (2009), 651-684.
 \bibitem{LZ2} T.-J. Li and W. Zhang, {\it Almost K\"{a}hler forms on rational 4-manifolds}, Amer. J. Math., {\bf 137} (2015), 1209-1256.
 \bibitem{MS} D. McDuff and D. Salamon, {\it J-Holomorphic Curves and Symplectic Topology}, American Mathematical Society, 2004.
 \bibitem{Mo}C.B. Morrey, {\it The analytic embedding of abstract real analytic manifolds}, Ann. of Math. (2), {\bf 68} (1958), 159-201.
    \bibitem{NW} A. Nijenhuis and W. B. Woolf,  {\it Some integration problems in almost-complex and complex manifolds}, Ann. of Math. (2), {\bf 77} (1963), 424-489.
 \bibitem{Pali} N. Pali, {\it Fonctions plurisousharmoniques et courants positifs de type $(1,1)$ sur une vari\'{e}t\'{e} complex},
           Manuscripta Math., {\bf 118} (2005), 311-337.
 \bibitem{RT2} T. Rivi\`{e}re and G. Tian, {\it The singular set of 1-1 integral currents}, Ann. of Math., {\bf 169} (2009), 741-794.
 \bibitem{Rosay} Jean-Pierre Rosay, {\it $J$-Holomorphic submanifolds are pluripolar}, Math. Z., {\bf 253} (2006), 659-665.
  \bibitem{Rosay2} Jean-Pierre Rosay, {\it Pluri-polarity in almost complex structures}, Math. Z., {\bf 265} (2010), 133-149.
   \bibitem{SY} R. Schoen and S-T Yau, {\it Lectures on differential geometry}, International Press, 1994.
   \bibitem{Sikorav} J. C. Sikorav, {\it Some properties of holomorphic curves in almost complex manifolds}, Holomorphic curves in symplectic geometry, 165-189, Progr. Math., {\bf 117},  Birkh\"{a}user, Basel, 1994.
 \bibitem{S1} L. Simon, {\it Lectures on Geometric Measure Theory}, Proc. Centre for Math. Anal. {\bf 3}. Australian National University, Canberra, 1983.
 \bibitem{S2} Y.-T. Siu, {\it Analyticity of sets associated to Lelong numbers and the extension of closed positive currents}, Invent. Math., {\bf 27} (1974), 53-156.
 \bibitem{S3} D. Sullivan, {\it Cycles for the dynamical study of foliated manifolds and complex manifolds}, Invent. Math., {\bf 36} (1976), 225-255.
 \bibitem{Sukhov} A. Sukhov, {\it Regularized maximum of strictly plurisubharmonic functions on an almost complex manifold}, Internat. J. Math., {\bf 24} (2013), 1350097.
 \bibitem{TWZZ} Q. Tan, H. Y. Wang, Y. Zhang and P. Zhu, {\it On cohomology of almost complex 4-manifolds}, J. Geom. Anal., {\bf 25} (2015), 1431-1443.
   \bibitem{TWZ} Q. Tan, H. Y. Wang and J. R. Zhou, {\it A note on the deformations of almost complex structures on closed four-manifolds},
        J. Geom. Anal., {\bf 27} (2017), 2700-2724.
 \bibitem{T1} C. H. Taubes, {\it The Seiberg-Witten and Gromov invariants}, Math. Res. Lett., {\bf 2} (1995), 221-238.
 \bibitem{T2} C. H. Taubes, {\it Tamed to compatible: Symplectic forms via moduli space integration}, J. Symplectic Geom., {\bf 9} (2011), 161-250.
 \bibitem{TWY} V. Tosatti, B. Weinkove and S.-T. Yau, {\it Taming symplectic forms and the Calabi-Yau equation}, Proc. Lond. Math. Soc., {\bf 97} (2008), 401-424.
  \bibitem{Wang} H. Y. Wang, {\it On $J$-anti-invariant cohomology of compact almost complex four-manifolds and applications (in Chinese)},
     Sci. Sin. Math., {\bf 46} (2016), 697-708.
 \bibitem{WZ2} H. Y. Wang and P. Zhu, {\it On a generalized Calabi-Yau equation},  Ann. Inst. Fourier, {\bf 60} (2010), 1595-1615.
 \bibitem{W} B. Weinkove, {\it The Calabi-Yau equation on almost K\"{a}hler four-manifolds}, J. Diff. Geom., {\bf 76} (2007), 317-349.
  \bibitem{X} X. W. Xu, {\it Private communications}, 2018.
 \bibitem{Y} S-T Yau, {\it On the Ricci curvature of a compact K\"{a}hler manifold and the complex Monge-Amp\`{e}re equation}, I, Comm. Pure Appl. Math., {\bf 31} (1978), 339-411.
 \bibitem{Zha} W. Zhang, {\it From Taubes current to almost K\"{a}hler forms}, Math. Ann., {\bf 356} (2013), 969-978.
 \end{thebibliography}
\end{document}